\documentclass{amsart}
\usepackage{amsmath,amssymb,amsthm}
\usepackage{amsfonts,amssymb}

\usepackage{natbib}

\usepackage[dvipsnames]{xcolor}
\usepackage{tikz}
\usepackage{tikz-cd}
\usetikzlibrary{fit,calc,backgrounds}
\tikzset{overcross/.style={double, line width=1.5, white, double=#1, double distance=\knotlinewidth},
    overcross/.default={black},
    knot/.style={line width=\knotlinewidth, baseline=-.5ex}}
\newcommand{\knotlinewidth}{.7pt}

\usepackage{float}
\usepackage{graphicx}
\usepackage{caption}
\usepackage[margin=10pt]{subcaption}
\let\emptyset\varnothing

\usepackage{hyperref}

\usepackage{enumitem}
\usepackage{placeins}

\newtheorem{theorem}{Theorem}[section]
\newtheorem{corollary}[theorem]{Corollary}
\newtheorem{proposition}[theorem]{Proposition}
\newtheorem{lemma}[theorem]{Lemma}

\theoremstyle{definition}
\newtheorem{definition}[theorem]{Definition}

\theoremstyle{remark}
\newtheorem*{remark}{Remark}

\newcommand{\spinc}{\mathrm{Spin}^c}
\newcommand{\ind}{\mathrm{ind}}
\newcommand{\Id}{\mathrm{Id}}
\newcommand{\PD}{\mathrm{PD}}
\newcommand{\hol}{\mathrm{Hol}}
\newcommand{\Kh}{\mathrm{Kh}}

\newcommand{\R}{\mathbb{R}}
\newcommand{\Z}{\mathbb{Z}}

\newcommand{\m}{\mathbf{m}}
\newcommand{\n}{\mathbf{n}}
\newcommand{\bfo}{\mathbf{o}}
\newcommand{\bfk}{\mathbf{k}}

\newcommand{\frakt}{\mathfrak{t}}

\newcommand{\circled}[1]{%
	\tikz[baseline=(char.base)]{
		\node[draw,circle,inner sep=1pt] (char) {#1};}}

\usepackage{geometry}

\title{Plane Floer homology and the odd Khovanov homology of 2-knots}

\author{Dean Spyropoulos}
\address{Department of Mathematical Sciences, New Mexico State University, Las Cruces, NM 88003}
\email{dspyro@nmsu.edu}

\author{Rithwik Susheel Vidyarthi}
\address{Department of Mathematics, Michigan State University, East Lansing, MI 48824}
\email{vidyart2@msu.edu}

\author{Chen Zhang}
\address{Simons Center for Geometry and Physics, State University of New York, Stony Brook, NY 11794}
\email{czhang@scgp.stonybrook.edu}

\begin{document}

\begin{abstract}
We prove a conjecture of Migdail and Wehrli regarding the odd Khovanov cobordism maps associated to knotted spheres. Our key tool is Daemi's plane Floer homology, which we use in place of a Lee deformation. Continuing the analogy with Lee homology, we see this work as a potential first step toward a genuinely functorial model for odd Khovanov homology.
\end{abstract}

\maketitle

\section{Introduction}

Khovanov homology \cite{10.1215/S0012-7094-00-10131-7} is a homological invariant of links which categorifies the Jones polynomial. Early in its development, Jacobsson \cite{MR2113903}, Bar-Natan \cite{MR2174270}, and Khovanov \cite{MR2171235} each gave separate proofs that Khovanov homology is functorial with respect to smooth link cobordisms, but only up to sign. The sign defect in functoriality was corrected by several authors \cite{MR2443094, MR2496052, MR2647055, MR4598808, MR4376719}, resulting in genuinely functorial models for Khovanov homology.

Shortly after the discovery of Khovanov homology, Ozsv\'ath and Szab\'o \cite{ozsvath2003heegaard} proved the existence of a spectral sequence from the (reduced) Khovanov homology of a link in $\mathbb{Z}/2\mathbb{Z}$-coefficients to the Heegaard Floer homology of the branched double cover:
\[
\widetilde{\mathrm{Kh}}(L; \mathbb{Z}/2\mathbb{Z}) \Rightarrow \widehat{\mathrm{HF}}(\Sigma(-L); \mathbb{Z}/2\mathbb{Z}).
\]
However, this statement comes with a warning: it is a consequence of an encompassing \emph{link surgeries spectral sequence}, occurring over the integers and, in this more general setting, the $E_2$ page is evidently distinct from $\widetilde{\mathrm{Kh}}(L; \mathbb{Z})$.

Ozsv\'ath, Rasmussen, and Szab\'o \cite{MR3071132} provided a new candidate for the $E_2$ page called \emph{odd Khovanov homology}; the original theory has been retroactively declared \emph{even Khovanov homology}. The two theories necessarily agree over $\mathbb{Z}/2\mathbb{Z}$ coefficients, but are distinct otherwise (see, \textit{e.g.}, \cite{MR2777025}) and both categorify the Jones polynomial.

Recently, Migdail and Wehrli \cite{migdail2024functoriality} proved the following.

\begin{theorem}[\cite{migdail2024functoriality}, Corollary 1]
\label{thm:MW24}
Odd Khovanov homology is functorial under smooth link cobordisms in $\mathbb{R}^3\times I$ up to sign.
\end{theorem}

We refer to Migdail and Wehrli's extension of odd Khovanov homology to a projective 2-functor as the \emph{odd Khovanov 2-functor}. See \cite{spyro25} for another proof generalizing this result to tangle cobordisms. At the time of writing, it is unknown whether odd Khovanov homology possesses a genuinely functorial model, although Schelstraete and Vaz have provided a potential candidate in \cite{MR4190457} and \cite{schelstraete2023odd}.

Given Theorem \ref{thm:MW24}, it makes sense to study cobordism maps on odd Khovanov homology. The purpose of this paper is to study the cobordism maps induced by knotted 2-spheres (\emph{i.e.}, 2-knots).

The projective TQFT (see \cite{MR3071132}) underlying odd Khovanov homology is defined as follows. Let $\mathfrak{C}$ denote the category whose objects are closed 1-manifolds and whose morphisms are compact, orientable cobordisms between them. Define a projective functor
\[
\mathcal{F}: \mathfrak{C} \to \mathbb{Z}\mathrm{Mod}^{\mathrm{gr}}
\]
assigning to objects $S \in \mathrm{Ob}(\mathfrak{C})$ a graded $\mathbb{Z}$-module $\mathcal{F}(S)$. Let $V(S)$ denote the free abelian group generated by the components of $S$. Set
\[
\mathcal{F}(S) := \Lambda^*(V(S)).
\]
We fix $\mathrm{gr}(1) = 1$ and $\mathrm{gr}(a_i) = -1$, where $a_i$ represents a component of $S$, and extend to all elements linearly. Morphisms in $\mathfrak{C}$ decompose as compositions of four elementary cobordisms, corresponding to 2-dimensional 0-, 1- and 2-handle attachments. Throughout this paper, we read cobordisms from bottom to top. One-handles are realized either by merges or splits. For merges, we define
\[
\mathcal{F}\left(
\tikz[baseline={([yshift=-.5ex]current bounding box.center)}, scale=.5]{
	\draw[dashed, knot] (0,0) .. controls (0,0.25) and (1, 0.25) .. (1,0);
	\draw[knot] (0,0) .. controls (0,-0.25) and (1, -0.25) .. (1,0);
	\draw[knot] (1,0) .. controls (1,1) and (2,1) .. (2,0);
	\draw[dashed, knot] (2,0) .. controls (2, 0.25) and (3, 0.25) .. (3,0);
	\draw[knot] (2,0) .. controls (2, -0.25) and (3, -0.25) .. (3,0);
	\draw[knot] (1,2) .. controls (1,2.25) and (2,2.25) .. (2,2);
	\draw[knot] (1,2) .. controls (1,1.75) and (2,1.75) .. (2,2);
	\draw[knot] (0,0) .. controls (0,0.75) and (1, 1.25) .. (1,2);
	\draw[knot] (3,0) .. controls (3, 0.75) and (2, 1.25) .. (2,2);
	\node at (0.5, -0.5) {$a_i$};
	\node at (2.5, -0.5) {$a_{i+1}$};
	\node at (1.5, 2.5) {$a$};
}
\right)
: \Lambda^*(V(S_1)) \to \Lambda^*\left(\frac{V(S_1)}{(a_i - a_{i+1})}\right) \xrightarrow{\cong} \Lambda^* (V(S_2))
\]
where the first map is induced by the projection $V(S_1) \twoheadrightarrow V(S_1)/(a_i - a_{i+1})$. Splits are defined only up to sign: we define
\[
\mathcal{F}\left(
\tikz[baseline={([yshift=-.5ex]current bounding box.center)}, scale=.5]{
	\draw[knot]  (1,2) .. controls (1,3) and (0,3) .. (0,4);
	\draw [knot] (2,2) .. controls (2,3) and (3,3) .. (3,4);
	\draw[knot] (1,4) .. controls (1,3) and (2,3) .. (2,4);
	\draw[knot] (0,4) .. controls (0,3.75) and (1,3.75) .. (1,4);
	\draw[knot] (0,4) .. controls (0,4.25) and (1,4.25) .. (1,4);
	\draw[knot] (2,4) .. controls (2,3.75) and (3,3.75) .. (3,4);
	\draw[knot] (2,4) .. controls (2,4.25) and (3,4.25) .. (3,4);
	\draw[knot] (1,2) .. controls (1,1.75) and (2,1.75) .. (2,2);
	\draw[dashed, knot] (1,2) .. controls (1,2.25) and (2,2.25) .. (2,2);
    \node at (0.5, 4.5) {$a_i$};
    \node at (2.5, 4.5) {$a_{i+1}$};
    \node at (1.5, 1.5) {$a$};
}
\right)
: \Lambda^*(V(S_1)) \xrightarrow{\cong} (a_i - a_{i+1}) \wedge \Lambda^* V(S_2) \hookrightarrow \Lambda^*(V(S_2))
\]
where the second map is induced by inclusion. This map sends $1 \in \Lambda^*(V(S_1))$ to $\pm(a_i - a_{i+1}) \in \Lambda^*(V(S_2))$. We specify the sign diagrammatically by the following convention. 
\[
\tikz[baseline={([yshift=-.5ex]current bounding box.center)}, scale=.5]{
	\draw[knot]   (1,2) .. controls (1,3) and (0,3) .. (0,4);
	\draw[knot]   (2,2) .. controls (2,3) and (3,3) .. (3,4);
	\draw[knot]  (1,4) .. controls (1,3) and (2,3) .. (2,4);
	\draw[knot]  (0,4) .. controls (0,3.75) and (1,3.75) .. (1,4);
	\draw[knot]  (0,4) .. controls (0,4.25) and (1,4.25) .. (1,4);
	\draw[knot]  (2,4) .. controls (2,3.75) and (3,3.75) .. (3,4);
	\draw[knot]  (2,4) .. controls (2,4.25) and (3,4.25) .. (3,4);
	\draw[knot]  (1,2) .. controls (1,1.75) and (2,1.75) .. (2,2);
	\draw[knot, dashed] (1,2) .. controls (1,2.25) and (2,2.25) .. (2,2);
    \draw[stealth-, red, thick](1.8,3.7) -- (1.2,2.8);
    \node at (0.5, 4.5) {$a_i$};
    \node at (2.5, 4.5) {$a_{i+1}$};
    \node at (1.5, 1.5) {$a$};
    \node at (1.5, 0) {$(a_i - a_{i+1})\wedge\cdots$}
}
\hspace{3cm}
\tikz[baseline={([yshift=-.5ex]current bounding box.center)}, scale=.5]{
	\draw[knot]   (1,2) .. controls (1,3) and (0,3) .. (0,4);
	\draw[knot]   (2,2) .. controls (2,3) and (3,3) .. (3,4);
	\draw[knot]  (1,4) .. controls (1,3) and (2,3) .. (2,4);
	\draw[knot]  (0,4) .. controls (0,3.75) and (1,3.75) .. (1,4);
	\draw[knot]  (0,4) .. controls (0,4.25) and (1,4.25) .. (1,4);
	\draw[knot]  (2,4) .. controls (2,3.75) and (3,3.75) .. (3,4);
	\draw[knot]  (2,4) .. controls (2,4.25) and (3,4.25) .. (3,4);
	\draw[knot]  (1,2) .. controls (1,1.75) and (2,1.75) .. (2,2);
	\draw[knot, dashed] (1,2) .. controls (1,2.25) and (2,2.25) .. (2,2);
    \draw[-stealth, red, thick](1.8,3.7) -- (1.2,2.8);
    \node at (0.5, 4.5) {$a_i$};
    \node at (2.5, 4.5) {$a_{i+1}$};
    \node at (1.5, 1.5) {$a$};
    \node at (1.5, 0) {$(a_{i+1} - a_{i}) \wedge \cdots$}
}
\]
The 0-handle attachments are referred to as births; we define
\[
\mathcal{F}\left(
\tikz[baseline={([yshift=-.5ex]current bounding box.center)}, scale=.5]{
	\draw[knot] (0,0) .. controls (0,0.25) and (1, 0.25) .. (1,0);
	\draw[knot] (0,0) .. controls (0,-0.25) and (1,-0.25) .. (1,0);
	\draw[knot] (0,0) .. controls (0,-1) and (1,-1) .. (1,0);
}
\right)
: \Lambda^*(V(S_1)) \to \Lambda^*(V(S_2))
\]
as the map induced by the inclusion $V(S_1)) \hookrightarrow V(S_2))$. Similarly, for 2-handle attachments, referred to as deaths, we define
\[
\mathcal{F}\left(
\tikz[baseline={([yshift=-.5ex]current bounding box.center)}, scale=.5]{
	\draw[dashed, knot] (0,0) .. controls (0,0.25) and (1, 0.25) .. (1,0);
	\draw[knot] (0,0) .. controls (0,-0.25) and (1,-0.25) .. (1,0);
	\draw[knot] (0,0) .. controls (0,1) and (1,1) .. (1,0);
	\node at (0.5, -0.5) {$a$};
}
\right)
: \Lambda^*(V(S_1)) \to \Lambda^*\left(\frac{V(S_1)}{(a)}\right) \xrightarrow{\cong} \Lambda^* (V(S_2))
\]
where the first map is induced by the projection $V(S_1) \twoheadrightarrow V(S_1)/ (a)$.

Given a link $L$, choose a generic planar projection $D$ onto $\mathbb{R}^2$ and take its hypercube of resolutions with vertices in $\{0,1\}^n$ according to the following schematic.
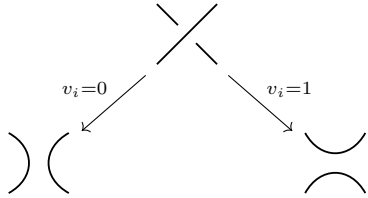
\[
\begin{tikzcd}
& \tikz[baseline=1.6ex, scale = .8]{
\draw[knot] (0,0) -- (1,1);
\draw[knot] (1,0) -- (.65,.35);
\draw[knot] (.35,.65) -- (0,1);
} \arrow[dl, "v_i = 0"'] \arrow[dr, "v_i = 1"] & 
\\
\tikz[baseline=1.6ex, scale = .8]{
\draw[knot] (0,0) .. controls (0.45, 0.25) and (0.45, 0.75) .. (0,1);
\draw[knot] (1,0) .. controls (0.55, 0.25) and (0.55, 0.75) .. (1,1);
} & & \tikz[baseline=1.6ex, scale = .8]{
\draw[knot] (0,1) .. controls (0.25, 0.55) and (0.75, 0.55) .. (1,1);
\draw[knot] (0,0) .. controls (0.25, 0.45) and (0.75, 0.45) .. (1,0);
}
\end{tikzcd}
\]
Up to sign, the \emph{odd Khovanov complex} $\mathcal{C}(D)$ of a link diagram is obtained by applying $\mathcal{F}$ to the vertices $D_v$ and edges of this hypercube. Hereafter, $\mathrm{Kh}(L)$ always denotes the homology of the odd Khovanov complex. 

There are a few ways to produce a reduced odd Khovanov homology \cite[Section 4]{MR3071132}---we work with the following version. Let $D$ be an $n$-crossing diagram for $L$. For any $v\in \{0,1\}^n$, we let $\Lambda_\circ^* (V(D_v))$ be the subalgebra generated by the kernel of the map $V(D_v) \to \mathbb{Z}$ defined by \[\sum_i n_i a_i \mapsto \sum_i n_i.\] Define the \emph{reduced odd Khovanov complex} be the complex defined as with $\mathcal{C}(D)$, but assigning $\Lambda_\circ^*(V(D_v))$ to each resolution in place of $\Lambda^* (V(D_v))$. The corresponding submodule of $\mathcal{C}(D)$ turns out to be a subcomplex. We denote its homology by $\widetilde{\mathrm{Kh}}(L)$. Importantly, one can show that this version of reduced odd Khovanov homology is preserved under smooth link cobordisms.

\subsection{Statement of results}

First, let $\mathcal{S}$ denote a (potentially knotted) closed surface smoothly embedded in $S^4$ viewed as a cobordism from the empty link to itself. Note that $\mathcal{F}(\emptyset) = \mathbb{Z}$ in bidegree $(0,0)$. For both even and odd Khovanov homology, the degree of the induced map $\mathrm{Kh}(\mathcal{S}): \mathbb{Z} \to \mathbb{Z}$ is $(0,\chi(\mathcal{S}))$; thus, $\Kh(\mathcal{S}) = 0$ unless $\chi(\mathcal{S}) = 0$. Therefore, in the closed case, the only potentially interesting surfaces are tori and Klein bottles. 

However, we can still derive an invariant of 2-knots by ``popping'' the 2-knot $\xi \subset S^4$: removing a small neighborhood of a generic point $p \in \xi$ results in a disk $\xi^*$ which induces a map $\mathcal{F}(\emptyset) \to \mathcal{F}(\mathcal{U})$, for $\mathcal{U}$ the unknot. Denote this map by
\[
\Kh(\xi^*): \mathbb{Z} \to \mathbb{Z}\{1, v\}
\]
where $\mathcal{F}(\mathcal{U}) \cong \mathbb{Z}\{1, v\}$ denotes $\mathbb{Z}$-linear combinations in $v$ and $1$. It follows that $\Kh(\xi^*)(1) = n1$ for some $n\in \mathbb{Z}$. Let $n(\xi)$ denote the absolute value of this integer. By Theorem \ref{thm:MW24}, $n(-)$ is an invariant of 2-knots.

This construction is possible in even Khovanov homology, but it yields a trivial invariant (see \cite{rasmussen2005khovanov, MR2240683}). In contrast, we will prove the following.

\begin{theorem}
\label{thm:main}
For any smooth 2-knot $\xi \subset S^4$, $n(\xi)$ is the number of $\mathrm{spin}^c$ structures on the branched double cover of $S^4$, branched along $\xi$.
\end{theorem}

Theorem \ref{thm:main} answers Conjecture 1 of \cite{migdail2024functoriality} in the affirmative. To see this, recall that the number of spin$^c$-structures on the branched double cover of $S^4$ branched along a surface is equal to the cardinality of its second cohomology group. Applying Poincar\'e duality and the universal coefficients theorem, we see that this is the order of the first homology, as in the cited article. Conversely, the map on odd Khovanov homology associated to a torus is zero, independent of embedding; we prove this as Corollary \ref{cor:extracredit}. In upcoming work \cite{mw1, mw2}, Migdail and Wehrli will reprove their conjecture for particular families of 2-knots---namely, ribbon 2-knots and even-twist spun knots---using techniques distinct from ours.

\subsection{Outline of proof}
\label{ss:outlineofproof}

To prove Theorem \ref{thm:main}, our main tool will be Daemi's \emph{plane Floer homology} \cite{daemi2015abelian}. The utility of (oriented) plane Floer homology, denoted $\widetilde{\mathrm{PFH}}$, is that it is (a) relatively easy to compute and (b) the $E_\infty$ page of a spectral sequence starting at odd Khovanov homology. We denote by $\widetilde{\Lambda}$ a particular ring of characteristic zero, which we ignore for now. We remark that we could have likely used framed instanton homology in place of plane Floer homology, in light of work by Scaduto \cite{MR3394316}, but plane Floer homology seems more natural given its similarities with Lee homology.

\begin{theorem}[See \cite{daemi2015abelian}, Theorem 4]
\label{thm:daemi}
There is a spectral sequence $\{\widetilde{\mathrm{PKE}}_r(L)\}$ with coefficients in $\widetilde{\Lambda}$ that abuts to $\widetilde{\mathrm{PFH}}(\Sigma(L))$. The second page of this spectral sequence is isomorphic to the reduced odd Khovanov homology of the mirror image of $L$ with coefficients in $\widetilde{\Lambda}$.
\end{theorem}

The spectral sequence is induced by a filtered complex $\widetilde{\mathrm{PKC}}(L)$. Our proof of Theorem \ref{thm:main} consists of two steps. First, we show that the cobordism maps are natural with respect to the spectral sequence. Let $\mathcal{E}_2$ denote the functor taking filtered complexes to the $E_2$ page of its associated spectral sequence.

\begin{proposition}
\label{prop:step1}
Suppose that $W$ is an elementary link cobordism; that is, $W$ is either a Reidemeister move or a Morse move. Then $\widetilde{\mathrm{PKC}}(W)$ is a filtered map of filtered complexes. Moreover, $\widetilde{\mathrm{PKE}}_2(W) = \mathcal{E}_2(\widetilde{\mathrm{PKC}}(W))$ agrees with the reduced odd Khovanov 2-functor of \cite{migdail2024functoriality}, up to unit of $\widetilde{\Lambda}$.
\end{proposition}

Using Proposition \ref{prop:step1}, we can conclude from Daemi's work that the maps on odd Khovanov homology and plane Floer homology are related to each other by the spectral sequence of Theorem \ref{thm:daemi}. In particular, if the spectral sequence collapses, these maps are equivalent. For a smooth 2-knot $\xi \subset S^4$, denote by $\xi^*$ the knotted disk obtained by removing a small neighborhood of any point $p\in \xi$, and view $\xi^*$ as a cobordism $\emptyset\to \mathcal{U}$. Since the $E_2$ page is \emph{reduced} odd Khovanov homology, the cobordism we consider is $\xi^* \sqcup C$ for $C: \mathcal{U} \to \mathcal{U}$ an unknotted cylinder. We note that $\widetilde{\mathrm{PKC}}(\xi^* \sqcup C)$ is a filtered map of filtered complexes by Proposition \ref{prop:step1}. Consider the following diagram.

\[
\begin{tikzcd}[row sep = large]
\widetilde{\Lambda} \otimes_\mathbb{Z} \widetilde{\mathrm{Kh}}(\mathcal{U}) \cong \widetilde{\Lambda} \arrow[d, "\widetilde{\mathrm{Kh}}\text{(}\xi^* \sqcup C\text{)}"'] \arrow[r, Rightarrow] & \widetilde{\mathrm{PFH}}(S^3) \cong \widetilde{\Lambda} \arrow[d, "\widetilde{\mathrm{PFH}}\text{(}\Sigma\text{(}\xi^* \sqcup C\text{)}\text{)}"]  \\
\widetilde{\Lambda} \otimes_\mathbb{Z} \widetilde{\mathrm{Kh}}(\mathcal{U} \sqcup \mathcal{U}) \arrow[r, Rightarrow] & \widetilde{\mathrm{PFH}}(S^1 \times S^2)
\end{tikzcd}
\]

Since the spectral sequence collapses in the cases above, Proposition \ref{prop:step1} tells us that in order to compute $\widetilde{\mathrm{Kh}}(\xi^*\sqcup C)$, and thus $n(\xi)$, it suffices to compute $\widetilde{\mathrm{PFH}}(\Sigma(\xi^* \sqcup C))$. 

\begin{proposition}
\label{prop:step2}
Assume $\xi: S^2 \hookrightarrow S^4$ is a smoothly embedded 2-knot in $S^4$. The map
\[
\widetilde{\mathrm{PFH}}(\Sigma(\xi^* \sqcup C)): \widetilde{\mathrm{PFH}}(S^3) \cong \widetilde{\Lambda} \to \widetilde{\mathrm{PFH}}(S^1 \times S^2) \cong \widetilde{\Lambda} \{1,v\}
\]
is given by $1 \mapsto \pm \left| H^2(\Sigma(\xi))\right| \cdot 1$ where $\Sigma(\xi)$ is the branched double cover of $S^4$ branched along $\xi$.
\end{proposition}

In summary, we have reduced the proof of Theorem \ref{thm:main} to the proofs of Propositions  \ref{prop:step1} and  \ref{prop:step2}.

\begin{remark}
In \cite{migdail2024functoriality}, Migdail and Wehrli proved that odd Khovanov homology fails to be functorial up to sign for general cobordisms in $S^3 \times I$. In particular, odd Khovanov homology is not invariant under the sweep-around move \cite{MR4562565}; this is exhibited by example in \cite[Section 5]{migdail2024functoriality}. However, the cobordisms relevant to the definition of $n(-)$ are disks properly embedded in $D^4 \cong S^4 - \nu(p)$ for $p \in \xi \subset S^4$. Viewing $D^4$ as the cone of $S^3$, the isotopies between the sweep-around cobordism and the identity cobordism can be pushed away from the point at infinity, so we need not consider them in our proof of Theorem \ref{thm:main}.
\end{remark}

\subsection{Outline of paper}

Section \ref{s:pfh} is devoted to giving a brief introduction to plane Floer homology. With Proposition \ref{prop:daemihammer} from Section \ref{s:pfh}, we can already prove Proposition \ref{prop:step2}. In Section \ref{s:ss}, we restrict to branched double covers of links and detail the spectral sequence relating reduced odd Khovanov homology and plane Floer homology. Finally, Section \ref{s:rmoves} is devoted to proving Proposition \ref{prop:step1}, and is the technical heart of our paper. 

\subsection*{Acknowledgments}

We would like to thank Aliakbar Daemi, Juan Muñoz Echániz, Robert Lipshitz, and Stephan Wehrli for their interest, as well as Minh Lam Nguyen and Matthew Stoffregen for supplementing key insight during the development of this article. This project was supported by the grant from the National Science Foundation: RTG: Algebraic and Geometric Topology at Michigan State (DMS 2135960).

\section{Plane Floer homology}
\label{s:pfh}

In this section, we give a brief overview of plane Floer homology; see \cite[Sections 1,\,2]{daemi2015abelian} for more details. We include descriptions of important cobordism maps.

\subsection{Definition of plane Floer homology and cobordism maps}

Plane Floer homology is a functor from the category of 3-manifolds and cobordisms between them to the category of $\mathbb{Z}$-graded $\tilde{\Lambda}$-modules, where
\[
\tilde{\Lambda} := \left\{ \sum_{i=1}^\infty a_i u^{q_i} : a_i \in \mathbb{Z}, q_i \in \mathbb{Q}, ~\text{and}~ \lim_{i\to \infty} q_i = \infty \right\}.
\]

Let $Y$ be a connected, closed Riemannian 3-manifold. A \emph{$Spin^c$ structure} $\mathfrak{t}$ on $Y$ is a principal $\mathrm{Spin}^{c}(3)$-bundle $P$ such that the $SO(3)$-bundle induced by the adjoint action $\mathrm{ad}: \mathrm{Spin}^c(3) \cong U(2) \to SO(3)$ is identified with the framed bundle of $TY$. A smooth connection $\underline{B}$ on $\mathfrak{t}$ is called a \emph{$Spin^c$ connection} if the induced connection on $ad(\mathfrak{t})$ is the Levi-Civita connection. The determinant map $\mathrm{det}: U(2) \to U(1)$ induces a complex line bundle $L_\mathfrak{t}$, which is called the \emph{determinant line bundle}. The $\mathrm{Spin}^c$ connection $\underline{B}$ on $\mathfrak{t}$ also induces a connection of $L_\mathfrak{t}$, called the \emph{central part} of $\underline{B}$; it is denoted by $B$. Since we fix the induced connection on $TY$ to be the Levi-Civita connection, a $\mathrm{Spin}^c$ connection $\underline{B}$ is determined by its central part. 

Plane Floer homology studies the space of $\mathrm{Spin}^c$ connections with flat central parts, modulo the action of the gauge group. Note that $L_{\frakt}$ admits flat connections if and only if $c_{1}(L_{\frakt})$ is torsion. Given a torsion $\spinc$ structure $\frakt$, let $A_f(Y, \mathfrak{t})$ denote the $\spinc$ connections with flat centrals parts, which is an affine space modeled on the space of closed 1-forms. Let $\mathcal{G}(Y)$ denote the gauge group of all automorphisms of the $\mathrm{Spin}^c$ structure $\mathfrak{t}$; that is, smooth automorphisms of $\mathfrak{t}$ as a principal $U(2)$-bundle which acts trivially on $TY$. $\mathcal{G}(Y)$ can be identified with the space of smooth maps $Y \to S^1$, which acts on the space of $\mathrm{Spin}^c$ connections by pulling back. Any element $u: Y \to S^1\in  \mathcal{G}(Y)$ acts on $\underline{B}$ by sending it to a connection with central part $B - 2u^{-1} du$. Define
\[
\mathcal{R}(Y, \mathfrak{t}) = A_f(Y, \mathfrak{t}) \big/ \mathcal{G}(Y).
\]
Fixing a basepoint in $A_f(Y, \mathfrak{t})$, the 1-form action identifies $\mathcal{R}(Y, \mathfrak{t})$ with a rescaling of the Jacobian torus of $Y$,
\[
J(Y) := H^1(Y; \mathbb{R}) \big/ 2 H^1(Y; \mathbb{Z}).
\]
We also define $\mathcal{R}(Y) := \bigcup_{\mathfrak{t}} \mathcal{R}(Y, \mathfrak{t}).$

Let $W^{\circ}: Y_0 \to Y_1$ be a cobordism. We equip it with a Riemannian metric such that its restriction to the collar neighborhood of the boundary is isometric to a product metric. Then, glue cylindrical ends to $W^\circ$ to produce a non-compact and complete Riemannian manifold, denoted $W$. Let $\mathfrak{s}$ be a $\spinc$ structure on $W$. The space of (weighted Sobolev) $\mathrm{Spin}^c$ connections, denoted $\mathcal{A}(W, \mathfrak{s})$, vanishes if the ends $\mathfrak{s}|_{Y_i} = \mathfrak{t}_i$ are non-torsion; therefore, from now on we assume that $\mathfrak{t}$ is torsion.

For a cobordism $(W, \mathfrak{s}): (Y_0, \mathfrak{t}_0) \to (Y_1, \mathfrak{t}_1)$, one can define the \textit{configuration space of (weighted Sobolev) $\mathrm{Spin}^c$ connections} on $\mathfrak{s}$, denoted $\mathcal{B}(W, \mathfrak{s})$, similar to the definition of $\mathcal{R}(Y)$. Namely:
\[
\mathcal{B}(W, \mathfrak{s}):= \mathcal{A}(W, \mathfrak{s}) \big/ \mathcal{G}(W).
\]
There are restriction maps $r_i: \mathcal{B}(W, \mathfrak{s}) \to \mathcal{R}(Y_i, \mathfrak{t}_i)$ for $i=0, 1$. Picking a generic $\nu \in L_{k-1, \delta}^2 (W, \Lambda^2)$ (called a \emph{perturbation term}), we can define the moduli space
\[
M_\nu (W, \mathfrak{s}) = \{[\underline{A}]\in \mathcal{B}(W, \mathfrak{s}) : F_c^+(\underline{A}) = i \nu^+\}
\] 
where $F_c^+(\underline{A})$ is the self dual part of the curvature of the central part of $\underline{A}$. We define $M_\nu(W) = \bigcup_\mathfrak{s}(W, \mathfrak{s})$. Fixing an element in $M_\nu (W, \mathfrak{s})$, we can also identify it with the Jacobian torus.

\begin{proposition}[\cite{daemi2015abelian}, Lemma 1.6]
If $b^+(W) >0$, $\nu$ can be chosen so that the moduli space $M_\nu(W, \mathfrak{s})$ is empty. If $b^+(W) = 0$, then $M_\nu(W, \mathfrak{s})$ is isomorphic to $J(W) := H^1(W; \mathbb{R})\big/ 2H^1(W; \mathbb{Z})$ for any choice of $\nu$.
\end{proposition}

The \emph{(oriented) plane Floer homology} of a 3-manifold $Y$, denoted $\widetilde{\mathrm{PFH}}(Y)$, is equal to the cohomology of $\mathcal{R}(Y)$; in particular,
\[
\widetilde{\mathrm{PFH}}(Y) \cong H^*(J(Y); \tilde{\Lambda})^{\oplus \left| \mathrm{Tor}(H_1(Y; \mathbb{Z}))\right|} \cong \tilde{\Lambda}^{2^{b_1(Y)} \cdot \left| \mathrm{Tor}(H_1(Y; \mathbb{Z})) \right|}.
\]
To lift $\widetilde{\mathrm{PFH}}(Y)$ to the level of chain complexes, let $f$ be a Morse-Smale function on $\mathcal{R}(Y)$ and let $(\widetilde{\mathrm{PFC}}(Y), d_p)$ denote the corresponding Morse complex, called the \emph{(oriented) plane Floer complex} of $Y$. We equip $\widetilde{\mathrm{PFC}}(Y)$ with a grading on critical points $\alpha$ given by
\[
\deg_p(\alpha) := 2 \cdot \mathrm{ind}(\alpha) - b_1(Y).
\]

As a primary example, note that $\widetilde{\mathrm{PFC}}(S^3)$ is the Morse complex of $H^1(S^3; \mathbb{R}) \big/ 2H^1(S^3; \mathbb{Z}) \simeq *$, so $\widetilde{\mathrm{PFC}}(S^3) \cong \tilde{\Lambda}$.

\begin{remark}
To define the cobordism maps, one needs to work with moduli spaces parametrized by a family of metrics and fatten these moduli spaces in order to ensure compactness. However, since we only consider a single cobordism map arising from a family of metrics in this paper, we provide this computation in Section \ref{family} and omit a more general definition. The interested reader can consult \cite[\S 1.3]{daemi2015abelian}.
\end{remark}

Over a single metric, the cobordism maps are defined as follows. Let $\alpha$ and $\beta$ be critical points of Morse functions $f_0$ and $f_1$ on $\mathcal{R}(Y_0)$ and $\mathcal{R}(Y_1)$ respectively. Let $U_\alpha$ and $S_\beta$ denote the unstable manifold of $\alpha$ and the stable manifold of $\beta$ respectively. Let $\nu$ be a perturbation term and consider the moduli space
\[
M_\nu(\alpha,\beta,\mathfrak{s})=\{z\in M_\nu(W,\mathfrak{s}):r_0(z)\in U_\alpha ,r_1(z)\in S_\beta\}
\]
where $r_0$ and $r_1$ are the restriction maps. One can choose $\nu$ so that $M_\nu(\alpha, \beta, \mathfrak{s})$ is transversely cut-out. We compute
\begin{align*}
	\dim (M_\nu(\alpha,\beta,\mathfrak{s})) &= \dim (M_\nu(W, \mathfrak{s}))+ \dim(U_\alpha \times S_\beta)- \dim\left(\mathcal{R}(Y_0, \mathfrak{t}_0) \times \mathcal{R}(Y_1, \mathfrak{t}_1) \right)\\
	&=\dim (M_\nu(W, \mathfrak{s}))-\dim (\mathcal{R}(Y_0, \mathfrak{t}_0)) + \ind(\alpha) - \ind (\beta) \\
	&= \dim(J(W))- \dim(J(Y_0))+ \ind (\alpha) - \ind (\beta)\\
	&= b_1(W)-b_1(Y_0)+ \ind (\alpha) -\ind (\beta).
\end{align*}
The map on the plane Floer complex induced by the cobordism $W$ is given by counting the $0$-dimensional moduli space:
 \[
 f^\nu ({\alpha}) =\sum_{\beta \in Crit(f_1)}\left( \sum_{p \in {M_\nu}(\alpha,\beta,\mathfrak{s})} \epsilon_{p}u^{e(p)}\right)\beta,
 \] where $\epsilon_{p}\in\{1,-1\}$ is the orientation on $p$, which is determined by a choice of homology orientation on the triple $(W, Y_0, Y_1)$.
Equivalently, the cobordism map can also be computed from the Jacobian torus as follows.
\begin{proposition}[\cite{daemi2015abelian}, Proposition 1.22]
\label{prop:daemihammer}
Suppose $W: Y_0 \to Y_1$. Then
\begin{enumerate}
	\item if $b^+(W) > 0$, $\widetilde{\mathrm{PFH}}(W) = 0$;
	\item if $b^+(W) = 0$, then
	\[
		\widetilde{\mathrm{PFH}}(W) = \sum_{\mathfrak{s}: \mathfrak{t}_0 \to \mathfrak{t}_1} u^{-c_1(\mathfrak{s}) \cdot c_1(\mathfrak{s})} D(W; Y_0, Y_1)(\mathfrak{t}_0, \mathfrak{t}_1)
	\]
\end{enumerate}
where $\mathfrak{s}: \mathfrak{t}_0 \to \mathfrak{t}_1$ denotes a $\mathrm{Spin}^c$ structure $\mathfrak{s}$ on $W$, with $\mathfrak{s}|_{Y_0} = \mathfrak{t}_0$ and $\mathfrak{s}|_{Y_1} = \mathfrak{t}_1$ being torsion $\mathrm{Spin}^c$ structures.
\end{proposition}

We briefly describe the map $D(W; Y_0, Y_1)(\mathfrak{t}_0, \mathfrak{t}_1)$; see \cite[\S 1.4]{daemi2015abelian}. First, identify the restriction maps $(r_0, r_1)$ with the maps $(i_0, i_1): J(W) \to J(Y_0) \times J(Y_1)$ induced by the inclusion $Y_i \hookrightarrow W$. In particular, we can identify the diagram
\[
\begin{tikzcd}[row sep = large]
& M_\nu(W, \mathfrak{s}) \arrow[dl, "r_0"'] \arrow[dr, "r_1"] & \\
\mathcal{R}(Y_0, \mathfrak{t}_0)& & \mathcal{R}(Y_1, \mathfrak{t}_1)
\end{tikzcd}
\qquad \text{with} \qquad
\begin{tikzcd}[row sep = large]
& J(W) \arrow[dl, "i_0"'] \arrow[dr, "i_1"] & \\
J(Y_0) & & J(Y_1)
\end{tikzcd}
\]
From the latter correspondence, we can define the ``pull-up-push-down'' map 
\[
F_{J(W)} = \mathrm{PD}_{J(Y_1)}^{-1} \circ (i_1)_* \circ \mathrm{PD}_{J(W)} \circ (i_0)^*: H^*(J(Y_0)) \to H^*(J(Y_1))
\]
where $\mathrm{PD}_X: H^*(X; \mathbb{Z}) \xrightarrow{\sim} H_*(X; \mathbb{Z})$ is the Poincar\'e duality isomorphism. Notice that the Poincar\'e duality isomorphism requires a choice of orientation on the Jacobian torus, which corresponds to a choice of homology orientation on the triple $(W, Y_0, Y_1)$. We define $D(W; Y_0, Y_1)(\mathfrak{t}_0, \mathfrak{t}_1)$ to be the map 
\[
F_{J(W)} \otimes \mathrm{id}_{\tilde{\Lambda}}: H^*(\mathcal{R}(Y_0, \mathfrak{t}_0); \tilde{\Lambda}) \to H^*(\mathcal{R}(Y_1, \mathfrak{t}_1); \tilde{\Lambda}).
\]
\subsubsection{Plane Floer homology of 2-knots}

Since $H_2(S^3 \times I) = 0$, a surface $\mathcal{S}$ in $S^3\times I$ determines a unique 2-fold branched cover $\Sigma(\mathcal{S})$ (see \cite[Corollary 2.10]{MR1772842} and \cite{MR4821360}). We are interested in branched double covers of $S^3 \times I$ branched over $\xi^*\sqcup C$ as described in \S \ref{ss:outlineofproof}. 

\begin{proof}[Proof of Proposition \ref{prop:step2}]
For any 2-knot $\xi: S^2 \hookrightarrow S^4$, $\Sigma(\xi^* \sqcup C)$ is diffeomorphic to $\Sigma(\xi^*) \natural (S^1 \times D^3) \setminus D^4$, so $b_2(\Sigma(\xi^* \sqcup C)) = 0$. View $\Sigma(\xi^* \sqcup C)$ as a cobordism from $S^3$ to $S^1 \times S^2$. Since both $S^3$ and $S^1 \times S^2$ have a unique torsion $\mathrm{Spin}^c$ structure (denoted $\mathfrak{t}_0$ and $\mathfrak{t}_1$ respectively), Proposition \ref{prop:daemihammer} tells us that
	\[
	\widetilde{\mathrm{PFH}}(\Sigma(\xi^* \sqcup C)) = \sum_{\mathfrak{s}} D(\Sigma(\xi^*\sqcup C); S^3, S^1 \times S^2)(\mathfrak{t}_0,\mathfrak{t}_1)
	\]
where the sum is taken over the $\spinc$ structures of $\Sigma(\xi^* \sqcup C)$. Notice that all $\spinc$ structures $\mathfrak{s}$ are torsion, since $b_2(\Sigma(\xi^* \sqcup C)) = 0$, and therefore have trivial energy. To compute $D(\Sigma(\xi^* \sqcup C); S^3, S^1 \times S^2)(\mathfrak{t}_0,\mathfrak{t}_1)$, consider the pull-up-push-down map associated to the correspondence
	\[
	\begin{tikzcd}[row sep = large]
		&  J(\Sigma(\xi^* \sqcup C)) \arrow[dl] \arrow[dr] & \\
		J(S^3) & & J(S^1 \times S^2)
	\end{tikzcd}
	\]
	It follows that $F_{J(\Sigma(\xi^*\sqcup C))}: \widetilde{\Lambda} \to \widetilde{\Lambda} \{1, v\}$ maps 1 to 1, concluding the proof.
\end{proof}

In contrast, if we assume the validity of Proposition \ref{prop:step1}, then Proposition \ref{prop:daemihammer} tells us that odd Khovanov homology does not distinguish between different embeddings of the torus in $S^4$.

\begin{corollary}
	\label{cor:extracredit}
	For any embedding of the torus $\mathcal{T}: S^1 \times S^1 \hookrightarrow S^4$, $n(\mathcal{T}) = 0$.
\end{corollary}

\begin{proof}
	Let $\mathcal{T}$ be an embedded torus in $S^4$. We have that $b^+(\Sigma(\mathcal{T} \sqcup C)) = 1 > 0$ so Proposition \ref{prop:daemihammer} gives us that $\widetilde{\mathrm{PFH}}(\Sigma(\mathcal{T}\sqcup C)) = 0$. Proposition \ref{prop:step1} implies the result.
\end{proof}

\subsubsection{Plane Floer homology with local coefficient systems}
The surgery exact triangle in Section \ref{s:ss} is formulated using plane Floer homology with a local coefficient system, denoted by $\widetilde{\mathrm{PFH}}(Y,\zeta)$, where $\zeta\subset Y$ is an embedded closed 1-manifold. Given a pair $(W,Z)$, where $W$ is a cobordism between $Y_0$ and $Y_1$, and $Z$ is a cobordism between $\zeta_0$ and $\zeta_1$, one can also define a cobordism map
\[
\widetilde{\mathrm{PFH}}(W,Z):\widetilde{\mathrm{PFH}}(Y_0,\zeta_0)\rightarrow\widetilde{\mathrm{PFH}}(Y_1,\zeta_1).
\]
See \cite[\S 1.4]{daemi2015abelian} for the definition. Since all local coefficient systems appearing in this paper are trivial, we omit the general definition and simply refer to the following lemma.

\begin{lemma}[\cite{daemi2015abelian}, Lemma 1.24]\label{lem:localtrivial}
	If $\zeta$ and ${\zeta}^{'}$ are 1-dimensional embedded sub-manifolds in a 3-manifold $Y$ representing the same element of $H_{1}(Y;\mathbb{Z}/2\mathbb{Z})$, then $\widetilde{\mathrm{PFH}}(Y,\zeta)\cong \widetilde{\mathrm{PFH}}(Y,\zeta^{'})$.
\end{lemma}

Let $[\zeta]$ denote the $\mathbb{Z}/2\mathbb{Z}$-homology class of the 1-manifold $\zeta$. Lemma \ref{lem:localtrivial} implies that, when $[\zeta]$ is trivial, $\widetilde{\mathrm{PFH}}(Y,\zeta)\cong \widetilde{\mathrm{PFH}}(Y)$. When the local coefficient system is trivial on both ends of the cobordism, $\widetilde{\mathrm{PFH}}(W,Z)$ differs from the ordinary cobordism map $\widetilde{\mathrm{PFH}}(W)$ by an extra sign determined by the evaluation of the first Chern class of the $\spinc$ structure on the cobordism surface.

\subsection{Cobordism maps on a single metric}
\label{single}

In this subsection we consider various elementary cobordisms between unlinks and the induced cobordism maps on plane Floer homology. The examples listed here will be crucial for proving Proposition \ref{prop:step1}. 

We use notation consistent with \cite[Section 4]{daemi2015abelian}. Let $U_n$ denote the standard  $n$-component unlink. Label the components of $U_n$ as $K_1,\ldots, K_n$ where $K_i$ is a circle of radius $\tfrac{1}{4}$ in the plane $\mathbb{R}^2\times \{0\}\subset \R^3$ at $(0,i)$. $\widetilde{H}^0(U_n)$ is generated by $v_i:H_0(U_n)\to \mathbb{Z}$ where
\[
v_i(K_j)=\delta_{ij} \text{ for } 1\leq j\leq n-1
\qquad\text{and}\qquad
v_i(K_n)=-1
\]
for each $i=1, \ldots, n-1$. $H_1(\Sigma(U_n); \mathbb{R})$ has a basis consisting of loops $\gamma_i\subset\Sigma(U_n)$ which arise as the lift of the path $p_i$ joining $K_i$ and $K_n$. Therefore $H^1(\Sigma(U_n); \R)$ has a basis consisting of the duals to $\gamma_i$. Then $\mathcal{R}(\Sigma(U_n))$ is identified with $S_1\times \cdots \times S_{n-1}$ where each $S_i$ corresponds to $
\gamma_i$. The cohomology of $\mathcal{R}(\Sigma(U_n))$ is generated by elements of the form $x_{i_1}\wedge \cdots \wedge x_{i_k}$ where $x_{i_j}$ denotes the Poincar\'e dual of $S_{i_j}$. We fix the following isomorphism:
\begin{align*}
	L_n: \widetilde{\mathrm{PFH}}(\Sigma(U_n)) &\to \Lambda^*(\widetilde{H}^0(U_n))\\
	x_{i_1}\wedge \cdots\wedge x_{i_k}&\mapsto (-1)^{k(n-1)}v_{i_1}\wedge \cdots \wedge v_{i_k}.
\end{align*}
 
Let $P_n:U_n\to U_{n+1}$ and $C_n:U_n\to U_{n+1}$ denote the split and birth cobordisms pictured below.
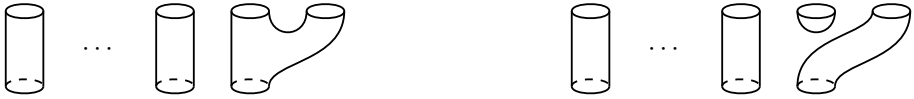
\[
\tikz[baseline={([yshift=-.5ex]current bounding box.center)}, scale=.5]{
    \draw[dashed, knot, -] (-6,0) .. controls (-6,.25) and (-5,.25) .. (-5,0);
    \draw[knot, -] (-6,0) .. controls (-6, -0.25) and (-5,-0.25) .. (-5,0);
    \draw[knot, -] (-6,2) .. controls (-6,2.25) and (-5,2.25) .. (-5,2);
    \draw[knot, -] (-6,2) .. controls (-6, 1.75) and (-5,1.75) .. (-5,2);
    \draw[knot, -] (-6,0) -- (-6,2);
    \draw[knot, -] (-5,0) -- (-5,2);
    \node at (-3.5,1) {$\cdots$};
    \draw[dashed, knot, -] (-2,0) .. controls (-2,.25) and (-1,.25) .. (-1,0);
    \draw[knot, -] (-2,0) .. controls (-2, -0.25) and (-1,-0.25) .. (-1,0);
    \draw[knot, -] (-2,2) .. controls (-2,2.25) and (-1,2.25) .. (-1,2);
    \draw[knot, -] (-2,2) .. controls (-2, 1.75) and (-1,1.75) .. (-1,2);
    \draw[knot, -] (-2,0) -- (-2,2);
    \draw[knot, -] (-1,0) -- (-1,2);
    \draw[knot, -] (0,2) .. controls (0,1.75) and (1,1.75) .. (1,2);
    \draw[knot, -] (0,2) .. controls (0,2.25) and (1,2.25) .. (1,2);
    \draw[knot, -] (2,2) .. controls (2,1.75) and (3,1.75) .. (3,2);
    \draw[knot, -] (2,2) .. controls (2,2.25) and (3,2.25) .. (3,2);
    \draw[knot, -] (0,0) .. controls (0,-.25) and (1,-.25) .. (1,0);
    \draw[dashed, knot, -] (0,0) .. controls (0,.25) and (1,.25) .. (1,0);
    \draw[knot, -] (1,0) .. controls (1,0.65) and (3,0.65) .. (3,2);
    \draw[knot, -] (0,0) -- (0,2);
    \draw[knot, -] (1,2) .. controls (1,1.25) and (2,1.25) .. (2,2);
}
\hspace{3cm}
\tikz[baseline={([yshift=-.5ex]current bounding box.center)}, scale=.5]{
    \draw[dashed, knot, -] (-6,0) .. controls (-6,.25) and (-5,.25) .. (-5,0);
    \draw[knot, -] (-6,0) .. controls (-6, -0.25) and (-5,-0.25) .. (-5,0);
    \draw[knot, -] (-6,2) .. controls (-6,2.25) and (-5,2.25) .. (-5,2);
    \draw[knot, -] (-6,2) .. controls (-6, 1.75) and (-5,1.75) .. (-5,2);
    \draw[knot, -] (-6,0) -- (-6,2);
    \draw[knot, -] (-5,0) -- (-5,2);
    \node at (-3.5,1) {$\cdots$};
    \draw[dashed, knot, -] (-2,0) .. controls (-2,.25) and (-1,.25) .. (-1,0);
    \draw[knot, -] (-2,0) .. controls (-2, -0.25) and (-1,-0.25) .. (-1,0);
    \draw[knot, -] (-2,2) .. controls (-2,2.25) and (-1,2.25) .. (-1,2);
    \draw[knot, -] (-2,2) .. controls (-2, 1.75) and (-1,1.75) .. (-1,2);
    \draw[knot, -] (-2,0) -- (-2,2);
    \draw[knot, -] (-1,0) -- (-1,2);
    \draw[knot, -] (0,2) .. controls (0,1.75) and (1,1.75) .. (1,2);
    \draw[knot, -] (0,2) .. controls (0,2.25) and (1,2.25) .. (1,2);
    \draw[knot, -] (2,2) .. controls (2,1.75) and (3,1.75) .. (3,2);
    \draw[knot, -] (2,2) .. controls (2,2.25) and (3,2.25) .. (3,2);
    \draw[knot, -] (0,0) .. controls (0,-.25) and (1,-.25) .. (1,0);
    \draw[dashed, knot, -] (0,0) .. controls (0,.25) and (1,.25) .. (1,0);
    \draw[knot,-] (0,2) .. controls (0,1.25) and (1,1.25) .. (1,2);
    \draw[knot, -] (1,0) .. controls (1,0.65) and (3,0.65) .. (3,2);
    \draw[knot, -] (0,0) .. controls (0,1.35) and (2,1.35) .. (2,2);
}
\]
The corresponding merge and death cobordisms, denoted by $\overline{P}_n: U_{n+1} \to U_n$ and $\overline{C}_n: U_{n+1} \to U_n$ respectively, are obtained by flipping $P_n$ and $C_n$. We have the following proposition from \cite{daemi2015abelian}.

\begin{proposition}[\cite{daemi2015abelian}, Proposition 4.2]
\label{4.2}
There is an appropriate choice of homology orientations on $C_n,P_n:U_n\to U_{n+1}$ such that the induced cobordism maps on plane Floer homology are as follows:
\begin{align*}
\widetilde{\mathrm{PFH}}(\Sigma(C_n))(v_{i_1}\wedge\cdots\wedge v_{i_k}) & = v_{i_1}\wedge\cdots\wedge v_{i_k},\\
\widetilde{\mathrm{PFH}}(\Sigma(P_n))(v_{i_1}\wedge\cdots\wedge v_{i_k}) & =v_n\wedge v_{i_1}\wedge\cdots\wedge v_{i_k}.
\end{align*}
Similarly, there is an appropriate choice of homology orientations on $\overline{C}_n, \overline{P}_n:U_{n+1}\to U_{n}$ such that the induced cobordism maps on plane Floer homology are as follows:
\begin{align*}
\widetilde{\mathrm{PFH}}(\Sigma(\overline{C}_n))(v_{i_1}\wedge\cdots\wedge v_{i_k}) & = 0, \hspace{3cm}\widetilde{\mathrm{PFH}}(\Sigma(\overline{P}_n))(v_{i_1}\wedge\cdots\wedge v_{i_k})=v_{i_1}\wedge\cdots\wedge v_{i_k},\\
\widetilde{\mathrm{PFH}}(\Sigma(\overline{C}_n))(v_n\wedge v_{i_1}\wedge\cdots\wedge v_{i_k}) & = v_{i_1}\wedge\cdots\wedge v_{i_k}, \hspace{1.2cm}\widetilde{\mathrm{PFH}}(\Sigma(\overline{P}_n))(v_n\wedge v_{i_1}\wedge\cdots\wedge v_{i_k}) =0,
\end{align*}
where $1\leq i_l\leq n-1$.
\end{proposition}

Let $\mathcal{U}$ denote an arbitrary unlink. In order to calculate the induced cobordism maps on plane Floer homology, we must identify the generators of $\widetilde{\text{PFH}}(\Sigma(\mathcal{U}))$ with those of $\Lambda^*(\widetilde{H}^0(\mathcal{U}))$ and make an appropriate choice of homology orientations for the elementary cobordisms. 

\begin{proposition}[\cite{daemi2015abelian}, Corollary 4.5]
\label{prop:unlinkisom}
If $\mathcal{U}$ is an arbitrary unlink, then $\widetilde{\mathrm{PFH}}(\Sigma(\mathcal{U}))$ is canonically isomorphic to $\Lambda^*(\widetilde{H}^0(\mathcal{U}))$.
\end{proposition}

For any isotopy $J$ between unlinks $\mathcal{U}_0$ and $\mathcal{U}_1$, fix the homology orientation of $\Sigma(J)$ to be the product homology orientation of $[0,1] \times \Sigma(\mathcal{U}_1)$. Choose an isotopy $J: \mathcal{U} \to U_n$, which should be thought of as an identification of components of $\mathcal{U}$ with the standard components of $U_n$. The isomorphism of Proposition \ref{prop:unlinkisom} is given as the following composition.
\[
\begin{tikzcd}
\widetilde{\mathrm{PFH}}(\Sigma(\mathcal{U}) \arrow[d, "\widetilde{\mathrm{PFH}}(\Sigma(J))"'] \arrow[r, dashed] & \Lambda^*(\widetilde{H}^0(\mathcal{U}))
\\
\widetilde{\mathrm{PFH}}(\Sigma(U_n) \arrow[r, "L_n"] & \Lambda^*(\widetilde{H}^0(U_n)) \arrow[u, "i_J"']
\end{tikzcd}
\]
Here, $i_J$ is the isomorphism $\Lambda^*(\widetilde{H}^0(U_n)) \cong \Lambda^*(\widetilde{H}^0(\mathcal{U}))$ induced by $J$. One can check that this isomorphism is independent of choice of $J$; see, for example \cite[Proposition 4.4]{daemi2015abelian}.

Similarly, to define the maps induced by arbitrary elementary cobordisms between arbitrary unlinks, we factor through isotopies between the standard unlinks. Let $\mathcal{F}$ be an elementary cobordism between unlinks $\mathcal{U}$ and $\mathcal{U}'$ and let $J: \mathcal{U}\to U$ and $J': \mathcal{U}'\to U'$ be isotopies of unlinks. Let $F$ denote a standard elementary cobordism between $U$ and $U'$. We say that $J$ and $J'$ are \emph{compatible} with $\mathcal{F}$ if $F\circ J$ and $J' \circ \mathcal{F}$ are isotopic to each other through an isotopy that fixes the ends.
\[
\begin{tikzcd}
\mathcal{U} \arrow[r, "J"] \arrow[d, "\mathcal{F}"'] & U \arrow[d, "F"] \\
\mathcal{U}' \arrow[r, "J'"] & U'
\end{tikzcd}
\]
Since there is a fixed homology orientation on $\Sigma(F)$ coming from Proposition \ref{4.2}, we can define a homology orientation on $\Sigma(\mathcal{F})$ by requiring that the homology orientations of the compositions $\Sigma(F)\circ\Sigma(J)$ and $\Sigma (J')\circ\Sigma(\mathcal{F})$ are equal. In this way, $J$ and $J'$ determine a homology orientation on $\Sigma(\mathcal{F})$.

Now that we have determined a set of generators for $\widetilde{\text{PFH}}(\Sigma(\mathcal{U}))$ and fixed homology orientations for cobordisms between arbitrary unlinks, we can describe the maps on plane Floer homology induced by elementary cobordisms. Let ${\mathcal{P}}$ denote a split cobordism between arbitrary unlinks $\mathcal{U}$ and $\mathcal{U'}$. Choose compatible isotopies of unlinks $J:\mathcal{U}\to U_{n}$ and $J':\mathcal{U}'\to U_{n+1}$. Compatibility implies that $J'$ is determined by $J$ on $n-1$ components of $\mathcal{U}'$. However, on the remaining two components $K_1$ and $K_2$ of $\mathcal{U}'$, we have two choices for $J'$ which differ by the permutation isotopy of $K_1$ and $K_2$. This gives rise to two possible homology orientations for the split cobordism $\Sigma(\mathcal{P})$. The induced map on plane Floer homology is given by
\begin{align*}
F_{{\mathcal{P}}}:\Lambda^*(\widetilde{H}^0(\mathcal{U})) &\to \Lambda^*(\widetilde{H}^0(\mathcal{U'})) \\
 x &\mapsto v\wedge x.
\end{align*}
Here $v\in \widetilde{H^0}(\mathcal{U'})$ is non-zero only on $K_1$ and $K_2$. Depending on the choice of $J'$, either $v(K_1)=1$ and $v(K_2)=-1$, or $v(K_1)=-1$ and $v(K_2)=1$. We can use the same decorations offered in the definition of odd Khovanov homology to distinguish these choices; see Figure \ref{fig:decorations}. Note that we have decorations on both the cobordism and the unlinks. These decorations serve as an identification of the unlinks with the boundary of the split cobordism.

\begin{figure}[h]
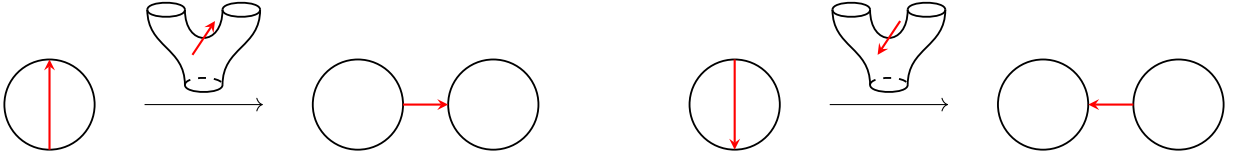

    \centering
\[
\tikz[]{
    \node(A) at (0,0) {$
    \tikz[scale=0.6]{
    \draw[knot] (0,0) circle (1cm);
    \draw[stealth-, red, thick] (0,1) -- (0,-1);
    }
    $};
    \node(B) at (5,0) {$
    \tikz[scale=0.6]{
    \draw[knot] (0,0) circle (1cm);
    \draw[knot] (3,0) circle (1cm);
    \draw[-stealth, red, thick] (1,0) -- (2,0);
    }
    $};
    \draw[->, white] (A) -- node[above, pos=0.5]{$
    \tikz[baseline={([yshift=-.5ex]current bounding box.center)}, scale=.5]{
	\draw[black, knot, -]   (1,2) .. controls (1,3) and (0,3) .. (0,4);
	\draw[black, knot, -]   (2,2) .. controls (2,3) and (3,3) .. (3,4);
	\draw[black, knot, -]  (1,4) .. controls (1,3) and (2,3) .. (2,4);
	\draw[black, knot, -]  (0,4) .. controls (0,3.75) and (1,3.75) .. (1,4);
	\draw[black, knot, -]  (0,4) .. controls (0,4.25) and (1,4.25) .. (1,4);
	\draw[black, knot, -]  (2,4) .. controls (2,3.75) and (3,3.75) .. (3,4);
	\draw[black, knot, -]  (2,4) .. controls (2,4.25) and (3,4.25) .. (3,4);
	\draw[black, knot, -]  (1,2) .. controls (1,1.75) and (2,1.75) .. (2,2);
	\draw[black, knot, dashed, -] (1,2) .. controls (1,2.25) and (2,2.25) .. (2,2);
    \draw[stealth-, red, thick](1.8,3.7) -- (1.2,2.8);
}
    $} (B);
    \draw[->, shorten >=15pt, shorten <=15pt] (A) -- (B);
}
\hspace{1.75cm}
\tikz[]{
    \node(A) at (0,0) {$
    \tikz[scale=0.6]{
    \draw[knot] (0,0) circle (1cm);
    \draw[-stealth, red, thick] (0,1) -- (0,-1);
    }
    $};
    \node(B) at (5,0) {$
    \tikz[scale=0.6]{
    \draw[knot] (0,0) circle (1cm);
    \draw[knot] (3,0) circle (1cm);
    \draw[stealth-, red, thick] (1,0) -- (2,0);
    }
    $};
    \draw[->, white] (A) -- node[above, pos=0.5]{$
    \tikz[baseline={([yshift=-.5ex]current bounding box.center)}, scale=.5]{
	\draw[black, knot, -]   (1,2) .. controls (1,3) and (0,3) .. (0,4);
	\draw[black, knot, -]   (2,2) .. controls (2,3) and (3,3) .. (3,4);
	\draw[black, knot, -]  (1,4) .. controls (1,3) and (2,3) .. (2,4);
	\draw[black, knot, -]  (0,4) .. controls (0,3.75) and (1,3.75) .. (1,4);
	\draw[black, knot, -]  (0,4) .. controls (0,4.25) and (1,4.25) .. (1,4);
	\draw[black, knot, -]  (2,4) .. controls (2,3.75) and (3,3.75) .. (3,4);
	\draw[black, knot, -]  (2,4) .. controls (2,4.25) and (3,4.25) .. (3,4);
	\draw[black, knot, -]  (1,2) .. controls (1,1.75) and (2,1.75) .. (2,2);
	\draw[black, knot, dashed, -] (1,2) .. controls (1,2.25) and (2,2.25) .. (2,2);
    \draw[-stealth, red, thick](1.8,3.7) -- (1.2,2.8);
}
    $} (B);
    \draw[->, shorten >=15pt, shorten <=15pt] (A) -- (B);
}
\]
    \caption{Decorations indicating the choice of homology orientation on the split cobordism.}
   \label{fig:decorations}
\end{figure}

Let $\overline{\mathcal{P}}$ denote a merge cobordism from $\mathcal{U}$ to $\mathcal{U'}$. Choose compatible isotopies of unlinks $J:\mathcal{U}\to U_{n+1}$ and $J':\mathcal{U}\to U_n$. Compatibility implies that $J'$ is determined by $J$. Since we have fixed a choice of homology orientation on each of $\Sigma(\overline{P}_n)$, $\Sigma(J)$, and $\Sigma(J')$, the homology orientation on $\Sigma(\overline{\mathcal{P}})$ is also fixed. Each element of $H^0(\mathcal{U})$ can be thought of as the Poincar\'e dual of the fundamental class of a connected component $K\subseteq \mathcal{U}$. $K$ is connected to exactly one component $K'\subseteq \mathcal{U}'$ by the cobordism $\overline{\mathcal{P}}$. Let $j:H^0(\mathcal{U})\to H^0(\mathcal{U'})$ be the map that sends $\mathrm{PD}[K]$ to $\mathrm{PD}[K']$. The induced map on plane Floer homology is given by
\begin{align*}
F_{\overline{\mathcal{P}}}:\Lambda^*(\widetilde{H}^0(\mathcal{U})) &\to \Lambda^*(\widetilde{H}^0(\mathcal{U}')), \\
 x &\mapsto j_*(x).
\end{align*}

Let $\mathcal{C}:\mathcal{U}\to \mathcal{U}'$ denote a birth cobordism. Just as for the merge cobordisms, a choice of compatible isotopies of unlinks determines a unique choice of homology orientation on $\Sigma({\mathcal{C})}$. Let $i:H^0(\mathcal{U})\to H^0(\mathcal{U'})$ be the map which extends $\phi\in H^0(\mathcal{U})$ to $H^0(\mathcal{U}')$ by declaring that $i(\phi) = 0$ on the extra component. The induced map on plane Floer homology is given by
\begin{align*}
F_{\mathcal{C}}:\Lambda^*(\widetilde{H}^0(\mathcal{U})) &\to \Lambda^*(\widetilde{H}^0(\mathcal{U'})), \\
 x &\mapsto i_*(x).
\end{align*}

Let $\overline{\mathcal{C}}:\mathcal{U}\to \mathcal{U}'$ denote a death cobordism. For a final time, a choice of compatible isotopies between unlinks is seen to determine a unique choice of homology orientation on $\Sigma(\overline{\mathcal{C}})$. Let $\iota_K:H^0(\mathcal{U})\to H^0(\mathcal{U'})$ be the contraction map with respect to the component of $\mathcal{U}$ that is dropped. The induced map on plane Floer homology is given by
\begin{align*}
F_{\overline{\mathcal{C}}}:\Lambda^*(\widetilde{H}^0(\mathcal{U})) &\to \Lambda^*(\widetilde{H}^0(\mathcal{U'})), \\
 x &\mapsto \iota_K(x).
\end{align*}

\subsubsection{A cobordism inducing the zero map}

Consider the pair $(V,Z)$, where $V$ is the 4-manifold $(I\times (\#^n S^1\times S^2)) \# \overline{\mathbb{C}P^2}$, and $Z$ is the embedded sphere in $\overline{\mathbb{C}P^2}$ with self-intersection $-1$. We view it as a cobordism from $(\#^n S^1\times S^2,\emptyset)$ to $(\#^n S^1\times S^2,\emptyset)$. The pull-up-push-down map is determined by the following correspondence.
\[
\begin{tikzcd}[row sep = large]
& J((I\times (\#^n S^1\times S^2)) \# \overline{\mathbb{C}P^2}) \arrow[dl, "i_0"'] \arrow[dr, "i_1"] & \\
J(\#^{n} S^1\times S^2) & & J(\#^{n} S^1\times S^2)
\end{tikzcd}
\]

We will compute the associated map on plane Floer homology in two steps. First, we study the map induced by $(\overline{\mathbb{C}P^2},Z)$ viewed as a cobordism from the empty set to the empty set. The $\mathrm{Spin}^c$ structures on $\overline{\mathbb{C}P^2}$, which are in bijection with the odd integers, are enumerated by $\mathfrak{s}_{2k+1}$ where $k\in \mathbb{Z}$. Since $b^+(\overline{\mathbb{C}P^2})=0$, we can take our perturbation term to be zero \cite[Lemma 1.6]{daemi2015abelian}. The moduli space $M(\overline{\mathbb{C}P^2}, \mathfrak{s}_{2k+1})$ contains a unique $\mathrm{Spin}^c$ connection $A_{2k+1}$ for each $\mathrm{Spin}^c$ structure $\mathfrak{s}_{2k+1}$. We have the following correspondence for each $\mathfrak{s}_{2k+1}$.
\[
\begin{tikzcd}[row sep = large]
& J(\overline{\mathbb{C}P^2}) \arrow[dl, "i_0"'] \arrow[dr, "i_1"] & \\
* & & *
\end{tikzcd}
\]
Therefore, the cobordism induces the identity map $D_{2k+1}:\tilde{\Lambda}_0 \to \tilde{\Lambda}_0$. The sign of the map associated to the conjugate $\mathrm{Spin}^c$ structure $\overline{\mathfrak{s}}_{2k+1} = \mathfrak{s}_{-2k-1}$ is opposite, since $\langle c_1(\mathfrak{s}_{2k+1}),Z\rangle=-\langle c_1(\mathfrak{s}_{-2k-1}),Z\rangle$. Since the energies of the conjugate $\mathrm{Spin}^c$ structures are equal, we conclude that the induced maps cancel and sum to the zero map.

Since $S^1 \times S^2$ has a unique torsion $\mathrm{Spin}^c$ structure, we can extend the argument above to the connected sum $V=(I\times (\#^n S^1\times S^2)) \# \overline{\mathbb{C}P^2}$. Therefore, $\widetilde{\mathrm{PFH}}(V,Z) = 0$.

\subsection{Cobordism maps defined over a family of metrics}\label{family}
In general, the moduli space $M(W, \mathfrak{s})$ is defined over a family of metrics corresponding to ``neck stretching'' along different ``cuts'' embedded in the 4-dimensional cobordism. One usually performs a ``fattening'' operation on the moduli space to ensure compactness. For a thorough description, see \cite[\S 1.3]{daemi2015abelian} and \cite{kronheimer2011khovanov}. In this article, we only discuss the following case, which is Example 1.13 in \cite{daemi2015abelian}. This example is essential to our proof of Proposition \ref{prop:step1} in Section \ref{s:rmoves}.

Consider the $4$ manifold $U= D^2 \times S^2 \# \overline{\mathbb{C}P^2}\setminus D^4$. Notice that after applying a handle slide, $U$ is diffeomorphic to the plumbing of two disk bundles over $S^2$ with Euler number  $-1$. Fix two spheres $S$ and $S'$ in $U$ which have self-intersection $-1$ and intersect each other exactly once. In addition, $S$ is required to be disjoint from the disk $D=D^{2}\times\{x\}\subset D^2\times S^2$, where $x$ is an arbitrary point in $S^2$, and the intersection number between $S'$ and $D$ is one. We consider the pair $(U,S)$, which is viewed as a cobordism from $(S^1 \times S^2,\emptyset)$ to $(S^3,\emptyset)$. The family of metrics on $U$ is defined via stretching as follows. Denote the tubular neighborhoods of $S$ and $S'$ by $T$ and $T'$ respectively. Now, stretch along $T$ and $T'$ to get a family of metrics parameterized by $G = [-\infty,\infty]$. Here, $-\infty$ (respectively, $\infty$) corresponds to a broken metric along $T$ (respectively, $T'$). This family of metrics will be denoted by $\mathbb{U}$.

Let $\mathfrak{t}_0$ be the unique torsion $\mathrm{Spin}^c$ structure on $S^1 \times S^2$ and $\mathfrak{t}_1$ be the unique $\mathrm{Spin}^c$ structure on $S^3$. Let $\mathfrak{s}$ be a $\mathrm{Spin}^c$ structure on $U$ which restricts to $\mathfrak{t}_0$ and $\mathfrak{t}_1$ on the boundary. Recall that $c_1(\mathfrak{s})$ is characteristic, \textit{i.e.}, $\langle c_1(\mathfrak{s}),\alpha\rangle\equiv \alpha^2 \mod 2$ for all $\alpha \in H_2(U;\mathbb{Z})$. Since $H_2(U;\mathbb{Z})$ has no 2-torsion, $\mathfrak{s}$ is determined by the evaluation $\langle c_1(\mathfrak{s}),[S]\rangle \equiv [S]^2\equiv 1 \mod 2$. We will denote the $\mathrm{Spin}^c$ structures by $\mathfrak{s}_{2k+1}$, where $\langle c_1(\mathfrak{s}),[S]\rangle=2k+1$.

Fix a $\mathrm{Spin}^c$ structure $\mathfrak{s}_{2k+1}$. Since $b^+(U)=0$, any perturbation will give rise to a regular moduli space \cite[Lemma 1.6]{daemi2015abelian}. In particular, we can take  $\nu$ to be the zero perturbation. We have the following correspondence.
\[
\begin{tikzcd}[row sep = large]
& M(\mathbb{U}, \mathfrak{s}_{2k+1})=\bigcup_{g\in G}M(U^g,\mathfrak{s}_{2k+1}) \arrow[dl, "r_0"'] \arrow[dr, "r_1"] & \\
\mathcal{R}(S^1 \times S^2, \mathfrak{t}_0)& & \mathcal{R}(S^3, \mathfrak{t}_1)
\end{tikzcd}
\]
The Jacobian torus $J(S^3)$ is a single point and $J(S^1\times S^2)=S^1$. For a fixed $g\in G$, the space $M(U^g,\mathfrak{s}_{2k+1})$ contains a single element $A^g$ which satisfies $F^+(A^g)=0$. Hence, as we vary $g\in [-\infty,\infty]$, we see that $M(\mathbb{U}, \mathfrak{s}_{2k+1})$ can be identified with an interval $\mathbb{I}$. Thus, the previous correspondence is identified with the following one.
\[
\begin{tikzcd}[row sep = large]
& \mathbb{I} \arrow[dl, "r_0"'] \arrow[dr, "r_1"] & \\
J(S^1 \times S^2)= S^1 & & J(S^3)=\{*\}
\end{tikzcd}
\]

Now we can compute the cobordism map. Given two critical points $\alpha$ and $\beta$, we can compute the dimension of the space as follows (also, see \cite[Remark 1.2]{daemi2015abelian}):
\begin{align*}
\dim (M_\nu(\mathbb{U},\alpha,\beta,\mathfrak{s}_{2k+1})) &=\dim(M_\nu(\mathbb{U}))+\dim(U_\alpha \times S_\beta)-\dim(\mathcal{R}(Y_0) \times \mathcal{R}(Y_1))\\
				&=\dim(M_\nu(\mathbb{U}))-b_1(Y_0)+\ind(\alpha)-\ind(\beta)\\
				&=\dim(G)-\frac{\chi(U)+\sigma(U)}{2} +\frac{b_1(Y_0)+b_1(Y_1)}{2}-b_1(Y_0)+\ind(\alpha)-\ind(\beta)\\
				&=\dim(G)-1+\ind(\alpha)-\ind(\beta)\\
				&=\ind(\alpha)-\ind(\beta),
\end{align*}
where $Y_0 = S^1 \times S^2$ and $Y_1 = S^3$. Therefore, the moduli space is 0-dimensional only when $\ind(\alpha)=\ind(\beta)$. 

Since $\mathcal{R}(S^3,\mathfrak{t}_1)$ is a point, it has a single critical point of index 0. Choose a Morse function on $S^1$ with two critical points: $\alpha$ of index 0 and $\alpha'$ of index 1. The map is non-trivial only on the index $0$ critical point, \textit{i.e.}, it is non-trivial on $\alpha$ and trivial on $\alpha'$. By \cite[Equation 27]{daemi2015abelian}, we have
\[
f_{\mathbb{U}}(\alpha)=\sum_{k\in \mathbb{Z}} \sum_{\beta \in Crit(f_1)}\left(\sum_{p \in M(\mathbb{U},\alpha,\beta,\mathfrak{s}_{2k+1})} \epsilon_p u^{e(p)}\right)\beta.
\]
In addition, the energy of the $\mathrm{Spin}^c$ structure $\mathfrak{s}_{2k+1}$ is
\begin{align*}
e(A^g)&=-c_1(\mathfrak{s}_{2k+1}) \cdot c_1(\mathfrak{s}_{2k+1})\\
&=(2k+1)^2.
\end{align*}

We conclude by studying the restriction map $r_0$. The image of the restriction map can be understood by identifying $\mathcal{R}(S^1\times S^2,\mathfrak{t}_0)$ with the determinant of the holonomy around the boundary of the disk $D$. By Gauss-Bonnet, we would like to understand
\[
\det(\hol(\underline{A}^g,\partial D))=\exp\left( \displaystyle \int_D F(A^g)\right).
\]

Consider $\tfrac{1}{2\pi i }F(A^g)$, which is a $g$-harmonic form representing $c_1(\mathfrak{s}_{2k+1})$. As $g$ goes to $\infty$ (respectively, $-\infty$), this form converges to $(2k+1)\mathrm{PD}(S_1)$ (respectively, $(2k+1)\mathrm{PD}(S_1')$). Now, we calculate the holonomies for $g=\pm \infty$. First, for $g=\infty$ (\textit{i.e.}, the broken metric along $T'$), 
\[
\int_D F(A^\infty)=0
\]
since $A^\infty$ is flat on $D$. Therefore, $\det(\hol(\underline{A}^{\infty},\partial D))=\exp\left( \displaystyle \int_D F(A^\infty)\right)=\exp(0)$. Next for $g=-\infty$ (\textit{i.e.}, the broken metric along $T$), we have that
\[
\int_D F(A^{-\infty})=2\pi i\langle c_1(\mathfrak{s}_{2k+1}),D\rangle =2\pi i\langle (2k+1)\PD(S_1'),D\rangle=2\pi i (2k+1)
\]
so $\det(\hol(\underline{A}^{-\infty},\partial D))=\exp\left(\displaystyle \int_D F(A^{-\infty})\right)=\exp(2\pi i(2k+1))$.

Under the identification with the Jacobian torus, the image of $r_0:M(\mathbb{U},\mathfrak{s}_{2k+1})\to \mathcal{R}(S^1\times S^2,\mathfrak{t}_0)$ consists of $2k+1$ half-circles. Observe that the image of the conjugate $\mathrm{Spin}^c$ structure $\overline{\mathfrak{s}}_{2k+1}=\mathfrak{s}_{-2k-1}$ by $r_0:M(\mathbb{U},\overline{\mathfrak{s}}_{2k+1})\to \mathcal{R}(S^1\times S^2,\mathfrak{t}_0)$ will consist of the complex conjugates of the previous half-circles. We conclude that the image of
$r_0:M(\mathbb{U},\mathfrak{s}_{2k+1})\cup M(\mathbb{U},\overline{\mathfrak{s}}_{2k+1}) \to \mathcal{R}(S^1\times S^2,\mathfrak{t}_0)$ is a circle that winds around $2k+1$ times.

Consider the critical point $\alpha$ of index 0. The unstable manifold at $\alpha$ is $\alpha$ itself; we compute
\begin{align*}
	M(\mathbb{U},\alpha,*,\mathfrak{s}_{2k+1})&=\{z\in M(\mathbb{U},\mathfrak{s}_{2k+1}):r_0(z)\in \alpha \}\\
	&=\{z\in M(\mathbb{U},\mathfrak{s}_{2k+1}):r_0(z)= \alpha \}\\
	&=r_0^{-1}(\alpha).
\end{align*}
Note that while the preimages of $r_0$ with respect to conjugate $\mathrm{Spin}^c$ structures have opposite signs, there is an additional sign arising from the cobordism of the local system $S$ which ensures that the sign counts on these preimages are consistent. Thus, we see that the map is given by 
\[
f_{\mathbb{U}}(\alpha)=\sum_{k\in \mathbb{Z}}\pm(2k+1 )u^{(2k+1)^2}.
\]
This element is invertible in $\tilde{\Lambda}$---we will refer to it as $c$. Hence, the cobordism map coincides with the birth map of Subsection \ref{single} up to multiplication by $c$.

Now, we will briefly describe the map associated to the cobordism $U$ turned around. The pull-up-push-down correspondence is as follows.
\[
\begin{tikzcd}[row sep = large]
& J(\mathbb{U})=\mathbb{I} \arrow[dl, "r_1"'] \arrow[dr, "r_0"] & \\
J(S^3)=\{*\}& & J(S^1 \times S^2)= S^1 
\end{tikzcd}
\]
Counting the dimension of the moduli space using the formula above, we see that it is dimension 0 only for the index $1$ critical point $\alpha'$ on $S^1$. The stable manifold corresponding to $\alpha'$ is, again, $\alpha'$ itself, and
 \begin{align*}
	M(\mathbb{U},*,\alpha',\mathfrak{s}_{2k+1})&=\{z\in M(\mathbb{U},\mathfrak{s}_{2k+1}):r_0(z)\in \alpha' \}\\
	&=\{z\in M(\mathbb{U},\mathfrak{s}_{2k+1}):r_0(z)= \alpha' \}\\
	&=r_0^{-1}(\alpha').
\end{align*}
Hence, we can describe the cobordism map as 
\[
f_{\mathbb{U}}(1)=\sum_{k\in \mathbb{Z}}\pm(2k+1 )u^{(2k+1)^2}\alpha'.
\]
Note that this map coincides with the death map of Subsection \ref{single} up to multiplication by $c$.

\section{Spectral sequence from odd Khovanov homology to plane Floer homology}
\label{s:ss}

The spectral sequence from the odd Khovanov homology of the mirror image of a knot to the plane Floer homology of the branched double cover of the knot fits into the framework of a Khovanov-Floer theory \cite{baldwin2019functoriality}. By iteratively applying the surgery exact triangle, one shows that the plane Floer homology of the branched double cover is chain homotopic to the corresponding surgery cube, and the spectral sequence is induced by the cube filtration. Throughout this section, we assume all links are oriented.

\subsection{Surgery exact triangle}
\label{ss:surgeryexacttriangle}
Given a pair $(Y,\zeta)$, where $\zeta$ is a link in $Y$, let $(L,\Lambda)=\bigcup_{1\leq i\leq n}(K_{i},\lambda_{i})$ be an oriented framed $n$-component link in $Y\backslash\zeta$. For each element $\m=(m_1,\ldots,m_n)$ in the lattice $\mathbb{Z}^{n}$, one can assign a pair $(Y_{\m},\zeta_{\m})$ as follows. Let $Y^L$ denote the complement of $L$ in $Y$, and (abusing notation slightly) let $\mu_i$ and $\lambda_i$ denote the meridian and longitude of the torus boundary corresponding to $K_i$. The orientation is chosen so that the outward-normal-first convention agrees with the orientation determined by the ordered pair $(\mu_i,\lambda_{i})$. Let $\gamma_{i}^0,\, \gamma_{i}^1,$ and $\gamma_{i}^2$ be geometric representatives of the homology classes of $\lambda_{i},-(\lambda_{i}+\mu_i),$ and $\mu_i$ respectively. We extend these curves to the family $\{\gamma_{i}^j\}_{j\in\mathbb{Z}}$ by requiring 3-periodicity in $j$. $Y_\m$ is defined to be the Dehn filling of $Y^L$ along the curves $\gamma_{i}^{m_i}$. Note that these Dehn fillings correspond to the $0$-, $1$-, and $\infty$-surgeries on $(K_i, \lambda_i) \subset (L, \Lambda)$; such a triple is called a \emph{surgery triad}. For each link component, we assign an embedded closed 1-manifold $\zeta_{m_i}$, where $\zeta_{m_i}$ is 
\begin{itemize}
\item the core of the filling solid torus when $m_i \equiv 2 \mod 3$, and
\item empty otherwise.
\end{itemize}
Denote the core of the filling solid torus by $D_{m_i}$. We define $\zeta_\m = \zeta \cup \bigcup_{1 \le i \le n} \zeta_{m_i}$.

To simply notation, we use the $L_1$ and $L_\infty$ norms:
\[|\m|_1:=\sum_{i=1}^{n}m_i\qquad \qquad |\m|_\infty:=\mathrm{max}\{m_1,\ldots,m_n\}.\]
For each pair $(\m,\n)$, with $\m\geq\n$, there is a cobordism of the pairs $(W_\m^\n,Z_\m^\n):(Y_\m,\zeta_{\m})\rightarrow(Y_\n,\zeta_{\n})$ defined as follows. First, when $|\m-\n|_1=1$, there exists an $i$ such that $m_i-n_i=1$ and $m_j-n_j=0$, when $j\neq i$. Then, there is a natural surgery cobordism between $Y_\m$ and $Y_\n$ given by the $2$-handle attachment along $D_{m_i}$. Denote this cobordism by $(W_\m^\n)^\circ$. The surface cobordism $(Z_\m^\n)^\circ$ is the union of $\zeta \times [0,1]$ with
\begin{itemize}
\item the core of the $2$-handle when $m_i\equiv 2$ mod $3$,
\item nothing when $m_i\equiv 1$ mod $3$, or
\item the cocore of the $2$-handle when $m_i\equiv 0$ mod $3$.
\end{itemize}
Now, glue cylindrical ends onto $((W_\m^\n)^\circ, (Z_\m^\n)^\circ)$ to obtain $(W_\m^\n, Z_\m^\n)$. For a general pair $(\m,\n)$, we can find a sequence 
\[
\m=\bfk_0>\bfk_1>\cdots > \bfk_{l-1}>\bfk_l=\n
\]
such that $|\bfk_i-\bfk_{i+1}|_1=1$. In this case we define 
\[
W_\m^\n=W_{\bfk_0}^{\bfk_1}\circ W_{\bfk_1}^{\bfk_2}\circ\cdots\circ W_{\bfk_{l-1}}^{\bfk_l} \qquad \text{and} \qquad Z_\m^\n=Z_{\bfk_0}^{\bfk_1}\circ Z_{\bfk_1}^{\bfk_2}\circ\cdots\circ Z_{\bfk_{l-1}}^{\bfk_l}.
\]
It is clear that the diffeomorphism type of $(W_\m^\n,Z_\m^\n)$ is independent on the choice of $\bfk_0, \ldots, \bfk_l$.

In \cite[\S 2.1]{daemi2015abelian}, Daemi assigned three different types of families of metrics to $(W_\m^\n,Z_\m^\n)$ when $\left| \m - \n \right|_\infty = 1, 2, 3$. These families of metrics are denoted by $G_\m^\n$, $N_\m^\n$, and $K_\m^\n$, with corresponding cobordism maps $\tilde{f}_{G_\m^\n}$, $\tilde{f}_{N_\m^\n}$, and $\tilde{f}_{K_\m^\n}$. As before, we omit the general definitions and describe only the cobordisms $(W_\m^\n, Z_\m^\n)$ used in this paper.

\textbf{Case 1}: $|\m-\n|_1=2$ and $|\m-\n|_\infty=1$. In this case, $(W_\m^\n, Z_\m^\n)$ is obtained by performing 2-handle attachments on two components $(K_1,\lambda_{1})$ and $(K_2,\lambda_{2})$. The corresponding family of metrics is defined over $[-\infty,\infty]$ such that when $t\rightarrow-\infty$, it corresponds to a neck stretching along $Y_{\lambda_1}(K_1)$, and when  $t\rightarrow \infty$, it corresponds to a neck stretching along $Y_{\lambda_2}(K_2)$.

\textbf{Case 2}: $|\m-\n|_1=2$ and $|\m-\n|_\infty=2$. In this case, $(W_\m^\n, Z_\m^\n)$ is obtained by attaching a 2-handle along a component $(K,\lambda)$ and then attaching a second 2-handle along the dual knot of $K$. Note that, by our choice of framing, there exists a $2$-sphere $S$ with self-intersection $-1$ which is the union of the cocore of the first 2-handle and the core the second 2-handle. The family of metrics is again defined over $[-\infty,\infty]$ such that when $t\rightarrow-\infty$, it corresponds to a neck stretching along $Y_{\lambda}(K)$, and when  $t\rightarrow \infty$, it corresponds to a neck stretching along the boundary of a tubular neighborhood of $S$.

With this setup, the surgery cube for a pair $(\m,\n)$, with $|\m-\n|_\infty=1$ is defined as
\[\left(C_\m^\n:=\bigoplus_{\m\geq\bfk\geq \n}\widetilde{\mathrm{PFC}}(Y_{\bfk},\zeta_\bfk),d_\m^\n:=(\tilde{f}_{G_{\bfk}^{{\bfk}^{'}}})_{\m\geq\bfk\geq{\bfk}^{'}\geq \n}\right).\]

Let $\m(i):=(i,m_1,\ldots,m_{n-1})\in\mathbb{Z}^{n}$. Given $|\m-\n|_\infty=1$, we view $C_{\m(i)}^{\n(i)}$ as an $(n-1)$-dimensional surgery cube. Define $g_{\m(i)}^{\n(i)}:C_{\m(i)}^{\n(i)}\rightarrow C_{\m(i-1)}^{\n(i-1)}$ by the following matrix presentation:
\[g_{\m(i)}^{\n(i)}:=\left(\tilde{f}_{G_{\bfk(i)}^{{\bfk}^{'}(i-1)}}\right)_{\m\geq\bfk\geq{\bfk}^{'}\geq \n}.\]
Using the $g_{\m(i)}^{\n(i)}$ map, we can define the mapping cone complex as
\[\mathrm{Cone}(g_{\m(i)}^{\n(i)}):=\left(C_{\m(i)}^{\n(i)}\oplus C_{\m(i-1)}^{\n(i-1)},\begin{pmatrix}
	d_{\m(i)}^{\n(i)} & 0\\
	g_{\m(i)}^{\n(i)} & d_{\m(i-1)}^{\n(i-1)}
\end{pmatrix}\right).\]

\begin{remark}
	$d_\m^\n$ and $ g_\m^\n$ are defined over $\Z$. Their signs are determined by a homology orientation on $W_\m^\n$ together with an orientation on $G_\m^\n$. It is proven in \cite[\S 2.1]{daemi2015abelian} that there exists a choice of orientation for which $d_\m^\n$ defines a differential and $g_\m^\n$ is an anti-chain map. Consequently, $(C_{\m}^{\n},d_{\m}^{\n})$ and $\mathrm{Cone}(g_{\m}^{\n})$ are well-defined chain complexes. We will not review the orientation conventions here; rather, we use only the fact that such a compatible choice exists.
\end{remark}
\begin{theorem}[\cite{daemi2015abelian}, Theorem 2.9]\label{thm:sec}
\label{thm:surgerytriangle}
	The chain complex $(C_{\m(i)}^{\n(i)},d_{\m(i)}^{\n(i)})$ and $\mathrm{Cone}(g_{\m(i-1)}^{\n(i-1)})$ have the same chain homotopy type.
\end{theorem}

Define $n_{\m(i)}^{\n(i)}:C_{\m(i)}^{\n(i)}\rightarrow C_{\m(i-2)}^{\n(i-2)}$ by the following matrix presentation:
\[n_{\m(i)}^{\n(i)}:=\left(\tilde{f}_{N_{\bfk(i)}^{{\bfk}^{'}(i-2)}}\right)_{\m\geq\bfk\geq{\bfk}^{'}\geq \n}.\]
The chain homotopy and its homotopy inverse are given by 
\begin{equation}
\label{eq:chainhomotopies}
\phi_i:=(g_{\m(i)}^{\n(i)},n_{\m(i)}^{\n(i)}):(C_{\m(i)}^{\n(i)},d_{\m(i)}^{\n(i)})\rightarrow\mathrm{Cone}(g_{\m(i-1)}^{\n(i-1)}), \quad  \psi_i:=\begin{pmatrix}
	n_{\m(i-1)}^{\n(i-1)}\\
	g_{\m(i-2)}^{\n(i-2)}
\end{pmatrix}:\mathrm{Cone}(g_{\m(i-1)}^{\n(i-1)})\rightarrow C_{\m(i)}^{\n(i)}.
\end{equation}

Applying Theorem \ref{thm:sec} iteratively, we have the following corollary:
\begin{corollary}\label{cor:sec}
	Let $\mathbf{p}=(2,\ldots, 2)$, $\m=(1,\ldots,1)$, and $\n=(0,\ldots,0)$. Then the chain complexes $(\widetilde{\mathrm{PFC}}(Y_\mathbf{p},\zeta_{\mathbf{p}}),d_{\mathbf{p}})$ and $(C_\m^\n,d_\m^\n)$ have the same chain homotopy type.
\end{corollary}
\subsection{Topology of branched double covers}\label{ss:topology}
A crossing $c$ for a link $L\in S^{3}$ is an orientation preserving embedding $c: D^3\rightarrow S^3$ such that $c(T)=c(D^3)\cap L$, where $T$ is the tangle in Figure \ref{subfig:skeinmovesA}. Let $L_0$ and $L_1$ be the corresponding 0- and 1-resolutions. Then $(\Sigma(L), \Sigma(L_0), \Sigma(L_1))$ forms a surgery triad. To see this, start by noticing that, away from the crossing region, the branched double covers agree. Additionally, the branched double cover of the 4-ended tangle in Figure \ref{subfig:skeinmovesA} is a solid torus where the core corresponds to the lift of the solid red path connecting the two arcs of $T$. Let $T_0$ and $T_1$ denote the tangles in Figure \ref{subfig:skeinmovesB}. Note that $T_0$ and $T_1$ are isotopic to $T$ via diffeomorphisms $f_0$ and $f_1$ of $D^3$. Therefore, $\Sigma(L_i)$ is obtained from $\Sigma(L)$ by cutting out $\Sigma(D^3, T)$ and gluing back $\Sigma(D^3, L_i)$ along a lift of $f_i$ for $i=0, 1$. This process can be interpreted as surgery on the same knot with different framings. More precisely, let $l_p$ denote the union of the solid and dotted red paths in Figure \ref{subfig:skeinmovesA}. It lifts to a framed knot in $\Sigma(L)$. Then $\Sigma(L)$ is the $\infty$-surgery, and $\Sigma(L_0)$ and $\Sigma(L_1)$ are the 0- and 1-surgeries. 

Denote resolutions of $L$ at the crossing $c$ by $L_{(i)}$ for $i \in \mathbb{Z}$, where $L_{(i)} = L_{i \,(\mathrm{mod}~ 3)}$ and $L_2 = L$. We also denote their branched double covers by $\Sigma(L)_{(i)} := \Sigma(L_{(i)})$. The cobordism map $W^{(i-1)}_{(i)}:\Sigma(L)_{(i)}\rightarrow \Sigma(L)_{(i-1)}$, defined in Subsection \ref{ss:surgeryexacttriangle}, is the branched double cover of the surface cobordism $S^{(i-1)}_{(i)}:L_{(i)}\rightarrow L_{(i-1)}$ given by the band attachments described in Figure \ref{fig:skeinmoves}. In general, $W^{(j)}_{(i)}$ and $S^{(j)}_{(i)}$ are given by the compositions of these elementary cobordisms. Let $\bar{S}^{(j)}_{(i)}:L_{(i)}\rightarrow L_{(j)}$ be the reverse of the cobordism  $S^{(j)}_{(i)}$. Then $S^{(i-2)}_{(i)}$, as a cobordism embedded in $I\times S^3$, is diffeomorphic to the connected sum of $\bar{S}^{(i)}_{(i+1)}$ and $S_0$, the standard embedding of $\mathbb{R}P^2$ in $S^4$ with self intersection 2 \cite{kronheimer2011khovanov}.

\begin{figure}[ht]
\begin{subfigure}[t]{.4\linewidth}
\centering
\tikz[scale=2]{
	\draw[<-, knot] (-0.5,1) -- node(e1)[pos=0.15]{} node(e2)[pos=0.25]{} (0.5,0);
	\draw[knot, overcross] (-0.5,0) -- (0.5,1);
    \draw[knot, ->] (-0.5,0) -- node(f1)[pos=0.85]{} node(f2)[pos=0.75]{}  (0.5,1);
	\node at (0, -0.5) {$L$};
    \draw[knot, red] (e1.center) to[out=45, in=135] (f1.center);
    \draw[knot, red, densely dotted] (e2.center) to[out=45, in=135] (f2.center);
    \shade[ball color = gray!10, opacity = 0.2] (0,0.5) circle (0.7071cm);
	\draw (0,0.5) circle (0.7071cm);
	\draw[very thin] (-0.7071cm,0.5) arc (180:360:0.7071 and 0.2);
	\draw[very thin, dashed] (0.7071cm,0.5) arc (0:180:0.7071 and 0.2);
}
\caption{A crossing of a link.}
\label{subfig:skeinmovesA}
\end{subfigure}%
\begin{subfigure}[t]{.6\linewidth}
\centering
\tikz[scale=2]{
	\draw[knot] (2,0) to[out=45, in=-45] node(a1)[pos=0.4]{} node(b1)[pos=0.6]{} (2,1);
	\draw[knot] (3,0) to[out=135, in=-135] node(a2)[pos=0.4]{} node(b2)[pos=0.6]{} (3,1);
    \draw[knot, red] (b1.center) to [out=0, in=180](a2.center);
    \node[white] at (2.5,0.5) {$\bullet$};
    \draw[knot, red, densely dotted] (a1.center) to[out=0, in=180] (b2.center);
	\node at (2.5, -0.5) {$L_0$};
    \shade[ball color = gray!10, opacity = 0.2] (2.5,0.5) circle (0.7071cm);
	\draw (2.5,0.5) circle (0.7071cm);
	\draw[very thin] (1.7929cm,0.5) arc (180:360:0.7071 and 0.2);
	\draw[very thin, dashed] (3.2071cm,0.5) arc (0:180:0.7071 and 0.2);
	\draw[knot] (4.5,0) to[out=45, in=135] node(c1)[pos=0.4]{} node(c2)[pos=0.6]{} (5.5,0);
	\draw[knot] (4.5,1) to[out=-45, in=-135] node(d1)[pos=0.4]{} node(d2)[pos=0.6]{} (5.5,1);
    \draw[knot, red] (c2.center) to[out=90, in=-90] (d2.center);
    \draw[knot, red, densely dotted] (c1.center) to[out=90, in=-90] (d1.center);
	\node at (5, -0.5) {$L_1$};
    \shade[ball color = gray!10, opacity = 0.2] (5,0.5) circle (0.7071cm);
	\draw (5,0.5) circle (0.7071cm);
	\draw[very thin] (4.2929cm,0.5) arc (180:360:0.7071 and 0.2);
	\draw[very thin, dashed] (5.7071cm,0.5) arc (0:180:0.7071 and 0.2);
}
\caption{The 0- and 1-resolutions of a crossing of a link.}
\label{subfig:skeinmovesB}
\end{subfigure}
\caption{}
\label{fig:skeinmoves}
\end{figure}

\subsection{Pseudo-diagrams and the plane knot complex}
\label{ss:pseudodiagrams}

In this section, we will describe Daemi's \emph{plane knot complex}, denoted $\widetilde{\mathrm{PKC}}(L)$, which is obtained from the plane Floer complex of the branched double cover of a link. Its filtered chain homotopy type is an invariant of $L$, but the invariance of the filtered chain homotopy type of $\widetilde{\mathrm{PKC}}(L)$ is understood in terms of \emph{pseudo-diagram} moves, rather than Reidemeister moves. Pseudo-diagrams were introduced by Kronheimer-Mrowka \cite{KMfiltrations}, and they offer flexibility which is useful for proving invariance for both instanton and plane Floer homology.

\begin{definition}
A \emph{pseudo-diagram} is a pair $(L, N)$ of a link $L$ and an ordered collection $N$ of crossings in a diagram of $L$ such that for all resolutions $v: N \to \{0, 1\}$, the link $L_v\subset \mathbb{R}^3$ is an unlink.
\end{definition}

For example, all planar diagrams are pseudo-diagrams, where $N$ is the collection of all crossings in the planar diagram; we denote this pseudo-diagram by $(L, D)$, for $D$ the planar diagram. Additionally, any choice of two crossings in a 3-crossing diagram for the trefoil is a pseudo-diagram. However,
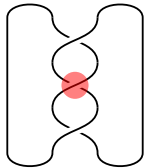
\[
\tikz[baseline={([yshift=-.5ex]current bounding box.center)}, scale=.6]{
	\draw[knot] (1,2) to[out=90, in=-90] (0,3);
	\draw[knot, overcross] (0,2) to[out=90, in=-90] (1,3);
	\draw[knot] (1,1) to[out=90, in=-90] (0,2);
	\draw[knot, overcross] (0,1) to[out=90, in=-90] (1,2);
	\draw[knot] (1,0) to[out=90, in=-90] (0,1);
	\draw[knot, overcross] (0,0) to[out=90, in=-90] (1,1);
	\draw[knot] (0,0) to[out=-90, in=-90] (-1,0) -- (-1,3) to[out=90, in=90] (0,3);
	\draw[knot] (1,0) to[out=-90, in=-90] (2,0) -- (2,3) to[out=90, in=90] (1,3);
	\fill[red, fill opacity=0.5] (0.5, 1.5) circle(0.3);
}
\]
is \emph{not} a pseudo-diagram; one resolution at the chosen crossing is the unknot, but the other resolution is the Hopf link. If a crossing of a planar diagram belongs to the set $N$, we highlight it in red as pictured above.

To prove invariance, one makes a further restriction:

\begin{definition}
A \emph{planar pseudo-diagram} is a pseudo-diagram $(L, N)$ that can be obtained by a planar diagram $(L, D)$ in any one of the following ways:
\begin{enumerate}
	\item $N = D$; \textit{i.e.}, $(L, N)$ is a planar diagram;
	\item $N$ is obtained by removing a crossing from $D$;
	\item $N$ is obtained by removing two consecutive under- or over-crossings from $D$.
\end{enumerate}
\end{definition}

Kronheimer and Mrowka \cite{KMfiltrations} proved the following topological statement.

\begin{proposition}
Any two planar pseudo-diagrams can be related to each other by a series of pseudo-diagram isotopies and adding and dropping crossings.
\end{proposition}

Let $N$ be a set of $n$ crossings for a link $L$. For any $\m = (m_1,\ldots, m_n) \in \mathbb{Z}^n$, the \emph{$\m$-resolution of $L$ relative $N$}, denoted $L_\m$, is the link obtained by performing an $m_i$-resolution of the crossing corresponding to the $i$th element in $N$ for each $i=1,\ldots, n$. For each crossing, one assigns a framed knot in $\Sigma(L)$ via the topological construction described in Section \ref{ss:topology}. Then, the branched double cover of $S^3$ branched along $L_\m$ is equivalent to surgery along the associated framed link in $\Sigma(L)$. To a pair $(\m, \n)$ with $\m \ge \n$, notice that the cobordism $W_\m^\n: \Sigma(L)_\m \to \Sigma(L)_\n$ is diffeomorphic to the branched double cover of a cobordism $S_\m^\n: L_\m \to L_\n$, described as follows.
\begin{enumerate}[label=(\roman*)]
	\item If $\left| \m - \n \right|_1 = 1$, then we assign the surface cobordism described in Subsection \ref{ss:topology}.
	\item Otherwise, $S_\m^\n$ is defined as the composition of $\left| \m - \n \right|_1$-many cobordisms of the previous type.
\end{enumerate} 
If $N$ is a pseudo-diagram for $L$, then for any $\m, \n \in \mathbb{Z}^n$ we may define $\sigma(\m, \n) \in 2\mathbb{Z}$ as the self-intersection number of the cobordism $S_\m^\n: L_\m \to L_\n$; this number is well-defined since $L_\m$ and $L_\n$ are both unlinks.

Finally, let $L$ be a link presented via a planar pseudo-diagram $N$ with $n$ crossings. For each crossing, one assigns a framed knot in $\Sigma(L)$ as described above. The collection of these framed knots is an $n$-component framed link which fits into the framework of Subsection \ref{ss:surgeryexacttriangle}. Define the \emph{plane knot complex} as
\[
\widetilde{\mathrm{PKC}}(L, N) := \bigoplus_{\mathbf{m}\in \{0,1\}^n} \widetilde{\mathrm{PFC}}(\Sigma(L_\mathbf{m})),
\]
\emph{i.e.}, the direct sum of Morse complexes associated to $\mathcal{R}(\Sigma(L_{\mathbf{m}}))$, with differential $d_p$ defined in Subsection \ref{ss:surgeryexacttriangle}.

Let $\mathbf{o} \in \{0,1\}^n$ denote the vertex such that $L_\mathbf{o}$ is the oriented resolution of $L$. The pseudo-diagram $N$ can be used to define a function
\begin{align*}
&h_N: \{0,1\}^n \to \mathbb{Z}\\
&h_N(\m):=-|\m|_1+\frac{1}{2}\sigma(\m,\mathbf{o})
\end{align*}
which, in turn, can be used to define a grading, called the \emph{homological degree}, on the chain complex $\mathrm{\widetilde{PKC}}(L, N)$:
\[\mathrm{deg}_{h}|_{\widetilde{\mathrm{PFC}}(\Sigma(L_m))}= h(\m)+n_{-}(N).
\]
Daemi also defines another grading on $(\widetilde{\mathrm{PKC}}(L, N), d_p)$, called the \emph{$\delta$-grading}, by
\[
\delta|_{\widetilde{\mathrm{PFC}}(\Sigma(L_m))} = -\frac{1}{2} \deg_p - \frac{1}{2} \left|m\right|_1 - \frac{1}{4} \sigma(\m, \mathbf{o}) + \frac{1}{2} n_+(N).
\]
One can verify that $\deg_\delta \in \mathbb{Z}\left[\tfrac{1}{2}\right]$, and that $2\deg_\delta \equiv b_0(L) - 1~\left(\mathrm{mod}~2\right)$. Therefore, $\deg_\delta$ is integrally supported whenever $L$ is a knot.

$\widetilde{\mathrm{PKC}}(L, N)$ is a filtered chain complex with respect to $\mathrm{deg}_h$. Recall that a \emph{filtered chain complex} $(C,d)$ is a chain complex with a filtration,
\[
C = \mathcal{F}_0 C \supseteq \cdots \supseteq \mathcal{F}_{i-1}C \supseteq \mathcal{F}_i C \supseteq \mathcal{F}_{i+1} \supseteq \cdots \supseteq \{0\},
\]
of the module $C$ such that $d$ maps each filtration level $\mathcal{F}_i C$ to itself (that is, the differential does not decrease filtration level). Thus, for each $i\ge 0$, $(\mathcal{F}_i C, d)$ is a chain complex. Moreover, suppose that there exists an $r \in \mathbb{R}$ such that each of the chain complexes $(\mathcal{F}_i, d)$ are graded with grading set $\{r\} + \mathbb{Z}$, and that the differential increases this grading by 1. In this case, we call $(C,d)$ a \emph{filtered $\mathbb{Z}$-graded chain complex}.

A \emph{map $f: (C, d) \to (C', d')$ of filtered $\mathbb{Z}$-graded chain complexes of degree $(a,b) \in \mathbb{Z} \times \mathbb{R}$} is a module homomorphism of degree $b$ which sends $\mathcal{F}_i C$ to $\mathcal{F}_{i+a} C'$. Thus, a map of degree $(a, b)$ has degree $(a', b)$ for any $a' \le a$. We define a morphism in the category of filtered $\mathbb{Z}$-graded chain complexes to be a chain map or an anti-chain map of degree $(0,0)$. We declare two morphisms to be \emph{filtered chain homotopy equivalent} if there is a chain homotopy between them of degree $(-1,-1)$. 

In the case that $N$ is a planar diagram, the self-intersection $\sigma(\m,\bfo)$ is trivial, and $\mathrm{deg}_h$ is reduced to
\[
\mathrm{deg}_h=-|\m|_{1}+n_{-}(N),
\]
which is equal to the cube filtration up to a global shift. Daemi proved the following.

\begin{lemma}[\cite{daemi2015abelian}, Corollary 3.8]
	If $\mathbf{m}, \mathbf{n}\in \{0,1\}^n$ satisfy $\mathbf{m} \ge \mathbf{n}$, then $h(\mathbf{n}) \ge h(\mathbf{m})$.
\end{lemma}
Hence, the plane knot complex is a filtered chain complex with respect to $\mathrm{deg}_h$.

\begin{proposition}[\cite{daemi2015abelian}, Proposition 3.11]
Suppose that a planar pseudo-diagram $N$ of a link $L$ is fixed. Then the filtered chain homotopy type of $(\widetilde{\mathrm{PKC}}(L, N), d_p)$ does not depend on auxiliary choices; \textit{i.e.}, it doesn't depend on the choice of family of metrics on $W_\m^\n$, Morse-Smale functions on the relevant representation variety, perturbations of the ASD equation, or homology orientations.
\end{proposition}

Thus, to obtain a well-defined plane knot complex, we must prove that the filtered chain homotopy type of $(\widetilde{\mathrm{PKC}}(L, N), d_p)$ is independent of choice of planar pseudo-diagram. Assume that $N$ and $N'$ are two planar pseudo-diagrams for $L$. Invariance is clear if the planar pseudo-diagrams differ only by pseudo-diagram isotopy. Assume that $N'$ is obtained from $N$ by adding a crossing. Following Daemi, we order the crossings of $N'$ so that the crossing in $N' \setminus N$ is the first one in $N'$. Then, using the notation of Subsection \ref{ss:surgeryexacttriangle}, we have that
\[
(\widetilde{\mathrm{PKC}}(L, N), d_p) = (C_{\m(2)}^{\n(2)}, d_{\m(2)}^{\n(2)})
\qquad \text{and} \qquad
(\widetilde{\mathrm{PKC}}(L, N'), d_p) = (C_{\m(1)}^{\n(0)}, d_{\m(1)}^{\n(0)}),
\]
where $\m = (1, \ldots, 1)$ and $\n = (0, \ldots, 0)$. The homological and $\delta$-gradings on the two complexes differ in the following way.

\begin{proposition}[\cite{daemi2015abelian}, Proposition 3.12]
\label{prop:homdelta}
Upon adding a crossing, the homological and $\delta$-gradings on $(\widetilde{\mathrm{PKC}}(L, N'), d_p)$ are obtained from those on $(\widetilde{\mathrm{PKC}}(L, N), d_p)$ by 
\[
\deg_{\delta'} = \deg_\delta - 1
\qquad \text{and} \qquad
\deg_{h'} = \deg_h - 1.
\]
Thus, $C_{\m(2)}^{\n(2)}$ is equal to $\widetilde{\mathrm{PKC}}(N) [1,1]$.
\end{proposition}

Theorem \ref{thm:surgerytriangle} states that $(C_{\m(2)}^{\n(2)}, d_{\m(2)}^{\n(2)})$ and $(C_{\m(1)}^{\n(0)}, d_{\m(1)}^{\n(0)})$ are chain homotopy equivalent with explicit chain homotopies 
\[
\Phi:= \phi_2 = \left(g_{\m(2)}^{\n(2)},n_{\m(2)}^{\n(2)}\right)
\qquad \text{and} \qquad
\Psi:= \psi_2 = 
\begin{pmatrix}
	n_{\m(1)}^{\n(1)}\\
	g_{\m(0)}^{\n(0)}
\end{pmatrix}
\]
appropriated from the family of chain homotopies stated in Equation (\ref{eq:chainhomotopies}).

\begin{theorem}[\cite{daemi2015abelian}, Theorem 2]
\label{thm:invariancedaemi}
The filtered chain homotopy type of $\widetilde{\mathrm{PKC}}(L, N)$, as a filtered $\mathbb{Z}$-graded chain complex, depends only on $L$.
\end{theorem}

Following Theorem \ref{thm:invariancedaemi}, we write $\widetilde{\mathrm{PKC}}(L)$ for the plane knot complex associated to $L$, suppressing $N$. A filtered chain complex naturally induces a spectral sequence, and a filtered chain map induces a map between the spectral sequences. The surgery exact triangle (Theorem \ref{thm:sec}) implies that $\widetilde{\mathrm{PKC}}(L)\cong \widetilde{\mathrm{PFC}}(\Sigma(L),\zeta)$. One can show that $[\zeta]$ vanishes in this case. Applying the spectral sequence induced from the cube filtration, we have the following.

\begin{theorem}[\cite{daemi2015abelian}, Theorems 3 and 4]
\label{thm:maindaemi}
The homological filtration of $\widetilde{\mathrm{PKC}}(L)$ induces a spectral sequence $\{\widetilde{\mathrm{PKE}}_r(L)\}$ with coefficients in $\tilde{\Lambda}$. This spectral sequence $\{\widetilde{\mathrm{PKE}}_r(L)\}$ converges to $\widetilde{\mathrm{PFH}}(\Sigma(L))$. Moreover, $\widetilde{\mathrm{PKE}}_2(L)$ is chain homotopic to the reduced odd Khovanov homology of the mirror of $L$ with coefficients in $\tilde{\Lambda}$.
\end{theorem}

The identification of $\widetilde{\mathrm{PKE}}_2(L)$ with the reduced odd Khovanov homology of the mirror (with coefficients in $\widetilde{\Lambda}$) can be described as follows. Fix a planar diagram with $n$ crossings for $L$. At each vertex $\m \in \{0,1\}^n$, the $E_2$ page of the spectral sequence is the homology of $\widetilde{\mathrm{PFC}}(\Sigma(L_\m))$. Since $L_\m$ is an unlink, Proposition \ref{prop:unlinkisom} allows us to identify it with the reduced odd Khovanov homology of $-L_\m$. This identifies the underlying $\widetilde{\Lambda}$-modules of the two chain complexes. By the computations provided in Subsection \ref{single}, the differentials also agree, at least up to sign. The signs on the edges of the surgery hypercube are determined by a choice of homology orientation and an orientation on $G_\m^\n$. We saw that the sign-assignments on births, deaths, and merges are independent of the choice of homology orientation. The only elementary cobordism which depends on the choice of homology orientations are splits. To fix this choice on \emph{both} reduced odd Khovanov and plane Floer homology, decorate the diagram for $L$ by placing arrows at each of its crossings. There are two choices at each crossing, and both induce decorations on band attachments.

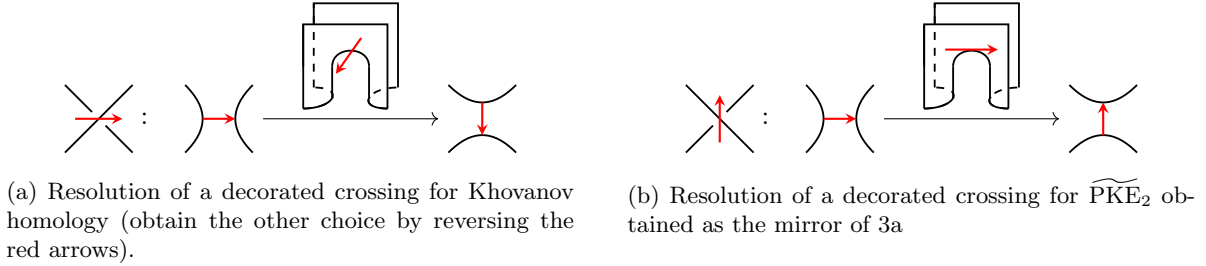
\begin{figure}[ht]
\begin{subfigure}[t]{.5\linewidth}
\centering
\tikz[]{
    \node(A) at (3,0) {$
    \tikz[baseline={([yshift=-.5ex]current bounding box.center)}, rotate=180, scale=0.9]{
    \draw[knot] (0,0) -- (1,1);
    \draw[knot] (1,0) -- (.6,.4);
    \draw[knot] (.4,.6) -- (0,1);
    \draw[stealth-, red, thick] (0.15,0.5) -- (0.85, 0.5);
    } :
    $};
    \node(B) at (4.5,0) {$
    \tikz[rotate=90, scale=0.9]{
    \draw[knot] (0,1) to[out=-45, in=-135, looseness=1.25] node(X)[pos=0.5]{} (1,1);
    \draw[knot] (0,0) to[out=45, in=135, looseness=1.25] node(Y)[pos=0.5]{} (1,0);
    \draw[-stealth, red, thick] (X.center) -- (Y.center);
    }
    $};
    \node(C) at (8,0) {$
    \tikz[rotate=0, scale=0.9]{
    \draw[knot] (0,1) to[out=-45, in=-135, looseness=1.25] node(X)[pos=0.5]{} (1,1);
    \draw[knot] (0,0) to[out=45, in=135, looseness=1.25] node(Y)[pos=0.5]{} (1,0);
    \draw[-stealth, red, thick] (X.center) -- (Y.center);
    }
    $};
    \draw[->] (B) -- node[above, pos=0.5]{$
    \tikz[baseline={([yshift=-.5ex]current bounding box.center)}, scale=0.28]{
	\draw[knot, -] (0,0) .. controls (1.5,0) and (2,1) .. (.5,1); 
    \draw[knot, -] (4,0) .. controls (2.5,0) and (3,1) .. (4.5,1); 
	\fill[fill=white] (.5,.55)  rectangle  (4.5,2);
	\draw[dashed, knot, -] (0,0) .. controls (1.5,0) and (2,1) .. (.5,1);
	\draw[dashed, knot, -] (4,0) .. controls (2.5,0) and (3,1) .. (4.5,1); 
	\draw[knot, -] (0,0) -- (0,4);
	\draw[dashed, knot, -] (.5,1) -- (.5,5);
	\draw[knot, -] (.5,4) -- (.5,5) -- (4.5,5) -- (4.5,4) -- (4.5,1);
	\draw[knot, -] (1.375,.5) -- (1.375,2);
	\draw[knot, -] (3.125,.5) -- (3.125,2);
	\draw[knot, -] (1.375,2) .. controls (1.375,3) and (3.125,3) .. (3.125,2);
	\draw[knot, -] (0,2) -- (0,4) -- (4,4) -- (4,0);
    \draw[-stealth, red, thick] (2.75,3.35) -- (1.55,1.65);
}
    $} (C);
}
\subcaption{Resolution of a decorated crossing for Khovanov homology (obtain the other choice by reversing the red arrows).}
\label{subfig:OKHdecors}
\end{subfigure}\hfill%
\begin{subfigure}[t]{.5\linewidth}
\centering
\tikz[]{
    \node(A) at (3,0) {$
    \tikz[baseline={([yshift=-.5ex]current bounding box.center)}, rotate=-90, scale=0.9]{
    \draw[knot] (0,0) -- (1,1);
    \draw[knot] (1,0) -- (.6,.4);
    \draw[knot] (.4,.6) -- (0,1);
    \draw[stealth-, red, thick] (0.15,0.5) -- (0.85, 0.5);
    } :
    $};
    \node(B) at (4.5,0) {$
    \tikz[rotate=90, scale=0.9]{
    \draw[knot] (0,1) to[out=-45, in=-135, looseness=1.25] node(X)[pos=0.5]{} (1,1);
    \draw[knot] (0,0) to[out=45, in=135, looseness=1.25] node(Y)[pos=0.5]{} (1,0);
    \draw[-stealth, red, thick] (X.center) -- (Y.center);
    }
    $};
    \node(C) at (8,0) {$
    \tikz[rotate=180, scale=0.9]{
    \draw[knot] (0,1) to[out=-45, in=-135, looseness=1.25] node(X)[pos=0.5]{} (1,1);
    \draw[knot] (0,0) to[out=45, in=135, looseness=1.25] node(Y)[pos=0.5]{} (1,0);
    \draw[-stealth, red, thick] (X.center) -- (Y.center);
    }
    $};
    \draw[->] (B) -- node[above, pos=0.5]{$
    \tikz[baseline={([yshift=-.5ex]current bounding box.center)}, scale=0.28]{
	\draw[knot, -] (0,0) .. controls (1.5,0) and (2,1) .. (.5,1); 
    \draw[knot, -] (4,0) .. controls (2.5,0) and (3,1) .. (4.5,1); 
	\fill[fill=white] (.5,.55)  rectangle  (4.5,2);
	\draw[dashed, knot, -] (0,0) .. controls (1.5,0) and (2,1) .. (.5,1);
	\draw[dashed, knot, -] (4,0) .. controls (2.5,0) and (3,1) .. (4.5,1); 
	\draw[knot, -] (0,0) -- (0,4);
	\draw[dashed, knot, -] (.5,1) -- (.5,5);
	\draw[knot, -] (.5,4) -- (.5,5) -- (4.5,5) -- (4.5,4) -- (4.5,1);
	\draw[knot, -] (1.375,.5) -- (1.375,2);
	\draw[knot, -] (3.125,.5) -- (3.125,2);
	\draw[knot, -] (1.375,2) .. controls (1.375,3) and (3.125,3) .. (3.125,2);
	\draw[knot, -] (0,2) -- (0,4) -- (4,4) -- (4,0);
    \draw[-stealth, red, thick] (1,2.8) -- (3.5,2.8);
}
    $} (C);
}
\subcaption{Resolution of a decorated crossing for $\widetilde{\mathrm{PKE}}_2$ obtained as the mirror of \ref{subfig:OKHdecors}}
\label{subfig:PFHdecors}
\end{subfigure}
\label{fig:decoratedresolutions}
\end{figure}
\noindent For plane Floer homology, the decorations determine the choice of homology orientation, and for reduced odd Khovanov homology, they determine the choice of split map (if the cobordism involved is a split). To ensure that the differentials are well-defined, one must assign extra signs  determined by an orientation of $G_\m^\n$ to each edge map; see \cite[\S 2.1]{daemi2015abelian}. This should be thought of as the edge assignment required in the definition of odd Khovanov homology; see \cite[Section 2]{MR3071132}.

\section{Reidemeister moves}
\label{s:rmoves}

In this section, we prove Proposition \ref{prop:step1}. In the case of Morse moves, the proposition follows from the work presented in Subsection \ref{single}. For Reidemeister moves, the argument is more complicated, due to the fact that the invariance of the filtered chain homotopy type of $\widetilde{\mathrm{PKC}}(L)$ is understood in terms of \emph{pseudo-diagram} moves, rather than Reidemeister moves; see Subsection \ref{ss:pseudodiagrams}. Therefore, to verify Proposition \ref{prop:step1}, we must compute the $E_2$ page of the plane Floer cobordism maps associated to Reidemeister moves from pseudo-diagram moves. In particular, we must ensure that these maps align with a choice of chain maps lifting odd Khovanov homology to a projective functor \cite{migdail2024functoriality}, up to sign.

\subsection{Reidemeister I moves}
\label{ss:r1}

Since the spectral sequence is induced by the $\deg_{h,N}$ filtration, any map which induces a nontrivial map on the $E_2$ page must preserve the $\deg_{h,N}$ grading. Therefore, we start by performing a computation to determine which maps are potentially nontrivial on the $E_2$ page due to these grading constraints. The maps which remain can be computed explicitly. Recall that, for a pseudo-diagram $N$ of a link $L$,
\[
\deg_{h,N}|_{\widetilde{\mathrm{PFC}}(\Sigma(L_\m))}=h_N(\m)+n_-(N).
\]
Assume that $N'$ is a pseudo-diagram obtained from $N$ by removing a crossing. From Proposition \ref{prop:homdelta}, we have that $\deg_{h,N'}-\deg_{h,N}=1$.

A Reidemeister I move can be described in terms of pseudo-diagram moves as the composition of pseudo-isotopies with adding/removing a single crossing. For example, adding a negative Reidemeister I move can be described as follows.
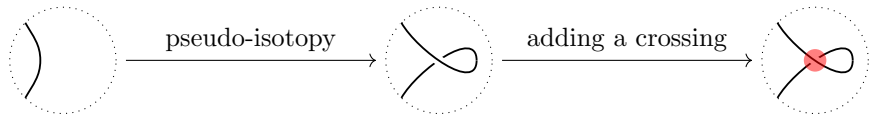
\[
\tikz[]{
	\node(A) at (0,0) {$
	\tikz[baseline={([yshift=-.5ex]current bounding box.center)}, scale=1]
        {
            \draw[dotted] (0.5,0.5) circle(0.707);
            \draw[knot] (0,0) to[out=60, in=-90] (0.2, 0.5) to[out=90, in=-60] (0,1);
        }
	$};
	\node(B) at (5,0) {$
	\tikz[baseline={([yshift=-.5ex]current bounding box.center)}, scale=1]
        {
            \draw[dotted] (.5,.5) circle(0.707);
            \draw[knot] (0,0) to[out=60, in =90] (1, 0.5);
            \draw[knot, overcross] (1, 0.5) to[out=-90, in=-60] (0,1);
            \draw[knot] (0.01, 0.99) -- (0,1);
        }
	$};
	\node(C) at (10,0) {$
	\tikz[baseline={([yshift=-.5ex]current bounding box.center)}, scale=1]
        {
            \draw[dotted] (.5,.5) circle(0.707);
            \draw[knot] (0,0) to[out=60, in =90] (1, 0.5);
            \draw[knot, overcross] (1, 0.5) to[out=-90, in=-60] (0,1);
            \draw[knot] (0.01, 0.99) -- (0,1);
            \fill[red, fill opacity=0.5] (0.5, 0.5) circle(0.15);
        }
	$};
	\draw[->] (A) -- node[above]{pseudo-isotopy} (B);
	\draw[->] (B) -- node[above]{adding a crossing} (C);
}
\]
The associated chain homotopy between plane Floer complexes is pictured in Figure \ref{fig:RIaddneg}. Recall that the chain map is given by $\Phi=(g_{\m(2)}^{\n(2)},n_{\m(2)}^{\n(2)})$ where $g_{\m(2)}^{\n(2)}=\{\tilde{f}_{G_{\bfk(2)}^{\bfk'(1)}}\}_{\m \geq \bfk \geq \bfk' \geq \n}$ and $n_{\m(2)}^{\n(2)}=\{\tilde{f}_{N_{\bfk(2)}^{\bfk'(0)}}\}_{\m \geq \bfk \geq \bfk' \geq \n}$. 
\begin{figure}[h]
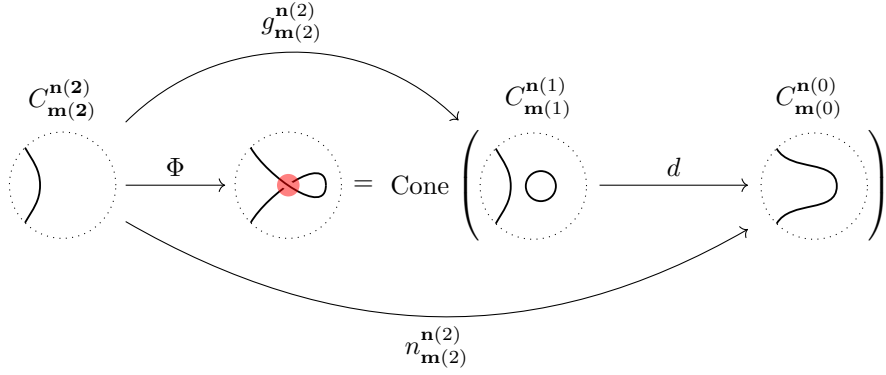

\centering
\tikz[baseline={([yshift=-.5ex]current bounding box.center)}]
{
    \node(A) at (0,0) {$
        \tikz[baseline={([yshift=-.5ex]current bounding box.center)}, scale=1]
        {
            \draw[dotted] (0.5,0.5) circle(0.707);
            \draw[knot] (0,0) to[out=60, in=-90] (0.2, 0.5) to[out=90, in=-60] (0,1);
        }
        $};
    \node(B) at (3,0) {$
        \tikz[baseline={([yshift=-.5ex]current bounding box.center)}, scale=1]
        {
            \draw[dotted] (.5,.5) circle(0.707);
            \draw[knot] (0,0) to[out=60, in =90] (1, 0.5);
            \draw[knot, overcross] (1, 0.5) to[out=-90, in=-60] (0,1);
            \draw[knot] (0.01, 0.99) -- (0,1);
            \fill[red, fill opacity=0.5] (0.5, 0.5) circle(0.15);
        }
        $};
    \node at (4,0) {$=$};
    \node at (4.75,0) {$\mathrm{Cone}$};
    \node(C) at (6.125,0) {$
        \left(
        \tikz[baseline={([yshift=-.5ex]current bounding box.center)}, scale=1]
        {
            \draw[dotted] (0.5,0.5) circle(0.707);
            \draw[knot] (0,0) to[out=60, in=-90] (0.2, 0.5) to[out=90, in=-60] (0,1);
            \draw[knot] (0.6, 0.5) circle(0.2);
        }
        \right.
    $};
    \node(D) at (10.125,0) {$
        \left.
        \tikz[baseline={([yshift=-.5ex]current bounding box.center)}, scale=1]
        {
            \draw[dotted] (.5,.5) circle(0.707);
            \draw[knot] (0,0) to[out=60, in=-90] (.8, .5) to[out=90, in=-60] (0,1);
        }
        \right)
    $};
    \node[above=-1mm] at (A.north) {$C_{\mathbf{m(2)}}^{\mathbf{n(2)}}$};
    \node[above=-1mm, xshift=2mm] at (C.north) {$C_{\mathbf{m}(1)}^{\mathbf{n}(1)}$};
    \node[above=-1mm, xshift=-2mm] at (D.north) {$C_{\mathbf{m}(0)}^{\mathbf{n}(0)}$};
    \draw[->] (A) to node[pos=0.5, above]{$\Phi$} (B);  
    \draw[->] (C) to node[pos=0.5, above]{$d$} (D);
    \draw[->] (A) to[out=45, in=135] node[pos=0.5, above]{$g_{\mathbf{m}(2)}^{\mathbf{n}(2)}$} (C);
    \draw[->] (A) to[out=-30, in=-150] node[pos=0.5, below]{$n_{\mathbf{m}(2)}^{\mathbf{n}(2)}$} (D);
}
\caption{Adding a crossing for the negative Reidemeister I move}
\label{fig:RIaddneg}
\end{figure}

We first calculate the grading change for $\tilde{f}_{N_{\mathbf{k}(2)}^{\mathbf{k'}(0)}}$; let $\mathbf{k} \geq \mathbf{k'}$. We compute
\begin{equation}
\label{eq:RIdegree}
\begin{aligned}
\deg_{h,N}(\mathbf{k}(2))-\deg_{h,N'}(\mathbf{k'}(0))&=\deg_{h,N'}(\bfk(2))+1-\deg_{h,N'}(\bfk'(0))\\
&=h_{N'}(\bfk(2))+1-h_{N'}(\bfk'(0))\\
&=-|\bfk(2)|_1+\frac{1}{2}\sigma(\bfk(2),0)+1-(-|\bfk'(0)|_1+\frac{1}{2}\sigma(\bfk'(0),0))\\
&=-|\bfk(2)-\bfk'(0)|_1+\frac{1}{2}\sigma(\bfk(2),\bfk'(0))+1\\
&=-|\bfk-\bfk'|_1-2+\frac{1}{2}\left(\sigma(\bfk(2),\bfk'(2))+\sigma(\bfk'(2),\bfk'(0))\right)+1\\
&=-|\mathbf{k}-\mathbf{k'}|_1-1+\frac{1}{2}(S_{\mathbf{k}(2)}^{\mathbf{k'}(2)}\cdot S_{\mathbf{k}(2)}^{\mathbf{k'}(2)}+S_{\mathbf{k'}(2)}^{\mathbf{k'}(0)}\cdot S_{\mathbf{k'}(2)}^{\mathbf{k'}(0)}).
\end{aligned}
\end{equation}
Since $N$ is a planar diagram, we have that $S_{\bfk(2)}^{\bfk'(2)} \cdot S_{\bfk(2)}^{\bfk'(2)}=0$. One can also verify that, since the crossing is negative, $S_{\bfk'(2)}^{\bfk'(0)} \cdot S_{\bfk'(2)}^{\bfk'(0)}=0$. Hence, $\deg_{h,N}(\bfk(2))-\deg_{h,N'}(\bfk'(0))<0$, so $\tilde{f}_{N_{\bfk(2)}^{\bfk'(0)}}$ cannot be grading preserving, and therefore does not survive to the $E_2$ page.

The grading change for $\tilde{f}_{G_{\bfk(2)}^{\bfk'(1)}}$ is given by

\begin{align*}
\deg_{h,N}(\bfk(2))-\deg_{h,N'}(\bfk'(1))&=\deg_{h,N'}(\bfk(2))+1-\deg_{h,N'}(\bfk'(1))\\
&=h_{N'}(\bfk(2))+1-h_{N'}(\bfk'(1))\\
&=-|\bfk(2)|_1+\frac{1}{2}\sigma(\bfk(2),0)+1-(-|\bfk'(1)|_1+\frac{1}{2}\sigma(\bfk'(1),0))\\
&=-|\bfk(2)-\bfk'(1)|_1+\frac{1}{2}\sigma(\bfk(2),\bfk'(1))+1\\
&=-|\bfk-\bfk'|_1-1+0+1\\
&=-|\bfk-\bfk'|_1.
\end{align*}
Hence, $\tilde{f}_{G_{\bfk(2)}^{\bfk'(1)}}$ is grading preserving only when $\bfk=\bfk'$. Denote the corresponding map by $\mathcal{E}_2(\tilde{f}_{G_{\bfk(2)}^{\bfk(1)}})$. Since $\bfk=\bfk'$, $\tilde{f}_{G_{\bfk(2)}^{\bfk'(1)}}$ is defined locally as the cobordism map from the 2-resolution to the 1-resolution of this negative crossing. The surface cobordism, described as a band-attachment in Subsection \ref{ss:topology}, is a split. The corresponding map is $F_\mathcal{P}$, which we computed explicitly in Subsection \ref{single}. The reader may verify that this map is exactly the homotopy equivalence for the positive (taking mirrors) Reidemeister I move on the odd Khovanov homology of $L$, up to sign.

The above argument can be used as an outline for all Reidemeister I moves. That is, we first show that one of $\mathcal{E}_2(g)$ or $\mathcal{E}_2(n)$ vanishes due to grading considerations, and then we compute $\mathcal{E}_2$ of the surviving map, up to sign. The following tables summarize the cases for Reidemeister I moves.

\begin{table}[h!]
\begin{minipage}{.45\linewidth}
\centering
\begin{tabular}{|c|c|c|}
\hline
 & $\mathcal{E}_2(g)$ & $\mathcal{E}_2(n)$ \\ \hline
 $\mathsf{RI}^+$ & 0 & 
$
\tikz[baseline={([yshift=-.5ex]current bounding box.center)}, scale=.35]{
    \draw[white] (-0.5,-1) rectangle (3.5,3);
    \draw[knot, -] (0,2) .. controls (0,1.75) and (1,1.75) .. (1,2);
    \draw[knot, -] (0,2) .. controls (0,2.25) and (1,2.25) .. (1,2);
    \draw[knot, -] (2,2) .. controls (2,1.75) and (3,1.75) .. (3,2);
    \draw[knot, -] (2,2) .. controls (2,2.25) and (3,2.25) .. (3,2);
    \draw[knot, -] (0,0) .. controls (0,-.25) and (1,-.25) .. (1,0);
    \draw[knot, -] (1,2) -- (1,0);
    \draw[knot, -] (2,2) .. controls (2,1) and (3,1) .. (3,2);
    \filldraw[white] (-0.1,-0.2) rectangle (0.46,2);
    \filldraw[white] (-0.1,2) rectangle (0.21, 2.2);
    \draw[knot] (0.5,-0.2) -- (0.5, 1.8);
    \draw[knot] (0.25,0.15) -- (0.25,2.15);
    \draw[knot] (0.5,0.2) to[out=180,in=15] (0.25,0.17);
}
$
\\ \hline
$\mathsf{RI}^-$ & 
$
\tikz[baseline={([yshift=-.5ex]current bounding box.center)}, scale=.35]{
    \draw[white] (-0.5,-1) rectangle (3.5,3);
    \draw[knot, -] (0,2) .. controls (0,1.75) and (1,1.75) .. (1,2);
    \draw[knot, -] (0,2) .. controls (0,2.25) and (1,2.25) .. (1,2);
    \draw[knot, -] (2,2) .. controls (2,1.75) and (3,1.75) .. (3,2);
    \draw[knot, -] (2,2) .. controls (2,2.25) and (3,2.25) .. (3,2);
    \draw[knot, -] (0,0) .. controls (0,-.25) and (1,-.25) .. (1,0);
    \draw[knot, -] (1,0) .. controls (1,0.65) and (3,0.65) .. (3,2);
    \draw[knot, -] (0,0) -- (0,2);
    \draw[knot, -] (1,2) .. controls (1,1.25) and (2,1.25) .. (2,2);
    \filldraw[white] (-0.1,-0.2) rectangle (0.46,2);
    \filldraw[white] (-0.1,2) rectangle (0.21, 2.2);
    \draw[knot] (0.5,-0.2) -- (0.5, 1.8);
    \draw[knot] (0.25,0.15) -- (0.25,2.15);
    \draw[knot] (0.5,0.2) to[out=180,in=15] (0.25,0.17);
}
$
& 0  \\ \hline
\end{tabular}
\caption{Adding a crossing}
\label{tab:RIadd}
\end{minipage}
\begin{minipage}{.45\linewidth}
\centering
\begin{tabular}{|c|c|c|}
\hline
 & $\mathcal{E}_2(g^{-1})$ & $\mathcal{E}_2(n^{-1})$ \\ \hline
 $\mathsf{RI}^+$ & 
$
\tikz[baseline={([yshift=-.5ex]current bounding box.center)}, scale=.35]{
    \draw[white] (-0.5,-1) rectangle (3.5,3);
    \draw[knot, -] (0,2) .. controls (0,1.75) and (1,1.75) .. (1,2);
    \draw[knot, -] (0,2) .. controls (0,2.25) and (1,2.25) .. (1,2);
    \draw[knot, -] (2,0) .. controls (2,-0.25) and (3,-0.25) .. (3,0);
    \draw[knot, -, dashed] (2,0) .. controls (2,0.25) and (3,0.25) .. (3,0);
    \draw[knot, -] (0,0) .. controls (0,-.25) and (1,-.25) .. (1,0);
    %
    \draw[knot, -] (0,0) -- (0,2);
    \draw[knot, -] (1,0).. controls (1, 0.75) and (2, 0.75) .. (2,0);
    \draw[knot, -] (1,2) .. controls (1,1.35) and (3,1.35) .. (3,0);
    \filldraw[white] (-0.1,-0.2) rectangle (0.46,2);
    \filldraw[white] (-0.1,2) rectangle (0.21, 2.2);
    \draw[knot] (0.5,-0.2) -- (0.5, 1.8);
    \draw[knot] (0.25,0.15) -- (0.25,2.15);
    \draw[knot] (0.5,0.2) to[out=180,in=15] (0.25,0.17);
}
$
 & 0 \\ \hline
 $\mathsf{RI}^-$ & 0 &
 $
\tikz[baseline={([yshift=-.5ex]current bounding box.center)}, scale=.35]{
    \draw[white] (-0.5,-1) rectangle (3.5,3);
    \draw[knot, -] (0,2) .. controls (0,1.75) and (1,1.75) .. (1,2);
    \draw[knot, -] (0,2) .. controls (0,2.25) and (1,2.25) .. (1,2);
    \draw[knot, -] (2,0) .. controls (2,-0.25) and (3,-0.25) .. (3,0);
    \draw[knot, -, dashed] (2,0) .. controls (2,0.25) and (3,0.25) .. (3,0);
    \draw[knot, -] (0,0) .. controls (0,-.25) and (1,-.25) .. (1,0);
    \draw[knot, -] (1,2) -- (1,0);
    \draw[knot, -] (2,0) .. controls (2,1) and (3,1) .. (3,0);
    \filldraw[white] (-0.1,-0.2) rectangle (0.46,2);
    \filldraw[white] (-0.1,2) rectangle (0.21, 2.2);
    \draw[knot] (0.5,-0.2) -- (0.5, 1.8);
    \draw[knot] (0.25,0.15) -- (0.25,2.15);
    \draw[knot] (0.5,0.2) to[out=180,in=15] (0.25,0.17);
}
$
   \\ \hline
\end{tabular}
\caption{Removing a crossing}
\label{tab:RIrem}
\end{minipage}
\end{table}

The reader can verify that these maps are, up to sign, exactly the chain homotopy equivalences presented in the definition of the odd Khovanov 2-functor; see, for example, diagrams (38) and (39) in \cite{migdail2024functoriality}. Note that the parity of the crossings in each Reidemeister I move is opposite Migdail and Wehrli's; this is because the spectral sequences requires that we take the mirror image of the link on the $E_\infty$-page. The remaining vanishing cases are left to the reader (they are exactly as in Equation (\ref{eq:RIdegree})). We conclude by calculating the non-vanishing entries of Tables \ref{tab:RIadd} and \ref{tab:RIrem}.

Consider the case of adding a positive Reidemeister I crossing (see the first row of Table \ref{tab:RIadd}). By grading considerations, the only map which survives to the $E_2$ page is $\tilde{f}_{N_{\bfk(2)}^{\bfk(0)}}$. The cobordism from $\Sigma(L)_{\bfk(2)}$ to $\Sigma(L)_{\bfk(0)}$ is described in Case 2 of Subsection \ref{ss:surgeryexacttriangle}. Explicitly, on the surface level, we first attach a band to $L_{\bfk(2)}$, as described in Section \ref{ss:topology}.  One can verify that the branched double cover is given by a 2-handle attachment along a $-1$ framed unknot $\mathcal{U}$. Then, glue a $-1$ framed 2-handle along the dual of $\mathcal{U}$. The composition of these cobordisms is the 4-manifold $U = D^2 \times S^2 \# \overline{\mathbb{C}P^2} \setminus D^4$ studied in Subsection \ref{family}. One can also verify that the family of metrics assigned to it in Subsection \ref{ss:surgeryexacttriangle} is $\mathbb{U}$. We showed that the induced map on plane Floer homology was the birth map of Subsection \ref{single}, up to the unit $c$ of $\widetilde{\Lambda}$.

Now we consider the cases where we are removing a Reidemeister I crossing. The homotopy equivalence on plane Floer homology is given by $\Psi=\begin{pmatrix}
	n_{\m(1)}^{\n(1)}\\
	g_{\m(0)}^{\n(0)}
\end{pmatrix}$; consider Figure \ref{fig:RIrempos}. Removing a positive Reidemeister I crossing (see the first row of Table \ref{tab:RIrem}) is the other case when $g$ is the surviving map.
\begin{figure}[ht]
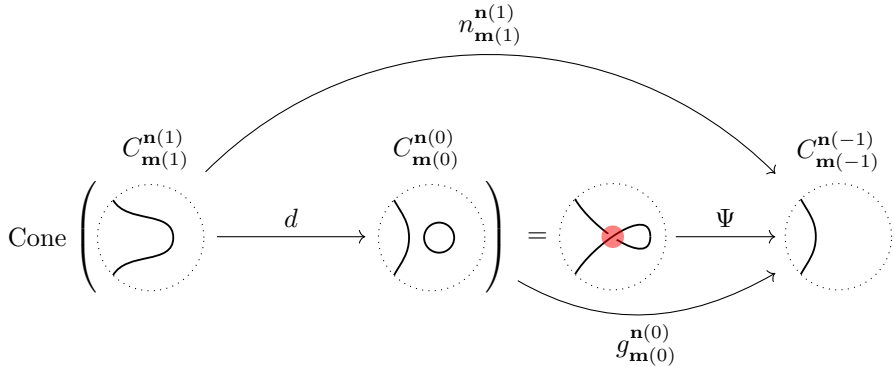

\centering
\tikz[baseline={([yshift=-.5ex]current bounding box.center)}]
{
    \node(A) at (10.65,0) {$
        \tikz[baseline={([yshift=-.5ex]current bounding box.center)}, scale=1]
        {
            \draw[dotted] (0.5,0.5) circle(0.707);
            \draw[knot] (0,0) to[out=60, in=-90] (0.2, 0.5) to[out=90, in=-60] (0,1);
        }
        $};
    \node(B) at (7.65,0) {$
        \tikz[baseline={([yshift=-.5ex]current bounding box.center)}, scale=1]
        {
            \draw[dotted] (.5,.5) circle(0.707);
            \draw[knot] (1, 0.5) to[out=-90, in=-60] (0,1);
            \draw[knot, overcross] (0,0) to[out=60, in =90] (1, 0.5);
            \draw[knot] (0.01, 0.99) -- (0,1);
            \fill[red, fill opacity=0.5] (0.5, 0.5) circle(0.15);
        }
        $};
    \node at (6.65,0) {$=$};
    \node at (0,0) {$\mathrm{Cone}$};
    \node(C) at (1.375,0) {$
        \left(
        \tikz[baseline={([yshift=-.5ex]current bounding box.center)}, scale=1]
        {
            \draw[dotted] (.5,.5) circle(0.707);
            \draw[knot] (0,0) to[out=60, in=-90] (.8, .5) to[out=90, in=-60] (0,1);
        }
        \right.
    $};
    \node(D) at (5.375,0) {$
        \left.
        \tikz[baseline={([yshift=-.5ex]current bounding box.center)}, scale=1]
        {
            \draw[dotted] (0.5,0.5) circle(0.707);
            \draw[knot] (0,0) to[out=60, in=-90] (0.2, 0.5) to[out=90, in=-60] (0,1);
            \draw[knot] (0.6, 0.5) circle(0.2);
        }
        \right)
    $};
    \node[above=-1mm] at (A.north) {$C_{\mathbf{m}(-1)}^{\mathbf{n}(-1)}$};
    \node[above=-1mm, xshift=2mm] at (C.north) {$C_{\mathbf{m}(1)}^{\mathbf{n}(1)}$};
    \node[above=-1mm, xshift=-2mm] at (D.north) {$C_{\mathbf{m}(0)}^{\mathbf{n}(0)}$};
    \draw[<-] (A) to node[pos=0.5, above]{$\Psi$} (B);  
    \draw[->] (C) to node[pos=0.5, above]{$d$} (D);
    \draw[<-] (A) to[out=135, in=45] node[pos=0.5, above]{$n_{\m(1)}^{\n(1)}$} (C);
    \draw[<-] (A) to[out=-150, in=-30] node[pos=0.5, below]{$g_{\m(0)}^{\n(0)}$} (D);
}
\caption{Removing a crossing for the positive Reidemeister I move}
\label{fig:RIrempos}
\end{figure}
It is straight forward to check that the band attachment inducing the surface cobordism from $L_{\bfk(0)}$ to $L_{\bfk(-1)}$ is a merge. We computed the corresponding map $F_{\overline{\mathcal{P}}}$ on plane Floer homology in Subsection \ref{single}.

Finally, we consider the case of removing a negative Reidemeister I crossing (see the second row of Table \ref{tab:RIrem}). By grading considerations, the only map which survives to the $E_2$ page is $\widetilde{f}_{N_{\bfk(1)}^{\bfk(-1)}}$. As before, we see that the cobordism is the 4-manifold $U$ from Subsection \ref{family}, this time turned around.  The induced map on plane Floer homology was the death map of Subsection \ref{single}, up to unit of $\widetilde{\Lambda}$.

\subsection{Reidemeister II moves}
\label{ss:r2}
In terms of pseudo-diagram moves, a Reidemeister II move can be described as consecutively adding or removing two crossings. The local picture for adding and removing crossings is given in Figures \ref{fig:RII+l} and \ref{fig:RII-l}.

\begin{figure}[h]
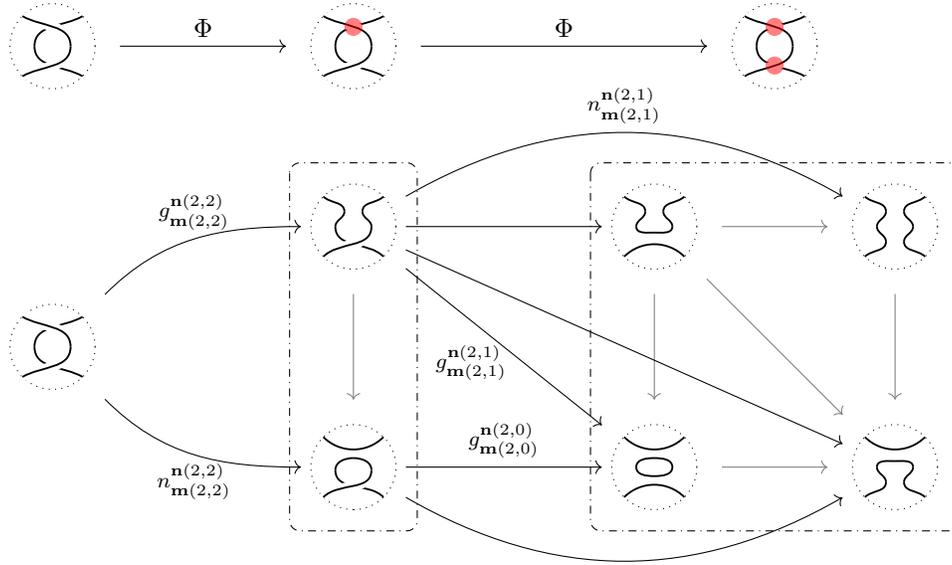

	\centering
	\tikz[baseline={([yshift=-.5ex]current bounding box.center)}, scale=0.8]
	{
		\node(22) at (0,0) {$
			\tikz[baseline={([yshift=-.5ex]current bounding box.center)}, scale=0.8]
			{
				\draw[dotted] (.5,.5) circle(0.707);
				\draw[knot] (1,0) to[out=150, in=-90] (0.2, 0.5);
				\draw[knot] (0.2, 0.5) to[out=90, in=210] (1,1);
				\draw[knot, overcross] (0,0) to[out=30, in=-90] (0.8, 0.5);
				\draw[knot, overcross] (0.8, 0.5) to[out=90, in=-30] (0,1);
				\draw[knot] (0.8, 0.499) -- (0.8, 0.501);
			}
			$};
		\node(12) at (5,2) {$
			\tikz[baseline={([yshift=-.5ex]current bounding box.center)}, scale=0.8]
			{
				\draw[dotted] (0.5, 0.5) circle (0.707);
				\draw[knot] (0,1) to[out=-45, in=90] (0.35, 0.75);
				\draw[knot] (0.35, 0.75) to[out=-90, in=90] (0.2, 0.5);
				\draw[knot] (0.2, 0.5) to[out=-90, in=135] (1,0);
				\draw[knot, overcross] (0.8, 0.5) to[out=-90, in=45] (0,0);
				\draw[knot] (0.8, 0.5) to[out=90, in=-90] (0.65, 0.75);
				\draw[knot] (0.65, 0.75) to[out=90, in=225] (1,1);
			}
			$};
		\node(02) at (5,-2) {$
			\tikz[baseline={([yshift=-.5ex]current bounding box.center)}, scale=0.8, rotate=-90]
			{
				\draw[dotted] (.5,.5) circle(0.707);
				\draw[knot] (0,0) to[out=45, in=-45] (0,1);
				\draw[knot] (0.55, 0.8) to[out=180, in=90] (0.35, 0.5);
				\draw[knot] (0.35, 0.5) to[out=-90, in=180] (0.55, 0.2);
				\draw[knot] (0.55, 0.2) to[out=0, in=225] (1, 1);
				\draw[knot, overcross] (1, 0) to[out=135, in=0] (0.55, 0.8);
			}
			$};
		\node(11) at (10,2) {$
			\tikz[baseline={([yshift=-.5ex]current bounding box.center)}, scale=0.8]
			{
				\draw[dotted] (0.5, 0.5) circle (0.707);
				\draw[knot] (0,0) to[out=45, in=135] (1,0);
				\draw[knot] (0,1) to[out=-45, in=90] (0.35, 0.75);
				\draw[knot] (0.35, 0.75) to[out=-90, in=90] (0.2, 0.5);
				\draw[knot] (0.2, 0.5) to[out=-90, in=180] (0.5, 0.4);
				\draw[knot] (0.5, 0.4) to[out=0, in=-90] (0.8, 0.5);
				\draw[knot] (0.8, 0.5) to[out=90, in=-90] (0.65, 0.75);
				\draw[knot] (0.65, 0.75) to[out=90, in=225] (1,1);
			}
			$};
		\node(10) at (14,2) {$
			\tikz[baseline={([yshift=-.5ex]current bounding box.center)}, scale=0.8]
			{
				\draw[dotted] (0.5, 0.5) circle (0.707);
				\draw[knot] (0,1) to[out=-45, in=90] (0.35, 0.75);
				\draw[knot] (0.35, 0.75) to[out=-90, in=90] (0.2, 0.5);
				\draw[knot] (0.2, 0.5) to[out=-90, in=90] (0.35, 0.25);
				\draw[knot] (0.35, 0.25) to[out=-90, in=45] (0,0);
				\draw[knot] (1,1) to[out=225, in=90] (0.65, 0.75);
				\draw[knot] (0.65, 0.75) to[out=-90, in=90] (0.8, 0.5);
				\draw[knot] (0.8, 0.5) to[out=-90, in=90] (0.65, 0.25);
				\draw[knot] (0.65, 0.25) to[out=-90, in=135] (1,0);
			}
			$};
		\node(01) at (10,-2) {$
			\tikz[baseline={([yshift=-.5ex]current bounding box.center)}, scale=0.8]
			{
				\draw[dotted] (0.5, 0.5) circle (0.707);
				\draw[knot] (0,0) to[out=45, in=135] (1,0);
				\draw[knot] (0,1) to[out=-45, in=-135] (1,1);
				\draw[knot] (0.5, 0.65) to[out=0, in=90] (0.8, 0.5) to[out=-90, in=0] (0.5, 0.35) to[out=180, in=-90] (0.2, 0.5) to[out=90, in=180] (0.5, 0.65);
			}
			$};
		\node(00) at (14,-2) {$
			\tikz[baseline={([yshift=-.5ex]current bounding box.center)}, scale=0.8]
			{
				\draw[dotted] (0.5, 0.5) circle (0.707);
				\draw[knot] (0,1) to[out=-45, in=-135] (1,1);
				\draw[knot] (0,0) to[out=45, in=-90] (0.35, 0.25);
				\draw[knot] (0.35, 0.25) to[out=90, in=-90] (0.2, 0.5);
				\draw[knot] (0.2, 0.5) to[out=90, in=180] (0.5, 0.6);
				\draw[knot] (0.5, 0.6) to[out=0, in=90] (0.8, 0.5);
				\draw[knot] (0.8, 0.5) to[out=-90, in=90] (0.65, 0.25);
				\draw[knot] (0.65, 0.25) to[out=-90, in=135] (1,0);
			}
			$};
		\node(1) at (0,5) {$
			\tikz[baseline={([yshift=-.5ex]current bounding box.center)}, scale=0.8]
			{
				\draw[dotted] (.5,.5) circle(0.707);
				\draw[knot] (1,0) to[out=150, in=-90] (0.2, 0.5);
				\draw[knot] (0.2, 0.5) to[out=90, in=210] (1,1);
				\draw[knot, overcross] (0,0) to[out=30, in=-90] (0.8, 0.5);
				\draw[knot, overcross] (0.8, 0.5) to[out=90, in=-30] (0,1);
				\draw[knot] (0.8, 0.499) -- (0.8, 0.501);	
			}
			$};
		\draw[dashdotted, rounded corners] (10-1.5*0.707,2+1.5*0.707) rectangle (14+1.5*0.707,-2-1.5*0.707);
		\node(2) at (5,5) {$
			\tikz[baseline={([yshift=-.5ex]current bounding box.center)}, scale=0.8]
			{
				\draw[dotted] (.5,.5) circle(0.707);
				\draw[knot] (1,0) to[out=150, in=-90] (0.2, 0.5);
				\draw[knot] (0.2, 0.5) to[out=90, in=210] (1,1);
				\draw[knot, overcross] (0,0) to[out=30, in=-90] (0.8, 0.5);
				\draw[knot, overcross] (0.8, 0.5) to[out=90, in=-30] (0,1);
				\draw[knot] (0.8, 0.499) -- (0.8, 0.501);
				\fill[red, fill opacity=0.5] (0.5, 0.8215) circle(0.15);
			}
			$};
		\draw[dashdotted, rounded corners] (5-1.5*0.707,2+1.5*0.707) rectangle (5+1.5*0.707,-2-1.5*0.707);
		\node(3) at (12,5) {$
			\tikz[baseline={([yshift=-.5ex]current bounding box.center)}, scale=0.8]
			{
				\draw[dotted] (.5,.5) circle(0.707);
				\draw[knot] (1,0) to[out=150, in=-90] (0.2, 0.5);
				\draw[knot] (0.2, 0.5) to[out=90, in=210] (1,1);
				\draw[knot, overcross] (0,0) to[out=30, in=-90] (0.8, 0.5);
				\draw[knot, overcross] (0.8, 0.5) to[out=90, in=-30] (0,1);
				\draw[knot] (0.8, 0.499) -- (0.8, 0.501);
				\fill[red, fill opacity=0.5] (0.5, 0.8215) circle(0.15);
				\fill[red, fill opacity=0.5] (0.5, 0.1785) circle(0.15);
			}
			$};
		\draw[->, shorten >=2mm, shorten <=2mm] (1) --node[above]{$\Phi$} (2);
		\draw[->, shorten >=2mm, shorten <=2mm] (2) --node[above]{$\Phi$} (3);
		\draw[->, shorten >=2mm, shorten <=2mm, gray] (12) -- (02);
		\draw[->, shorten >=2mm, shorten <=2mm, gray] (11) -- (10);
		\draw[->, shorten >=2mm, shorten <=2mm, gray] (11) -- (01);
		\draw[->, shorten >=2mm, shorten <=2mm, gray] (10) -- (00);
		\draw[->, shorten >=2mm, shorten <=2mm, gray] (01) -- (00);
		\draw[->, gray] (11) -- (00);
		\draw[->] (22) to[out=45, in=180] node[above]{\footnotesize$g_{\m(2,2)}^{\n(2,2)}$} (12);
		\draw[->] (22) to[out=-45, in=180] node[below]{\footnotesize$n_{\m(2,2)}^{\n(2,2)}$} (02);
		\draw[->] (12) -- (11);
		\draw[->] (12) to[out=30, in=150] node[above]{\footnotesize$n_{\m(2,1)}^{\n(2,1)}$} (10);
		\draw[->] (12) -- node[left=1mm, pos=0.6]{\footnotesize$g_{\m(2,1)}^{\n(2,1)}$} (01);
		\draw[->] (12) -- (00);
		\draw[->] (02) -- node[above]{\footnotesize$g_{\m(2,0)}^{\n(2,0)}$} (01);
		\draw[->] (02) to[out=-30, in=-150] (00);
	}
	\caption{Local picture of the plane knot complex for adding two crossings.}
	\label{fig:RII+l}
\end{figure}

\begin{figure}[h!]
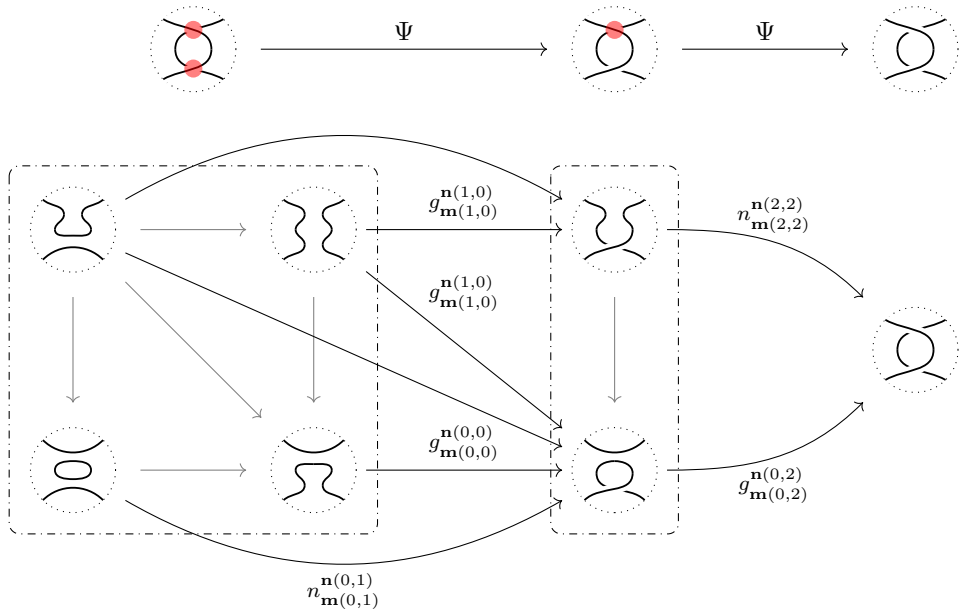

	\centering
	\tikz[baseline={([yshift=-.5ex]current bounding box.center)}, scale=0.8]
	{
		\node(22) at (14,0) {$
			\tikz[baseline={([yshift=-.5ex]current bounding box.center)}, scale=0.8]
			{
				\draw[dotted] (.5,.5) circle(0.707);
				\draw[knot] (1,0) to[out=150, in=-90] (0.2, 0.5);
				\draw[knot] (0.2, 0.5) to[out=90, in=210] (1,1);
				\draw[knot, overcross] (0,0) to[out=30, in=-90] (0.8, 0.5);
				\draw[knot, overcross] (0.8, 0.5) to[out=90, in=-30] (0,1);
				\draw[knot] (0.8, 0.499) -- (0.8, 0.501);
			}
			$};
		\node(12) at (9,2) {$
			\tikz[baseline={([yshift=-.5ex]current bounding box.center)}, scale=0.8]
			{
				\draw[dotted] (0.5, 0.5) circle (0.707);
				\draw[knot] (0,1) to[out=-45, in=90] (0.35, 0.75);
				\draw[knot] (0.35, 0.75) to[out=-90, in=90] (0.2, 0.5);
				\draw[knot] (0.2, 0.5) to[out=-90, in=135] (1,0);
				\draw[knot, overcross] (0.8, 0.5) to[out=-90, in=45] (0,0);
				\draw[knot] (0.8, 0.5) to[out=90, in=-90] (0.65, 0.75);
				\draw[knot] (0.65, 0.75) to[out=90, in=225] (1,1);
			}
			$};
		\node(02) at (9,-2) {$
			\tikz[baseline={([yshift=-.5ex]current bounding box.center)}, scale=0.8, rotate=-90]
			{
				\draw[dotted] (.5,.5) circle(0.707);
				\draw[knot] (0,0) to[out=45, in=-45] (0,1);
				\draw[knot] (0.55, 0.8) to[out=180, in=90] (0.35, 0.5);
				\draw[knot] (0.35, 0.5) to[out=-90, in=180] (0.55, 0.2);
				\draw[knot] (0.55, 0.2) to[out=0, in=225] (1, 1);
				\draw[knot, overcross] (1, 0) to[out=135, in=0] (0.55, 0.8);
			}
			$};
		\node(11) at (0,2) {$
			\tikz[baseline={([yshift=-.5ex]current bounding box.center)}, scale=0.8]
			{
				\draw[dotted] (0.5, 0.5) circle (0.707);
				\draw[knot] (0,0) to[out=45, in=135] (1,0);
				\draw[knot] (0,1) to[out=-45, in=90] (0.35, 0.75);
				\draw[knot] (0.35, 0.75) to[out=-90, in=90] (0.2, 0.5);
				\draw[knot] (0.2, 0.5) to[out=-90, in=180] (0.5, 0.4);
				\draw[knot] (0.5, 0.4) to[out=0, in=-90] (0.8, 0.5);
				\draw[knot] (0.8, 0.5) to[out=90, in=-90] (0.65, 0.75);
				\draw[knot] (0.65, 0.75) to[out=90, in=225] (1,1);
			}
			$};
		\node(10) at (4,2) {$
			\tikz[baseline={([yshift=-.5ex]current bounding box.center)}, scale=0.8]
			{
				\draw[dotted] (0.5, 0.5) circle (0.707);
				\draw[knot] (0,1) to[out=-45, in=90] (0.35, 0.75);
				\draw[knot] (0.35, 0.75) to[out=-90, in=90] (0.2, 0.5);
				\draw[knot] (0.2, 0.5) to[out=-90, in=90] (0.35, 0.25);
				\draw[knot] (0.35, 0.25) to[out=-90, in=45] (0,0);
				\draw[knot] (1,1) to[out=225, in=90] (0.65, 0.75);
				\draw[knot] (0.65, 0.75) to[out=-90, in=90] (0.8, 0.5);
				\draw[knot] (0.8, 0.5) to[out=-90, in=90] (0.65, 0.25);
				\draw[knot] (0.65, 0.25) to[out=-90, in=135] (1,0);
			}
			$};
		\node(01) at (0,-2) {$
			\tikz[baseline={([yshift=-.5ex]current bounding box.center)}, scale=0.8]
			{
				\draw[dotted] (0.5, 0.5) circle (0.707);
				\draw[knot] (0,0) to[out=45, in=135] (1,0);
				\draw[knot] (0,1) to[out=-45, in=-135] (1,1);
				\draw[knot] (0.5, 0.65) to[out=0, in=90] (0.8, 0.5) to[out=-90, in=0] (0.5, 0.35) to[out=180, in=-90] (0.2, 0.5) to[out=90, in=180] (0.5, 0.65);
			}
			$};
		\node(00) at (4,-2) {$
			\tikz[baseline={([yshift=-.5ex]current bounding box.center)}, scale=0.8]
			{
				\draw[dotted] (0.5, 0.5) circle (0.707);
				\draw[knot] (0,1) to[out=-45, in=-135] (1,1);
				\draw[knot] (0,0) to[out=45, in=-90] (0.35, 0.25);
				\draw[knot] (0.35, 0.25) to[out=90, in=-90] (0.2, 0.5);
				\draw[knot] (0.2, 0.5) to[out=90, in=180] (0.5, 0.6);
				\draw[knot] (0.5, 0.6) to[out=0, in=90] (0.8, 0.5);
				\draw[knot] (0.8, 0.5) to[out=-90, in=90] (0.65, 0.25);
				\draw[knot] (0.65, 0.25) to[out=-90, in=135] (1,0);
			}
			$};
		\node(1) at (2,5) {$
			\tikz[baseline={([yshift=-.5ex]current bounding box.center)}, scale=0.8]
			{
				\draw[dotted] (.5,.5) circle(0.707);
				\draw[knot] (1,0) to[out=150, in=-90] (0.2, 0.5);
				\draw[knot] (0.2, 0.5) to[out=90, in=210] (1,1);
				\draw[knot, overcross] (0,0) to[out=30, in=-90] (0.8, 0.5);
				\draw[knot, overcross] (0.8, 0.5) to[out=90, in=-30] (0,1);
				\draw[knot] (0.8, 0.499) -- (0.8, 0.501);	
				\fill[red, fill opacity=0.5] (0.5, 0.8215) circle(0.15);
				\fill[red, fill opacity=0.5] (0.5, 0.1785) circle(0.15);
			}
			$};
		\draw[dashdotted, rounded corners] (-1.5*0.707,2+1.5*0.707) rectangle (4+1.5*0.707,-2-1.5*0.707);
		\node(2) at (9,5) {$
			\tikz[baseline={([yshift=-.5ex]current bounding box.center)}, scale=0.8]
			{
				\draw[dotted] (.5,.5) circle(0.707);
				\draw[knot] (1,0) to[out=150, in=-90] (0.2, 0.5);
				\draw[knot] (0.2, 0.5) to[out=90, in=210] (1,1);
				\draw[knot, overcross] (0,0) to[out=30, in=-90] (0.8, 0.5);
				\draw[knot, overcross] (0.8, 0.5) to[out=90, in=-30] (0,1);
				\draw[knot] (0.8, 0.499) -- (0.8, 0.501);
				\fill[red, fill opacity=0.5] (0.5, 0.8215) circle(0.15);
			}
			$};
		\draw[dashdotted, rounded corners] (9-1.5*0.707,2+1.5*0.707) rectangle (9+1.5*0.707,-2-1.5*0.707);
		\node(3) at (14,5) {$
			\tikz[baseline={([yshift=-.5ex]current bounding box.center)}, scale=0.8]
			{
				\draw[dotted] (.5,.5) circle(0.707);
				\draw[knot] (1,0) to[out=150, in=-90] (0.2, 0.5);
				\draw[knot] (0.2, 0.5) to[out=90, in=210] (1,1);
				\draw[knot, overcross] (0,0) to[out=30, in=-90] (0.8, 0.5);
				\draw[knot, overcross] (0.8, 0.5) to[out=90, in=-30] (0,1);
				\draw[knot] (0.8, 0.499) -- (0.8, 0.501);
			}
			$};
		\draw[->, shorten >=2mm, shorten <=2mm] (1) --node[above]{$\Psi$} (2);
		\draw[->, shorten >=2mm, shorten <=2mm] (2) --node[above]{$\Psi$} (3);
		\draw[->, shorten >=2mm, shorten <=2mm, gray] (12) -- (02);
		\draw[->, shorten >=2mm, shorten <=2mm, gray] (11) -- (10);
		\draw[->, shorten >=2mm, shorten <=2mm, gray] (11) -- (01);
		\draw[->, shorten >=2mm, shorten <=2mm, gray] (10) -- (00);
		\draw[->, shorten >=2mm, shorten <=2mm, gray] (01) -- (00);
		\draw[->, gray] (11) -- (00);
        \draw[->] (12) to[out=0, in=135] node[above]{\footnotesize$n_{\m(2,2)}^{\n(2,2)}$} (22);
        \draw[->] (02) to[out=0, in=-135] node[below]{\footnotesize$g_{\m(0,2)}^{\n(0,2)}$} (22);
        \draw[->] (10) -- node[above]{\footnotesize$g_{\m(1,0)}^{\n(1,0)}$} (12);
        \draw[->] (10) -- node[above=0.4cm]{\footnotesize$g_{\m(1,0)}^{\n(1,0)}$} (02);
        \draw[->] (00) -- node[above]{\footnotesize$g_{\m(0,0)}^{\n(0,0)}$} (02);
        \draw[->] (01) to[out=-30, in=210] node[below]{\footnotesize$n_{\m(0,1)}^{\n(0,1)}$} (02);
        \draw[->] (11) to[out=30, in=150] (12);
        \draw[->] (11) -- (02);
	}
	\caption{Local picture of the plane knot complex for removing two crossings.}
	\label{fig:RII-l}
\end{figure}

We call crossings in the RII region \emph{local} and crossing outside of the RII region \emph{external}. We will first verify the claim for simple closures with no external crossings and then extend the argument to the general case of a knot diagram containing external crossings. Locally, the maps induced by a Reidemeister II move have either source or target equal to a surgery square coming from the RII diagram with two crossing activated. Therefore, the map may (and will) have more than one nontrivial component. Although we can perform a global sign change, we still need to determine the relative signs of these maps. We will first compute the maps up to overall sign, and then use algebraic constraints to fix the signs consistently. Assume that the crossings in an RII region always take one of the following forms:
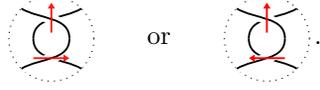
\[
\tikz[baseline={([yshift=-.5ex]current bounding box.center)}, scale=0.8]
	{
	\draw[dotted] (.5,.5) circle(0.707);
	\draw[knot] (1,0) to[out=150, in=-90] (0.2, 0.5);
	\draw[knot] (0.2, 0.5) to[out=90, in=210] (1,1);
	\draw[knot, overcross] (0,0) to[out=30, in=-90] (0.8, 0.5);
	\draw[knot, overcross] (0.8, 0.5) to[out=90, in=-30] (0,1);
	\draw[knot] (0.8, 0.499) -- (0.8, 0.501);	
	\draw[-{Triangle[scale=0.5]}, knot, red, opacity=0.9] (0.5,0.6) -- (0.5,1.1);
	\draw[-{Triangle[scale=0.5]}, knot, red, opacity=0.9] (0.2,0.175) -- (0.8, 0.175);
	}
\qquad \text{or} \qquad
\tikz[baseline={([yshift=-.5ex]current bounding box.center)}, scale=0.8]
	{
	\draw[dotted] (.5,.5) circle(0.707);
	\draw[knot] (1,0) to[out=150, in=-90] (0.2, 0.5);
	\draw[knot] (0.2, 0.5) to[out=90, in=210] (1,1);
	\draw[knot, overcross] (0,0) to[out=30, in=-90] (0.8, 0.5);
	\draw[knot, overcross] (0.8, 0.5) to[out=90, in=-30] (0,1);
	\draw[knot] (0.8, 0.499) -- (0.8, 0.501);	
	\draw[-{Triangle[scale=0.5]}, knot, red] (0.5,0.6) -- (0.5,1.1);
	\draw[{Triangle[scale=0.5]}-, knot, red] (0.2,0.175) -- (0.8, 0.175);
	}
\,.
\]
Then, fix an additional parameter $a\in\{1,-1\}$ where $a=1$ if the decorations of the RII diagram are as in the lefthand diagram, and $a=-1$ if decorations are as the righthand diagram.

We start our analysis by noticing that, regardless of the closure, the grading changes are as follows:
\begin{itemize}
	\item $\mathrm{deg}_{h}(\bfk(2,2))-\mathrm{deg}_{h}(\bfk'(1,1))=-|\bfk-\bfk'|_1+1$
	\item $\mathrm{deg}_{h}(\bfk(2,2))-\mathrm{deg}_{h}(\bfk'(1,0))=-|\bfk-\bfk'|_1$
	\item $\mathrm{deg}_{h}(\bfk(2,2))-\mathrm{deg}_{h}(\bfk'(0,1))=-|\bfk-\bfk'|_1$
	\item $\mathrm{deg}_{h}(\bfk(2,2))-\mathrm{deg}_{h}(\bfk'(0,0))=-|\bfk-\bfk'|_1-1$
\end{itemize}

\begin{remark}
\label{rmk:mattmagic}
Matt Stoffregen made us aware of an alternative, purely algebraic proof. The maps we compute for simple closures are filtered maps between complexes associated to unlinks, which have cube filtration concentrated at zero. There is a unique filtered map up to chain homotopy for such complexes. Hence, instead of performing explicit computations, one can pick a homotopy representative. This places strong algebraic constraints on the cases with external crossings, which can be used to prove Proposition \ref{lem:trivial}.
\end{remark}

\subsubsection{First simple closure} 

This closure yields the 2-component unlink; we convert Figure \ref{fig:RII+l} into Figure \ref{fig:RII+1}.
\vspace{-3mm}

\begin{figure}[htbp]
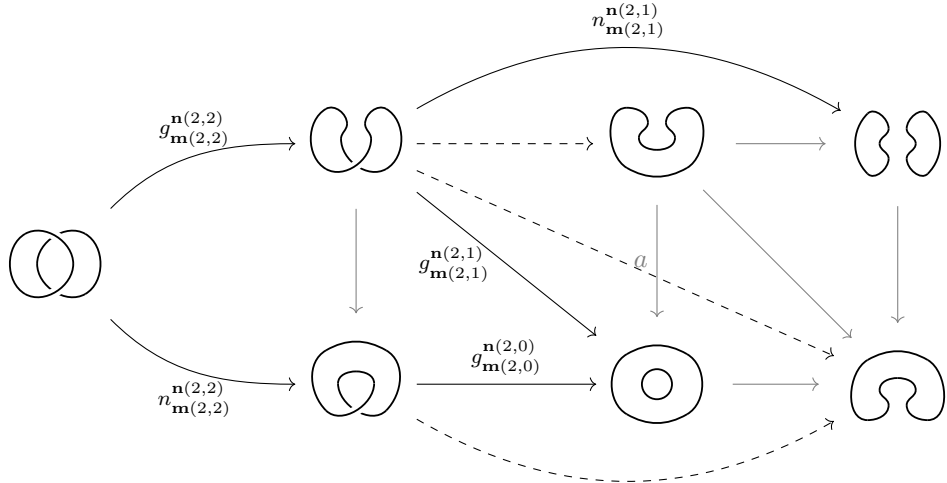

	\centering
\tikz[baseline={([yshift=-.5ex]current bounding box.center)}, scale=0.8]
	{
		\node(22) at (0,0) {$
			\tikz[baseline={([yshift=-.5ex]current bounding box.center)}, scale=0.8]
			{
				\draw[knot] (1,0) to[out=210, in=-90] (0.2, 0.5);
				\draw[knot] (0.2, 0.5) to[out=90, in=150] (1,1);
				\draw[knot, overcross] (0,0) to[out=-30, in=-90] (0.8, 0.5);
				\draw[knot, overcross] (0.8, 0.5) to[out=90, in=30] (0,1);
				\draw[knot] (0.8, 0.499) -- (0.8, 0.501);
                \draw[knot] (0,0) to[out=150, in=210] (0,1);
                \draw[knot] (1,0) to[out=30, in=-30] (1,1);
			}
			$};
		\node(12) at (5,2) {$
			\tikz[baseline={([yshift=-.5ex]current bounding box.center)}, scale=0.8]
			{
				\draw[knot] (0,1) to[out=30, in=90] (0.35, 0.75);
				\draw[knot] (0.35, 0.75) to[out=-90, in=90] (0.25, 0.5);
				\draw[knot] (0.25, 0.5) to[out=-90, in=210] (1,0);
				\draw[knot, overcross] (0.75, 0.5) to[out=-90, in=-30] (0,0);
				\draw[knot] (0.75, 0.5) to[out=90, in=-90] (0.65, 0.75);
				\draw[knot] (0.65, 0.75) to[out=90, in=150] (1,1);
                \draw[knot] (0,0) to[out=150, in=210] (0,1); 
                \draw[knot] (1,0) to[out=30, in=-30] (1,1);
			}
			$};
		\node(02) at (5,-2) {$
			\tikz[baseline={([yshift=-.5ex]current bounding box.center)}, scale=0.8, rotate=-90]
			{
				\draw[knot] (0,0) to[out=120, in=240] (0,1);
				\draw[knot] (0.55, 0.8) to[out=180, in=90] (0.35, 0.5);
				\draw[knot] (0.35, 0.5) to[out=-90, in=180] (0.55, 0.2);
				\draw[knot] (0.55, 0.2) to[out=0, in=-45] (1, 1);
				\draw[knot, overcross] (1, 0) to[out=45, in=0] (0.55, 0.8);
                \draw[knot] (0,0) to[out=300, in=-135] (1,0);
                \draw[knot] (0,1) to[out=60, in=135] (1,1);
			}
			$};
		\node(11) at (10,2) {$
			\tikz[baseline={([yshift=-.5ex]current bounding box.center)}, scale=0.8]
			{
				\draw[knot] (0,0) to[out=-30, in=-150] (1,0);
				\draw[knot] (0,1) to[out=0, in=90] (0.35, 0.75);
				\draw[knot] (0.35, 0.75) to[out=-90, in=90] (0.2, 0.5);
				\draw[knot] (0.2, 0.5) to[out=-90, in=180] (0.5, 0.3);
				\draw[knot] (0.5, 0.3) to[out=0, in=-90] (0.8, 0.5);
				\draw[knot] (0.8, 0.5) to[out=90, in=-90] (0.65, 0.75);
				\draw[knot] (0.65, 0.75) to[out=90, in=180] (1,1);
                \draw[knot] (0,0) to[out=-210, in=180] (0,1);
                \draw[knot] (1,0) to[out=-330, in=0] (1,1);
			}
			$};
		\node(10) at (14,2) {$
			\tikz[baseline={([yshift=-.5ex]current bounding box.center)}, scale=0.8]
			{
				\draw[knot] (0,1) to[out=45, in=90] (0.35, 0.75);
				\draw[knot] (0.35, 0.75) to[out=-90, in=90] (0.2, 0.5);
				\draw[knot] (0.2, 0.5) to[out=-90, in=90] (0.35, 0.25);
				\draw[knot] (0.35, 0.25) to[out=-90, in=-45] (0,0);
				\draw[knot] (1,1) to[out=135, in=90] (0.65, 0.75);
				\draw[knot] (0.65, 0.75) to[out=-90, in=90] (0.8, 0.5);
				\draw[knot] (0.8, 0.5) to[out=-90, in=90] (0.65, 0.25);
				\draw[knot] (0.65, 0.25) to[out=-90, in=-135] (1,0);
                \draw[knot] (0,0) to[out=135, in=-135] (0,1);
                \draw[knot] (1,0) to[out=45, in=-45] (1,1);
			}
			$};
		\node(01) at (10,-2) {$
			\tikz[baseline={([yshift=-.5ex]current bounding box.center)}, scale=0.8]
			{
				\draw[knot] (0,0) to[out=-30, in=-150] (1,0);
				\draw[knot] (0,1) to[out=30, in=150] (1,1);
                \draw[knot] (0.5,0.5) circle (0.25cm);
                \draw[knot] (0,0) to[out=-210, in=-150] (0,1);
                \draw[knot] (1,0) to[out=30, in=-30] (1,1);
			}
			$};
		\node(00) at (14,-2) {$
			\tikz[baseline={([yshift=-.5ex]current bounding box.center)}, scale=0.8, rotate=180]
			{
				\draw[knot] (0,0) to[out=-30, in=-150] (1,0);
				\draw[knot] (0,1) to[out=0, in=90] (0.35, 0.75);
				\draw[knot] (0.35, 0.75) to[out=-90, in=90] (0.2, 0.5);
				\draw[knot] (0.2, 0.5) to[out=-90, in=180] (0.5, 0.3);
				\draw[knot] (0.5, 0.3) to[out=0, in=-90] (0.8, 0.5);
				\draw[knot] (0.8, 0.5) to[out=90, in=-90] (0.65, 0.75);
				\draw[knot] (0.65, 0.75) to[out=90, in=180] (1,1);
                \draw[knot] (0,0) to[out=-210, in=180] (0,1);
                \draw[knot] (1,0) to[out=-330, in=0] (1,1);
			}
			$};
		\draw[->, shorten >=2mm, shorten <=2mm, gray] (12) -- (02);
		\draw[->, shorten >=2mm, shorten <=2mm, gray] (11) -- (10);
		\draw[->, shorten >=2mm, shorten <=2mm, gray] (11) -- (01) node[midway, left] {$a$};
		\draw[->, shorten >=2mm, shorten <=2mm, gray] (10) -- (00);
		\draw[->, shorten >=2mm, shorten <=2mm, gray] (01) -- (00);
		\draw[->, gray] (11) -- (00);
		\draw[->] (22) to[out=45, in=180] node[above]{\footnotesize$g_{\m(2,2)}^{\n(2,2)}$} (12);
		\draw[->] (22) to[out=-45, in=180] node[below]{\footnotesize$n_{\m(2,2)}^{\n(2,2)}$} (02);
		\draw[->, dashed] (12) -- (11);
		\draw[->] (12) to[out=30, in=150] node[above]{\footnotesize$n_{\m(2,1)}^{\n(2,1)}$} (10);
		\draw[->] (12) -- node[left=1mm]{\footnotesize$g_{\m(2,1)}^{\n(2,1)}$} (01);
		\draw[->, dashed] (12) -- (00);
		\draw[->] (02) -- node[above]{\footnotesize$g_{\m(2,0)}^{\n(2,0)}$} (01);
		\draw[->, dashed] (02) to[out=-30, in=-150] (00);
	}
	\caption{Chain maps on $\widetilde{\mathrm{PKC}}$ for a positive Reidemeister II move with the first closure.}
	\label{fig:RII+1}
\end{figure}

For grading reasons, the maps whose target is the $(0,0)$ resolution are trivial. Similarly, the map whose target is the $(1,1)$ resolution requires another change of resolution to preserve grading. Since we assume that there are no external crossings in this case, this map is also trivial. The remaining maps can be computed explicitly. Note that we have computed each of $g_{\m(2,2)}^{\n(2,2)},\, g_{\m(2,0)}^{\n(2,0)},\, n_{\m(2,2)}^{\n(2,2)}$, and $n_{\m(2,1)}^{\n(1,1)}$ in Subsection \ref{ss:r1}. For the diagonal map $g_{\m(2,1)}^{\n(2,1)}$, the cobordism $W_{\m(2,1)}^{\n(2,1)}$ is given by the branched double cover of the handle attachment in Figure \ref{fig:top1}. The family of metrics on $W_{\m(2,1)}^{\n(2,1)}$ is given by stretching along the two 2-spheres with self-intersection $-1$, which is identical the family of metrics $\mathbb{U}$ in Subsection \ref{family}. Hence, $g_{\m(2,1)}^{\n(2,1)}$ is also equal to a birth map up to a unit. The induced chain map is given in Figure \ref{fig:chainmapRII+1}. We equip these maps with appropriate signs so that the resulting map is a chain map. Note that the signs on the bottom surgery square depend on the decorations of the two crossings. However, once we fix this choice, the choice of signs on the chain maps is also fixed. Hence we omit the decorations on this figure. Also note that the cobordism maps require an identification of components; we identify the $(2,2)$ resolution with the bottoms of the outgoing cobordism maps from left-to-right.

\begin{figure}[htbp]
	\centering
\tikz[baseline={([yshift=-.5ex]current bounding box.center)}, scale=01]
    {
    \node (X) at (0,0) {
    \tikz[baseline={([yshift=-.5ex]current bounding box.center)}, scale=1.15*0.8]
    {
    \draw[line width=8pt] (1.5, 1) to[out=120, in=60] (-0.5, 1);
    \draw[line width=6.5pt, white] (1.5, 1) to[out=120, in=60] (-0.5, 1);
    \draw[line width=8pt] (1.75, 1) -- (1.75, 1.2) to[out=90, in=90] (-0.75, 1.2) -- (-0.75, 1);
    \draw[line width=6.5pt, white] (1.75, 1) -- (1.75, 1.2) to[out=90, in=90] (-0.75, 1.2) -- (-0.75, 1);
    \draw[knot, fill=white] (1,1) to[out=90, in=90] (2,1) -- (2,0) to[out=-90, in=-90] (1,0);
    \draw[knot, fill=white] (0,1) to[out=90, in=90] (-1,1) -- (-1,0) to[out=-90, in=-90] (0,0);
    \draw[knot] (1,0) to[out=90, in=-90] (0,1);
    \draw[knot, overcross] (0,0) to[out=90, in=-90] (1,1);
    }
    };
    \node (Y) at (5,0) {
    \tikz[baseline={([yshift=-.5ex]current bounding box.center)}, scale=0.6*0.95]{
    \begin{scope}[shift={(1,0)}]
    \draw[knot] (0.5,2.5) -- (0.5, 4.15) to[out=90, in=90] (3,4.15) -- (3,4) to[out=-90, in=90] (2.5,3.5) -- (2.5,2.5);
    \draw[knot, overcross] (1,2.5) -- (1,4) to[out=90, in=90] (2.5,4) to[out=-90, in=90] (3,3.5) -- (3,2.5);
    \end{scope}
    \coordinate (A) at (0.5, 4.15);
    \coordinate (B) at (1,4);
    \coordinate (C) at (3,4.15);
    \coordinate (D) at (2.5,4);
    \draw[fill=white, opacity=.9, draw=white, thick] (A) to [out=90, in=90] (C) -- (D) to[out=90, in=90] (B);
    \draw[knot, overcross] (0.5,2.5) -- (0.5, 4.15) to[out=90, in=90] (3,4.15) -- (3,4) to[out=-90, in=90] (2.5,3.5) -- (2.5,2.5);
    \draw[knot, overcross] (1,2.5) -- (1,4) to[out=90, in=90] (2.5,4) to[out=-90, in=90] (3,3.5) -- (3,2.5);
    \draw[knot, rounded corners=25pt, fill=white] (0,0) rectangle (4.5,3);
    }
    };
    \node (Z) at (12,0) {
    \tikz[baseline={([yshift=-.5ex]current bounding box.center)}, scale=0.8*0.8]
    {
    \draw[knot] (0,0) to[out=90, in=-90] (1,1);
    \draw[knot, overcross] (1,0) to[out=90, in=-90] (0,1);
    \draw[knot] (0,1) to[out=90, in=-90] (1,2);
    \draw[knot, overcross] (1,1) to[out=90, in=-90] (0,2);
    \draw[knot] (0,0) to[out=-90, in=-90] (-1,0);
    \draw[knot] (-1,0) -- (-1,2);
    \draw[knot] (-1,2) to[out=90, in=90] node[pos=0.5, above]{$-1$} (0,2);
    \draw[knot] (1,0) to[out=-90, in=-90] (2,0);
    \draw[knot] (2,0) -- (2,2);
    \draw[knot] (2,2) to[out=90, in=90] node[pos=0.5, above]{$-1$} (1,2);
    }
    };
    \draw[->] (X) to node[above]{\text{isotopy}} (Y);
    \draw[->] (Y) to node[above]{\text{2-fold branched cover}} (Z);
    }
	\caption{Cobordism for $g_{\m(2,1)}^{\n(2,1)}$.}
	\label{fig:top1}
\end{figure}

\begin{figure}[htbp]
	\centering
\tikz[baseline={([yshift=-.5ex]current bounding box.center)}, scale=1.5]
    {
    \node (A) at (0,0) {$
    \tikz[baseline={([yshift=-.5ex]current bounding box.center)}, scale=0.8]
			{
				\draw[knot] (0,0) to[out=-30, in=-150] (1,0);
				\draw[knot] (0,1) to[out=0, in=90] (0.35, 0.75);
				\draw[knot] (0.35, 0.75) to[out=-90, in=90] (0.2, 0.5);
				\draw[knot] (0.2, 0.5) to[out=-90, in=180] (0.5, 0.3);
				\draw[knot] (0.5, 0.3) to[out=0, in=-90] (0.8, 0.5);
				\draw[knot] (0.8, 0.5) to[out=90, in=-90] (0.65, 0.75);
				\draw[knot] (0.65, 0.75) to[out=90, in=180] (1,1);
                \draw[knot] (0,0) to[out=-210, in=180] (0,1);
                \draw[knot] (1,0) to[out=-330, in=0] (1,1);
			}
    $};
    \node (B) at (2.75,1) {$
    			\tikz[baseline={([yshift=-.5ex]current bounding box.center)}, scale=0.8]
			{
				\draw[knot] (0,1) to[out=45, in=90] (0.35, 0.75);
				\draw[knot] (0.35, 0.75) to[out=-90, in=90] (0.2, 0.5);
				\draw[knot] (0.2, 0.5) to[out=-90, in=90] (0.35, 0.25);
				\draw[knot] (0.35, 0.25) to[out=-90, in=-45] (0,0);
				\draw[knot] (1,1) to[out=135, in=90] (0.65, 0.75);
				\draw[knot] (0.65, 0.75) to[out=-90, in=90] (0.8, 0.5);
				\draw[knot] (0.8, 0.5) to[out=-90, in=90] (0.65, 0.25);
				\draw[knot] (0.65, 0.25) to[out=-90, in=-135] (1,0);
                \draw[knot] (0,0) to[out=135, in=-135] (0,1);
                \draw[knot] (1,0) to[out=45, in=-45] (1,1);
			}
    $};
    \node (C) at (4.25,-1) {$
    			\tikz[baseline={([yshift=-.5ex]current bounding box.center)}, scale=0.8]
			{
				\draw[knot] (0,0) to[out=-30, in=-150] (1,0);
				\draw[knot] (0,1) to[out=30, in=150] (1,1);
                \draw[knot] (0.5,0.5) circle (0.25cm);
                \draw[knot] (0,0) to[out=-210, in=-150] (0,1);
                \draw[knot] (1,0) to[out=30, in=-30] (1,1);
			}
    $};
    \node (D) at (7,0) {$
    			\tikz[baseline={([yshift=-.5ex]current bounding box.center)}, scale=0.8, rotate=180]
			{
				\draw[knot] (0,0) to[out=-30, in=-150] (1,0);
				\draw[knot] (0,1) to[out=0, in=90] (0.35, 0.75);
				\draw[knot] (0.35, 0.75) to[out=-90, in=90] (0.2, 0.5);
				\draw[knot] (0.2, 0.5) to[out=-90, in=180] (0.5, 0.3);
				\draw[knot] (0.5, 0.3) to[out=0, in=-90] (0.8, 0.5);
				\draw[knot] (0.8, 0.5) to[out=90, in=-90] (0.65, 0.75);
				\draw[knot] (0.65, 0.75) to[out=90, in=180] (1,1);
                \draw[knot] (0,0) to[out=-210, in=180] (0,1);
                \draw[knot] (1,0) to[out=-330, in=0] (1,1);
			}
    $};
    \node (E) at (3.5,3.15) {$
        			\tikz[baseline={([yshift=-.5ex]current bounding box.center)}, scale=0.8]
			{
				\draw[knot] (0,1) to[out=45, in=90] (0.35, 0.75);
				\draw[knot] (0.35, 0.75) to[out=-90, in=90] (0.2, 0.5);
				\draw[knot] (0.2, 0.5) to[out=-90, in=90] (0.35, 0.25);
				\draw[knot] (0.35, 0.25) to[out=-90, in=-45] (0,0);
				\draw[knot] (1,1) to[out=135, in=90] (0.65, 0.75);
				\draw[knot] (0.65, 0.75) to[out=-90, in=90] (0.8, 0.5);
				\draw[knot] (0.8, 0.5) to[out=-90, in=90] (0.65, 0.25);
				\draw[knot] (0.65, 0.25) to[out=-90, in=-135] (1,0);
                \draw[knot] (0,0) to[out=135, in=-135] (0,1);
                \draw[knot] (1,0) to[out=45, in=-45] (1,1);
			}
    $};
    \draw[->] (A) to node[above]{$\epsilon_{1\star}\,\tikz[baseline={([yshift=-.5ex]current bounding box.center)}, scale=.35]{
	\draw[knot, -]  (1,2) .. controls (1,3) and (0,3) .. (0,4);
	\draw[knot, -]  (2,2) .. controls (2,3) and (3,3) .. (3,4);
	\draw[knot, -] (1,4) .. controls (1,3) and (2,3) .. (2,4);
	\draw[knot, -] (0,4) .. controls (0,3.75) and (1,3.75) .. (1,4);
	\draw[knot, -] (0,4) .. controls (0,4.25) and (1,4.25) .. (1,4);
	\draw[knot, -] (2,4) .. controls (2,3.75) and (3,3.75) .. (3,4);
	\draw[knot, -] (2,4) .. controls (2,4.25) and (3,4.25) .. (3,4);
	\draw[knot, -] (1,2) .. controls (1,1.75) and (2,1.75) .. (2,2);
	\draw[knot, dashed, -] (1,2) .. controls (1,2.25) and (2,2.25) .. (2,2);
    }$} (B);
    \draw[->] (B) to node[above=5pt, pos=0.7] {$\epsilon_{\star0}\,
    \tikz[baseline={([yshift=-.5ex]current bounding box.center)}, scale=.35]{
	\draw[knot, -] (0,0) .. controls (0,1) and (1,1) .. (1,2);
	\draw[knot, -] (1,0) .. controls (1,1) and (2,1) .. (2,0);
	\draw[knot, -] (3,0) .. controls (3,1) and (2,1) .. (2,2);
	\draw[knot, -] (0,0) .. controls (0,-.25) and (1,-.25) .. (1,0);
	\draw[dashed, knot, -] (0,0) .. controls (0,.25) and (1,.25) .. (1,0);
	\draw[knot, -] (2,0) .. controls (2,-.25) and (3,-.25) .. (3,0);
	\draw[dashed, knot, -] (2,0) .. controls (2,.25) and (3,.25) .. (3,0);
	\draw[knot, -] (1,2) .. controls (1,1.75) and (2,1.75) .. (2,2);
	\draw[knot, -] (1,2) .. controls (1,2.25) and (2,2.25) .. (2,2);
}
    $} (D);
    \draw[->] (A) to node[below=5pt, pos=0.3]{$
    \epsilon_{\star1} \, \tikz[baseline={([yshift=-.5ex]current bounding box.center)}, scale=.35]{
	\draw[knot, -]  (1,2) .. controls (1,3) and (0,3) .. (0,4);
	\draw[knot, -]  (2,2) .. controls (2,3) and (3,3) .. (3,4);
	\draw[knot, -] (1,4) .. controls (1,3) and (2,3) .. (2,4);
	\draw[knot, -] (0,4) .. controls (0,3.75) and (1,3.75) .. (1,4);
	\draw[knot, -] (0,4) .. controls (0,4.25) and (1,4.25) .. (1,4);
	\draw[knot, -] (2,4) .. controls (2,3.75) and (3,3.75) .. (3,4);
	\draw[knot, -] (2,4) .. controls (2,4.25) and (3,4.25) .. (3,4);
	\draw[knot, -] (1,2) .. controls (1,1.75) and (2,1.75) .. (2,2);
	\draw[knot, dashed, -] (1,2) .. controls (1,2.25) and (2,2.25) .. (2,2);
    }
    $}  (C);
    \draw[->] (C) to node[below] {$\epsilon_{0\star}\,
    \tikz[baseline={([yshift=-.5ex]current bounding box.center)}, scale=.35]{
	\draw[knot, -] (0,0) .. controls (0,1) and (1,1) .. (1,2);
	\draw[knot, -] (1,0) .. controls (1,1) and (2,1) .. (2,0);
	\draw[knot, -] (3,0) .. controls (3,1) and (2,1) .. (2,2);
	\draw[knot, -] (0,0) .. controls (0,-.25) and (1,-.25) .. (1,0);
	\draw[dashed, knot, -] (0,0) .. controls (0,.25) and (1,.25) .. (1,0);
	\draw[knot, -] (2,0) .. controls (2,-.25) and (3,-.25) .. (3,0);
	\draw[dashed, knot, -] (2,0) .. controls (2,.25) and (3,.25) .. (3,0);
	\draw[knot, -] (1,2) .. controls (1,1.75) and (2,1.75) .. (2,2);
	\draw[knot, -] (1,2) .. controls (1,2.25) and (2,2.25) .. (2,2);
}
    $} (D);
    \draw[->] (E) to[bend right=15pt] node[left=5pt] {$ c\, 
    \tikz[baseline={([yshift=-.5ex]current bounding box.center)}, scale=.35]{
    \draw[knot, -] (0,2) .. controls (0,1.75) and (1,1.75) .. (1,2);
    \draw[knot, -] (0,2) .. controls (0,2.25) and (1,2.25) .. (1,2);
    \draw[knot, -] (2,2) .. controls (2,1.75) and (3,1.75) .. (3,2);
    \draw[knot, -] (2,2) .. controls (2,2.25) and (3,2.25) .. (3,2);
    \draw[knot, -] (0,0) .. controls (0,-.25) and (1,-.25) .. (1,0);
    \draw[dashed, knot, -] (0,0) .. controls (0,.25) and (1,.25) .. (1,0);
    \draw[knot, -] (2,0) .. controls (2,-.25) and (3,-.25) .. (3,0);
    \draw[dashed, knot, -] (2,0) .. controls (2,.25) and (3,.25) .. (3,0);
    \draw[knot, -] (2,2) .. controls (2,1.25) and (3,1.25) .. (3,2);
    \draw[knot, -] (3,0) .. controls (3,1.35) and (1,1.35) .. (1,2);
    \draw[knot, -] (0,0) -- (0,2);
    \draw[knot, -] (1,0) .. controls (1,0.75) and (2,0.75) .. (2,0);
}
    $} (B);
    \draw[->, white, line width=5pt] (A) -- (C);
    \draw[->] (A) -- (C);
    \draw[->, white, line width=5pt] (E) to[bend left=10pt] (C);
    \draw[->] (E) to[bend left=10pt] node[right=7.5pt, pos=0.15]{$
    -c\epsilon_{0\star}\epsilon_{\star0}\left(
\tikz[baseline={([yshift=-.5ex]current bounding box.center)}, scale=.35]{
    \draw[knot, -] (0,2) .. controls (0,1.75) and (1,1.75) .. (1,2);
    \draw[knot, -] (0,2) .. controls (0,2.25) and (1,2.25) .. (1,2);
    \draw[knot, -] (2,2) .. controls (2,1.75) and (3,1.75) .. (3,2);
    \draw[knot, -] (2,2) .. controls (2,2.25) and (3,2.25) .. (3,2);
    \draw[knot, -] (0,0) .. controls (0,-.25) and (1,-.25) .. (1,0);
    \draw[dashed, knot, -] (0,0) .. controls (0,.25) and (1,.25) .. (1,0);
    \draw[knot, -] (2,0) .. controls (2,-.25) and (3,-.25) .. (3,0);
    \draw[dashed, knot, -] (2,0) .. controls (2,.25) and (3,.25) .. (3,0);
    \draw[knot, -] (0,0) .. controls (0,0.75) and (1,0.75) .. (1,0);
    \draw[knot, -] (2,0) .. controls (2,0.65) and (0,0.65) .. (0,2);
    \draw[knot, -] (3,0) -- (3,2);
    \draw[knot, -] (1,2) .. controls (1,1.25) and (2,1.25) .. (2,2);
}
+
\tikz[baseline={([yshift=-.5ex]current bounding box.center)}, scale=.35]{
    \draw[knot, -] (0,2) .. controls (0,1.75) and (1,1.75) .. (1,2);
    \draw[knot, -] (0,2) .. controls (0,2.25) and (1,2.25) .. (1,2);
    \draw[knot, -] (2,2) .. controls (2,1.75) and (3,1.75) .. (3,2);
    \draw[knot, -] (2,2) .. controls (2,2.25) and (3,2.25) .. (3,2);
    \draw[knot, -] (0,0) .. controls (0,-.25) and (1,-.25) .. (1,0);
    \draw[dashed, knot, -] (0,0) .. controls (0,.25) and (1,.25) .. (1,0);
    \draw[knot, -] (2,0) .. controls (2,-.25) and (3,-.25) .. (3,0);
    \draw[dashed, knot, -] (2,0) .. controls (2,.25) and (3,.25) .. (3,0);
    \draw[knot, -] (2,2) .. controls (2,1.25) and (3,1.25) .. (3,2);
    \draw[knot, -] (3,0) .. controls (3,1.35) and (1,1.35) .. (1,2);
    \draw[knot, -] (0,0) -- (0,2);
    \draw[knot, -] (1,0) .. controls (1,0.75) and (2,0.75) .. (2,0);
}
\right)
    $} (C);
    \draw[->, red, thick] (E) to[bend right] node[left=20pt]{$c\epsilon_{1\star}\
    \tikz[baseline={([yshift=-.5ex]current bounding box.center)}, scale=.35]{
    \draw[knot, -] (0,0) .. controls (0,-0.25) and (1,-0.25) .. (1,0);
    \draw[knot, -, dashed] (0,0) .. controls (0,0.25) and (1,0.25) .. (1,0);
    \draw[knot, -] (2,2) .. controls (2,1.75) and (3,1.75) .. (3,2);
    \draw[knot, -] (2,2) .. controls (2,2.25) and (3,2.25) .. (3,2);
    \draw[knot, -] (2,0) .. controls (2,-.25) and (3,-.25) .. (3,0);
    \draw[dashed, knot, -] (2,0) .. controls (2,.25) and (3,.25) .. (3,0);
    \draw[knot, -] (0,0) .. controls (0,1.25) and (1,1.25) .. (1,0);
    \draw[knot, -] (2,0) -- (2,2);
    \draw[knot, -] (3,0) -- (3,2);
}
    $} (A);
    }
	\caption{Chain maps on $\widetilde{\mathrm{PKE}}_2$ for a positive Reidemeister II move with the first closure.}
	\label{fig:chainmapRII+1}
\end{figure}

\begin{figure}[htbp]
	\centering
\tikz[baseline={([yshift=-.5ex]current bounding box.center)}, scale=0.8]
	{
		\node(22) at (14,0) {$
			\tikz[baseline={([yshift=-.5ex]current bounding box.center)}, scale=0.8]
			{
				\draw[knot] (1,0) to[out=210, in=-90] (0.2, 0.5);
				\draw[knot] (0.2, 0.5) to[out=90, in=150] (1,1);
				\draw[knot, overcross] (0,0) to[out=-30, in=-90] (0.8, 0.5);
				\draw[knot, overcross] (0.8, 0.5) to[out=90, in=30] (0,1);
				\draw[knot] (0.8, 0.499) -- (0.8, 0.501);
                \draw[knot] (0,0) to[out=150, in=210] (0,1);
                \draw[knot] (1,0) to[out=30, in=-30] (1,1);
			}
			$};
		\node(12) at (9,2) {$
			\tikz[baseline={([yshift=-.5ex]current bounding box.center)}, scale=0.8]
			{
				\draw[knot] (0,1) to[out=30, in=90] (0.35, 0.75);
				\draw[knot] (0.35, 0.75) to[out=-90, in=90] (0.25, 0.5);
				\draw[knot] (0.25, 0.5) to[out=-90, in=210] (1,0);
				\draw[knot, overcross] (0.75, 0.5) to[out=-90, in=-30] (0,0);
				\draw[knot] (0.75, 0.5) to[out=90, in=-90] (0.65, 0.75);
				\draw[knot] (0.65, 0.75) to[out=90, in=150] (1,1);
                \draw[knot] (0,0) to[out=150, in=210] (0,1); 
                \draw[knot] (1,0) to[out=30, in=-30] (1,1);
			}
			$};
		\node(02) at (9,-2) {$
			\tikz[baseline={([yshift=-.5ex]current bounding box.center)}, scale=0.8, rotate=-90]
			{
				\draw[knot] (0,0) to[out=120, in=240] (0,1);
				\draw[knot] (0.55, 0.8) to[out=180, in=90] (0.35, 0.5);
				\draw[knot] (0.35, 0.5) to[out=-90, in=180] (0.55, 0.2);
				\draw[knot] (0.55, 0.2) to[out=0, in=-45] (1, 1);
				\draw[knot, overcross] (1, 0) to[out=45, in=0] (0.55, 0.8);
                \draw[knot] (0,0) to[out=300, in=-135] (1,0);
                \draw[knot] (0,1) to[out=60, in=135] (1,1);
			}
			$};
		\node(11) at (0,2) {$
			\tikz[baseline={([yshift=-.5ex]current bounding box.center)}, scale=0.8]
			{
				\draw[knot] (0,0) to[out=-30, in=-150] (1,0);
				\draw[knot] (0,1) to[out=0, in=90] (0.35, 0.75);
				\draw[knot] (0.35, 0.75) to[out=-90, in=90] (0.2, 0.5);
				\draw[knot] (0.2, 0.5) to[out=-90, in=180] (0.5, 0.3);
				\draw[knot] (0.5, 0.3) to[out=0, in=-90] (0.8, 0.5);
				\draw[knot] (0.8, 0.5) to[out=90, in=-90] (0.65, 0.75);
				\draw[knot] (0.65, 0.75) to[out=90, in=180] (1,1);
                \draw[knot] (0,0) to[out=-210, in=180] (0,1);
                \draw[knot] (1,0) to[out=-330, in=0] (1,1);
			}
			$};
		\node(10) at (4,2) {$
			\tikz[baseline={([yshift=-.5ex]current bounding box.center)}, scale=0.8]
			{
				\draw[knot] (0,1) to[out=45, in=90] (0.35, 0.75);
				\draw[knot] (0.35, 0.75) to[out=-90, in=90] (0.2, 0.5);
				\draw[knot] (0.2, 0.5) to[out=-90, in=90] (0.35, 0.25);
				\draw[knot] (0.35, 0.25) to[out=-90, in=-45] (0,0);
				\draw[knot] (1,1) to[out=135, in=90] (0.65, 0.75);
				\draw[knot] (0.65, 0.75) to[out=-90, in=90] (0.8, 0.5);
				\draw[knot] (0.8, 0.5) to[out=-90, in=90] (0.65, 0.25);
				\draw[knot] (0.65, 0.25) to[out=-90, in=-135] (1,0);
                \draw[knot] (0,0) to[out=135, in=-135] (0,1);
                \draw[knot] (1,0) to[out=45, in=-45] (1,1);
			}
			$};
		\node(01) at (0,-2) {$
			\tikz[baseline={([yshift=-.5ex]current bounding box.center)}, scale=0.8]
			{
				\draw[knot] (0,0) to[out=-30, in=-150] (1,0);
				\draw[knot] (0,1) to[out=30, in=150] (1,1);
                \draw[knot] (0.5,0.5) circle (0.25cm);
                \draw[knot] (0,0) to[out=-210, in=-150] (0,1);
                \draw[knot] (1,0) to[out=30, in=-30] (1,1);
			}
			$};
		\node(00) at (4,-2) {$
			\tikz[baseline={([yshift=-.5ex]current bounding box.center)}, scale=0.8, rotate=180]
			{
				\draw[knot] (0,0) to[out=-30, in=-150] (1,0);
				\draw[knot] (0,1) to[out=0, in=90] (0.35, 0.75);
				\draw[knot] (0.35, 0.75) to[out=-90, in=90] (0.2, 0.5);
				\draw[knot] (0.2, 0.5) to[out=-90, in=180] (0.5, 0.3);
				\draw[knot] (0.5, 0.3) to[out=0, in=-90] (0.8, 0.5);
				\draw[knot] (0.8, 0.5) to[out=90, in=-90] (0.65, 0.75);
				\draw[knot] (0.65, 0.75) to[out=90, in=180] (1,1);
                \draw[knot] (0,0) to[out=-210, in=180] (0,1);
                \draw[knot] (1,0) to[out=-330, in=0] (1,1);
			}
			$};
		\draw[->, shorten >=2mm, shorten <=2mm, gray] (12) -- (02);
		\draw[->, shorten >=2mm, shorten <=2mm, gray] (11) -- (10);
		\draw[->, shorten >=2mm, shorten <=2mm, gray] (11) -- (01);
		\draw[->, shorten >=2mm, shorten <=2mm, gray] (10) -- (00);
		\draw[->, shorten >=2mm, shorten <=2mm, gray] (01) -- (00);
		\draw[->, gray] (11) -- (00);
        \draw[->] (12) to[out=0, in=135] node[above]{\footnotesize$n_{\m(2,2)}^{\n(2,2)}$} (22);
        \draw[->] (02) to[out=0, in=-135] node[below]{\footnotesize$g_{\m(0,2)}^{\n(0,2)}$} (22);
        \draw[->] (10) -- node[above]{\footnotesize$g_{\m(1,0)}^{\n(1,0)}$} (12);
        \draw[->] (10) -- node[above=0.4cm]{\footnotesize$g_{\m(1,0)}^{\n(1,0)}$} (02);
        \draw[->, dashed] (00) -- (02);
        \draw[->] (01) to[out=-30, in=210] node[below]{\footnotesize$n_{\m(0,1)}^{\n(0,1)}$} (02);
        \draw[->, dashed] (11) to[out=30, in=150] (12);
        \draw[->, dashed] (11) -- (02);
	}
	\caption{Chain maps on $\widetilde{\mathrm{PKC}}$ for a negative Reidemeister II move with the first closure.}
	\label{fig:RII-1}
\end{figure}

For removing two crossings, Figure \ref{fig:RII-l} is rewritten as Figure \ref{fig:RII-1}. By a similar grading change computation, the dashed differentials are trivial on the $E_2$ page. We assign signs to these maps according to the decorations of the two crossings. The additional parameter $a$ ensures that the resulting map, pictured in Figure \ref{fig:chainmapRII-1}, is a chain map. The component identifications can be fixed as above.

\begin{figure}[htbp]
	\centering
\tikz[baseline={([yshift=-.5ex]current bounding box.center)}, scale=1.5]
    {
    \node (A) at (0,0) {$
    \tikz[baseline={([yshift=-.5ex]current bounding box.center)}, scale=0.8]
			{
				\draw[knot] (0,0) to[out=-30, in=-150] (1,0);
				\draw[knot] (0,1) to[out=0, in=90] (0.35, 0.75);
				\draw[knot] (0.35, 0.75) to[out=-90, in=90] (0.2, 0.5);
				\draw[knot] (0.2, 0.5) to[out=-90, in=180] (0.5, 0.3);
				\draw[knot] (0.5, 0.3) to[out=0, in=-90] (0.8, 0.5);
				\draw[knot] (0.8, 0.5) to[out=90, in=-90] (0.65, 0.75);
				\draw[knot] (0.65, 0.75) to[out=90, in=180] (1,1);
                \draw[knot] (0,0) to[out=-210, in=180] (0,1);
                \draw[knot] (1,0) to[out=-330, in=0] (1,1);
			}
    $};
    \node (B) at (2.75,1) {$
    			\tikz[baseline={([yshift=-.5ex]current bounding box.center)}, scale=0.8]
			{
				\draw[knot] (0,1) to[out=45, in=90] (0.35, 0.75);
				\draw[knot] (0.35, 0.75) to[out=-90, in=90] (0.2, 0.5);
				\draw[knot] (0.2, 0.5) to[out=-90, in=90] (0.35, 0.25);
				\draw[knot] (0.35, 0.25) to[out=-90, in=-45] (0,0);
				\draw[knot] (1,1) to[out=135, in=90] (0.65, 0.75);
				\draw[knot] (0.65, 0.75) to[out=-90, in=90] (0.8, 0.5);
				\draw[knot] (0.8, 0.5) to[out=-90, in=90] (0.65, 0.25);
				\draw[knot] (0.65, 0.25) to[out=-90, in=-135] (1,0);
                \draw[knot] (0,0) to[out=135, in=-135] (0,1);
                \draw[knot] (1,0) to[out=45, in=-45] (1,1);
			}
    $};
    \node (C) at (4.25,-1) {$
    			\tikz[baseline={([yshift=-.5ex]current bounding box.center)}, scale=0.8]
			{
				\draw[knot] (0,0) to[out=-30, in=-150] (1,0);
				\draw[knot] (0,1) to[out=30, in=150] (1,1);
                \draw[knot] (0.5,0.5) circle (0.25cm);
                \draw[knot] (0,0) to[out=-210, in=-150] (0,1);
                \draw[knot] (1,0) to[out=30, in=-30] (1,1);
			}
    $};
    \node (D) at (7,0) {$
    			\tikz[baseline={([yshift=-.5ex]current bounding box.center)}, scale=0.8, rotate=180]
			{
				\draw[knot] (0,0) to[out=-30, in=-150] (1,0);
				\draw[knot] (0,1) to[out=0, in=90] (0.35, 0.75);
				\draw[knot] (0.35, 0.75) to[out=-90, in=90] (0.2, 0.5);
				\draw[knot] (0.2, 0.5) to[out=-90, in=180] (0.5, 0.3);
				\draw[knot] (0.5, 0.3) to[out=0, in=-90] (0.8, 0.5);
				\draw[knot] (0.8, 0.5) to[out=90, in=-90] (0.65, 0.75);
				\draw[knot] (0.65, 0.75) to[out=90, in=180] (1,1);
                \draw[knot] (0,0) to[out=-210, in=180] (0,1);
                \draw[knot] (1,0) to[out=-330, in=0] (1,1);
			}
    $};
    \node (F) at (3.5,-3.15) {$
        			\tikz[baseline={([yshift=-.5ex]current bounding box.center)}, scale=0.8]
			{
				\draw[knot] (0,1) to[out=45, in=90] (0.35, 0.75);
				\draw[knot] (0.35, 0.75) to[out=-90, in=90] (0.2, 0.5);
				\draw[knot] (0.2, 0.5) to[out=-90, in=90] (0.35, 0.25);
				\draw[knot] (0.35, 0.25) to[out=-90, in=-45] (0,0);
				\draw[knot] (1,1) to[out=135, in=90] (0.65, 0.75);
				\draw[knot] (0.65, 0.75) to[out=-90, in=90] (0.8, 0.5);
				\draw[knot] (0.8, 0.5) to[out=-90, in=90] (0.65, 0.25);
				\draw[knot] (0.65, 0.25) to[out=-90, in=-135] (1,0);
                \draw[knot] (0,0) to[out=135, in=-135] (0,1);
                \draw[knot] (1,0) to[out=45, in=-45] (1,1);
			}
    $};
    \draw[->] (A) to node[above]{$\epsilon_{1\star}\,\tikz[baseline={([yshift=-.5ex]current bounding box.center)}, scale=.35]{
	\draw[knot, -]  (1,2) .. controls (1,3) and (0,3) .. (0,4);
	\draw[knot, -]  (2,2) .. controls (2,3) and (3,3) .. (3,4);
	\draw[knot, -] (1,4) .. controls (1,3) and (2,3) .. (2,4);
	\draw[knot, -] (0,4) .. controls (0,3.75) and (1,3.75) .. (1,4);
	\draw[knot, -] (0,4) .. controls (0,4.25) and (1,4.25) .. (1,4);
	\draw[knot, -] (2,4) .. controls (2,3.75) and (3,3.75) .. (3,4);
	\draw[knot, -] (2,4) .. controls (2,4.25) and (3,4.25) .. (3,4);
	\draw[knot, -] (1,2) .. controls (1,1.75) and (2,1.75) .. (2,2);
	\draw[knot, dashed, -] (1,2) .. controls (1,2.25) and (2,2.25) .. (2,2);
    }$} (B);
    \draw[->] (B) to node[above=5pt, pos=0.7] {$\epsilon_{\star0}\,
    \tikz[baseline={([yshift=-.5ex]current bounding box.center)}, scale=.35]{
	\draw[knot, -] (0,0) .. controls (0,1) and (1,1) .. (1,2);
	\draw[knot, -] (1,0) .. controls (1,1) and (2,1) .. (2,0);
	\draw[knot, -] (3,0) .. controls (3,1) and (2,1) .. (2,2);
	\draw[knot, -] (0,0) .. controls (0,-.25) and (1,-.25) .. (1,0);
	\draw[dashed, knot, -] (0,0) .. controls (0,.25) and (1,.25) .. (1,0);
	\draw[knot, -] (2,0) .. controls (2,-.25) and (3,-.25) .. (3,0);
	\draw[dashed, knot, -] (2,0) .. controls (2,.25) and (3,.25) .. (3,0);
	\draw[knot, -] (1,2) .. controls (1,1.75) and (2,1.75) .. (2,2);
	\draw[knot, -] (1,2) .. controls (1,2.25) and (2,2.25) .. (2,2);
}
    $} (D);
    \draw[->] (A) to node[below=5pt, pos=0.3]{$
    \epsilon_{\star1} \, \tikz[baseline={([yshift=-.5ex]current bounding box.center)}, scale=.35]{
	\draw[knot, -]  (1,2) .. controls (1,3) and (0,3) .. (0,4);
	\draw[knot, -]  (2,2) .. controls (2,3) and (3,3) .. (3,4);
	\draw[knot, -] (1,4) .. controls (1,3) and (2,3) .. (2,4);
	\draw[knot, -] (0,4) .. controls (0,3.75) and (1,3.75) .. (1,4);
	\draw[knot, -] (0,4) .. controls (0,4.25) and (1,4.25) .. (1,4);
	\draw[knot, -] (2,4) .. controls (2,3.75) and (3,3.75) .. (3,4);
	\draw[knot, -] (2,4) .. controls (2,4.25) and (3,4.25) .. (3,4);
	\draw[knot, -] (1,2) .. controls (1,1.75) and (2,1.75) .. (2,2);
	\draw[knot, dashed, -] (1,2) .. controls (1,2.25) and (2,2.25) .. (2,2);
    }
    $}  (C);
    \draw[->] (C) to node[below] {$\epsilon_{0\star}\,
    \tikz[baseline={([yshift=-.5ex]current bounding box.center)}, scale=.35]{
	\draw[knot, -] (0,0) .. controls (0,1) and (1,1) .. (1,2);
	\draw[knot, -] (1,0) .. controls (1,1) and (2,1) .. (2,0);
	\draw[knot, -] (3,0) .. controls (3,1) and (2,1) .. (2,2);
	\draw[knot, -] (0,0) .. controls (0,-.25) and (1,-.25) .. (1,0);
	\draw[dashed, knot, -] (0,0) .. controls (0,.25) and (1,.25) .. (1,0);
	\draw[knot, -] (2,0) .. controls (2,-.25) and (3,-.25) .. (3,0);
	\draw[dashed, knot, -] (2,0) .. controls (2,.25) and (3,.25) .. (3,0);
	\draw[knot, -] (1,2) .. controls (1,1.75) and (2,1.75) .. (2,2);
	\draw[knot, -] (1,2) .. controls (1,2.25) and (2,2.25) .. (2,2);
}
    $} (D);
    \draw[<-] (F) to[bend left=10pt] node[left=5pt,pos=0.2]{$
    c \left(
\tikz[baseline={([yshift=-.5ex]current bounding box.center)}, scale=.35]{
    \draw[knot, -] (0,2) .. controls (0,1.75) and (1,1.75) .. (1,2);
    \draw[knot, -] (0,2) .. controls (0,2.25) and (1,2.25) .. (1,2);
    \draw[knot, -] (2,2) .. controls (2,1.75) and (3,1.75) .. (3,2);
    \draw[knot, -] (2,2) .. controls (2,2.25) and (3,2.25) .. (3,2);
    \draw[knot, -] (0,0) .. controls (0,-.25) and (1,-.25) .. (1,0);
    \draw[dashed, knot, -] (0,0) .. controls (0,.25) and (1,.25) .. (1,0);
    \draw[knot, -] (2,0) .. controls (2,-.25) and (3,-.25) .. (3,0);
    \draw[dashed, knot, -] (2,0) .. controls (2,.25) and (3,.25) .. (3,0);
    \draw[knot, -] (0,0) .. controls (0,0.75) and (1,0.75) .. (1,0);
    \draw[knot, -] (2,0) .. controls (2,0.65) and (0,0.65) .. (0,2);
    \draw[knot, -] (3,0) -- (3,2);
    \draw[knot, -] (1,2) .. controls (1,1.25) and (2,1.25) .. (2,2);
}
+
\tikz[baseline={([yshift=-.5ex]current bounding box.center)}, scale=.35]{
    \draw[knot, -] (0,2) .. controls (0,1.75) and (1,1.75) .. (1,2);
    \draw[knot, -] (0,2) .. controls (0,2.25) and (1,2.25) .. (1,2);
    \draw[knot, -] (2,2) .. controls (2,1.75) and (3,1.75) .. (3,2);
    \draw[knot, -] (2,2) .. controls (2,2.25) and (3,2.25) .. (3,2);
    \draw[knot, -] (0,0) .. controls (0,-.25) and (1,-.25) .. (1,0);
    \draw[dashed, knot, -] (0,0) .. controls (0,.25) and (1,.25) .. (1,0);
    \draw[knot, -] (2,0) .. controls (2,-.25) and (3,-.25) .. (3,0);
    \draw[dashed, knot, -] (2,0) .. controls (2,.25) and (3,.25) .. (3,0);
    \draw[knot, -] (2,2) .. controls (2,1.25) and (3,1.25) .. (3,2);
    \draw[knot, -] (3,0) .. controls (3,1.35) and (1,1.35) .. (1,2);
    \draw[knot, -] (0,0) -- (0,2);
    \draw[knot, -] (1,0) .. controls (1,0.75) and (2,0.75) .. (2,0);
}
\right)
    $} (B);
    \draw[<-] (F) to[bend right=15pt] node[right]{$-a\epsilon_{1\star}\epsilon_{\star1}c
    \tikz[baseline={([yshift=-.5ex]current bounding box.center)}, scale=.35]{
    \draw[knot, -] (0,2) .. controls (0,1.75) and (1,1.75) .. (1,2);
    \draw[knot, -] (0,2) .. controls (0,2.25) and (1,2.25) .. (1,2);
    \draw[knot, -] (2,2) .. controls (2,1.75) and (3,1.75) .. (3,2);
    \draw[knot, -] (2,2) .. controls (2,2.25) and (3,2.25) .. (3,2);
    \draw[knot, -] (0,0) .. controls (0,-.25) and (1,-.25) .. (1,0);
    \draw[dashed, knot, -] (0,0) .. controls (0,.25) and (1,.25) .. (1,0);
    \draw[knot, -] (2,0) .. controls (2,-.25) and (3,-.25) .. (3,0);
    \draw[dashed, knot, -] (2,0) .. controls (2,.25) and (3,.25) .. (3,0);
    \draw[knot, -] (0,0) .. controls (0,0.75) and (1,0.75) .. (1,0);
    \draw[knot, -] (2,0) .. controls (2,0.65) and (0,0.65) .. (0,2);
    \draw[knot, -] (3,0) -- (3,2);
    \draw[knot, -] (1,2) .. controls (1,1.25) and (2,1.25) .. (2,2);
}
    $} (C);
    \draw[->, white, line width=5pt] (A) -- (C);
    \draw[->] (A) -- (C);    
    }
	\caption{Chain maps on $\widetilde{\mathrm{PKE}}_2$ for a negative Reidemeister II move with the first closure.}
	\label{fig:chainmapRII-1}
\end{figure}

We conclude this case by providing a chain homotopy equivalence between $\widetilde{\mathrm{PKE}}_2$ and odd Khovanov homology. The maps on odd Khovanov homology induced by a Reidemeister II move are given in Figure \ref{fig:RIIodd}. The signs of these maps also depend on a choice of decorations. We again use $a$ to distinguish these sign assignments. It is straightforward to check that the chain maps for removing two crossings are identical, and the chain maps for adding two crossings are chain homotopic by the homotopy pictured in red in Figure \ref{fig:chainmapRII+1}. Hence, we have proven that plane Floer homology induces the same map as odd Khovanov homology for the two component unlink.

\begin{figure}[htbp]
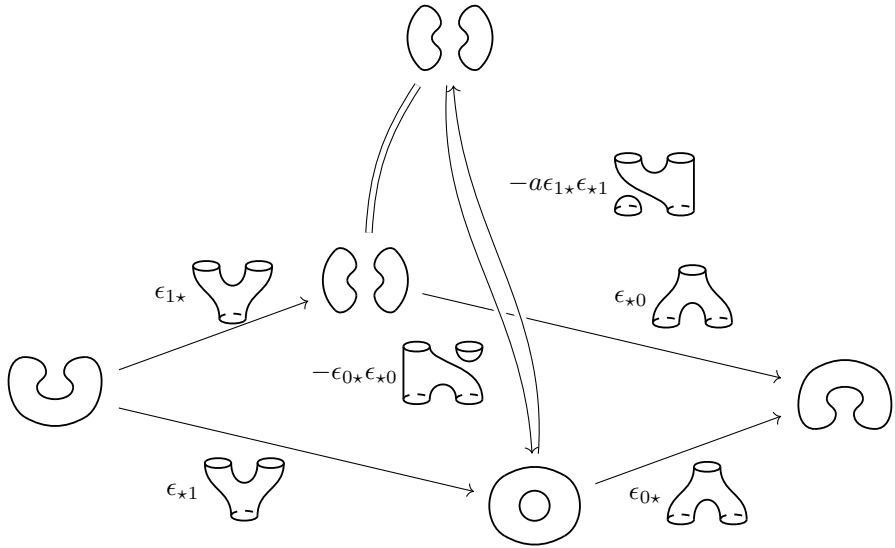

	\centering
\tikz[baseline={([yshift=-.5ex]current bounding box.center)}, scale=1.5]
    {
    \node (A) at (0,0) {$
    \tikz[baseline={([yshift=-.5ex]current bounding box.center)}, scale=0.8]
			{
				\draw[knot] (0,0) to[out=-30, in=-150] (1,0);
				\draw[knot] (0,1) to[out=0, in=90] (0.35, 0.75);
				\draw[knot] (0.35, 0.75) to[out=-90, in=90] (0.2, 0.5);
				\draw[knot] (0.2, 0.5) to[out=-90, in=180] (0.5, 0.3);
				\draw[knot] (0.5, 0.3) to[out=0, in=-90] (0.8, 0.5);
				\draw[knot] (0.8, 0.5) to[out=90, in=-90] (0.65, 0.75);
				\draw[knot] (0.65, 0.75) to[out=90, in=180] (1,1);
                \draw[knot] (0,0) to[out=-210, in=180] (0,1);
                \draw[knot] (1,0) to[out=-330, in=0] (1,1);
			}
    $};
    \node (B) at (2.75,1) {$
    			\tikz[baseline={([yshift=-.5ex]current bounding box.center)}, scale=0.8]
			{
				\draw[knot] (0,1) to[out=45, in=90] (0.35, 0.75);
				\draw[knot] (0.35, 0.75) to[out=-90, in=90] (0.2, 0.5);
				\draw[knot] (0.2, 0.5) to[out=-90, in=90] (0.35, 0.25);
				\draw[knot] (0.35, 0.25) to[out=-90, in=-45] (0,0);
				\draw[knot] (1,1) to[out=135, in=90] (0.65, 0.75);
				\draw[knot] (0.65, 0.75) to[out=-90, in=90] (0.8, 0.5);
				\draw[knot] (0.8, 0.5) to[out=-90, in=90] (0.65, 0.25);
				\draw[knot] (0.65, 0.25) to[out=-90, in=-135] (1,0);
                \draw[knot] (0,0) to[out=135, in=-135] (0,1);
                \draw[knot] (1,0) to[out=45, in=-45] (1,1);
			}
    $};
    \node (C) at (4.25,-1) {$
    			\tikz[baseline={([yshift=-.5ex]current bounding box.center)}, scale=0.8]
			{
				\draw[knot] (0,0) to[out=-30, in=-150] (1,0);
				\draw[knot] (0,1) to[out=30, in=150] (1,1);
                \draw[knot] (0.5,0.5) circle (0.25cm);
                \draw[knot] (0,0) to[out=-210, in=-150] (0,1);
                \draw[knot] (1,0) to[out=30, in=-30] (1,1);
			}
    $};
    \node (D) at (7,0) {$
    			\tikz[baseline={([yshift=-.5ex]current bounding box.center)}, scale=0.8, rotate=180]
			{
				\draw[knot] (0,0) to[out=-30, in=-150] (1,0);
				\draw[knot] (0,1) to[out=0, in=90] (0.35, 0.75);
				\draw[knot] (0.35, 0.75) to[out=-90, in=90] (0.2, 0.5);
				\draw[knot] (0.2, 0.5) to[out=-90, in=180] (0.5, 0.3);
				\draw[knot] (0.5, 0.3) to[out=0, in=-90] (0.8, 0.5);
				\draw[knot] (0.8, 0.5) to[out=90, in=-90] (0.65, 0.75);
				\draw[knot] (0.65, 0.75) to[out=90, in=180] (1,1);
                \draw[knot] (0,0) to[out=-210, in=180] (0,1);
                \draw[knot] (1,0) to[out=-330, in=0] (1,1);
			}
    $};
    \node (E) at (3.5,3.15) {$
        			\tikz[baseline={([yshift=-.5ex]current bounding box.center)}, scale=0.8]
			{
				\draw[knot] (0,1) to[out=45, in=90] (0.35, 0.75);
				\draw[knot] (0.35, 0.75) to[out=-90, in=90] (0.2, 0.5);
				\draw[knot] (0.2, 0.5) to[out=-90, in=90] (0.35, 0.25);
				\draw[knot] (0.35, 0.25) to[out=-90, in=-45] (0,0);
				\draw[knot] (1,1) to[out=135, in=90] (0.65, 0.75);
				\draw[knot] (0.65, 0.75) to[out=-90, in=90] (0.8, 0.5);
				\draw[knot] (0.8, 0.5) to[out=-90, in=90] (0.65, 0.25);
				\draw[knot] (0.65, 0.25) to[out=-90, in=-135] (1,0);
                \draw[knot] (0,0) to[out=135, in=-135] (0,1);
                \draw[knot] (1,0) to[out=45, in=-45] (1,1);
			}
    $};
    \draw[->] (A) to node[above]{$\epsilon_{1\star}\,\tikz[baseline={([yshift=-.5ex]current bounding box.center)}, scale=.35]{
	\draw[knot, -]  (1,2) .. controls (1,3) and (0,3) .. (0,4);
	\draw[knot, -]  (2,2) .. controls (2,3) and (3,3) .. (3,4);
	\draw[knot, -] (1,4) .. controls (1,3) and (2,3) .. (2,4);
	\draw[knot, -] (0,4) .. controls (0,3.75) and (1,3.75) .. (1,4);
	\draw[knot, -] (0,4) .. controls (0,4.25) and (1,4.25) .. (1,4);
	\draw[knot, -] (2,4) .. controls (2,3.75) and (3,3.75) .. (3,4);
	\draw[knot, -] (2,4) .. controls (2,4.25) and (3,4.25) .. (3,4);
	\draw[knot, -] (1,2) .. controls (1,1.75) and (2,1.75) .. (2,2);
	\draw[knot, dashed, -] (1,2) .. controls (1,2.25) and (2,2.25) .. (2,2);
    }$} (B);
    \draw[->] (B) to node[above=5pt, pos=0.7] {$\epsilon_{\star0}\,
    \tikz[baseline={([yshift=-.5ex]current bounding box.center)}, scale=.35]{
	\draw[knot, -] (0,0) .. controls (0,1) and (1,1) .. (1,2);
	\draw[knot, -] (1,0) .. controls (1,1) and (2,1) .. (2,0);
	\draw[knot, -] (3,0) .. controls (3,1) and (2,1) .. (2,2);
	\draw[knot, -] (0,0) .. controls (0,-.25) and (1,-.25) .. (1,0);
	\draw[dashed, knot, -] (0,0) .. controls (0,.25) and (1,.25) .. (1,0);
	\draw[knot, -] (2,0) .. controls (2,-.25) and (3,-.25) .. (3,0);
	\draw[dashed, knot, -] (2,0) .. controls (2,.25) and (3,.25) .. (3,0);
	\draw[knot, -] (1,2) .. controls (1,1.75) and (2,1.75) .. (2,2);
	\draw[knot, -] (1,2) .. controls (1,2.25) and (2,2.25) .. (2,2);
}
    $} (D);
    \draw[->] (A) to node[below=5pt, pos=0.3]{$
    \epsilon_{\star1} \, \tikz[baseline={([yshift=-.5ex]current bounding box.center)}, scale=.35]{
	\draw[knot, -]  (1,2) .. controls (1,3) and (0,3) .. (0,4);
	\draw[knot, -]  (2,2) .. controls (2,3) and (3,3) .. (3,4);
	\draw[knot, -] (1,4) .. controls (1,3) and (2,3) .. (2,4);
	\draw[knot, -] (0,4) .. controls (0,3.75) and (1,3.75) .. (1,4);
	\draw[knot, -] (0,4) .. controls (0,4.25) and (1,4.25) .. (1,4);
	\draw[knot, -] (2,4) .. controls (2,3.75) and (3,3.75) .. (3,4);
	\draw[knot, -] (2,4) .. controls (2,4.25) and (3,4.25) .. (3,4);
	\draw[knot, -] (1,2) .. controls (1,1.75) and (2,1.75) .. (2,2);
	\draw[knot, dashed, -] (1,2) .. controls (1,2.25) and (2,2.25) .. (2,2);
    }
    $}  (C);
    \draw[->] (C) to node[below] {$\epsilon_{0\star}\,
    \tikz[baseline={([yshift=-.5ex]current bounding box.center)}, scale=.35]{
	\draw[knot, -] (0,0) .. controls (0,1) and (1,1) .. (1,2);
	\draw[knot, -] (1,0) .. controls (1,1) and (2,1) .. (2,0);
	\draw[knot, -] (3,0) .. controls (3,1) and (2,1) .. (2,2);
	\draw[knot, -] (0,0) .. controls (0,-.25) and (1,-.25) .. (1,0);
	\draw[dashed, knot, -] (0,0) .. controls (0,.25) and (1,.25) .. (1,0);
	\draw[knot, -] (2,0) .. controls (2,-.25) and (3,-.25) .. (3,0);
	\draw[dashed, knot, -] (2,0) .. controls (2,.25) and (3,.25) .. (3,0);
	\draw[knot, -] (1,2) .. controls (1,1.75) and (2,1.75) .. (2,2);
	\draw[knot, -] (1,2) .. controls (1,2.25) and (2,2.25) .. (2,2);
}
    $} (D);
    \draw[double, double distance=2pt] (E) to[bend right=15pt] (B);
    \draw[->, white, line width=5pt] (A) -- (C);
    \draw[->] (A) -- (C);
    %
    %
    \draw[->, white, line width=4pt] (E) to[out=-94,in=93] (C);
    \draw[->, white, line width=4pt] (C) to[out=86,in=-86] (E);
    \draw[->] (E) to[out=-94,in=93] node[left=10pt, pos=0.8]{$
    -\epsilon_{0\star}\epsilon_{\star0}\,\tikz[baseline={([yshift=-.5ex]current bounding box.center)}, scale=.35]{
    \draw[knot, -] (0,2) .. controls (0,1.75) and (1,1.75) .. (1,2);
    \draw[knot, -] (0,2) .. controls (0,2.25) and (1,2.25) .. (1,2);
    \draw[knot, -] (2,2) .. controls (2,1.75) and (3,1.75) .. (3,2);
    \draw[knot, -] (2,2) .. controls (2,2.25) and (3,2.25) .. (3,2);
    \draw[knot, -] (0,0) .. controls (0,-.25) and (1,-.25) .. (1,0);
    \draw[dashed, knot, -] (0,0) .. controls (0,.25) and (1,.25) .. (1,0);
    \draw[knot, -] (2,0) .. controls (2,-.25) and (3,-.25) .. (3,0);
    \draw[dashed, knot, -] (2,0) .. controls (2,.25) and (3,.25) .. (3,0);
    \draw[knot, -] (2,2) .. controls (2,1.25) and (3,1.25) .. (3,2);
    \draw[knot, -] (3,0) .. controls (3,1.35) and (1,1.35) .. (1,2);
    \draw[knot, -] (0,0) -- (0,2);
    \draw[knot, -] (1,0) .. controls (1,0.75) and (2,0.75) .. (2,0);
}
    $}(C);
    \draw[->] (C) to[out=86,in=-86] node[right=10pt, pos=0.75]{$
    -a\epsilon_{1\star}\epsilon_{\star1}\,\tikz[baseline={([yshift=-.5ex]current bounding box.center)}, scale=.35]{
    \draw[knot, -] (0,2) .. controls (0,1.75) and (1,1.75) .. (1,2);
    \draw[knot, -] (0,2) .. controls (0,2.25) and (1,2.25) .. (1,2);
    \draw[knot, -] (2,2) .. controls (2,1.75) and (3,1.75) .. (3,2);
    \draw[knot, -] (2,2) .. controls (2,2.25) and (3,2.25) .. (3,2);
    \draw[knot, -] (0,0) .. controls (0,-.25) and (1,-.25) .. (1,0);
    \draw[dashed, knot, -] (0,0) .. controls (0,.25) and (1,.25) .. (1,0);
    \draw[knot, -] (2,0) .. controls (2,-.25) and (3,-.25) .. (3,0);
    \draw[dashed, knot, -] (2,0) .. controls (2,.25) and (3,.25) .. (3,0);
    \draw[knot, -] (0,0) .. controls (0,0.75) and (1,0.75) .. (1,0);
    \draw[knot, -] (2,0) .. controls (2,0.65) and (0,0.65) .. (0,2);
    \draw[knot, -] (3,0) -- (3,2);
    \draw[knot, -] (1,2) .. controls (1,1.25) and (2,1.25) .. (2,2);
}
    $} (E);
    }
	\caption{Chain maps for a Reidemeister II move on odd Khovanov homology.}
	\label{fig:RIIodd}
\end{figure}

\subsubsection{Second simple closure}
For the second closure, the chain map is given in Figure \ref{fig:RII2}. For adding two crossings, since the map $g_{\m(2,2)}^{\n(2,2)}$ is defined over $\overline{\mathbb{C}P^{2}}$, it is trivial. Hence, the only nontrivial map is given by $g_{\m(2,0)}^{\n(2,0)}\circ n_{\m(2,2)}^{\n(2,2)}$, which is equal to the map induced by the composition of two Reidemeister I moves. By naturality \cite{migdail2024functoriality}, the latter is chain homotopic to the map on odd Khovanov homology induced by a Reidemeister II move.

For removing two crossings, the map $n_{\m(1,2)}^{\n(1,2)}$ is defined over $\overline{\mathbb{C}P^{2}}$, so it's trivial. The reader can verify that the remaining maps are identical to the Reidemeister II chain homotopy maps on odd Khovanov homology. This concludes the argument for the second closure.

\begin{figure}[h!]
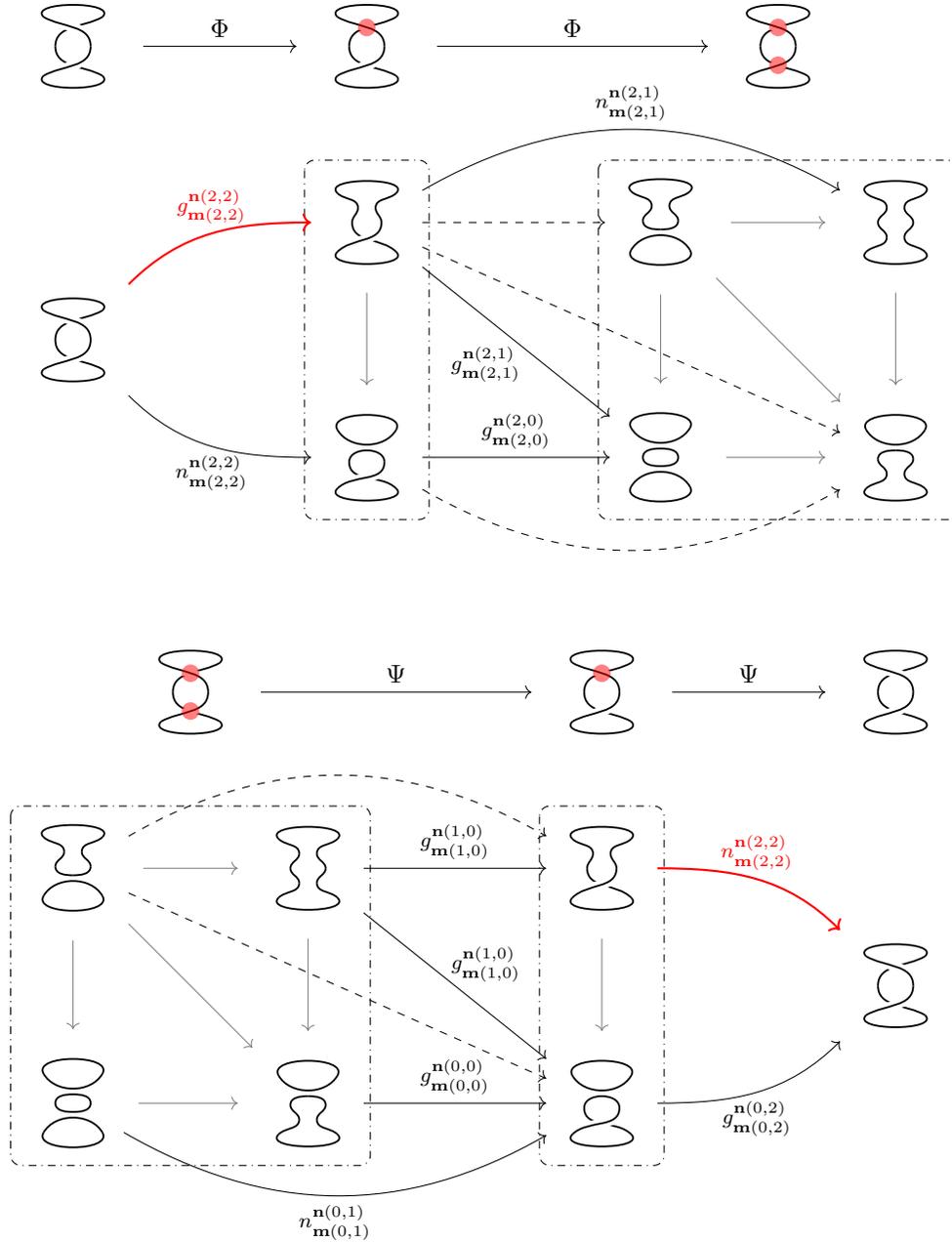

	\centering
\tikz[baseline={([yshift=-.5ex]current bounding box.center)}, scale=0.8]
	{
		\node(22) at (0,0) {$
			\tikz[baseline={([yshift=-.5ex]current bounding box.center)}, scale=0.8]
			{
				\draw[knot] (1,0) to[out=150, in=-90] (0.2, 0.5);
				\draw[knot] (0.2, 0.5) to[out=90, in=210] (1,1);
				\draw[knot, overcross] (0,0) to[out=30, in=-90] (0.8, 0.5);
				\draw[knot, overcross] (0.8, 0.5) to[out=90, in=-30] (0,1);
				\draw[knot] (0.8, 0.499) -- (0.8, 0.501);
                \draw[knot] (0,0) to[out=225, in=-45] (1,0);
                \draw[knot] (0,1) to[out=135, in=45] (1,1);
			}
			$};
		\node(12) at (5,2) {$
			\tikz[baseline={([yshift=-.5ex]current bounding box.center)}, scale=0.8]
			{
				\draw[knot] (0,1) to[out=-45, in=90] (0.35, 0.75);
				\draw[knot] (0.35, 0.75) to[out=-90, in=90] (0.25, 0.5);
				\draw[knot] (0.25, 0.5) to[out=-90, in=135] (1,0);
				\draw[knot, overcross] (0.75, 0.5) to[out=-90, in=45] (0,0);
				\draw[knot] (0.75, 0.5) to[out=90, in=-90] (0.65, 0.75);
				\draw[knot] (0.65, 0.75) to[out=90, in=225] (1,1);
                \draw[knot] (0,0) to[out=225, in=-45] (1,0);
                \draw[knot] (0,1) to[out=135, in=45] (1,1);
			}
			$};
		\node(02) at (5,-2) {$
			\tikz[baseline={([yshift=-.5ex]current bounding box.center)}, scale=0.8, rotate=-90]
			{
				\draw[knot] (0,0) to[out=30, in=-30] (0,1);
				\draw[knot] (0.55, 0.8) to[out=180, in=90] (0.35, 0.5);
				\draw[knot] (0.35, 0.5) to[out=-90, in=180] (0.55, 0.2);
				\draw[knot] (0.55, 0.2) to[out=0, in=225] (1, 1);
				\draw[knot, overcross] (1, 0) to[out=135, in=0] (0.55, 0.8);
                \draw[knot] (0,0) to[out=210, in=150] (0,1);
                \draw[knot] (1,0) to[out=-45, in=45] (1,1);
			}
			$};
		\node(11) at (10,2) {$
			\tikz[baseline={([yshift=-.5ex]current bounding box.center)}, scale=0.8]
			{
				\draw[knot] (0,0) to[out=60, in=120] (1,0);
				\draw[knot] (0,1) to[out=-45, in=90] (0.35, 0.75);
				\draw[knot] (0.35, 0.75) to[out=-90, in=90] (0.2, 0.5);
				\draw[knot] (0.2, 0.5) to[out=-90, in=180] (0.5, 0.35);
				\draw[knot] (0.5, 0.35) to[out=0, in=-90] (0.8, 0.5);
				\draw[knot] (0.8, 0.5) to[out=90, in=-90] (0.65, 0.75);
				\draw[knot] (0.65, 0.75) to[out=90, in=225] (1,1);
                \draw[knot] (0,0) to[out=240, in=-60] (1,0);
                \draw[knot] (0,1) to[out=135, in=45] (1,1);
			}
			$};
		\node(10) at (14,2) {$
			\tikz[baseline={([yshift=-.5ex]current bounding box.center)}, scale=0.8]
			{
				\draw[knot] (0,1) to[out=-45, in=90] (0.35, 0.75);
				\draw[knot] (0.35, 0.75) to[out=-90, in=90] (0.2, 0.5);
				\draw[knot] (0.2, 0.5) to[out=-90, in=90] (0.35, 0.25);
				\draw[knot] (0.35, 0.25) to[out=-90, in=45] (0,0);
				\draw[knot] (1,1) to[out=225, in=90] (0.65, 0.75);
				\draw[knot] (0.65, 0.75) to[out=-90, in=90] (0.8, 0.5);
				\draw[knot] (0.8, 0.5) to[out=-90, in=90] (0.65, 0.25);
				\draw[knot] (0.65, 0.25) to[out=-90, in=135] (1,0);
                \draw[knot] (0,0) to[out=225, in=-45] (1,0);
                \draw[knot] (0,1) to[out=135, in=45] (1,1);
			}
			$};
		\node(01) at (10,-2) {$
			\tikz[baseline={([yshift=-.5ex]current bounding box.center)}, scale=0.8]
			{
				\draw[knot] (0,0) to[out=60, in=120] (1,0);
				\draw[knot] (0,1) to[out=-60, in=-120] (1,1);
				\draw[knot] (0.5, 0.65) to[out=0, in=90] (0.8, 0.5) to[out=-90, in=0] (0.5, 0.35) to[out=180, in=-90] (0.2, 0.5) to[out=90, in=180] (0.5, 0.65);
                \draw[knot] (0,0) to[out=240, in=-60] (1,0);
                \draw[knot] (0,1) to[out=-240, in=60] (1,1);
			}
			$};
		\node(00) at (14,-2) {$
			\tikz[baseline={([yshift=-.5ex]current bounding box.center)}, scale=0.8, rotate=180]
			{
				\draw[knot] (0,0) to[out=60, in=120] (1,0);
				\draw[knot] (0,1) to[out=-45, in=90] (0.35, 0.75);
				\draw[knot] (0.35, 0.75) to[out=-90, in=90] (0.2, 0.5);
				\draw[knot] (0.2, 0.5) to[out=-90, in=180] (0.5, 0.35);
				\draw[knot] (0.5, 0.35) to[out=0, in=-90] (0.8, 0.5);
				\draw[knot] (0.8, 0.5) to[out=90, in=-90] (0.65, 0.75);
				\draw[knot] (0.65, 0.75) to[out=90, in=225] (1,1);
                \draw[knot] (0,0) to[out=240, in=-60] (1,0);
                \draw[knot] (0,1) to[out=135, in=45] (1,1);
			}
			$};
		\node(1) at (0,5) {$
			\tikz[baseline={([yshift=-.5ex]current bounding box.center)}, scale=0.8]
			{
				\draw[knot] (1,0) to[out=150, in=-90] (0.2, 0.5);
				\draw[knot] (0.2, 0.5) to[out=90, in=210] (1,1);
				\draw[knot, overcross] (0,0) to[out=30, in=-90] (0.8, 0.5);
				\draw[knot, overcross] (0.8, 0.5) to[out=90, in=-30] (0,1);
				\draw[knot] (0.8, 0.499) -- (0.8, 0.501);
                \draw[knot] (0,0) to[out=225, in=-45] (1,0);
                \draw[knot] (0,1) to[out=135, in=45] (1,1);
			}
			$};
		\draw[dashdotted, rounded corners] (10-1.5*0.707,2+1.5*0.707) rectangle (14+1.5*0.707,-2-1.5*0.707);
		\node(2) at (5,5) {$
			\tikz[baseline={([yshift=-.5ex]current bounding box.center)}, scale=0.8]
			{
				\draw[knot] (1,0) to[out=150, in=-90] (0.2, 0.5);
				\draw[knot] (0.2, 0.5) to[out=90, in=210] (1,1);
				\draw[knot, overcross] (0,0) to[out=30, in=-90] (0.8, 0.5);
				\draw[knot, overcross] (0.8, 0.5) to[out=90, in=-30] (0,1);
				\draw[knot] (0.8, 0.499) -- (0.8, 0.501);
                \draw[knot] (0,0) to[out=225, in=-45] (1,0);
                \draw[knot] (0,1) to[out=135, in=45] (1,1);
				\fill[red, fill opacity=0.5] (0.5, 0.8215) circle(0.15);
			}
			$};
		\draw[dashdotted, rounded corners] (5-1.5*0.707,2+1.5*0.707) rectangle (5+1.5*0.707,-2-1.5*0.707);
		\node(3) at (12,5) {$
			\tikz[baseline={([yshift=-.5ex]current bounding box.center)}, scale=0.8]
			{
				\draw[knot] (1,0) to[out=150, in=-90] (0.2, 0.5);
				\draw[knot] (0.2, 0.5) to[out=90, in=210] (1,1);
				\draw[knot, overcross] (0,0) to[out=30, in=-90] (0.8, 0.5);
				\draw[knot, overcross] (0.8, 0.5) to[out=90, in=-30] (0,1);
				\draw[knot] (0.8, 0.499) -- (0.8, 0.501);
                \draw[knot] (0,0) to[out=225, in=-45] (1,0);
                \draw[knot] (0,1) to[out=135, in=45] (1,1);
				\fill[red, fill opacity=0.5] (0.5, 0.8215) circle(0.15);
				\fill[red, fill opacity=0.5] (0.5, 0.1785) circle(0.15);
			}
			$};
		\draw[->, shorten >=2mm, shorten <=2mm] (1) --node[above]{$\Phi$} (2);
		\draw[->, shorten >=2mm, shorten <=2mm] (2) --node[above]{$\Phi$} (3);
		\draw[->, shorten >=2mm, shorten <=2mm, gray] (12) -- (02);
		\draw[->, shorten >=2mm, shorten <=2mm, gray] (11) -- (10);
		\draw[->, shorten >=2mm, shorten <=2mm, gray] (11) -- (01);
		\draw[->, shorten >=2mm, shorten <=2mm, gray] (10) -- (00);
		\draw[->, shorten >=2mm, shorten <=2mm, gray] (01) -- (00);
		\draw[->, gray] (11) -- (00);
		\draw[->, red, thick] (22) to[out=45, in=180] node[above]{\footnotesize$g_{\m(2,2)}^{\n(2,2)}$} (12);
		\draw[->] (22) to[out=-45, in=180] node[below]{\footnotesize$n_{\m(2,2)}^{\n(2,2)}$} (02);
		\draw[->, dashed] (12) -- (11);
		\draw[->] (12) to[out=30, in=150] node[above]{\footnotesize$n_{\m(2,1)}^{\n(2,1)}$} (10);
		\draw[->] (12) -- node[left=0.2cm, pos=0.65] {\footnotesize$g_{\m(2,1)}^{\n(2,1)}$} (01);
		\draw[->, dashed] (12) -- (00);
		\draw[->] (02) -- node[above]{\footnotesize$g_{\m(2,0)}^{\n(2,0)}$} (01);
		\draw[->, dashed] (02) to[out=-30, in=-150] (00);
\begin{scope}[shift={(0,-11)}]
		\node(22) at (14,0) {$
			\tikz[baseline={([yshift=-.5ex]current bounding box.center)}, scale=0.8]
			{
				\draw[knot] (1,0) to[out=150, in=-90] (0.2, 0.5);
				\draw[knot] (0.2, 0.5) to[out=90, in=210] (1,1);
				\draw[knot, overcross] (0,0) to[out=30, in=-90] (0.8, 0.5);
				\draw[knot, overcross] (0.8, 0.5) to[out=90, in=-30] (0,1);
				\draw[knot] (0.8, 0.499) -- (0.8, 0.501);
                \draw[knot] (0,0) to[out=225, in=-45] (1,0);
                \draw[knot] (0,1) to[out=135, in=45] (1,1);
			}
			$};
		\node(12) at (9,2) {$
			\tikz[baseline={([yshift=-.5ex]current bounding box.center)}, scale=0.8]
			{
				\draw[knot] (0,1) to[out=-45, in=90] (0.35, 0.75);
				\draw[knot] (0.35, 0.75) to[out=-90, in=90] (0.25, 0.5);
				\draw[knot] (0.25, 0.5) to[out=-90, in=135] (1,0);
				\draw[knot, overcross] (0.75, 0.5) to[out=-90, in=45] (0,0);
				\draw[knot] (0.75, 0.5) to[out=90, in=-90] (0.65, 0.75);
				\draw[knot] (0.65, 0.75) to[out=90, in=225] (1,1);
                \draw[knot] (0,0) to[out=225, in=-45] (1,0);
                \draw[knot] (0,1) to[out=135, in=45] (1,1);
			}
			$};
		\node(02) at (9,-2) {$
			\tikz[baseline={([yshift=-.5ex]current bounding box.center)}, scale=0.8, rotate=-90]
			{
				\draw[knot] (0,0) to[out=30, in=-30] (0,1);
				\draw[knot] (0.55, 0.8) to[out=180, in=90] (0.35, 0.5);
				\draw[knot] (0.35, 0.5) to[out=-90, in=180] (0.55, 0.2);
				\draw[knot] (0.55, 0.2) to[out=0, in=225] (1, 1);
				\draw[knot, overcross] (1, 0) to[out=135, in=0] (0.55, 0.8);
                \draw[knot] (0,0) to[out=210, in=150] (0,1);
                \draw[knot] (1,0) to[out=-45, in=45] (1,1);
            }
			$};
		\node(11) at (0,2) {$
			\tikz[baseline={([yshift=-.5ex]current bounding box.center)}, scale=0.8]
			{
				\draw[knot] (0,0) to[out=60, in=120] (1,0);
				\draw[knot] (0,1) to[out=-45, in=90] (0.35, 0.75);
				\draw[knot] (0.35, 0.75) to[out=-90, in=90] (0.2, 0.5);
				\draw[knot] (0.2, 0.5) to[out=-90, in=180] (0.5, 0.35);
				\draw[knot] (0.5, 0.35) to[out=0, in=-90] (0.8, 0.5);
				\draw[knot] (0.8, 0.5) to[out=90, in=-90] (0.65, 0.75);
				\draw[knot] (0.65, 0.75) to[out=90, in=225] (1,1);
                \draw[knot] (0,0) to[out=240, in=-60] (1,0);
                \draw[knot] (0,1) to[out=135, in=45] (1,1);
			}
			$};
		\node(10) at (4,2) {$
			\tikz[baseline={([yshift=-.5ex]current bounding box.center)}, scale=0.8]
			{
				\draw[knot] (0,1) to[out=-45, in=90] (0.35, 0.75);
				\draw[knot] (0.35, 0.75) to[out=-90, in=90] (0.2, 0.5);
				\draw[knot] (0.2, 0.5) to[out=-90, in=90] (0.35, 0.25);
				\draw[knot] (0.35, 0.25) to[out=-90, in=45] (0,0);
				\draw[knot] (1,1) to[out=225, in=90] (0.65, 0.75);
				\draw[knot] (0.65, 0.75) to[out=-90, in=90] (0.8, 0.5);
				\draw[knot] (0.8, 0.5) to[out=-90, in=90] (0.65, 0.25);
				\draw[knot] (0.65, 0.25) to[out=-90, in=135] (1,0);
                \draw[knot] (0,0) to[out=225, in=-45] (1,0);
                \draw[knot] (0,1) to[out=135, in=45] (1,1);
			}
			$};
		\node(01) at (0,-2) {$
			\tikz[baseline={([yshift=-.5ex]current bounding box.center)}, scale=0.8]
			{
				\draw[knot] (0,0) to[out=60, in=120] (1,0);
				\draw[knot] (0,1) to[out=-60, in=-120] (1,1);
				\draw[knot] (0.5, 0.65) to[out=0, in=90] (0.8, 0.5) to[out=-90, in=0] (0.5, 0.35) to[out=180, in=-90] (0.2, 0.5) to[out=90, in=180] (0.5, 0.65);
                \draw[knot] (0,0) to[out=240, in=-60] (1,0);
                \draw[knot] (0,1) to[out=-240, in=60] (1,1);
			}
			$};
		\node(00) at (4,-2) {$
			\tikz[baseline={([yshift=-.5ex]current bounding box.center)}, scale=0.8, rotate=180]
			{
				\draw[knot] (0,0) to[out=60, in=120] (1,0);
				\draw[knot] (0,1) to[out=-45, in=90] (0.35, 0.75);
				\draw[knot] (0.35, 0.75) to[out=-90, in=90] (0.2, 0.5);
				\draw[knot] (0.2, 0.5) to[out=-90, in=180] (0.5, 0.35);
				\draw[knot] (0.5, 0.35) to[out=0, in=-90] (0.8, 0.5);
				\draw[knot] (0.8, 0.5) to[out=90, in=-90] (0.65, 0.75);
				\draw[knot] (0.65, 0.75) to[out=90, in=225] (1,1);
                \draw[knot] (0,0) to[out=240, in=-60] (1,0);
                \draw[knot] (0,1) to[out=135, in=45] (1,1);
			}
			$};
		\node(1) at (2,5) {$
			\tikz[baseline={([yshift=-.5ex]current bounding box.center)}, scale=0.8]
			{
				\draw[knot] (1,0) to[out=150, in=-90] (0.2, 0.5);
				\draw[knot] (0.2, 0.5) to[out=90, in=210] (1,1);
				\draw[knot, overcross] (0,0) to[out=30, in=-90] (0.8, 0.5);
				\draw[knot, overcross] (0.8, 0.5) to[out=90, in=-30] (0,1);
				\draw[knot] (0.8, 0.499) -- (0.8, 0.501);
                \draw[knot] (0,0) to[out=225, in=-45] (1,0);
                \draw[knot] (0,1) to[out=135, in=45] (1,1);
				\fill[red, fill opacity=0.5] (0.5, 0.8215) circle(0.15);
				\fill[red, fill opacity=0.5] (0.5, 0.1785) circle(0.15);
			}
			$};
		\draw[dashdotted, rounded corners] (-1.5*0.707,2+1.5*0.707) rectangle (4+1.5*0.707,-2-1.5*0.707);
		\node(2) at (9,5) {$
			\tikz[baseline={([yshift=-.5ex]current bounding box.center)}, scale=0.8]
			{
				\draw[knot] (1,0) to[out=150, in=-90] (0.2, 0.5);
				\draw[knot] (0.2, 0.5) to[out=90, in=210] (1,1);
				\draw[knot, overcross] (0,0) to[out=30, in=-90] (0.8, 0.5);
				\draw[knot, overcross] (0.8, 0.5) to[out=90, in=-30] (0,1);
				\draw[knot] (0.8, 0.499) -- (0.8, 0.501);
                \draw[knot] (0,0) to[out=225, in=-45] (1,0);
                \draw[knot] (0,1) to[out=135, in=45] (1,1);
				\fill[red, fill opacity=0.5] (0.5, 0.8215) circle(0.15);
			}
			$};
		\draw[dashdotted, rounded corners] (9-1.5*0.707,2+1.5*0.707) rectangle (9+1.5*0.707,-2-1.5*0.707);
		\node(3) at (14,5) {$
			\tikz[baseline={([yshift=-.5ex]current bounding box.center)}, scale=0.8]
			{
				\draw[knot] (1,0) to[out=150, in=-90] (0.2, 0.5);
				\draw[knot] (0.2, 0.5) to[out=90, in=210] (1,1);
				\draw[knot, overcross] (0,0) to[out=30, in=-90] (0.8, 0.5);
				\draw[knot, overcross] (0.8, 0.5) to[out=90, in=-30] (0,1);
				\draw[knot] (0.8, 0.499) -- (0.8, 0.501);
                \draw[knot] (0,0) to[out=225, in=-45] (1,0);
                \draw[knot] (0,1) to[out=135, in=45] (1,1);
			}
			$};
		\draw[->, shorten >=2mm, shorten <=2mm] (1) --node[above]{$\Psi$} (2);
		\draw[->, shorten >=2mm, shorten <=2mm] (2) --node[above]{$\Psi$} (3);
		\draw[->, shorten >=2mm, shorten <=2mm, gray] (12) -- (02);
		\draw[->, shorten >=2mm, shorten <=2mm, gray] (11) -- (10);
		\draw[->, shorten >=2mm, shorten <=2mm, gray] (11) -- (01);
		\draw[->, shorten >=2mm, shorten <=2mm, gray] (10) -- (00);
		\draw[->, shorten >=2mm, shorten <=2mm, gray] (01) -- (00);
		\draw[->, gray] (11) -- (00);
        \draw[->, red, thick] (12) to[out=0, in=135] node[above]{\footnotesize$n_{\m(2,2)}^{\n(2,2)}$} (22);
        \draw[->] (02) to[out=0, in=-135] node[below]{\footnotesize$g_{\m(0,2)}^{\n(0,2)}$} (22);
        \draw[->] (10) -- node[above]{\footnotesize$g_{\m(1,0)}^{\n(1,0)}$} (12);
        \draw[->] (10) -- node[right=0.2cm, pos=0.35]{\footnotesize$g_{\m(1,0)}^{\n(1,0)}$} (02);
        \draw[->] (00) -- node[above]{\footnotesize$g_{\m(0,0)}^{\n(0,0)}$} (02);
        \draw[->] (01) to[out=-30, in=210] node[below]{\footnotesize$n_{\m(0,1)}^{\n(0,1)}$} (02);
        \draw[->, dashed] (11) to[out=30, in=150] (12);
        \draw[->, dashed] (11) -- (02);
\end{scope}
	}
	\caption{Plane knot complexes and chain maps for Reidemeister II moves of the second closure.}
	\label{fig:RII2}
\end{figure}

\subsubsection{External crossings}
Suppose there are $n$ external crossings. We may regard the chain complex as an $n$ dimensional hypercube, whose vertices are given by the chain complexes associated to a simple closure of a Reidemeister II diagram together with the external components, resolved with respect to the corresponding cube index. We have shown that the map induced by a Reidemeister II move is chain homotopic to the one induced on odd Khovanov homology when we restrict to the components whose domain and target have the same cube index. However, there are two types of maps whose target and source can have cube indexes which differ by 1. These are depicted in Figure \ref{fig:R2extra}. The red arrows consist of compositions of differentials (which change the cube index by $1$) and the chain homotopies described previously. The blue arrows consist of maps which were neglected in the simple closure cases, since they failed to preserve gradings in the absence of external crossings. When external crossings are included, however, these maps may become nontrivial. Hence, to generalize the chain homotopy to the case with external crossings, it remains to show that these additional maps sum to zero.
\vspace{-7.5mm}
\begin{figure}[h]
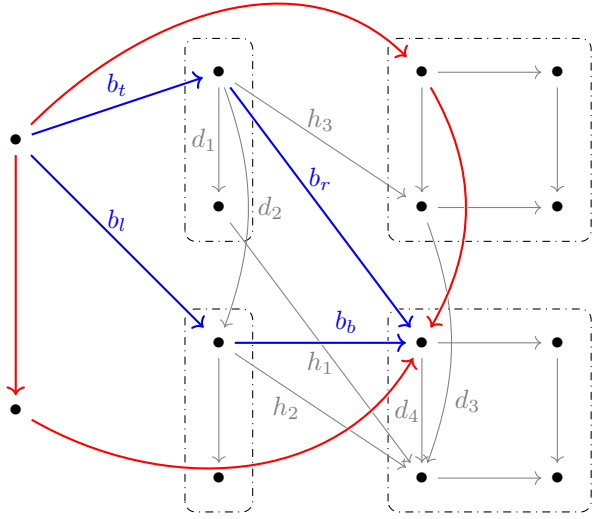

	\tikz[scale=0.9]{
		\draw[rounded corners, dashdotted] (2.5,-1.5) rectangle (3.5,1.5) node[pos=0.5] (rectA1){};
		\draw[rounded corners, dashdotted] (5.5, -1.5) rectangle (8.5,1.5) node[pos=0.5] (rectA2){};
		\node(A22) at (0,0) {$\bullet$};
		\node(A02) at (3,1) {$\bullet$};
		\node(A12) at (3,-1) {$\bullet$};
		\node(A00) at (6,1) {$\bullet$};
		\node(A10) at (6,-1) {$\bullet$};
		\node(A01) at (8,1) {$\bullet$};
		\node(A11) at (8,-1) {$\bullet$};
		\draw[->, gray] (A02) -- node[left=-1mm]{$d_1$}(A12); 
		\draw[->, gray] (A00) -- (A10);
		\draw[->, gray] (A00) -- (A01);
		\draw[->, gray] (A10) -- (A11);
		\draw[->, gray] (A01) -- (A11);
		\draw[rounded corners, dashdotted] (2.5,-5.5) rectangle (3.5,-2.5) node[pos=0.5] (rectB1){};
		\draw[rounded corners, dashdotted] (5.5, -5.5) rectangle (8.5,-2.5) node[pos=0.5] (rectB2){};
		\node(B22) at (0,-4) {$\bullet$};
		\node(B02) at (3,-3) {$\bullet$};
		\node(B12) at (3,-5) {$\bullet$};
		\node(B00) at (6,-3) {$\bullet$};
		\node(B10) at (6,-5) {$\bullet$};
		\node(B01) at (8,-3) {$\bullet$};
		\node(B11) at (8,-5) {$\bullet$};
		\draw[->, gray] (B02) -- (B12); 
		\draw[->, gray] (B00) --  node[left=-1mm]{$d_4$} (B10);
		\draw[->, gray] (B00) -- (B01);
		\draw[->, gray] (B10) -- (B11);
		\draw[->, gray] (B01) -- (B11);
		\draw[->, gray] (A02) to[bend left=20] node[right]{$d_2$} (B02);
		\draw[->, gray] (A12) -- node[below]{$h_1$} (B10);
		\draw[->, gray] (A02) -- node[above]{$h_3$} (A10);
		\draw[->, gray] (B02) -- node[below, pos=0.3]{$h_2$} (B10);
		\draw[->, gray] (A10) to[bend left=20] node[right, pos=0.75]{$d_3$} (B10);
		\draw[->, ultra thick, white] (A22) -- (A02);
		\draw[->, ultra thick, white] (A02) -- (B00);
		\draw[->, ultra thick, white] (A22) -- (B02);
		\draw[->, ultra thick, white] (B02) -- (B00);
		\draw[->, thick, blue] (A22) -- node[above]{$b_t$}(A02);
		\draw[->, thick, blue] (A02) -- node[above=1mm]{$b_r$}(B00);
		\draw[->, thick, blue] (A22) -- node[above]{$b_l$}(B02);
		\draw[->, thick, blue] (B02) -- node[above, pos=0.65]{$b_b$}(B00);
		\draw[->, ultra thick, white] (A22) to[out=45, in=135] (A00);
		\draw[->, ultra thick, white] (A00) to[out=-60, in=60] (B00);
		\draw[->, ultra thick, white] (A22) to[] (B22);
		\draw[->, ultra thick, white] (B22) to[out=-30, in=-120] (B00);
		\draw[->, thick, red] (A22) to[out=45, in=135] (A00);
		\draw[->, thick, red] (A00) to[out=-60, in=60] (B00);
		\draw[->, thick, red] (A22) to[] (B22);
		\draw[->, thick, red] (B22) to[out=-30, in=-120] (B00);
	}
	\caption{Extra maps which changes the cube index by $1$.}
	\label{fig:R2extra}
\end{figure}

\begin{figure}[p!]
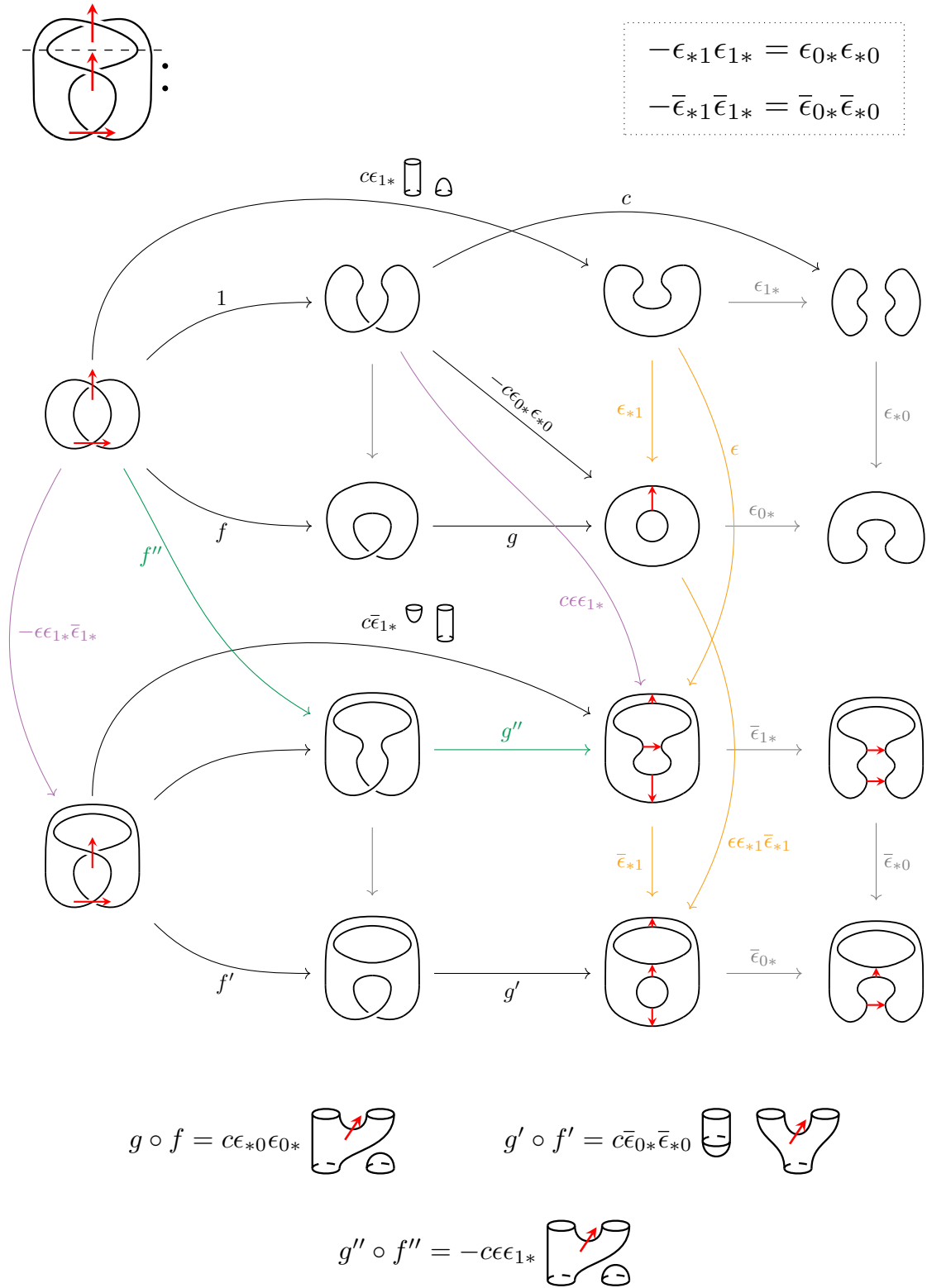

	\tikz[baseline={([yshift=-.5ex]current bounding box.center)}, scale=0.9]
	{
		\node(22) at (0,0) {$
			\tikz[baseline={([yshift=-.5ex]current bounding box.center)}, scale=1]
			{
				\draw[knot] (1,0) to[out=210, in=-90] (0.2, 0.5);
				\draw[knot] (0.2, 0.5) to[out=90, in=150] (1,1);
				\draw[knot, overcross] (0,0) to[out=-30, in=-90] (0.8, 0.5);
				\draw[knot, overcross] (0.8, 0.5) to[out=90, in=30] (0,1);
				\draw[knot] (0.8, 0.499) -- (0.8, 0.501);
				\draw[knot] (0,0) to[out=150, in=210] (0,1);
				\draw[knot] (1,0) to[out=30, in=-30] (1,1);
				\draw[-stealth, red, thick] (0.5,0.73) -- (0.5,1.23);
				\draw[-stealth, red, thick] (0.2,0.035) -- (0.8,0.035);
			}
			$};
		\node(12) at (5,2) {$
			\tikz[baseline={([yshift=-.5ex]current bounding box.center)}, scale=1]
			{
				\draw[knot] (0,1) to[out=30, in=90] (0.35, 0.75);
				\draw[knot] (0.35, 0.75) to[out=-90, in=90] (0.25, 0.5);
				\draw[knot] (0.25, 0.5) to[out=-90, in=210] (1,0);
				\draw[knot, overcross] (0.75, 0.5) to[out=-90, in=-30] (0,0);
				\draw[knot] (0.75, 0.5) to[out=90, in=-90] (0.65, 0.75);
				\draw[knot] (0.65, 0.75) to[out=90, in=150] (1,1);
				\draw[knot] (0,0) to[out=150, in=210] (0,1); 
				\draw[knot] (1,0) to[out=30, in=-30] (1,1);
			}
			$};
		\node(02) at (5,-2) {$
			\tikz[baseline={([yshift=-.5ex]current bounding box.center)}, scale=1, rotate=-90]
			{
				\draw[knot] (0,0) to[out=120, in=240] (0,1);
				\draw[knot] (0.55, 0.8) to[out=180, in=90] (0.35, 0.5);
				\draw[knot] (0.35, 0.5) to[out=-90, in=180] (0.55, 0.2);
				\draw[knot] (0.55, 0.2) to[out=0, in=-45] (1, 1);
				\draw[knot, overcross] (1, 0) to[out=45, in=0] (0.55, 0.8);
				\draw[knot] (0,0) to[out=300, in=-135] (1,0);
				\draw[knot] (0,1) to[out=60, in=135] (1,1);
			}
			$};
		\node(11) at (10,2) {$
			\tikz[baseline={([yshift=-.5ex]current bounding box.center)}, scale=1]
			{
				\draw[knot] (0,0) to[out=-30, in=-150] (1,0);
				\draw[knot] (0,1) to[out=0, in=90] (0.35, 0.75);
				\draw[knot] (0.35, 0.75) to[out=-90, in=90] (0.2, 0.5);
				\draw[knot] (0.2, 0.5) to[out=-90, in=180] (0.5, 0.3);
				\draw[knot] (0.5, 0.3) to[out=0, in=-90] (0.8, 0.5);
				\draw[knot] (0.8, 0.5) to[out=90, in=-90] (0.65, 0.75);
				\draw[knot] (0.65, 0.75) to[out=90, in=180] (1,1);
				\draw[knot] (0,0) to[out=-210, in=180] (0,1);
				\draw[knot] (1,0) to[out=-330, in=0] (1,1);
			}
			$};
		\node(10) at (14,2) {$
			\tikz[baseline={([yshift=-.5ex]current bounding box.center)}, scale=1]
			{
				\draw[knot] (0,1) to[out=45, in=90] (0.35, 0.75);
				\draw[knot] (0.35, 0.75) to[out=-90, in=90] (0.2, 0.5);
				\draw[knot] (0.2, 0.5) to[out=-90, in=90] (0.35, 0.25);
				\draw[knot] (0.35, 0.25) to[out=-90, in=-45] (0,0);
				\draw[knot] (1,1) to[out=135, in=90] (0.65, 0.75);
				\draw[knot] (0.65, 0.75) to[out=-90, in=90] (0.8, 0.5);
				\draw[knot] (0.8, 0.5) to[out=-90, in=90] (0.65, 0.25);
				\draw[knot] (0.65, 0.25) to[out=-90, in=-135] (1,0);
				\draw[knot] (0,0) to[out=135, in=-135] (0,1);
				\draw[knot] (1,0) to[out=45, in=-45] (1,1);
			}
			$};
		\node(01) at (10,-2) {$
			\tikz[baseline={([yshift=-.5ex]current bounding box.center)}, scale=1]
			{
				\draw[knot] (0,0) to[out=-30, in=-150] (1,0);
				\draw[knot] (0,1) to[out=30, in=150] (1,1);
				\draw[knot] (0.5,0.5) circle (0.25cm);
				\draw[knot] (0,0) to[out=-210, in=-150] (0,1);
				\draw[knot] (1,0) to[out=30, in=-30] (1,1);
				\draw[-stealth, red, thick] (0.5,0.75) -- (0.5,1.15);
			}
			$};
		\node(00) at (14,-2) {$
			\tikz[baseline={([yshift=-.5ex]current bounding box.center)}, scale=1, rotate=180]
			{
				\draw[knot] (0,0) to[out=-30, in=-150] (1,0);
				\draw[knot] (0,1) to[out=0, in=90] (0.35, 0.75);
				\draw[knot] (0.35, 0.75) to[out=-90, in=90] (0.2, 0.5);
				\draw[knot] (0.2, 0.5) to[out=-90, in=180] (0.5, 0.3);
				\draw[knot] (0.5, 0.3) to[out=0, in=-90] (0.8, 0.5);
				\draw[knot] (0.8, 0.5) to[out=90, in=-90] (0.65, 0.75);
				\draw[knot] (0.65, 0.75) to[out=90, in=180] (1,1);
				\draw[knot] (0,0) to[out=-210, in=180] (0,1);
				\draw[knot] (1,0) to[out=-330, in=0] (1,1);
			}
			$};
		\draw[->, shorten >=2mm, shorten <=2mm, gray] (12) -- (02);
		\draw[->, shorten >=2mm, shorten <=2mm, gray] (11) -- node[above]{$\epsilon_{1*}$} (10);
		\draw[->, shorten >=2mm, shorten <=2mm, YellowOrange] (11) -- (01) node[midway, left] {$\epsilon_{*1}$};
		\draw[->, shorten >=2mm, shorten <=2mm, gray] (10) -- node[right]{$\epsilon_{*0}$} (00);
		\draw[->, shorten >=2mm, shorten <=2mm, gray] (01) -- node[above]{$\epsilon_{0*}$} (00);
		\draw[->] (22) to[out=45, in=180] node[above]{$1$} (12);
		\draw[->] (22) to[out=-45, in=180]node[below]{$f$} (02);
		\draw[->] (12) to[out=30, in=150] node[above]{$c$} (10);
		\draw[->] (12) -- node [midway, above, sloped]{$-c \epsilon_{0*}\epsilon_{*0}$} (01);
		\draw[->] (02) -- node[below]{$g$} (01);
		\draw[->] (22) to[out=90, in=150] node[above, pos=0.7]{$c \epsilon_{1*} ~
			\tikz[baseline={([yshift=-.5ex]current bounding box.center)}, scale=.25]{
				\draw[knot, -, black] (0,0) .. controls (0,-0.25) and (1,-0.25) .. (1,0);
				\draw[knot, -, black, dashed] (0,0) .. controls (0,0.25) and (1,0.25) .. (1,0);
				\draw[knot, -, black] (0,2) .. controls (0,1.75) and (1,1.75) .. (1,2);
				\draw[knot, -, black] (0,2) .. controls (0,2.25) and (1,2.25) .. (1,2);
				\draw[knot, -, black] (2,0) .. controls (2,-.25) and (3,-.25) .. (3,0);
				\draw[dashed, knot, -, black] (2,0) .. controls (2,.25) and (3,.25) .. (3,0);
				\draw[knot, -, black] (2,0) .. controls (2,1.25) and (3,1.25) .. (3,0);
				\draw[knot, -, black] (0,0) -- (0,2);
				\draw[knot, -, black] (1,0) -- (1,2);
			}
			$} (11);
		\begin{scope}[yshift=-8cm]
			\node(B22) at (0,0) {$
				\tikz[baseline={([yshift=-.5ex]current bounding box.center)}, scale=1]
				{
					\draw[knot,white] (0,0) to[out=150, in=210]node(P)[pos=0.5, color=black]{} (0,1);
					\draw[knot, white] (1,0) to[out=30, in=-30]node(Q)[pos=0.5, color=black]{} (1,1);
					\draw[knot] (0,1) to[out=150, in=180, looseness=1.5] (0.5,1.45) to[out=0,in=30, looseness=1.5] (1,1);
					\draw[knot] (P.center) to[out=90, in=180, looseness=1.5] (0.5, 1.6) to[out=0, in=90, looseness=1.5] (Q.center);
					\draw[knot] (0,0) to[out=150, in=-90] (P.center);
					\draw[knot] (1,0) to[out=30, in=-90] (Q.center);
					\draw[knot] (1,0) to[out=210, in=-90] (0.2, 0.5);
					\draw[knot] (0.2, 0.5) to[out=90, in=210] (1,1);
					\draw[knot, overcross] (0,0) to[out=-30, in=-90] (0.8, 0.5);
					\draw[knot, overcross] (0.8, 0.5) to[out=90, in=-30] (0,1);
					\draw[knot] (0.8, 0.499) -- (0.8, 0.501);
					\draw[-stealth, red, thick] (0.5,0.575) -- (0.5,1.075);
					\draw[-stealth, red, thick] (0.2,0.035) -- (0.8,0.035);
				}
				$};
			\node(B12) at (5,2) {$
				\tikz[baseline={([yshift=-.5ex]current bounding box.center)}, scale=1]
				{
					\draw[knot] (0,1) to[out=-30, in=90] (0.35, 0.75);
					\draw[knot] (0.35, 0.75) to[out=-90, in=90] (0.25, 0.5);
					\draw[knot] (0.25, 0.5) to[out=-90, in=210] (1,0);
					\draw[knot, overcross] (0.75, 0.5) to[out=-90, in=-30] (0,0);
					\draw[knot] (0.75, 0.5) to[out=90, in=-90] (0.65, 0.75);
					\draw[knot] (0.65, 0.75) to[out=90, in=210] (1,1);
					\draw[knot, white] (0,0) to[out=150, in=210] node(P)[pos=0.5]{} (0,1); 
					\draw[knot, white] (1,0) to[out=30, in=-30] node(Q)[pos=0.5]{} (1,1);
					\draw[knot] (0,1) to[out=150, in=180, looseness=1.5] (0.5,1.45) to[out=0,in=30, looseness=1.5] (1,1);
					\draw[knot] (P.center) to[out=90, in=180, looseness=1.5] (0.5, 1.6) to[out=0, in=90, looseness=1.5] (Q.center);
					\draw[knot] (0,0) to[out=150, in=-90] (P.center);
					\draw[knot] (1,0) to[out=30, in=-90] (Q.center);
				}
				$};
			\node(B02) at (5,-2) {$
				\tikz[baseline={([yshift=-.5ex]current bounding box.center)}, scale=1, rotate=-90]
				{
					\begin{scope}[rotate=90,yshift=-1cm]
						\draw[knot, white] (0,0) to[out=150, in=210] node(P)[pos=0.5]{} (0,1); 
						\draw[knot, white] (1,0) to[out=30, in=-30] node(Q)[pos=0.5]{} (1,1);
						\draw[knot] (0,1) to[out=150, in=180, looseness=1.5] (0.5,1.45) to[out=0,in=30, looseness=1.5] (1,1);
						\draw[knot] (P.center) to[out=90, in=180, looseness=1.5] (0.5, 1.6) to[out=0, in=90, looseness=1.5] (Q.center);
						\draw[knot] (0,0) to[out=150, in=-90] (P.center);
						\draw[knot] (1,0) to[out=30, in=-90] (Q.center);
					\end{scope}
					\draw[knot] (0,0) to[out=60, in=-60] (0,1);
					\draw[knot] (0.55, 0.8) to[out=180, in=90] (0.35, 0.5);
					\draw[knot] (0.35, 0.5) to[out=-90, in=180] (0.55, 0.2);
					\draw[knot] (0.55, 0.2) to[out=0, in=-60] (1, 1);
					\draw[knot, overcross] (1, 0) to[out=60, in=0] (0.55, 0.8);
				}
				$};
			\node(B11) at (10,2) {$
				\tikz[baseline={([yshift=-.5ex]current bounding box.center)}, scale=1]
				{
					\draw[knot] (0,0) to[out=-30, in=-150] (1,0);
					\draw[knot, white] (0,0) to[out=150, in=210] node(P)[pos=0.5]{} (0,1); 
					\draw[knot, white] (1,0) to[out=30, in=-30] node(Q)[pos=0.5]{} (1,1);
					\draw[knot] (0,1) to[out=150, in=180, looseness=1.5] (0.5,1.45) to[out=0,in=30, looseness=1.5] (1,1);
					\draw[knot] (P.center) to[out=90, in=180, looseness=1.5] (0.5, 1.6) to[out=0, in=90, looseness=1.5] (Q.center);
					\draw[knot] (0,0) to[out=150, in=-90] (P.center);
					\draw[knot] (1,0) to[out=30, in=-90] (Q.center);
					\draw[knot] (0,1) to[out=-30, in=90] (0.35, 0.75);
					\draw[knot] (0.35, 0.75) to[out=-90, in=90] (0.2, 0.5);
					\draw[knot] (0.2, 0.5) to[out=-90, in=180] (0.5, 0.3);
					\draw[knot] (0.5, 0.3) to[out=0, in=-90] (0.8, 0.5);
					\draw[knot] (0.8, 0.5) to[out=90, in=-90] (0.65, 0.75);
					\draw[knot] (0.65, 0.75) to[out=90, in=210] (1,1);
					\draw[stealth-, red, thick] (0.5,-0.15) -- (0.5,0.3);
					\draw[-stealth, red, thick] (0.35, 0.75) -- (0.65,0.75);
					\draw[-stealth, red, thick] (0.5,1.45) -- (0.5,1.6);
				}
				$};
			\node(B10) at (14,2) {$
				\tikz[baseline={([yshift=-.5ex]current bounding box.center)}, scale=1]
				{
					\draw[knot, white] (0,0) to[out=150, in=210] node(P)[pos=0.5]{} (0,1); 
					\draw[knot, white] (1,0) to[out=30, in=-30] node(Q)[pos=0.5]{} (1,1);
					\draw[knot] (0,1) to[out=150, in=180, looseness=1.5] (0.5,1.45) to[out=0,in=30, looseness=1.5] (1,1);
					\draw[knot] (P.center) to[out=90, in=180, looseness=1.5] (0.5, 1.6) to[out=0, in=90, looseness=1.5] (Q.center);
					\draw[knot] (0,0) to[out=150, in=-90] (P.center);
					\draw[knot] (1,0) to[out=30, in=-90] (Q.center);
					\draw[knot] (0,1) to[out=-30, in=90] (0.35, 0.75);
					\draw[knot] (0.35, 0.75) to[out=-90, in=90] (0.2, 0.5);
					\draw[knot] (0.2, 0.5) to[out=-90, in=90] (0.35, 0.25);
					\draw[knot] (0.35, 0.25) to[out=-90, in=-30] (0,0);
					\draw[knot] (1,1) to[out=210, in=90] (0.65, 0.75);
					\draw[knot] (0.65, 0.75) to[out=-90, in=90] (0.8, 0.5);
					\draw[knot] (0.8, 0.5) to[out=-90, in=90] (0.65, 0.25);
					\draw[knot] (0.65, 0.25) to[out=-90, in=-150] (1,0);
					\draw[-stealth, red, thick] (0.35, 0.75) -- (0.65,0.75);
					\draw[-stealth, red, thick] (0.35, 0.25) -- (0.65,0.25);
				}
				$};
			\node(B01) at (10,-2) {$
				\tikz[baseline={([yshift=-.5ex]current bounding box.center)}, scale=1]
				{
					\draw[knot] (0.5,0.4) circle (0.25cm);
					\draw[knot, white] (0,0) to[out=150, in=210] node(P)[pos=0.5]{} (0,1); 
					\draw[knot, white] (1,0) to[out=30, in=-30] node(Q)[pos=0.5]{} (1,1);
					\draw[knot] (0,1) to[out=150, in=180, looseness=1.5] (0.5,1.45) to[out=0,in=30, looseness=1.5] (1,1);
					\draw[knot] (P.center) to[out=90, in=180, looseness=1.5] (0.5, 1.6) to[out=0, in=90, looseness=1.5] (Q.center);
					\draw[knot] (0,0) to[out=150, in=-90] (P.center);
					\draw[knot] (1,0) to[out=30, in=-90] (Q.center);
					\draw[knot] (0,0) to[out=-30, in=-150] (1,0);
					\draw[knot] (0,1) to[out=-30, in=210] (1,1);
					\draw[-stealth, red, thick] (0.5,1.45) -- (0.5,1.6);
					\draw[-stealth, red, thick] (0.5,0.65) -- (0.5,0.85);
					\draw[-stealth, red, thick] (0.5,0.15) -- (0.5,-0.15);
				}
				$};
			\node(B00) at (14,-2) {$
				\tikz[baseline={([yshift=-.5ex]current bounding box.center)}, scale=1, rotate=180]
				{
					\draw[knot, white] (0,0) to[out=150, in=210] node(P)[pos=0.5]{} (0,1); 
					\draw[knot, white] (1,0) to[out=30, in=-30] node(Q)[pos=0.5]{} (1,1);
					\draw[knot] (0,0) to[out=210, in=180, looseness=1.5] (0.5,-0.45) to[out=0, in=-30, looseness=1.5] (1,0);
					\draw[knot] (0,1) to[out=210, in=90] (P.center);
					\draw[knot] (1,1) to[out=-30, in=90] (Q.center);
					\draw[knot] (P.center) to[out=-90, in=180, looseness=1.5] (0.5, -0.6) to[out=0, in=-90, looseness=1.5] (Q.center);
					\draw[knot] (0,0) to[out=30, in=150] (1,0);
					\draw[knot] (0,1) to[out=30, in=90] (0.35, 0.75);
					\draw[knot] (0.35, 0.75) to[out=-90, in=90] (0.2, 0.5);
					\draw[knot] (0.2, 0.5) to[out=-90, in=180] (0.5, 0.3);
					\draw[knot] (0.5, 0.3) to[out=0, in=-90] (0.8, 0.5);
					\draw[knot] (0.8, 0.5) to[out=90, in=-90] (0.65, 0.75);
					\draw[knot] (0.65, 0.75) to[out=90, in=150] (1,1);
					\draw[stealth-, red, thick] (0.35, 0.75) -- (0.65,0.75);
					\draw[stealth-, red, thick] (0.5, 0.15) -- (0.5, 0.3);
				}
				$};
		\end{scope}
		\draw[->, shorten >=2mm, shorten <=2mm, gray] (B12) -- (B02);
		\draw[->, shorten >=2mm, shorten <=2mm, gray] (B11) -- node[above]{$\overline{\epsilon}_{1*}$} (B10);
		\draw[->, shorten >=2mm, shorten <=2mm, YellowOrange] (B11) -- (B01) node[midway, left] {$\overline{\epsilon}_{*1}$};
		\draw[->, shorten >=2mm, shorten <=2mm, gray] (B10) -- node[right]{$\overline{\epsilon}_{*0}$} (B00);
		\draw[->, shorten >=2mm, shorten <=2mm, gray] (B01) -- node[above]{$\overline{\epsilon}_{0*}$} (B00);
		\draw[->] (B22) to[out=45, in=180] node[above]{} (B12);
		\draw[->] (B22) to[out=-45, in=180] node[below]{$f'$} (B02);
		\draw[->, ForestGreen] (B12) --node[above]{$g''$} (B11);
		\draw[->] (B02) -- node[below]{$g'$} (B01);
		\draw[->] (B22) to[out=90, in=150] node[above, pos=0.7]{$c \overline{\epsilon}_{1*} ~
			\tikz[baseline={([yshift=-.5ex]current bounding box.center)}, scale=.25]{
				\draw[knot, -, black] (0,2) .. controls (0,1.75) and (1,1.75) .. (1,2);
				\draw[knot, -, black] (0,2) .. controls (0,2.25) and (1,2.25) .. (1,2);
				\draw[knot, -, black] (2,2) .. controls (2,1.75) and (3,1.75) .. (3,2);
				\draw[knot, -, black] (2,2) .. controls (2,2.25) and (3,2.25) .. (3,2);
				\draw[knot, -, black] (2,0) .. controls (2,-.25) and (3,-.25) .. (3,0);
				\draw[dashed, knot, -, black] (2,0) .. controls (2,.25) and (3,.25) .. (3,0);
				\draw[knot, -, black] (0,2) .. controls (0,0.75) and (1,0.75) .. (1,2);
				\draw[knot, -, black] (2,0) -- (2,2);
				\draw[knot, -, black] (3,0) -- (3,2);
			}
			$} (B11);
		\draw[->, Orchid] (22) to[out=-120, in=120] node[right]{$-\epsilon\epsilon_{1*}\overline{\epsilon}_{1*}$} (B22);
		\draw[->, Orchid] (12) to[out=-60, in=100] node[left, pos=0.8] {$c\epsilon\epsilon_{1*}$} (B11); 
		\draw[->, ForestGreen] (22) to[out=-60, in=150] node[left, pos=0.3]{$f''$} (B12);
		\draw[->, YellowOrange] (11) to[out=-60, in=60] node[right, pos=0.3]{$\epsilon$} (B11);
		\draw[->, YellowOrange] (01) to[out=-60, in=60] node[right, pos=0.8] {$\epsilon \epsilon_{*1} \overline{\epsilon}_{*1}$} (B01);
		\node[scale=1.25] at (7,-13) {$
			g \circ f = c \epsilon_{*0} \epsilon_{0*} ~\tikz[baseline={([yshift=-.5ex]current bounding box.center)}, scale=.35]{
				\draw[knot, -] (0,2) .. controls (0,1.75) and (1,1.75) .. (1,2);
				\draw[knot, -] (0,2) .. controls (0,2.25) and (1,2.25) .. (1,2);
				\draw[knot, -] (2,2) .. controls (2,1.75) and (3,1.75) .. (3,2);
				\draw[knot, -] (2,2) .. controls (2,2.25) and (3,2.25) .. (3,2);
				\draw[knot, -] (0,0) .. controls (0,-.25) and (1,-.25) .. (1,0);
				\draw[dashed, knot, -] (0,0) .. controls (0,.25) and (1,.25) .. (1,0);
				\draw[knot, -] (2,0) .. controls (2,-.25) and (3,-.25) .. (3,0);
				\draw[dashed, knot, -] (2,0) .. controls (2,.25) and (3,.25) .. (3,0);
				\draw[knot, -] (2,0) .. controls (2,0.75) and (3,0.75) .. (3,0);
				\draw[knot, -] (1,0) .. controls (1,0.65) and (3,0.65) .. (3,2);
				\draw[knot, -] (0,0) -- (0,2);
				\draw[knot, -] (1,2) .. controls (1,1.25) and (2,1.25) .. (2,2);
				\draw[stealth-, red, thick] (1.8,3.7-1.75) -- (1.2,2.8-1.75);
			}
			\hspace{1.4cm}
			g' \circ f' = c \overline{\epsilon}_{0*} \overline{\epsilon}_{*0} ~
			\tikz[baseline={([yshift=-.5ex]current bounding box.center)}, scale=.35]{
				\draw[knot, -]  (1,2) .. controls (1,3) and (0,3) .. (0,4);
				\draw[knot, -]  (2,2) .. controls (2,3) and (3,3) .. (3,4);
				\draw[knot, -] (1,4) .. controls (1,3) and (2,3) .. (2,4);
				\draw[knot, -] (0,4) .. controls (0,3.75) and (1,3.75) .. (1,4);
				\draw[knot, -] (0,4) .. controls (0,4.25) and (1,4.25) .. (1,4);
				\draw[knot, -] (2,4) .. controls (2,3.75) and (3,3.75) .. (3,4);
				\draw[knot, -] (2,4) .. controls (2,4.25) and (3,4.25) .. (3,4);
				\draw[knot, -] (1,2) .. controls (1,1.75) and (2,1.75) .. (2,2);
				\draw[knot, dashed, -] (1,2) .. controls (1,2.25) and (2,2.25) .. (2,2);
				\draw[knot] (-2,4) .. controls (-2,4.25) and (-1,4.25) .. (-1,4);
				\draw[knot] (-2,4) .. controls (-2,3.75) and (-1,3.75) .. (-1,4);
				\draw[knot] (-2,4) -- (-2,3);
				\draw[knot, dashed] (-2,3) .. controls (-2,3.25) and (-1,3.25) .. (-1,3);
				\draw[knot] (-2,3) .. controls (-2,2.75) and (-1,2.75) .. (-1,3);
				\draw[knot] (-1,4) -- (-1,3);
				\draw[knot] (-2,3) .. controls (-2,2.25) and (-1,2.25) .. (-1,3);
				\draw[stealth-, red, thick] (1.8,3.7+.1) -- (1.2,2.8+.1);
			}
			$}; 
		\node[scale=1.25] at (7,-15) {$g'' \circ f''=
			-c \epsilon \epsilon_{1*}
			~\tikz[baseline={([yshift=-.5ex]current bounding box.center)}, scale=.35]{
				\draw[knot, -] (0,2) .. controls (0,1.75) and (1,1.75) .. (1,2);
				\draw[knot, -] (0,2) .. controls (0,2.25) and (1,2.25) .. (1,2);
				\draw[knot, -] (2,2) .. controls (2,1.75) and (3,1.75) .. (3,2);
				\draw[knot, -] (2,2) .. controls (2,2.25) and (3,2.25) .. (3,2);
				\draw[knot, -] (0,0) .. controls (0,-.25) and (1,-.25) .. (1,0);
				\draw[dashed, knot, -] (0,0) .. controls (0,.25) and (1,.25) .. (1,0);
				\draw[knot, -] (2,0) .. controls (2,-.25) and (3,-.25) .. (3,0);
				\draw[dashed, knot, -] (2,0) .. controls (2,.25) and (3,.25) .. (3,0);
				\draw[knot, -] (2,0) .. controls (2,0.75) and (3,0.75) .. (3,0);
				\draw[knot, -] (1,0) .. controls (1,0.65) and (3,0.65) .. (3,2);
				\draw[knot, -] (0,0) -- (0,2);
				\draw[knot, -] (1,2) .. controls (1,1.25) and (2,1.25) .. (2,2);
				\draw[stealth-, red, thick] (1.8,3.7-1.75) -- (1.2,2.8-1.75);
			}
			$};
		\node[scale=1.25] at (0,6) {$
			\tikz[baseline={([yshift=-.5ex]current bounding box.center)}, scale=1]
			{
				\draw[knot,white] (0,0) to[out=150, in=210]node(P)[pos=0.5, color=black]{} (0,1);
				\draw[knot, white] (1,0) to[out=30, in=-30]node(Q)[pos=0.5, color=black]{} (1,1);
				\draw[knot] (P.center) to[out=90, in=-90] (-0.26,1);
				\draw[knot] (Q.center) to[out=90, in=-90] (1.26,1);
				\draw[knot] (0,1) to[out=150, in=180] (0.85,1.5) to[out=0, in=90] (1.26,1);
				\draw[knot, overcross] (1,1) to[out=30, in=0] (0.15,1.5) to[out=180, in=90] (-0.26,1);
				\draw[dashed] (-0.4,1.1) -- (1.4,1.1);
				\draw[knot] (0,0) to[out=150, in=-90] (P.center);
				\draw[knot] (1,0) to[out=30, in=-90] (Q.center);
				\draw[knot] (1,0) to[out=210, in=-90] (0.2, 0.5);
				\draw[knot] (0.2, 0.5) to[out=90, in=210] (1,1);
				\draw[knot, overcross] (0,0) to[out=-30, in=-90] (0.8, 0.5);
				\draw[knot, overcross] (0.8, 0.5) to[out=90, in=-30] (0,1);
				\draw[knot] (0.8, 0.499) -- (0.8, 0.501);
				\draw[-stealth, red, thick] (0.5,1.2) -- (0.5, 1.7);
				\draw[-stealth, red, thick] (0.5,0.575) -- (0.5,1.075);
				\draw[-stealth, red, thick] (0.2,0.035) -- (0.8,0.035);
			}
			$};
		\node[scale=2.5] at (1.3,6.2) {$\cdot$};
		\node[scale=2.5] at (1.3,5.8) {$\cdot$};
		\draw[dotted] (9.5,5) rectangle (14.5,7);
		\node[scale=1.5] at (12,6.5) {$-\epsilon_{*1}\epsilon_{1*} = \epsilon_{0*}\epsilon_{*0}$};
		\node[scale=1.5] at (12,5.5) {$-\overline{\epsilon}_{*1}\overline{\epsilon}_{1*} = \overline{\epsilon}_{0*}\overline{\epsilon}_{*0}$};
	}
	\caption{Extra crossing which changes first simple closure to the second.}
	\label{fig:R2extra1}
\end{figure}

\begin{proposition}\label{lem:trivial}
	The maps in Figure \ref{fig:R2extra} whose target and source have cube indexes which differ by 1 sum to zero.
\end{proposition}
\begin{proof}
	There are several cases to consider, but they naturally fall into two categories. In the first category, the two resolutions of the external crossing yield different simple closures. In the second category, the two resolutions produce the same simple closure. The arguments in each category are analogous, so we present the proof for one representative case in each category and leave the remaining cases to the reader.
	
	We provide an example calculation falling into the first category. See Figure \ref{fig:R2extra1}. In this case, we can compute all the maps explicitly. Note that we have provided decorations on both of the diagrams and the cobordisms; these can be used to make the proper component identifications. We pin down the signs as follows. The signs on the face with yellow edges are fixed since this face anti-commutes. The signs on the purple and green arrows are fixed using the fact that the map induced by forgetting two crossings is a chain map. With these signs, the extra maps cancel with each other.
	
	As a sample computation falling into the second category, we consider the case of adding two crossings on the first simple closure (Figure \ref{fig:chainmapRII+1}). Notice that the red chain homotopy maps are equal to each other. Hence, compositions of these maps with the differential associated to the external crossing either commute or anti-commute. As before, we can fix signs which make the red arrows sum to zero. We label the remaining maps as in Figure \ref{fig:R2extra}. First, note that $b_b\circ b_l=0$, since $b_b=0$. We can also show that $b_r=0$. Since the map induced from adding the second crossing is a chain map, we have
	\[h_1\circ d_1+h_2\circ d_2=d_3\circ h_3+d_4\circ b_r.\]
Since $h_1\circ d_1= d_3\circ h_3$, it follows that $d_4\circ b_r=h_1\circ d_1=0$. Note that $d_4$ is nontrivial (it is the split map on odd Khovanov homology), so we conclude that $b_r=0$.
\end{proof}

\subsection{Reidemeister III moves}
When two planar diagrams $N_1$, $N_2$ differ from each other by a Reidemeister III move, we use the following pseudo-diagram moves to pass from $N_1$ to $N_2$; refer to Figure \ref{R3picture}.
\begin{enumerate}[label=(\arabic*)]
\item First, remove the two under-crossing $c_1$ and $c_2$ from the crossing set $N_1$ to obtain a smaller set of crossings $N'$.
\item Second, apply a pseudo-diagram isotopy to move the strand past the central crossing.
\item Third, add two new crossings to $N'$ to obtain the crossing-set $N_2$.
\end{enumerate}

\begin{figure}[h!]
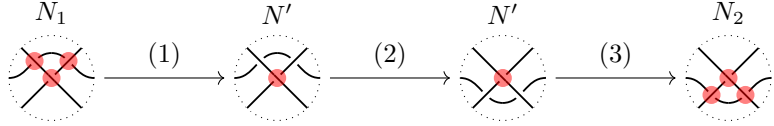

	\centering
	\tikz[baseline={([yshift=-.5ex]current bounding box.center)}]
	{
		\node(A) at (0,0) {$
		\tikz[baseline={([yshift=-.5ex]current bounding box.center)}, scale=0.8]
                {
        \draw[dotted] (.5,.5) circle(0.707);
        \draw[knot] (-0.207, 0.5) to[out=0, in=180] (0.5, 0.915);
        \draw[knot] (1.207, 0.5) to[out=180, in=0] (0.5, 0.915);
        \draw[knot, overcross] (0,0) -- (0.45, 0.45);
        \draw[knot] (0.45, 0.45) -- (0.55, 0.55);
        \draw[knot, overcross] (0.55, 0.55) -- (1,1);
        \draw[knot, overcross] (0,1) -- (0.45, 0.55);
        \draw[knot] (0,1) -- (0.01, 0.99);
        \draw[knot, overcross] (0.45, 0.55) -- (0.55, 0.45);
        \draw[knot] (0.55, 0.45) -- (1,0);
        \fill[red, fill opacity=0.5] (0.5, 0.5) circle(0.15);
        \fill[red, fill opacity=0.5] (0.782, 0.782) circle(0.15);
        \fill[red, fill opacity=0.5] (0.5-0.282, 0.782) circle(0.15);
            }
			$};
		\node(B) at (3,0) {$
		\tikz[baseline={([yshift=-.5ex]current bounding box.center)}, scale=0.8]
			{
        \draw[dotted] (.5,.5) circle(0.707);
        \draw[knot] (-0.207, 0.5) to[out=0, in=180] (0.5, 0.915);
        \draw[knot] (1.207, 0.5) to[out=180, in=0] (0.5, 0.915);
        \draw[knot, overcross] (0,0) -- (0.45, 0.45);
        \draw[knot] (0.45, 0.45) -- (0.55, 0.55);
        \draw[knot, overcross] (0.55, 0.55) -- (1,1);
        \draw[knot, overcross] (0,1) -- (0.45, 0.55);
        \draw[knot] (0,1) -- (0.01, 0.99);
        \draw[knot, overcross] (0.45, 0.55) -- (0.55, 0.45);
        \draw[knot] (0.55, 0.45) -- (1,0);
        \fill[red, fill opacity=0.5] (0.5, 0.5) circle(0.15);
			}
			$};
		\node(C) at (6,0) {$
		\tikz[baseline={([yshift=-.5ex]current bounding box.center)}, scale=0.8]
			{
        \draw[dotted] (.5,.5) circle(0.707);
        \draw[knot] (-0.207, 0.5) to[out=0, in=180] (0.5, 0.085);
        \draw[knot] (1.207, 0.5) to[out=180, in=0] (0.5, 0.085);
        \draw[knot, overcross] (0,0) -- (0.45, 0.45);
        \draw[knot] (0.45, 0.45) -- (0.55, 0.55);
        \draw[knot, overcross] (0.55, 0.55) -- (1,1);
        \draw[knot, overcross] (0,1) -- (0.45, 0.55);
        \draw[knot] (0,1) -- (0.01, 0.99);
        \draw[knot, overcross] (0.45, 0.55) -- (0.55, 0.45);
        \draw[knot, overcross] (0.55, 0.45) -- (1,0);
        \fill[red, fill opacity=0.5] (0.5, 0.5) circle(0.15);			}
			$};
		\node(D) at (9,0) {$
		\tikz[baseline={([yshift=-.5ex]current bounding box.center)}, scale=0.8]
			{
        \draw[dotted] (.5,.5) circle(0.707);
        \draw[knot] (-0.207, 0.5) to[out=0, in=180] (0.5, 0.085);
        \draw[knot] (1.207, 0.5) to[out=180, in=0] (0.5, 0.085);
        \draw[knot, overcross] (0,0) -- (0.45, 0.45);
        \draw[knot] (0.45, 0.45) -- (0.55, 0.55);
        \draw[knot, overcross] (0.55, 0.55) -- (1,1);
        \draw[knot, overcross] (0,1) -- (0.45, 0.55);
        \draw[knot] (0,1) -- (0.01, 0.99);
        \draw[knot, overcross] (0.45, 0.55) -- (0.55, 0.45);
        \draw[knot, overcross] (0.55, 0.45) -- (1,0);
        \fill[red, fill opacity=0.5] (0.5, 0.5) circle(0.15);
        \fill[red, fill opacity=0.5] (0.218, 0.218) circle(0.15);
        \fill[red, fill opacity=0.5] (0.782, 0.218) circle(0.15);
            }
			$};
		\node[above=-1mm] at (A.north) {$N_1$};
		\node[above=-1mm] at (B.north) {$N'$};
		\node[above=-1mm] at (C.north) {$N'$};
		\node[above=-1mm] at (D.north) {$N_2$};
		\draw[->] (A) to node[pos=0.5, above]{(1)} (B);  
		\draw[->] (B) to node[pos=0.5, above]{(2)} (C);
		\draw[->] (C) to node[pos=0.5, above]{(3)} (D);
	}
	\caption{Sequence of pseudo-moves for Reidemeister III.}
	\label{R3picture}
\end{figure}

To compare with the chain homotopy of odd Khovanov homology in \cite[\S 3.7]{migdail2024functoriality}, we treat the cube of resolutions as a mapping cone of the map associated to the resolution at the central crossing. To avoid confusion, we use $d$ (respectively, $d'$) to denote the differential associated to the central crossing before (respectively, after) the Reidemeister III move, and $D$ for the internal differentials.

\[
\tikz[baseline={([yshift=-.5ex]current bounding box.center)}, scale=0.8]
			{
        \draw[dotted] (.5,.5) circle(0.707);
        \draw[knot] (-0.207, 0.5) to[out=0, in=180] (0.5, 0.915);
        \draw[knot] (1.207, 0.5) to[out=180, in=0] (0.5, 0.915);
        \draw[knot, overcross] (0,0) -- (0.45, 0.45);
        \draw[knot] (0.45, 0.45) -- (0.55, 0.55);
        \draw[knot, overcross] (0.55, 0.55) -- (1,1);
        \draw[knot, overcross] (0,1) -- (0.45, 0.55);
        \draw[knot] (0,1) -- (0.01, 0.99);
        \draw[knot, overcross] (0.45, 0.55) -- (0.55, 0.45);
        \draw[knot] (0.55, 0.45) -- (1,0);
        \fill[red, fill opacity=0.5] (0.5, 0.5) circle(0.15);			
        }
= \mathrm{Cone} \left(
\begin{tikzcd}
\tikz[baseline={([yshift=-.5ex]current bounding box.center)}, scale=0.8]
			{
        \draw[dotted] (.5,.5) circle(0.707);
        \draw[knot] (-0.207, 0.5) to[out=0, in=180] (0.5, 0.915);
        \draw[knot] (1.207, 0.5) to[out=180, in=0] (0.5, 0.915);
        \draw[knot, overcross] (0,0) to[out=45, in=-90] (0.375,0.5) to[out=90, in=-45] (0,1);
        \draw[knot, overcross] (1,0) to[out=135, in=-90] (0.625,0.5) to[out=90, in=-135] (1,1);
            }
\arrow[r, "d"] & \tikz[baseline={([yshift=-.5ex]current bounding box.center)}, scale=0.8]
			{
        \draw[dotted] (.5,.5) circle(0.707);
        \draw[knot] (-0.207, 0.5) to[out=0, in=180] (0.5, 0.915);
        \draw[knot] (1.207, 0.5) to[out=180, in=0] (0.5, 0.915);
        \draw[knot, overcross] (0,0) to[out=45, in=180] (0.5,0.375) to[out=0, in=135] (1,0);
        \draw[knot, overcross] (0,1) to[out=-45, in=180] (0.5,0.625) to[out=0, in=-135] (1,1);
            }
\end{tikzcd}
\right)
\qquad
\tikz[baseline={([yshift=-.5ex]current bounding box.center)}, scale=0.8]
			{
        \draw[dotted] (.5,.5) circle(0.707);
        \draw[knot] (-0.207, 0.5) to[out=0, in=180] (0.5, 0.085);
        \draw[knot] (1.207, 0.5) to[out=180, in=0] (0.5, 0.085);
        \draw[knot, overcross] (0,0) -- (0.45, 0.45);
        \draw[knot] (0.45, 0.45) -- (0.55, 0.55);
        \draw[knot, overcross] (0.55, 0.55) -- (1,1);
        \draw[knot, overcross] (0,1) -- (0.45, 0.55);
        \draw[knot] (0,1) -- (0.01, 0.99);
        \draw[knot, overcross] (0.45, 0.55) -- (0.55, 0.45);
        \draw[knot, overcross] (0.55, 0.45) -- (1,0);
        \fill[red, fill opacity=0.5] (0.5, 0.5) circle(0.15);
            } = \mathrm{Cone} \left(
\begin{tikzcd}
\tikz[baseline={([yshift=-.5ex]current bounding box.center)}, scale=0.8]
			{
        \draw[dotted] (.5,.5) circle(0.707);
        \draw[knot] (-0.207, 0.5) to[out=0, in=180] (0.5, 0.085);
        \draw[knot] (1.207, 0.5) to[out=180, in=0] (0.5, 0.085);
        \draw[knot, overcross] (0,0) to[out=45, in=-90] (0.375,0.5) to[out=90, in=-45] (0,1);
        \draw[knot, overcross] (1,0) to[out=135, in=-90] (0.625,0.5) to[out=90, in=-135] (1,1);
            }
 \arrow[r, "d'"] &
 \tikz[baseline={([yshift=-.5ex]current bounding box.center)}, scale=0.8]
			{
        \draw[dotted] (.5,.5) circle(0.707);
        \draw[knot] (-0.207, 0.5) to[out=0, in=180] (0.5, 0.085);
        \draw[knot] (1.207, 0.5) to[out=180, in=0] (0.5, 0.085);
        \draw[knot, overcross] (0,0) to[out=45, in=180] (0.5,0.375) to[out=0, in=135] (1,0);
        \draw[knot, overcross] (0,1) to[out=-45, in=180] (0.5,0.625) to[out=0, in=-135] (1,1);
            }
\end{tikzcd}
\right)
\]

The chain homotopy on the plane knot complex induced by the sequence of pseudo-diagram moves is given by the composition of $\Phi_{(3)}\circ\Phi_{(2)}\circ\Phi_{(1)}$, which is depicted in Figure \ref{fig:R3CH1}. For simplicity, let $f$ denote the homotopy equivalence induced by adding two crossings to subcomplex corresponding to the 1-resolution of the central crossing. Similarly, let $e$ and $e'$ denote the homotopy equivalence induced by adding two crossings to the subcomplexes corresponding to the 0-resolution of the central crossing, as in Figure \ref{fig:R3CH1}. Finally, we let $f^{-1}$, $e^{-1}$, and $(e')^{-1}$ denote their homotopy inverses; \emph{i.e.}, the corresponding maps for removing two crossings. The diagonal maps will be denoted by $h$ with varying subscripts. We provide a compressed version of this sequence of chain homotopy equivalences on the righthand side of Figure \ref{fig:R3CH1}.

\begin{figure}[h]
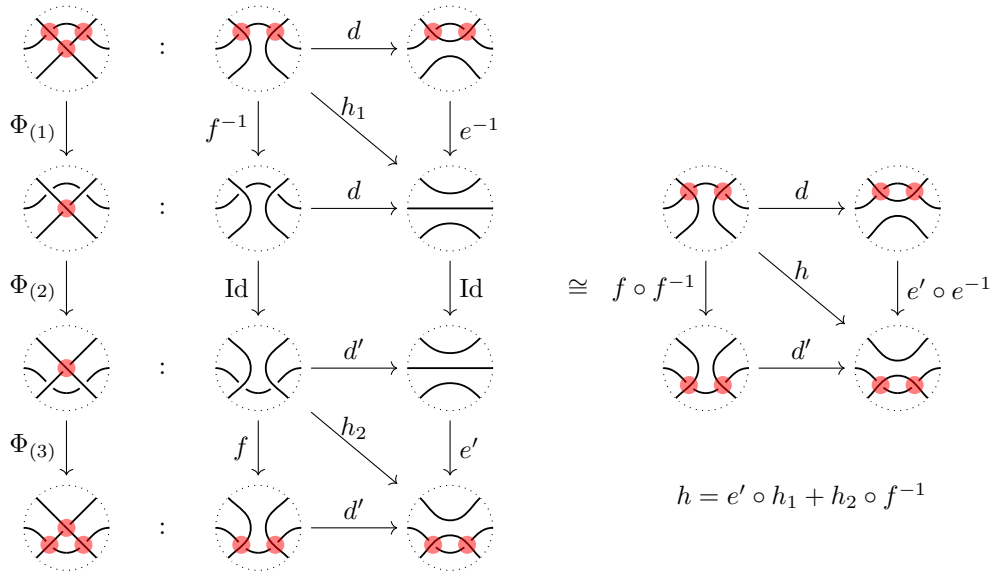

	\centering
	\tikz[baseline={([yshift=-.5ex]current bounding box.center)}, scale=0.85]
	{
		\node(A) at (3,0) {$
			\tikz[baseline={([yshift=-.5ex]current bounding box.center)}, scale=0.8]
			{
	\draw[dotted] (.5,.5) circle(0.707);
        \draw[knot] (-0.207, 0.5) to[out=0, in=180] (0.5, 0.085);
        \draw[knot] (1.207, 0.5) to[out=180, in=0] (0.5, 0.085);
        \draw[knot, overcross] (0,0) to[out=45, in=-90] (0.375,0.5) to[out=90, in=-45] (0,1);
        \draw[knot, overcross] (1,0) to[out=135, in=-90] (0.625,0.5) to[out=90, in=-135] (1,1);
        \fill[red, fill opacity=0.5] (0.218, 0.218) circle(0.15);
        \fill[red, fill opacity=0.5] (0.782, 0.218) circle(0.15);
        			}
			$};
		\node(B) at (6,0) {$
			\tikz[baseline={([yshift=-.5ex]current bounding box.center)}, scale=0.8]
			{
	\draw[dotted] (.5,.5) circle(0.707);
        \draw[knot] (-0.207, 0.5) to[out=0, in=180] (0.5, 0.085);
        \draw[knot] (1.207, 0.5) to[out=180, in=0] (0.5, 0.085);
        \draw[knot, overcross] (0,0) to[out=45, in=180] (0.5,0.375) to[out=0, in=135] (1,0);
        \draw[knot, overcross] (0,1) to[out=-45, in=180] (0.5,0.625) to[out=0, in=-135] (1,1);
        \fill[red, fill opacity=0.5] (0.218, 0.218) circle(0.15);
        \fill[red, fill opacity=0.5] (0.782, 0.218) circle(0.15);
        			}
			$};
		\node(C) at (3,2.5) {$
			\tikz[baseline={([yshift=-.5ex]current bounding box.center)}, scale=0.8]
			{
	\draw[dotted] (.5,.5) circle(0.707);
        \draw[knot] (-0.207, 0.5) to[out=0, in=180] (0.5, 0.085);
        \draw[knot] (1.207, 0.5) to[out=180, in=0] (0.5, 0.085);
        \draw[knot, overcross] (0,0) to[out=45, in=-90] (0.375,0.5) to[out=90, in=-45] (0,1);
        \draw[knot, overcross] (1,0) to[out=135, in=-90] (0.625,0.5) to[out=90, in=-135] (1,1);
        			}
			$};
		\node(D) at (6,2.5) {$
			\tikz[baseline={([yshift=-.5ex]current bounding box.center)}, scale=0.8]
			{
	\draw[dotted] (.5,.5) circle(0.707);
        \draw[knot] (-0.207, 0.5) -- (1.207, 0.5);
        \draw[knot, overcross] (0,0) to[out=45, in=180] (0.5,0.25) to[out=0, in=135] (1,0);
        \draw[knot, overcross] (0,1) to[out=-45, in=180] (0.5,0.75) to[out=0, in=-135] (1,1);
        			}
			$};
		\node(E) at (3,5) {$
			\tikz[baseline={([yshift=-.5ex]current bounding box.center)}, scale=0.8]
			{
	\draw[dotted] (.5,.5) circle(0.707);
        \draw[knot] (-0.207, 0.5) to[out=0, in=180] (0.5, 0.915);
        \draw[knot] (1.207, 0.5) to[out=180, in=0] (0.5, 0.915);
        \draw[knot, overcross] (0,0) to[out=45, in=-90] (0.375,0.5) to[out=90, in=-45] (0,1);
        \draw[knot, overcross] (1,0) to[out=135, in=-90] (0.625,0.5) to[out=90, in=-135] (1,1);
        			}
			$};
		\node(F) at (6,5) {$
			\tikz[baseline={([yshift=-.5ex]current bounding box.center)}, scale=0.8]
			{
	\draw[dotted] (.5,.5) circle(0.707);
        \draw[knot] (-0.207, 0.5) -- (1.207, 0.5);
        \draw[knot, overcross] (0,0) to[out=45, in=180] (0.5,0.25) to[out=0, in=135] (1,0);
        \draw[knot, overcross] (0,1) to[out=-45, in=180] (0.5,0.75) to[out=0, in=-135] (1,1);
        			}
			$};
		\node(G) at (3,7.5) {$
			\tikz[baseline={([yshift=-.5ex]current bounding box.center)}, scale=0.8]
			{
	\draw[dotted] (.5,.5) circle(0.707);
        \draw[knot] (-0.207, 0.5) to[out=0, in=180] (0.5, 0.915);
        \draw[knot] (1.207, 0.5) to[out=180, in=0] (0.5, 0.915);
        \draw[knot, overcross] (0,0) to[out=45, in=-90] (0.375,0.5) to[out=90, in=-45] (0,1);
        \draw[knot, overcross] (1,0) to[out=135, in=-90] (0.625,0.5) to[out=90, in=-135] (1,1);
        \fill[red, fill opacity=0.5] (0.782, 0.782) circle(0.15);
        \fill[red, fill opacity=0.5] (0.5-0.282, 0.782) circle(0.15);
        			}
			$};
		\node(H) at (6,7.5) {$
			\tikz[baseline={([yshift=-.5ex]current bounding box.center)}, scale=0.8]
			{
	\draw[dotted] (.5,.5) circle(0.707);
        \draw[knot] (-0.207, 0.5) to[out=0, in=180] (0.5, 0.915);
        \draw[knot] (1.207, 0.5) to[out=180, in=0] (0.5, 0.915);
        \draw[knot, overcross] (0,0) to[out=45, in=180] (0.5,0.375) to[out=0, in=135] (1,0);
        \draw[knot, overcross] (0,1) to[out=-45, in=180] (0.5,0.625) to[out=0, in=-135] (1,1);
        \fill[red, fill opacity=0.5] (0.782, 0.782) circle(0.15);
        \fill[red, fill opacity=0.5] (0.5-0.282, 0.782) circle(0.15);
        			}
			$};
		\node(I) at (0,0) {$
			\tikz[baseline={([yshift=-.5ex]current bounding box.center)}, scale=0.8]
			{
	\draw[dotted] (.5,.5) circle(0.707);
        \draw[knot] (-0.207, 0.5) to[out=0, in=180] (0.5, 0.085);
        \draw[knot] (1.207, 0.5) to[out=180, in=0] (0.5, 0.085);
        \draw[knot, overcross] (0,0) -- (0.45, 0.45);
        \draw[knot] (0.45, 0.45) -- (0.55, 0.55);
        \draw[knot, overcross] (0.55, 0.55) -- (1,1);
        \draw[knot, overcross] (0,1) -- (0.45, 0.55);
        \draw[knot] (0,1) -- (0.01, 0.99);
        \draw[knot, overcross] (0.45, 0.55) -- (0.55, 0.45);
        \draw[knot, overcross] (0.55, 0.45) -- (1,0);	
        \fill[red, fill opacity=0.5] (0.5, 0.5) circle(0.15);
        \fill[red, fill opacity=0.5] (0.218, 0.218) circle(0.15);
        \fill[red, fill opacity=0.5] (0.782, 0.218) circle(0.15);
        		}
			$};
		\node at (1.5,0) {$:$};
		\node(J) at (0,2.5) {$
			\tikz[baseline={([yshift=-.5ex]current bounding box.center)}, scale=0.8]
			{
	\draw[dotted] (.5,.5) circle(0.707);
        \draw[knot] (-0.207, 0.5) to[out=0, in=180] (0.5, 0.085);
        \draw[knot] (1.207, 0.5) to[out=180, in=0] (0.5, 0.085);
        \draw[knot, overcross] (0,0) -- (0.45, 0.45);
        \draw[knot] (0.45, 0.45) -- (0.55, 0.55);
        \draw[knot, overcross] (0.55, 0.55) -- (1,1);
        \draw[knot, overcross] (0,1) -- (0.45, 0.55);
        \draw[knot] (0,1) -- (0.01, 0.99);
        \draw[knot, overcross] (0.45, 0.55) -- (0.55, 0.45);
        \draw[knot, overcross] (0.55, 0.45) -- (1,0);
        \fill[red, fill opacity=0.5] (0.5, 0.5) circle(0.15);
        			}
			$};
		\node at (1.5,2.5) {$:$};
		\node(K) at (0,5) {$
			\tikz[baseline={([yshift=-.5ex]current bounding box.center)}, scale=0.8]
			{
 	\draw[dotted] (.5,.5) circle(0.707);
        \draw[knot] (-0.207, 0.5) to[out=0, in=180] (0.5, 0.915);
        \draw[knot] (1.207, 0.5) to[out=180, in=0] (0.5, 0.915);
        \draw[knot, overcross] (0,0) -- (0.45, 0.45);
        \draw[knot] (0.45, 0.45) -- (0.55, 0.55);
        \draw[knot, overcross] (0.55, 0.55) -- (1,1);
        \draw[knot, overcross] (0,1) -- (0.45, 0.55);
        \draw[knot] (0,1) -- (0.01, 0.99);
        \draw[knot, overcross] (0.45, 0.55) -- (0.55, 0.45);
        \draw[knot] (0.55, 0.45) -- (1,0);
        \fill[red, fill opacity=0.5] (0.5, 0.5) circle(0.15);			
        }
			$};
		\node at (1.5,5) {$:$};
		\node(L) at (0,7.5) {$
			\tikz[baseline={([yshift=-.5ex]current bounding box.center)}, scale=0.8]
			{
	 \draw[dotted] (.5,.5) circle(0.707);
        \draw[knot] (-0.207, 0.5) to[out=0, in=180] (0.5, 0.915);
        \draw[knot] (1.207, 0.5) to[out=180, in=0] (0.5, 0.915);
        \draw[knot, overcross] (0,0) -- (0.45, 0.45);
        \draw[knot] (0.45, 0.45) -- (0.55, 0.55);
        \draw[knot, overcross] (0.55, 0.55) -- (1,1);
        \draw[knot, overcross] (0,1) -- (0.45, 0.55);
        \draw[knot] (0,1) -- (0.01, 0.99);
        \draw[knot, overcross] (0.45, 0.55) -- (0.55, 0.45);
        \draw[knot] (0.55, 0.45) -- (1,0);
        \fill[red, fill opacity=0.5] (0.5, 0.5) circle(0.15);
        \fill[red, fill opacity=0.5] (0.782, 0.782) circle(0.15);
        \fill[red, fill opacity=0.5] (0.5-0.282, 0.782) circle(0.15);
			}
			$};
		\node at (1.5,7.5) {$:$};
		
		\node(A1) at (10,2.5) {$
			\tikz[baseline={([yshift=-.5ex]current bounding box.center)}, scale=0.8]
			{
	\draw[dotted] (.5,.5) circle(0.707);
        \draw[knot] (-0.207, 0.5) to[out=0, in=180] (0.5, 0.085);
        \draw[knot] (1.207, 0.5) to[out=180, in=0] (0.5, 0.085);
        \draw[knot, overcross] (0,0) to[out=45, in=-90] (0.375,0.5) to[out=90, in=-45] (0,1);
        \draw[knot, overcross] (1,0) to[out=135, in=-90] (0.625,0.5) to[out=90, in=-135] (1,1);
        \fill[red, fill opacity=0.5] (0.218, 0.218) circle(0.15);
        \fill[red, fill opacity=0.5] (0.782, 0.218) circle(0.15);
        			}
			$};
		\node(B1) at (13,2.5) {$
			\tikz[baseline={([yshift=-.5ex]current bounding box.center)}, scale=0.8]
			{
	\draw[dotted] (.5,.5) circle(0.707);
        \draw[knot] (-0.207, 0.5) to[out=0, in=180] (0.5, 0.085);
        \draw[knot] (1.207, 0.5) to[out=180, in=0] (0.5, 0.085);
        \draw[knot, overcross] (0,0) to[out=45, in=180] (0.5,0.375) to[out=0, in=135] (1,0);
        \draw[knot, overcross] (0,1) to[out=-45, in=180] (0.5,0.625) to[out=0, in=-135] (1,1);
        \fill[red, fill opacity=0.5] (0.218, 0.218) circle(0.15);
        \fill[red, fill opacity=0.5] (0.782, 0.218) circle(0.15);
        			}
			$};
		\node(C1) at (10,5) {$
			\tikz[baseline={([yshift=-.5ex]current bounding box.center)}, scale=0.8]
			{
	\draw[dotted] (.5,.5) circle(0.707);
        \draw[knot] (-0.207, 0.5) to[out=0, in=180] (0.5, 0.915);
        \draw[knot] (1.207, 0.5) to[out=180, in=0] (0.5, 0.915);
        \draw[knot, overcross] (0,0) to[out=45, in=-90] (0.375,0.5) to[out=90, in=-45] (0,1);
        \draw[knot, overcross] (1,0) to[out=135, in=-90] (0.625,0.5) to[out=90, in=-135] (1,1);
        \fill[red, fill opacity=0.5] (0.782, 0.782) circle(0.15);
        \fill[red, fill opacity=0.5] (0.5-0.282, 0.782) circle(0.15);			
        }
			$};
		\node(D1) at (13,5) {$
			\tikz[baseline={([yshift=-.5ex]current bounding box.center)}, scale=0.8]
			{
	\draw[dotted] (.5,.5) circle(0.707);
        \draw[knot] (-0.207, 0.5) to[out=0, in=180] (0.5, 0.915);
        \draw[knot] (1.207, 0.5) to[out=180, in=0] (0.5, 0.915);
        \draw[knot, overcross] (0,0) to[out=45, in=180] (0.5,0.375) to[out=0, in=135] (1,0);
        \draw[knot, overcross] (0,1) to[out=-45, in=180] (0.5,0.625) to[out=0, in=-135] (1,1);
        \fill[red, fill opacity=0.5] (0.782, 0.782) circle(0.15);
        \fill[red, fill opacity=0.5] (0.5-0.282, 0.782) circle(0.15);
        			}
			$};
		\node at (8,3.75) {$\cong$};
		\node at (11.5,0.5) {$h=e'\circ h_{1}+h_2\circ f^{-1}$};
		\draw[->] (A) to node[pos=0.5, above]{$d'$} (B);
		\draw[->] (C) to node[pos=0.5, above]{$d'$} (D);
		\draw[->] (E) to node[pos=0.5, above]{$d$} (F);
		\draw[->] (G) to node[pos=0.5, above]{$d$} (H);
		\draw[->] (C) to node[pos=0.5, left]{$f$} (A);
		\draw[->] (D) to node[pos=0.5, right]{$e'$} (B);
		\draw[->] (E) to node[pos=0.5, left]{$\Id$} (C);
		\draw[->] (F) to node[pos=0.5, right]{$\Id$} (D);
		\draw[->] (G) to node[pos=0.5, left]{$f^{-1}$} (E);
		\draw[->] (H) to node[pos=0.5, right]{$e^{-1}$} (F);
		\draw[->] (C) to node[pos=0.5, above]{$h_2$} (B);
		\draw[->] (G) to node[pos=0.5, above]{$h_1$} (F);
		\draw[->] (J) to node[pos=0.5, left]{$\Phi_{(3)}$} (I);
		\draw[->] (K) to node[pos=0.5, left]{$\Phi_{(2)}$} (J);
		\draw[->] (L) to node[pos=0.5, left]{$\Phi_{(1)}$} (K);
		\draw[->] (A1) to node[pos=0.5, above]{$d'$} (B1);
		\draw[->] (C1) to node[pos=0.5, above]{$d$} (D1);
		\draw[->] (C1) to node[pos=0.5, left]{$f\circ f^{-1}$} (A1);
		\draw[->] (D1) to node[pos=0.5, right]{$e'\circ e^{-1}$} (B1);
		\draw[->] (C1) to node[pos=0.5, above]{$h$} (B1);	
	}
	\caption{Chain maps for the sequence of pseudo-moves from Figure \ref{R3picture}.}
	\label{fig:R3CH1}
\end{figure}

The chain homotopy on odd Khovanov homology corresponding to the Reidemeister III move is given as follows. First, take the mapping cone of odd Khovanov complexes associated to resolutions of the central crossing. Let $\bar{e}$ and $\bar{e}'$ denote the chain homotopy associated to Reidemeister II moves. These chain maps give rise to homotopy equivalent mapping cones, denoted by $C_1$ and $C_2$ in Figure \ref{decomphi1}. It can be shown that,

\begin{lemma}[\cite{migdail2024functoriality}, Lemma 8]
	$C_1$ and $C_2$ are identical, up to sign.
\end{lemma}

 Hence, by composing these chain maps, we obtain the chain homotopy equivalence corresponding to the Reidemeister III move. This is illustrated on the righthand side of Figure \ref{decomphi1}, where $\bar{E}$ is a chain homotopy between $\bar{e}'\circ (\bar{e}')^{-1}$ and identity map.
\begin{figure}[H]
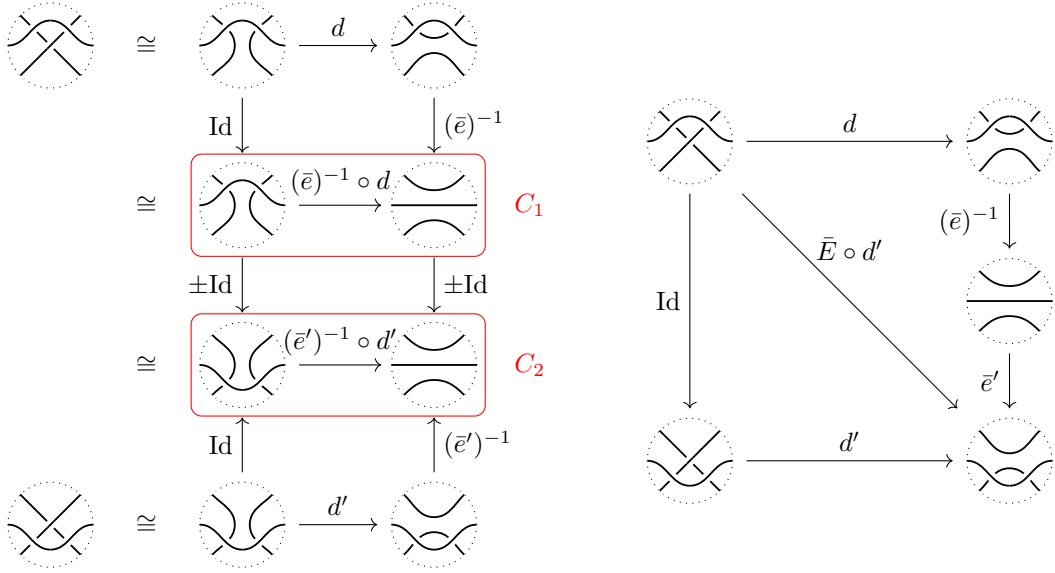

	\centering
	\tikz[baseline={([yshift=-.5ex]current bounding box.center)}, scale=0.85]
	{
		\node(A1) at (3,6) {$
			\tikz[baseline={([yshift=-.5ex]current bounding box.center)}, scale=0.8]
			{
	\draw[dotted] (.5,.5) circle(0.707);
        \draw[knot] (-0.207, 0.5) to[out=0, in=180] (0.5, 0.915);
        \draw[knot] (1.207, 0.5) to[out=180, in=0] (0.5, 0.915);
        \draw[knot, overcross] (0,0) to[out=45, in=-90] (0.375,0.5) to[out=90, in=-45] (0,1);
        \draw[knot, overcross] (1,0) to[out=135, in=-90] (0.625,0.5) to[out=90, in=-135] (1,1);
        \draw[knot, overcross] (-0.207, 0.5) to[out=0, in=180] (0.5, 0.915);
        \draw[knot, overcross] (1.207, 0.5) to[out=180, in=0] (0.5, 0.915);
        			}
			$};
		\node(B1) at (6,6) {$
			\tikz[baseline={([yshift=-.5ex]current bounding box.center)}, scale=0.8]
			{
	\draw[dotted] (.5,.5) circle(0.707);
        \draw[knot] (-0.207, 0.5) -- (1.207, 0.5);
        \draw[knot, overcross] (0,0) to[out=45, in=180] (0.5,0.25) to[out=0, in=135] (1,0);
        \draw[knot, overcross] (0,1) to[out=-45, in=180] (0.5,0.75) to[out=0, in=-135] (1,1);
        			}
			$};
		\node(C1) at (3,8.5) {$
			\tikz[baseline={([yshift=-.5ex]current bounding box.center)}, scale=0.8]
			{
	\draw[dotted] (.5,.5) circle(0.707);
        \draw[knot] (-0.207, 0.5) to[out=0, in=180] (0.5, 0.915);
        \draw[knot] (1.207, 0.5) to[out=180, in=0] (0.5, 0.915);
        \draw[knot, overcross] (0,0) to[out=45, in=-90] (0.375,0.5) to[out=90, in=-45] (0,1);
        \draw[knot, overcross] (1,0) to[out=135, in=-90] (0.625,0.5) to[out=90, in=-135] (1,1);
        \draw[knot, overcross] (-0.207, 0.5) to[out=0, in=180] (0.5, 0.915);
        \draw[knot, overcross] (1.207, 0.5) to[out=180, in=0] (0.5, 0.915);
        			}
			$};
		\node(D1) at (6,8.5) {$
			\tikz[baseline={([yshift=-.5ex]current bounding box.center)}, scale=0.8]
			{
	\draw[dotted] (.5,.5) circle(0.707);
        \draw[knot, overcross] (0,0) to[out=45, in=180] (0.5,0.375) to[out=0, in=135] (1,0);
        \draw[knot, overcross] (0,1) to[out=-45, in=180] (0.5,0.625) to[out=0, in=-135] (1,1);\
        \draw[knot, overcross] (-0.207, 0.5) to[out=0, in=180] (0.5, 0.915);
        \draw[knot, overcross] (1.207, 0.5) to[out=180, in=0] (0.5, 0.915);
        			}
			$};
		\node(E1) at (0,8.5) {$
		\tikz[baseline={([yshift=-.5ex]current bounding box.center)}, scale=0.8]
		{
	\draw[dotted] (.5,.5) circle(0.707);
        \draw[knot] (0,1) -- (0.45, 0.55);
        \draw[knot] (0,1) -- (0.01, 0.99);
        \draw[knot] (0.45, 0.55) -- (0.55, 0.45);
        \draw[knot] (0.55, 0.45) -- (1,0);
        \draw[knot, overcross] (0,0) -- (0.45, 0.45);
        \draw[knot, overcross] (0.45, 0.45) -- (0.55, 0.55);
        \draw[knot, overcross] (0.55, 0.55) -- (1,1);
        \draw[knot, overcross] (-0.207, 0.5) to[out=0, in=180] (0.5, 0.915);
        \draw[knot, overcross] (1.207, 0.5) to[out=180, in=0] (0.5, 0.915);
        		}
		$};
		\node at (1.5,6){$\cong$};
		\node at (1.5,8.5){$\cong$};
		
		\node(Cone5)[red] at (7.5,6) {$C_1$};
		
		\node(A) at (3,1) {$
			\tikz[baseline={([yshift=-.5ex]current bounding box.center)}, scale=0.8]
			{
	\draw[dotted] (.5,.5) circle(0.707);
        \draw[knot, overcross] (0,0) to[out=45, in=-90] (0.375,0.5) to[out=90, in=-45] (0,1);
        \draw[knot, overcross] (1,0) to[out=135, in=-90] (0.625,0.5) to[out=90, in=-135] (1,1);
        \draw[knot, overcross] (-0.207, 0.5) to[out=0, in=180] (0.5, 0.085);
        \draw[knot, overcross] (1.207, 0.5) to[out=180, in=0] (0.5, 0.085);
        			}
			$};
		\node(B) at (6,1) {$
			\tikz[baseline={([yshift=-.5ex]current bounding box.center)}, scale=0.8]
			{
	\draw[dotted] (.5,.5) circle(0.707);
        \draw[knot, overcross] (0,0) to[out=45, in=180] (0.5,0.375) to[out=0, in=135] (1,0);
        \draw[knot, overcross] (0,1) to[out=-45, in=180] (0.5,0.625) to[out=0, in=-135] (1,1);
        \draw[knot, overcross] (-0.207, 0.5) to[out=0, in=180] (0.5, 0.085);
        \draw[knot, overcross] (1.207, 0.5) to[out=180, in=0] (0.5, 0.085);
        			}
			$};
		\node(C) at (3,3.5) {$
			\tikz[baseline={([yshift=-.5ex]current bounding box.center)}, scale=0.8]
			{
	\draw[dotted] (.5,.5) circle(0.707);
        \draw[knot, overcross] (0,0) to[out=45, in=-90] (0.375,0.5) to[out=90, in=-45] (0,1);
        \draw[knot, overcross] (1,0) to[out=135, in=-90] (0.625,0.5) to[out=90, in=-135] (1,1);
        \draw[knot, overcross] (-0.207, 0.5) to[out=0, in=180] (0.5, 0.085);
        \draw[knot, overcross] (1.207, 0.5) to[out=180, in=0] (0.5, 0.085);
        			}
			$};
		\node(D) at (6,3.5) {$
			\tikz[baseline={([yshift=-.5ex]current bounding box.center)}, scale=0.8]
			{
	\draw[dotted] (.5,.5) circle(0.707);
        \draw[knot] (-0.207, 0.5) -- (1.207, 0.5);
        \draw[knot, overcross] (0,0) to[out=45, in=180] (0.5,0.25) to[out=0, in=135] (1,0);
        \draw[knot, overcross] (0,1) to[out=-45, in=180] (0.5,0.75) to[out=0, in=-135] (1,1);
        			}
			$};
		\node(E) at (0,1) {$
      		\tikz[baseline={([yshift=-.5ex]current bounding box.center)}, scale=0.8]
			{
	\draw[dotted] (.5,.5) circle(0.707);
        \draw[knot] (0,1) -- (0.45, 0.55);
        \draw[knot] (0,1) -- (0.01, 0.99);
        \draw[knot] (0.45, 0.55) -- (0.55, 0.45);
        \draw[knot] (0.55, 0.45) -- (1,0);
        \draw[knot, overcross] (0,0) -- (0.45, 0.45);
        \draw[knot, overcross] (0.45, 0.45) -- (0.55, 0.55);
        \draw[knot, overcross] (0.55, 0.55) -- (1,1);
        \draw[knot, overcross] (-0.207, 0.5) to[out=0, in=180] (0.5, 0.085);
        \draw[knot, overcross] (1.207, 0.5) to[out=180, in=0] (0.5, 0.085);
        			}
			$};
		\node at (1.5,3.5){$\cong$};
		\node at (1.5,1){$\cong$};
		
		\node(Cone2)[red] at (7.5,3.5) {$C_2$};
		
		\node(A2) at (10,2) {$
			\tikz[baseline={([yshift=-.5ex]current bounding box.center)}, scale=0.8]
			{
	\draw[dotted] (.5,.5) circle(0.707);
        \draw[knot] (0,1) -- (0.45, 0.55);
        \draw[knot] (0,1) -- (0.01, 0.99);
        \draw[knot] (0.45, 0.55) -- (0.55, 0.45);
        \draw[knot] (0.55, 0.45) -- (1,0);
        \draw[knot, overcross] (0,0) -- (0.45, 0.45);
        \draw[knot, overcross] (0.45, 0.45) -- (0.55, 0.55);
        \draw[knot, overcross] (0.55, 0.55) -- (1,1);
        \draw[knot, overcross] (-0.207, 0.5) to[out=0, in=180] (0.5, 0.085);
        \draw[knot, overcross] (1.207, 0.5) to[out=180, in=0] (0.5, 0.085);
        			}
			$};
		\node(B2) at (10,7) {$
			\tikz[baseline={([yshift=-.5ex]current bounding box.center)}, scale=0.8]
			{
	\draw[dotted] (.5,.5) circle(0.707);
        \draw[knot] (0,1) -- (0.45, 0.55);
        \draw[knot] (0,1) -- (0.01, 0.99);
        \draw[knot] (0.45, 0.55) -- (0.55, 0.45);
        \draw[knot] (0.55, 0.45) -- (1,0);
        \draw[knot, overcross] (0,0) -- (0.45, 0.45);
        \draw[knot, overcross] (0.45, 0.45) -- (0.55, 0.55);
        \draw[knot, overcross] (0.55, 0.55) -- (1,1);
        \draw[knot, overcross] (-0.207, 0.5) to[out=0, in=180] (0.5, 0.915);
        \draw[knot, overcross] (1.207, 0.5) to[out=180, in=0] (0.5, 0.915);
        			}
			$};
		\node(C2) at (15,2) {$
			\tikz[baseline={([yshift=-.5ex]current bounding box.center)}, scale=0.8]
			{
	\draw[dotted] (.5,.5) circle(0.707);
        \draw[knot, overcross] (0,0) to[out=45, in=180] (0.5,0.375) to[out=0, in=135] (1,0);
        \draw[knot, overcross] (0,1) to[out=-45, in=180] (0.5,0.625) to[out=0, in=-135] (1,1);
        \draw[knot, overcross] (-0.207, 0.5) to[out=0, in=180] (0.5, 0.085);
        \draw[knot, overcross] (1.207, 0.5) to[out=180, in=0] (0.5, 0.085);
			}
			$};
		\node(D2) at (15,4.5) {$
			\tikz[baseline={([yshift=-.5ex]current bounding box.center)}, scale=0.8]
			{
	\draw[dotted] (.5,.5) circle(0.707);
        \draw[knot] (-0.207, 0.5) -- (1.207, 0.5);
        \draw[knot, overcross] (0,0) to[out=45, in=180] (0.5,0.25) to[out=0, in=135] (1,0);
        \draw[knot, overcross] (0,1) to[out=-45, in=180] (0.5,0.75) to[out=0, in=-135] (1,1);
        			}
			$};
		\node(E2) at (15,7) {$
			\tikz[baseline={([yshift=-.5ex]current bounding box.center)}, scale=0.8]
			{
	\draw[dotted] (.5,.5) circle(0.707);
        \draw[knot, overcross] (0,0) to[out=45, in=180] (0.5,0.375) to[out=0, in=135] (1,0);
        \draw[knot, overcross] (0,1) to[out=-45, in=180] (0.5,0.625) to[out=0, in=-135] (1,1);\
        \draw[knot, overcross] (-0.207, 0.5) to[out=0, in=180] (0.5, 0.915);
        \draw[knot, overcross] (1.207, 0.5) to[out=180, in=0] (0.5, 0.915);
        			}
			$};

		\draw[->] (A1) to  node[pos=0.5, above]{$(\bar{e})^{-1}\circ d$} (B1);
		\draw[->] (C1) to  node[pos=0.5, above]{$d$} (D1);
		\draw[->] (C1) to  node[pos=0.5, left]{Id} (A1);
		\draw[->] (D1) to  node[pos=0.5, right]{$(\bar{e})^{-1}$} (B1);
		\begin{scope}[on background layer]
			\draw[red,rounded corners]
			($(A1)+(-0.8,-0.8)$) rectangle ($(B1)+(0.8,0.8)$);
		\end{scope}
		\draw[->] (A) to  node[pos=0.5, above]{$d'$} (B);
		\draw[->] (C) to  node[pos=0.5, above]{$(\bar{e}')^{-1}\circ d'$} (D);
		\draw[->] (A) to  node[pos=0.5, left]{Id} (C);
		\draw[->] (B) to  node[pos=0.5, right]{$(\bar{e}')^{-1}$} (D);
		\begin{scope}[on background layer]
			\draw[red,rounded corners]
			($(C)+(-0.8,-0.8)$) rectangle ($(D)+(0.8,0.8)$);
		\end{scope}
		
		\draw[->] (A1) to  node[pos=0.5, left]{$\pm\Id$} (C);
		\draw[->] (B1) to  node[pos=0.5, right]{$\pm\Id$} (D);
		
		\draw[->] (B2) to  node[pos=0.5, left]{Id} (A2);
		\draw[->] (A2) to  node[pos=0.5, above]{$d'$} (C2);
		\draw[->] (B2) to  node[pos=0.5, above]{$d$} (E2);
		\draw[->] (B2) to  node[pos=0.5, above=1.2em]{$\bar{E}\circ d'$} (C2);
		\draw[->] (E2) to  node[pos=0.5, left]{$(\bar{e})^{-1}$} (D2);
		\draw[->] (D2) to  node[pos=0.5, left]{$\bar{e}'$} (C2);
		
	}
	\caption{Chain homotopy on the odd Khovanov complex for RIII.}
	\label{decomphi1}
\end{figure}

To construct a chain homotopy between the maps on $\widetilde{\mathrm{PKE}}_2$ and odd Khovanov homology, we first compare the chain map on $\widetilde{\mathrm{PKE}}_2$ with the chain map in Figure \ref{R3comparison}. Since the comparison map in Figure \ref{R3comparison} is the composition of forgetting two crossing and adding them back, it is chain homotopic to the identity map. Hence, we can find a triple of maps $(F,H,G)$ such that
\begin{align*}
	&D\circ F+F\circ D=f\circ f^{-1}-\Id,\\
	&D\circ G+G\circ D=e'\circ (e')^{-1}-\Id,\\
	&D\circ H+H\circ D+d'\circ F+G\circ d=\tilde{h}.
\end{align*}

\begin{figure}[h]
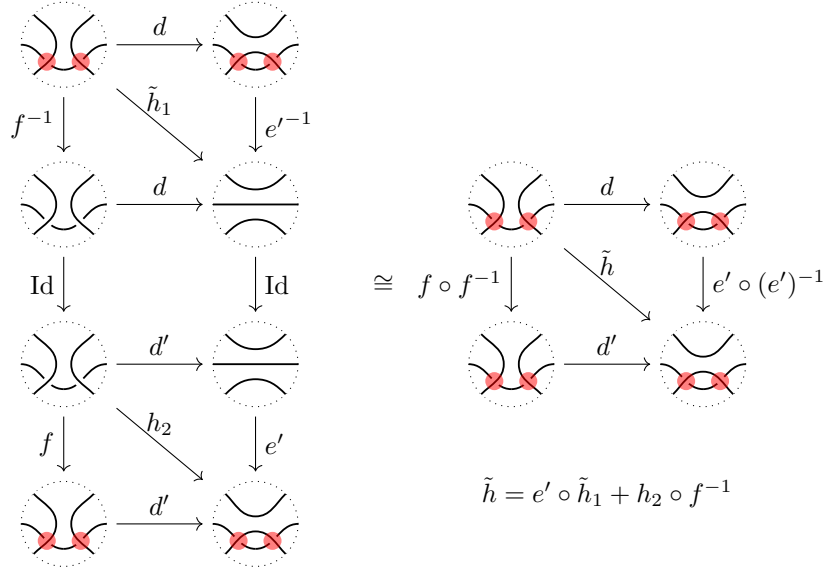

	\centering
	\tikz[baseline={([yshift=-.5ex]current bounding box.center)}, scale=0.85]
	{
		\node(A) at (3,0) {$
			\tikz[baseline={([yshift=-.5ex]current bounding box.center)}, scale=0.8]
			{
	\draw[dotted] (.5,.5) circle(0.707);
        \draw[knot] (-0.207, 0.5) to[out=0, in=180] (0.5, 0.085);
        \draw[knot] (1.207, 0.5) to[out=180, in=0] (0.5, 0.085);
        \draw[knot, overcross] (0,0) to[out=45, in=-90] (0.375,0.5) to[out=90, in=-45] (0,1);
        \draw[knot, overcross] (1,0) to[out=135, in=-90] (0.625,0.5) to[out=90, in=-135] (1,1);
        \fill[red, fill opacity=0.5] (0.218, 0.218) circle(0.15);
        \fill[red, fill opacity=0.5] (0.782, 0.218) circle(0.15);
			}
			$};
		\node(B) at (6,0) {$
			\tikz[baseline={([yshift=-.5ex]current bounding box.center)}, scale=0.8]
			{
	\draw[dotted] (.5,.5) circle(0.707);
        \draw[knot] (-0.207, 0.5) to[out=0, in=180] (0.5, 0.085);
        \draw[knot] (1.207, 0.5) to[out=180, in=0] (0.5, 0.085);
        \draw[knot, overcross] (0,0) to[out=45, in=180] (0.5,0.375) to[out=0, in=135] (1,0);
        \draw[knot, overcross] (0,1) to[out=-45, in=180] (0.5,0.625) to[out=0, in=-135] (1,1);
        \fill[red, fill opacity=0.5] (0.218, 0.218) circle(0.15);
        \fill[red, fill opacity=0.5] (0.782, 0.218) circle(0.15);
			}
			$};
		\node(C) at (3,2.5) {$
			\tikz[baseline={([yshift=-.5ex]current bounding box.center)}, scale=0.8]
			{
	\draw[dotted] (.5,.5) circle(0.707);
        \draw[knot] (-0.207, 0.5) to[out=0, in=180] (0.5, 0.085);
        \draw[knot] (1.207, 0.5) to[out=180, in=0] (0.5, 0.085);
        \draw[knot, overcross] (0,0) to[out=45, in=-90] (0.375,0.5) to[out=90, in=-45] (0,1);
        \draw[knot, overcross] (1,0) to[out=135, in=-90] (0.625,0.5) to[out=90, in=-135] (1,1);
			}
			$};
		\node(D) at (6,2.5) {$
			\tikz[baseline={([yshift=-.5ex]current bounding box.center)}, scale=0.8]
			{
	\draw[dotted] (.5,.5) circle(0.707);
        \draw[knot] (-0.207, 0.5) -- (1.207, 0.5);
        \draw[knot, overcross] (0,0) to[out=45, in=180] (0.5,0.25) to[out=0, in=135] (1,0);
        \draw[knot, overcross] (0,1) to[out=-45, in=180] (0.5,0.75) to[out=0, in=-135] (1,1);
			}
			$};
		\node(E) at (3,5) {$
			\tikz[baseline={([yshift=-.5ex]current bounding box.center)}, scale=0.8]
			{
	\draw[dotted] (.5,.5) circle(0.707);
        \draw[knot] (-0.207, 0.5) to[out=0, in=180] (0.5, 0.085);
        \draw[knot] (1.207, 0.5) to[out=180, in=0] (0.5, 0.085);
        \draw[knot, overcross] (0,0) to[out=45, in=-90] (0.375,0.5) to[out=90, in=-45] (0,1);
        \draw[knot, overcross] (1,0) to[out=135, in=-90] (0.625,0.5) to[out=90, in=-135] (1,1);
			}
			$};
		\node(F) at (6,5) {$
			\tikz[baseline={([yshift=-.5ex]current bounding box.center)}, scale=0.8]
			{
	\draw[dotted] (.5,.5) circle(0.707);
        \draw[knot] (-0.207, 0.5) -- (1.207, 0.5);
        \draw[knot, overcross] (0,0) to[out=45, in=180] (0.5,0.25) to[out=0, in=135] (1,0);
        \draw[knot, overcross] (0,1) to[out=-45, in=180] (0.5,0.75) to[out=0, in=-135] (1,1);
			}
			$};
		\node(G) at (3,7.5) {$
			\tikz[baseline={([yshift=-.5ex]current bounding box.center)}, scale=0.8]
			{
	\draw[dotted] (.5,.5) circle(0.707);
        \draw[knot] (-0.207, 0.5) to[out=0, in=180] (0.5, 0.085);
        \draw[knot] (1.207, 0.5) to[out=180, in=0] (0.5, 0.085);
        \draw[knot, overcross] (0,0) to[out=45, in=-90] (0.375,0.5) to[out=90, in=-45] (0,1);
        \draw[knot, overcross] (1,0) to[out=135, in=-90] (0.625,0.5) to[out=90, in=-135] (1,1);
        \fill[red, fill opacity=0.5] (0.218, 0.218) circle(0.15);
        \fill[red, fill opacity=0.5] (0.782, 0.218) circle(0.15);
			}
			$};
		\node(H) at (6,7.5) {$
			\tikz[baseline={([yshift=-.5ex]current bounding box.center)}, scale=0.8]
			{
	\draw[dotted] (.5,.5) circle(0.707);
        \draw[knot] (-0.207, 0.5) to[out=0, in=180] (0.5, 0.085);
        \draw[knot] (1.207, 0.5) to[out=180, in=0] (0.5, 0.085);
        \draw[knot, overcross] (0,0) to[out=45, in=180] (0.5,0.375) to[out=0, in=135] (1,0);
        \draw[knot, overcross] (0,1) to[out=-45, in=180] (0.5,0.625) to[out=0, in=-135] (1,1);
        \fill[red, fill opacity=0.5] (0.218, 0.218) circle(0.15);
        \fill[red, fill opacity=0.5] (0.782, 0.218) circle(0.15);
			}
			$};
		
		\node(A1) at (10,2.5) {$
			\tikz[baseline={([yshift=-.5ex]current bounding box.center)}, scale=0.8]
			{
	\draw[dotted] (.5,.5) circle(0.707);
        \draw[knot] (-0.207, 0.5) to[out=0, in=180] (0.5, 0.085);
        \draw[knot] (1.207, 0.5) to[out=180, in=0] (0.5, 0.085);
        \draw[knot, overcross] (0,0) to[out=45, in=-90] (0.375,0.5) to[out=90, in=-45] (0,1);
        \draw[knot, overcross] (1,0) to[out=135, in=-90] (0.625,0.5) to[out=90, in=-135] (1,1);
        \fill[red, fill opacity=0.5] (0.218, 0.218) circle(0.15);
        \fill[red, fill opacity=0.5] (0.782, 0.218) circle(0.15);	
			}
			$};
		\node(B1) at (13,2.5) {$
			\tikz[baseline={([yshift=-.5ex]current bounding box.center)}, scale=0.8]
			{
	\draw[dotted] (.5,.5) circle(0.707);
        \draw[knot] (-0.207, 0.5) to[out=0, in=180] (0.5, 0.085);
        \draw[knot] (1.207, 0.5) to[out=180, in=0] (0.5, 0.085);
        \draw[knot, overcross] (0,0) to[out=45, in=180] (0.5,0.375) to[out=0, in=135] (1,0);
        \draw[knot, overcross] (0,1) to[out=-45, in=180] (0.5,0.625) to[out=0, in=-135] (1,1);
        \fill[red, fill opacity=0.5] (0.218, 0.218) circle(0.15);
        \fill[red, fill opacity=0.5] (0.782, 0.218) circle(0.15);
			}
			$};
		\node(C1) at (10,5) {$
			\tikz[baseline={([yshift=-.5ex]current bounding box.center)}, scale=0.8]
			{
	\draw[dotted] (.5,.5) circle(0.707);
        \draw[knot] (-0.207, 0.5) to[out=0, in=180] (0.5, 0.085);
        \draw[knot] (1.207, 0.5) to[out=180, in=0] (0.5, 0.085);
        \draw[knot, overcross] (0,0) to[out=45, in=-90] (0.375,0.5) to[out=90, in=-45] (0,1);
        \draw[knot, overcross] (1,0) to[out=135, in=-90] (0.625,0.5) to[out=90, in=-135] (1,1);
        \fill[red, fill opacity=0.5] (0.218, 0.218) circle(0.15);
        \fill[red, fill opacity=0.5] (0.782, 0.218) circle(0.15);
			}
			$};
		\node(D1) at (13,5) {$
			\tikz[baseline={([yshift=-.5ex]current bounding box.center)}, scale=0.8]
			{
	\draw[dotted] (.5,.5) circle(0.707);
        \draw[knot] (-0.207, 0.5) to[out=0, in=180] (0.5, 0.085);
        \draw[knot] (1.207, 0.5) to[out=180, in=0] (0.5, 0.085);
        \draw[knot, overcross] (0,0) to[out=45, in=180] (0.5,0.375) to[out=0, in=135] (1,0);
        \draw[knot, overcross] (0,1) to[out=-45, in=180] (0.5,0.625) to[out=0, in=-135] (1,1);
        \fill[red, fill opacity=0.5] (0.218, 0.218) circle(0.15);
        \fill[red, fill opacity=0.5] (0.782, 0.218) circle(0.15);
			}
			$};
		\node at (8,3.75) {$\cong$};
		\node at (11.5,0.5) {$\tilde{h}=e'\circ \tilde{h}_{1}+h_2\circ f^{-1}$};
		\draw[->] (A) to node[pos=0.5, above]{$d'$} (B);
		\draw[->] (C) to node[pos=0.5, above]{$d'$} (D);
		\draw[->] (E) to node[pos=0.5, above]{$d$} (F);
		\draw[->] (G) to node[pos=0.5, above]{$d$} (H);
		\draw[->] (C) to node[pos=0.5, left]{$f$} (A);
		\draw[->] (D) to node[pos=0.5, right]{$e'$} (B);
		\draw[->] (E) to node[pos=0.5, left]{Id} (C);
		\draw[->] (F) to node[pos=0.5, right]{Id} (D);
		\draw[->] (G) to node[pos=0.5, left]{$f^{-1}$} (E);
		\draw[->] (H) to node[pos=0.5, right]{${e'}^{-1}$} (F);
		\draw[->] (C) to node[pos=0.5, above]{$h_2$} (B);
		\draw[->] (G) to node[pos=0.5, above]{$\tilde{h}_1$} (F);
		\draw[->] (A1) to node[pos=0.5, above]{$d'$} (B1);
		\draw[->] (C1) to node[pos=0.5, above]{$d$} (D1);
		\draw[->] (C1) to node[pos=0.5, left]{$f\circ f^{-1}$} (A1);
		\draw[->] (D1) to node[pos=0.5, right]{$e'\circ (e')^{-1}$} (B1);
		\draw[->] (C1) to node[pos=0.5, above]{$\tilde{h}$} (B1);	
	}
	\caption{Comparison map.}
	\label{R3comparison}
\end{figure}

Since the chain maps on $\widetilde{\mathrm{PKE}}_2$ induced by Reidemeister II moves are homotopic to the chain maps on odd Khovanov homology induced by Reidemeister II moves (see Subsection \ref{ss:r2}), we can find corresponding chain homotopies $E_{-1}'$, $E'$, and $E_{-1}$ such that
\begin{align*}
	&(e')^{-1}-(\bar{e}')^{-1}=D\circ E_{-1}'+E_{-1}'\circ D,\\
	&e'-\bar{e}'=D\circ E'+E'\circ D,\\
	&(e)^{-1}-(\bar{e})^{-1}=D\circ E_{-1}+E_{-1}\circ D.
\end{align*}

Define $g'=e'\circ E_{-1}'+E'\circ(\bar{e}')^{-1}$. It is straightforward to verify that $G-g'$ is a chain homotopy between the identity map and $\bar{e}'\circ(\bar{e}')^{-1}$.

\begin{figure}[h]
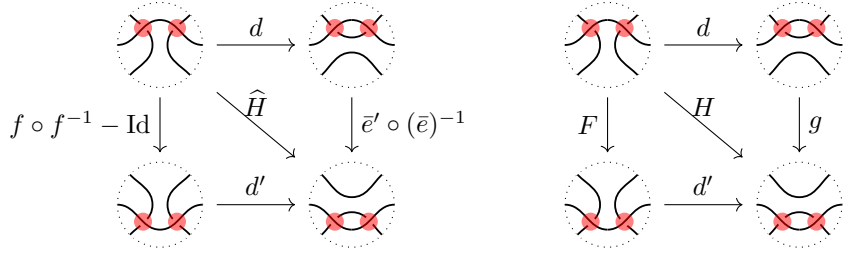

	\centering
	\tikz[baseline={([yshift=-.5ex]current bounding box.center)}, scale=0.85]
	{
		\node(C) at (3,2.5) {$
			\tikz[baseline={([yshift=-.5ex]current bounding box.center)}, scale=0.8]
			{
				\draw[dotted] (.5,.5) circle(0.707);
				\draw[knot, overcross] (0,0) to[out=45, in=-90] (0.375,0.5) to[out=90, in=-45] (0,1);
				\draw[knot, overcross] (1,0) to[out=135, in=-90] (0.625,0.5) to[out=90, in=-135] (1,1);
				\draw[knot, overcross] (-0.207, 0.5) to[out=0, in=180] (0.5, 0.085);
				\draw[knot, overcross] (1.207, 0.5) to[out=180, in=0] (0.5, 0.085);
				\fill[red, fill opacity=0.5] (0.218, 0.218) circle(0.15);
				\fill[red, fill opacity=0.5] (0.782, 0.218) circle(0.15);
			}
			$};
		\node(D) at (6,2.5) {$
			\tikz[baseline={([yshift=-.5ex]current bounding box.center)}, scale=0.8]
			{
				\draw[dotted] (.5,.5) circle(0.707);
				\draw[knot, overcross] (0,0) to[out=45, in=180] (0.5,0.375) to[out=0, in=135] (1,0);
				\draw[knot, overcross] (0,1) to[out=-45, in=180] (0.5,0.625) to[out=0, in=-135] (1,1);
				\draw[knot, overcross] (-0.207, 0.5) to[out=0, in=180] (0.5, 0.085);
				\draw[knot, overcross] (1.207, 0.5) to[out=180, in=0] (0.5, 0.085);
				\fill[red, fill opacity=0.5] (0.218, 0.218) circle(0.15);
				\fill[red, fill opacity=0.5] (0.782, 0.218) circle(0.15);
			}
			$};
		\node(E) at (3,5) {$
			\tikz[baseline={([yshift=-.5ex]current bounding box.center)}, scale=0.8]
			{
				\draw[dotted] (.5,.5) circle(0.707);
				\draw[knot] (-0.207, 0.5) to[out=0, in=180] (0.5, 0.915);
				\draw[knot] (1.207, 0.5) to[out=180, in=0] (0.5, 0.915);
				\draw[knot, overcross] (0,0) to[out=45, in=-90] (0.375,0.5) to[out=90, in=-45] (0,1);
				\draw[knot, overcross] (1,0) to[out=135, in=-90] (0.625,0.5) to[out=90, in=-135] (1,1);
				\draw[knot, overcross] (-0.207, 0.5) to[out=0, in=180] (0.5, 0.915);
				\draw[knot, overcross] (1.207, 0.5) to[out=180, in=0] (0.5, 0.915);
				\fill[red, fill opacity=0.5] (0.782, 0.782) circle(0.15);
				\fill[red, fill opacity=0.5] (0.5-0.282, 0.782) circle(0.15);
			}
			$};
		\node(F) at (6,5) {$
			\tikz[baseline={([yshift=-.5ex]current bounding box.center)}, scale=0.8]
			{
				\draw[dotted] (.5,.5) circle(0.707);
				\draw[knot, overcross] (0,0) to[out=45, in=180] (0.5,0.375) to[out=0, in=135] (1,0);
				\draw[knot, overcross] (0,1) to[out=-45, in=180] (0.5,0.625) to[out=0, in=-135] (1,1);\
				\draw[knot, overcross] (-0.207, 0.5) to[out=0, in=180] (0.5, 0.915);
				\draw[knot, overcross] (1.207, 0.5) to[out=180, in=0] (0.5, 0.915);
				\fill[red, fill opacity=0.5] (0.782, 0.782) circle(0.15);
				\fill[red, fill opacity=0.5] (0.5-0.282, 0.782) circle(0.15);
			}
			$};
		
		\node(A1) at (10,2.5) {$
			\tikz[baseline={([yshift=-.5ex]current bounding box.center)}, scale=0.8]
			{
				\draw[dotted] (.5,.5) circle(0.707);
				\draw[knot, overcross] (0,0) to[out=45, in=-90] (0.375,0.5) to[out=90, in=-45] (0,1);
				\draw[knot, overcross] (1,0) to[out=135, in=-90] (0.625,0.5) to[out=90, in=-135] (1,1);
				\draw[knot, overcross] (-0.207, 0.5) to[out=0, in=180] (0.5, 0.085);
				\draw[knot, overcross] (1.207, 0.5) to[out=180, in=0] (0.5, 0.085);
				\fill[red, fill opacity=0.5] (0.218, 0.218) circle(0.15);
				\fill[red, fill opacity=0.5] (0.782, 0.218) circle(0.15);
			}
			$};
		\node(B1) at (13,2.5) {$
			\tikz[baseline={([yshift=-.5ex]current bounding box.center)}, scale=0.8]
			{
				\draw[dotted] (.5,.5) circle(0.707);
				\draw[knot, overcross] (0,0) to[out=45, in=180] (0.5,0.375) to[out=0, in=135] (1,0);
				\draw[knot, overcross] (0,1) to[out=-45, in=180] (0.5,0.625) to[out=0, in=-135] (1,1);
				\draw[knot, overcross] (-0.207, 0.5) to[out=0, in=180] (0.5, 0.085);
				\draw[knot, overcross] (1.207, 0.5) to[out=180, in=0] (0.5, 0.085);
				\fill[red, fill opacity=0.5] (0.218, 0.218) circle(0.15);
				\fill[red, fill opacity=0.5] (0.782, 0.218) circle(0.15);
			}
			$};
		\node(C1) at (10,5) {$
			\tikz[baseline={([yshift=-.5ex]current bounding box.center)}, scale=0.8]
			{
				\draw[dotted] (.5,.5) circle(0.707);
				\draw[knot] (-0.207, 0.5) to[out=0, in=180] (0.5, 0.915);
				\draw[knot] (1.207, 0.5) to[out=180, in=0] (0.5, 0.915);
				\draw[knot, overcross] (0,0) to[out=45, in=-90] (0.375,0.5) to[out=90, in=-45] (0,1);
				\draw[knot, overcross] (1,0) to[out=135, in=-90] (0.625,0.5) to[out=90, in=-135] (1,1);
				\draw[knot, overcross] (-0.207, 0.5) to[out=0, in=180] (0.5, 0.915);
				\draw[knot, overcross] (1.207, 0.5) to[out=180, in=0] (0.5, 0.915);
				\fill[red, fill opacity=0.5] (0.782, 0.782) circle(0.15);
				\fill[red, fill opacity=0.5] (0.5-0.282, 0.782) circle(0.15);
			}
			$};
		\node(D1) at (13,5) {$
			\tikz[baseline={([yshift=-.5ex]current bounding box.center)}, scale=0.8]
			{
				\draw[dotted] (.5,.5) circle(0.707);
				\draw[knot, overcross] (0,0) to[out=45, in=180] (0.5,0.375) to[out=0, in=135] (1,0);
				\draw[knot, overcross] (0,1) to[out=-45, in=180] (0.5,0.625) to[out=0, in=-135] (1,1);\
				\draw[knot, overcross] (-0.207, 0.5) to[out=0, in=180] (0.5, 0.915);
				\draw[knot, overcross] (1.207, 0.5) to[out=180, in=0] (0.5, 0.915);
				\fill[red, fill opacity=0.5] (0.782, 0.782) circle(0.15);
				\fill[red, fill opacity=0.5] (0.5-0.282, 0.782) circle(0.15);
			}
			$};
		\draw[->] (F) to node[pos=0.5, right]{$\bar{e}'\circ (\bar{e})^{-1}$} (D);
		\draw[->] (C) to node[pos=0.5, above]{$d'$} (D);
		\draw[->] (E) to node[pos=0.5, above]{$d$} (F);
		\draw[->] (E) to node[pos=0.5, left]{$f\circ f^{-1}-\text{Id}$} (C);
		\draw[->] (E) to node[pos=0.5, above]{$\widehat{H}$} (D);
		\draw[->] (A1) to node[pos=0.5, above]{$d'$} (B1);
		\draw[->] (C1) to node[pos=0.5, above]{$d$} (D1);
		\draw[->] (C1) to node[pos=0.5, left]{$F$} (A1);
		\draw[->] (D1) to node[pos=0.5, right]{$g$} (B1);
		\draw[->] (C1) to node[pos=0.5, above]{$H$} (B1);	
	}
	\caption{Chain homotopy.}
	\label{fig:R3chainhomotopy}
\end{figure}

With this set up, we construct a chain homotopy between the Reidemeister III maps on odd Khovanov homology and the maps induced from plane Floer homology. The chain homotopy is illustrated on the righthand side of Figure \ref{fig:R3chainhomotopy}, where $g:=e'\circ E_{-1}+E'\circ (\bar{e})^{-1}$. Applying this chain homotopy to the Reidemeister III map on $\widetilde{\mathrm{PKE}}_2$ from Figure \ref{fig:R3CH1}, we get the chain map on the lefthand side of Figure \ref{fig:R3chainhomotopy}, where
\begin{align*}
	\widehat{H}&=h-(D\circ H+H\circ D+d'\circ F+g\circ d)\\
	       &=h-(h'-G\circ d'+g\circ d)\\
	       &=(h-h')+(G\circ d'-g'\circ d')+(g'\circ d'-g\circ d)\\
	       &=e'\circ(h_1-\tilde{h}_1)-e'\circ(E_{-1}\circ d-E_{-1}'\circ d')+(G-g')\circ d'.
\end{align*}

We use the following lemma to prove that the diagonal term $\bar{E}\circ d'$ is chain homotopic to $\widehat{H}$.

\begin{lemma}\label{lemma:R3computation}
	$h_1-\tilde{h}_1=E_{-1}\circ d-E_{-1}'\circ d'$.
\end{lemma}

\begin{proof}
		Order the three crossings so that the central one is third. By grading constrains, the only nontrivial maps coming from $h_1$ and $\tilde{h}_1$ on the $E_2$ page are those whose source is the $(0,0,1)$ resolution and whose target is the $(2,2,0)$ resolution. These maps are the compositions of the blue arrows in Figure \ref{fig:R3local}; denote them by $h_{2}\circ h_1$ and $h_4\circ h_3$, where the subscript is given by the circled number. Note that $h_{4a}\circ h_{3a} = h_{4b}\circ h_{3b}$ since they are induced by locally isotopic cobordisms, independent of the choice of closure. Thus, we have 
\begin{align*}
h_1 - \tilde{h}_1 &= (h_{4a}\circ h_{3a} + h_{2a}\circ h_{1a}) - (h_{4b}\circ h_{3b} + h_{2b}\circ h_{1b})\\
	&= h_{2a}\circ h_{1a} - h_{2b}\circ h_{1b}.
\end{align*}
Therefore, the proof is reduced to providing $E_{-1}$ and $E_{-1}'$ such that $h_{2a}\circ h_{1a} - h_{2b}\circ h_{1b} = E_{-1}\circ d-E_{-1}'\circ d'.$ We complete this computation case-by-case with respect to the possible closures; see Table \ref{tab:Etable}. 

\null
\vfill
\begin{figure}[h]
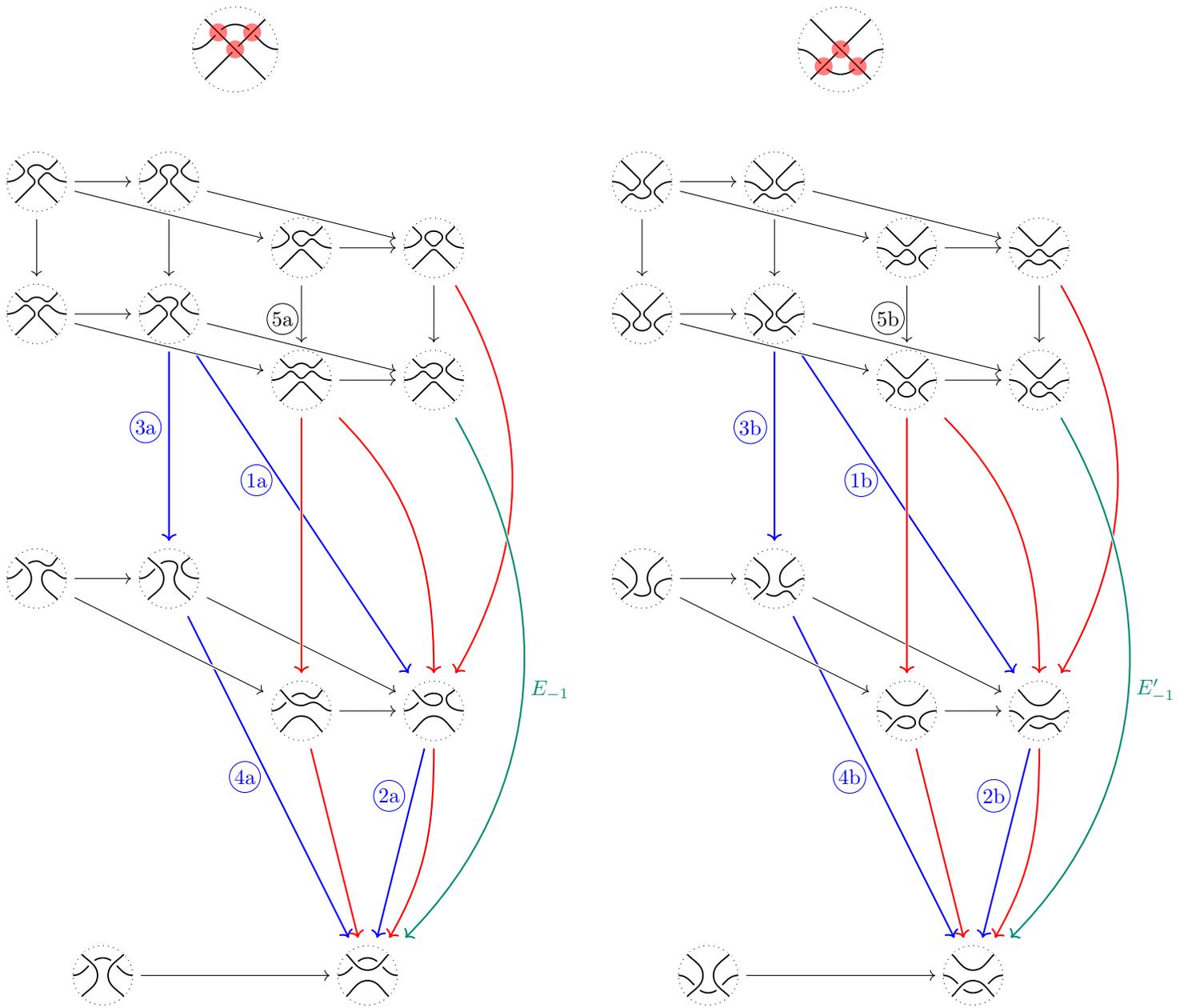

	\makebox[\textwidth][c]{
		\tikz[scale=2.2]{
			\node at (1.5,1) {$
				\tikz[baseline={([yshift=-.5ex]current bounding box.center)}, scale=1]
				{
					\draw[dotted] (.5,.5) circle(0.707);
					\draw[knot] (-0.207, 0.5) to[out=0, in=180] (0.5, 0.915);
					\draw[knot] (1.207, 0.5) to[out=180, in=0] (0.5, 0.915);
					\draw[knot, overcross] (0,0) -- (0.45, 0.45);
					\draw[knot] (0.45, 0.45) -- (0.55, 0.55);
					\draw[knot, overcross] (0.55, 0.55) -- (1,1);
					\draw[knot, overcross] (0,1) -- (0.45, 0.55);
					\draw[knot] (0,1) -- (0.01, 0.99);
					\draw[knot, overcross] (0.45, 0.55) -- (0.55, 0.45);
					\draw[knot] (0.55, 0.45) -- (1,0);
					\fill[red, fill opacity=0.5] (0.5, 0.5) circle(0.15);
					\fill[red, fill opacity=0.5] (0.782, 0.782) circle(0.15);
					\fill[red, fill opacity=0.5] (0.5-0.282, 0.782) circle(0.15);
				}
				$};
			\node(111) at (0,0) {$
				\tikz[baseline={([yshift=-.5ex]current bounding box.center)}, scale=0.7, rotate=-90]
				{
					\draw[dotted] (.5,.5) circle(0.707);
					\draw[knot] (0.5, 1.207) to[out=-90, in=45] (0.325, 0.825);
					\draw[knot] (0.175, 0.675) to[out=225, in=135] (0.175, 0.325);
					\draw[knot] (0.325, 0.175) to[out=-45, in=90] (0.5, -0.207);
					\draw[knot] (0,0) to (0.175, 0.175);
					\draw[knot] (0.325, 0.325) to (0.425, 0.425);
					\draw[knot] (0.575, 0.575) to (1,1);
					\draw[knot] (0,1) to (0.175, 0.825);
					\draw[knot] (0.325, 0.675) to (0.425, 0.575);
					\draw[knot] (0.575, 0.425) to (1,0);
					\draw[knot] (0.175, 0.175) to[out=45, in=135] (0.325, 0.175);
					\draw[knot] (0.175, 0.325) to[out=-45, in=-135] (0.325, 0.325);
					\draw[knot] (0.425, 0.425) to[out=45, in=135] (0.575, 0.425);
					\draw[knot] (0.425, 0.575) to[out=-45, in=-135] (0.575, 0.575);
					\draw[knot] (0.175, 0.675) to[out=45, in=-45] (0.175, 0.825);
					\draw[knot] (0.325, 0.675) to[out=135, in=-135] (0.325, 0.825);
				}
				$};
			\node(101) at (1,0) {$
				\tikz[baseline={([yshift=-.5ex]current bounding box.center)}, scale=0.7, rotate=-90]
				{
					\draw[dotted] (.5,.5) circle(0.707);
					\draw[knot] (0.5, 1.207) to[out=-90, in=45] (0.325, 0.825);
					\draw[knot] (0.175, 0.675) to[out=225, in=135] (0.175, 0.325);
					\draw[knot] (0.325, 0.175) to[out=-45, in=90] (0.5, -0.207);
					\draw[knot] (0,0) to (0.175, 0.175);
					\draw[knot] (0.325, 0.325) to (0.425, 0.425);
					\draw[knot] (0.575, 0.575) to (1,1);
					\draw[knot] (0,1) to (0.175, 0.825);
					\draw[knot] (0.325, 0.675) to (0.425, 0.575);
					\draw[knot] (0.575, 0.425) to (1,0);
					\draw[knot] (0.175, 0.175) to[out=45, in=135] (0.325, 0.175);
					\draw[knot] (0.175, 0.325) to[out=-45, in=-135] (0.325, 0.325);
					\draw[knot] (0.425, 0.425) to[out=45, in=135] (0.575, 0.425);
					\draw[knot] (0.425, 0.575) to[out=-45, in=-135] (0.575, 0.575);
					\draw[knot] (0.175, 0.825) to[out=-45, in=-135] (0.325, 0.825);
					\draw[knot] (0.175, 0.675) to[out=45, in=135] (0.325, 0.675);
				}
				$};
			\node(110) at (2,-.5) {$
				\tikz[baseline={([yshift=-.5ex]current bounding box.center)}, scale=0.7, rotate=-90]
				{
					\draw[dotted] (.5,.5) circle(0.707);
					\draw[knot] (0.5, 1.207) to[out=-90, in=45] (0.325, 0.825);
					\draw[knot] (0.175, 0.675) to[out=225, in=135] (0.175, 0.325);
					\draw[knot] (0.325, 0.175) to[out=-45, in=90] (0.5, -0.207);
					\draw[knot] (0,0) to (0.175, 0.175);
					\draw[knot] (0.325, 0.325) to (0.425, 0.425);
					\draw[knot] (0.575, 0.575) to (1,1);
					\draw[knot] (0,1) to (0.175, 0.825);
					\draw[knot] (0.325, 0.675) to (0.425, 0.575);
					\draw[knot] (0.575, 0.425) to (1,0);
					\draw[knot] (0.175, 0.175) to[out=45, in=135] (0.325, 0.175);
					\draw[knot] (0.175, 0.325) to[out=-45, in=-135] (0.325, 0.325);
					\draw[knot] (0.425, 0.425) to[out=45, in=-45] (0.425, 0.575); 
					\draw[knot] (0.575, 0.425) to[out=135, in=-135] (0.575, 0.575);
					\draw[knot] (0.175, 0.675) to[out=45, in=-45] (0.175, 0.825);
					\draw[knot] (0.325, 0.675) to[out=135, in=-135] (0.325, 0.825);
				}
				$};
			\node(100) at (3,-.5) {$
				\tikz[baseline={([yshift=-.5ex]current bounding box.center)}, scale=0.7, rotate=-90]
				{
					\draw[dotted] (.5,.5) circle(0.707);
					\draw[knot] (0.5, 1.207) to[out=-90, in=45] (0.325, 0.825);
					\draw[knot] (0.175, 0.675) to[out=225, in=135] (0.175, 0.325);
					\draw[knot] (0.325, 0.175) to[out=-45, in=90] (0.5, -0.207);
					\draw[knot] (0,0) to (0.175, 0.175);
					\draw[knot] (0.325, 0.325) to (0.425, 0.425);
					\draw[knot] (0.575, 0.575) to (1,1);
					\draw[knot] (0,1) to (0.175, 0.825);
					\draw[knot] (0.325, 0.675) to (0.425, 0.575);
					\draw[knot] (0.575, 0.425) to (1,0);
					\draw[knot] (0.175, 0.175) to[out=45, in=135] (0.325, 0.175);
					\draw[knot] (0.175, 0.325) to[out=-45, in=-135] (0.325, 0.325);
					\draw[knot] (0.425, 0.425) to[out=45, in=-45] (0.425, 0.575); 
					\draw[knot] (0.575, 0.425) to[out=135, in=-135] (0.575, 0.575);
					\draw[knot] (0.175, 0.825) to[out=-45, in=-135] (0.325, 0.825);
					\draw[knot] (0.175, 0.675) to[out=45, in=135] (0.325, 0.675);
				}
				$};
			\node(011) at (0,-1) {$
				\tikz[baseline={([yshift=-.5ex]current bounding box.center)}, scale=0.7, rotate=-90]
				{
					\draw[dotted] (.5,.5) circle(0.707);
					\draw[knot] (0.5, 1.207) to[out=-90, in=45] (0.325, 0.825);
					\draw[knot] (0.175, 0.675) to[out=225, in=135] (0.175, 0.325);
					\draw[knot] (0.325, 0.175) to[out=-45, in=90] (0.5, -0.207);
					\draw[knot] (0,0) to (0.175, 0.175);
					\draw[knot] (0.325, 0.325) to (0.425, 0.425);
					\draw[knot] (0.575, 0.575) to (1,1);
					\draw[knot] (0,1) to (0.175, 0.825);
					\draw[knot] (0.325, 0.675) to (0.425, 0.575);
					\draw[knot] (0.575, 0.425) to (1,0);
					\draw[knot] (0.175, 0.175) to[out=45, in=-45] (0.175, 0.325);
					\draw[knot] (0.325, 0.175) to[out=135, in=-135] (0.325, 0.325);
					\draw[knot] (0.425, 0.425) to[out=45, in=135] (0.575, 0.425);
					\draw[knot] (0.425, 0.575) to[out=-45, in=-135] (0.575, 0.575);
					\draw[knot] (0.175, 0.675) to[out=45, in=-45] (0.175, 0.825);
					\draw[knot] (0.325, 0.675) to[out=135, in=-135] (0.325, 0.825);
				}
				$};
			\node(001) at (1,-1) {$
				\tikz[baseline={([yshift=-.5ex]current bounding box.center)}, scale=0.7, rotate=-90]
				{
					\draw[dotted] (.5,.5) circle(0.707);
					\draw[knot] (0.5, 1.207) to[out=-90, in=45] (0.325, 0.825);
					\draw[knot] (0.175, 0.675) to[out=225, in=135] (0.175, 0.325);
					\draw[knot] (0.325, 0.175) to[out=-45, in=90] (0.5, -0.207);
					\draw[knot] (0,0) to (0.175, 0.175);
					\draw[knot] (0.325, 0.325) to (0.425, 0.425);
					\draw[knot] (0.575, 0.575) to (1,1);
					\draw[knot] (0,1) to (0.175, 0.825);
					\draw[knot] (0.325, 0.675) to (0.425, 0.575);
					\draw[knot] (0.575, 0.425) to (1,0);
					\draw[knot] (0.175, 0.175) to[out=45, in=-45] (0.175, 0.325);
					\draw[knot] (0.325, 0.175) to[out=135, in=-135] (0.325, 0.325);
					\draw[knot] (0.425, 0.425) to[out=45, in=135] (0.575, 0.425);
					\draw[knot] (0.425, 0.575) to[out=-45, in=-135] (0.575, 0.575);
					\draw[knot] (0.175, 0.825) to[out=-45, in=-135] (0.325, 0.825);
					\draw[knot] (0.175, 0.675) to[out=45, in=135] (0.325, 0.675);
				}
				$};
			\node(010) at (2,-1.5) {$
				\tikz[baseline={([yshift=-.5ex]current bounding box.center)}, scale=0.7, rotate=-90]
				{
					\draw[dotted] (.5,.5) circle(0.707);
					\draw[knot] (0.5, 1.207) to[out=-90, in=45] (0.325, 0.825);
					\draw[knot] (0.175, 0.675) to[out=225, in=135] (0.175, 0.325);
					\draw[knot] (0.325, 0.175) to[out=-45, in=90] (0.5, -0.207);
					\draw[knot] (0,0) to (0.175, 0.175);
					\draw[knot] (0.325, 0.325) to (0.425, 0.425);
					\draw[knot] (0.575, 0.575) to (1,1);
					\draw[knot] (0,1) to (0.175, 0.825);
					\draw[knot] (0.325, 0.675) to (0.425, 0.575);
					\draw[knot] (0.575, 0.425) to (1,0);
					\draw[knot] (0.175, 0.175) to[out=45, in=-45] (0.175, 0.325);
					\draw[knot] (0.325, 0.175) to[out=135, in=-135] (0.325, 0.325);
					\draw[knot] (0.425, 0.425) to[out=45, in=-45] (0.425, 0.575); 
					\draw[knot] (0.575, 0.425) to[out=135, in=-135] (0.575, 0.575);
					\draw[knot] (0.175, 0.675) to[out=45, in=-45] (0.175, 0.825);
					\draw[knot] (0.325, 0.675) to[out=135, in=-135] (0.325, 0.825);
				}
				$};
			\node(000) at (3,-1.5) {$
				\tikz[baseline={([yshift=-.5ex]current bounding box.center)}, scale=0.7, rotate=-90]
				{
					\draw[dotted] (.5,.5) circle(0.707);
					\draw[knot] (0.5, 1.207) to[out=-90, in=45] (0.325, 0.825);
					\draw[knot] (0.175, 0.675) to[out=225, in=135] (0.175, 0.325);
					\draw[knot] (0.325, 0.175) to[out=-45, in=90] (0.5, -0.207);
					\draw[knot] (0,0) to (0.175, 0.175);
					\draw[knot] (0.325, 0.325) to (0.425, 0.425);
					\draw[knot] (0.575, 0.575) to (1,1);
					\draw[knot] (0,1) to (0.175, 0.825);
					\draw[knot] (0.325, 0.675) to (0.425, 0.575);
					\draw[knot] (0.575, 0.425) to (1,0);
					\draw[knot] (0.175, 0.175) to[out=45, in=-45] (0.175, 0.325);
					\draw[knot] (0.325, 0.175) to[out=135, in=-135] (0.325, 0.325);
					\draw[knot] (0.425, 0.425) to[out=45, in=-45] (0.425, 0.575); 
					\draw[knot] (0.575, 0.425) to[out=135, in=-135] (0.575, 0.575);
					\draw[knot] (0.175, 0.825) to[out=-45, in=-135] (0.325, 0.825);
					\draw[knot] (0.175, 0.675) to[out=45, in=135] (0.325, 0.675);
				}
				$};
			\node(211) at (0,-3) {$
				\tikz[baseline={([yshift=-.5ex]current bounding box.center)}, scale=0.7]
				{
					\draw[dotted] (.5,.5) circle(0.707);
					\draw[knot] (-0.207, 0.5) to[out=0, in=180] (0.5, 0.915) to[out=0, in=180] (0.782, 0.8) to[out=0, in=-135] (1,1);
					\draw[knot] (1.207, 0.5) to[out=180, in=0] (0.782, 0.7) to[out=180, in=90] (0.625,0.5) to[out=-90, in=135] (1,0);
					\draw[knot, overcross] (0,0) to[out=45, in=-90] (0.375,0.5) to[out=90, in=-45] (0,1);
				}
				$};
			\node(201) at (1,-3) {$
				\tikz[baseline={([yshift=-.5ex]current bounding box.center)}, scale=0.7]
				{
					\draw[dotted] (.5,.5) circle(0.707);
					\draw[knot] (-0.207, 0.5) to[out=0, in=180] (0.5, 0.915) to[out=0, in=90] (0.7, 0.782) to[out=-90, in=90] (0.625,0.5) to[out=-90, in=135] (1,0);
					\draw[knot] (1,1) to[out=-135, in=90] (0.8, 0.782) to[out=-90, in=180] (1.207, 0.5);
					\draw[knot, overcross] (0,0) to[out=45, in=-90] (0.375,0.5) to[out=90, in=-45] (0,1);
				}
				$};
			\node(210) at (2,-4) {$
				\tikz[baseline={([yshift=-.5ex]current bounding box.center)}, scale=0.7]
				{
					\draw[dotted] (.5,.5) circle(0.707);
					\draw[knot] (-0.207, 0.5) to[out=0, in=180] (0.5, 0.915);
					\draw[knot, overcross] (0,0) to[out=45, in=180] (0.5,0.375) to[out=0, in=135] (1,0);
					\draw[knot, overcross] (0,1) to[out=-45, in=180] (0.5,0.625);
					\fill[white, fill opacity=1] (0.782, 0.782) circle(0.15);
					\draw[knot] (0.5, 0.915) to[out=0, in=180] (0.782, 0.8) to[out=0, in=-135] (1,1);
					\draw[knot] (0.5,0.625) to [out=0, in=180] (0.782, 0.7) to[out=0, in=180] (1.207, 0.5);
				}
				$};
			\node(200) at (3,-4) {$
				\tikz[baseline={([yshift=-.5ex]current bounding box.center)}, scale=0.7]
				{
					\draw[dotted] (.5,.5) circle(0.707);
					\draw[knot] (-0.207, 0.5) to[out=0, in=180] (0.5, 0.915);
					\draw[knot, overcross] (0,0) to[out=45, in=180] (0.5,0.375) to[out=0, in=135] (1,0);
					\draw[knot, overcross] (0,1) to[out=-45, in=180] (0.5,0.625);
					\fill[white, fill opacity=1] (0.782, 0.782) circle(0.15);
					\draw[knot] (0.5, 0.915) to[out=0, in=90] (0.7, 0.782) to[out=-90, in=0] (0.5,0.625);
					\draw[knot] (1,1) to[out=-135, in=90] (0.8, 0.782) to[out=-90, in=180] (1.207, 0.5);
				}
				$};
			\node(222) at (0.5,-6) {$
				\tikz[baseline={([yshift=-.5ex]current bounding box.center)}, scale=0.7]
				{
					\draw[dotted] (.5,.5) circle(0.707);
					\draw[knot] (-0.207, 0.5) to[out=0, in=180] (0.5, 0.915);
					\draw[knot] (1.207, 0.5) to[out=180, in=0] (0.5, 0.915);
					\draw[knot, overcross] (0,0) to[out=45, in=-90] (0.375,0.5) to[out=90, in=-45] (0,1);
					\draw[knot, overcross] (1,0) to[out=135, in=-90] (0.625,0.5) to[out=90, in=-135] (1,1);
				}
				$};
			\node(220) at (2.5,-6) {$
				\tikz[baseline={([yshift=-.5ex]current bounding box.center)}, scale=0.7]
				{
					\draw[dotted] (.5,.5) circle(0.707);
					\draw[knot] (-0.207, 0.5) to[out=0, in=180] (0.5, 0.915);
					\draw[knot] (1.207, 0.5) to[out=180, in=0] (0.5, 0.915);
					\draw[knot, overcross] (0,0) to[out=45, in=180] (0.5,0.375) to[out=0, in=135] (1,0);
					\draw[knot, overcross] (0,1) to[out=-45, in=180] (0.5,0.625) to[out=0, in=-135] (1,1);
				}
				$};
			\draw[->] (201) -- (200);
			\draw[->, thick, blue] (001) -- node[left, pos=0.4]{$\circled{3a}$} (201);
			\draw[white, ultra thick] (001) -- (200);
			\draw[->, thick, blue] (001) -- node[left, pos=0.4]{$\circled{1a}$} (200);
			\draw[->, thick, blue] (201) -- node[left, pos=0.5]{$\circled{4a}$}(220);
			\draw[->, thick, blue] (200) -- node[left, pos=0.25]{$\circled{2a}$} (220);
			\draw[->] (111) -- (110);
			\draw[->] (101) -- (100);
			\draw[white, ultra thick] (011) -- (010);
			\draw[->] (011) -- (010);
			\draw[->] (001) -- node[above, pos=0.4]{$\circled{5a}$}(000);
			\draw[->] (111) -- (101);
			\draw[->] (111) -- (011);
			\draw[->] (011) -- (001);
			\draw[->] (101) -- (001);
			\draw[->] (110) -- (100);
			\draw[->] (110) -- (010);
			\draw[->] (010) -- (000);
			\draw[->] (100) -- (000);
			\draw[->] (211) -- (201);
			\draw[white, ultra thick] (211) -- (210);
			\draw[white, ultra thick] (210) -- (200);
			\draw[->] (211) -- (210);
			\draw[->] (210) -- (200);
			\draw[->] (222) -- (220);
			\draw[->, thick, red] (100) to[out=-60, in=60] node[right]{} (200);
			\draw[white, ultra thick] (010) -- (210);
			\draw[->, thick, red] (010) --node[right, pos=0.8]{} (210);
			\draw[->, thick, red] (010) to[out=-45, in=90] node[right]{} (200);
			\draw[->, thick, red] (210) -- node[right]{} (220);
			\draw[->, thick, red] (200) to[out=-90,in=60]node[right]{} (220);
			\draw[->, ultra thick, white] (000) to[out=-60, in=45] (220);
			\draw[->, thick, PineGreen] (000) to[out=-60, in=45]node[right]{$E_{-1}$} (220);
		}
		\quad
		\tikz[scale=2.2]{
			\node at (1.5,1) {$
				\tikz[baseline={([yshift=-.5ex]current bounding box.center)}, scale=1]
				{
					\draw[dotted] (.5,.5) circle(0.707);
					\draw[knot] (-0.207, 0.5) to[out=0, in=180] (0.5, 0.085);
					\draw[knot] (1.207, 0.5) to[out=180, in=0] (0.5, 0.085);
					\draw[knot, overcross] (0,0) -- (0.45, 0.45);
					\draw[knot] (0.45, 0.45) -- (0.55, 0.55);
					\draw[knot, overcross] (0.55, 0.55) -- (1,1);
					\draw[knot, overcross] (0,1) -- (0.45, 0.55);
					\draw[knot] (0,1) -- (0.01, 0.99);
					\draw[knot, overcross] (0.45, 0.55) -- (0.55, 0.45);
					\draw[knot, overcross] (0.55, 0.45) -- (1,0);	
					\fill[red, fill opacity=0.5] (0.5, 0.5) circle(0.15);
					\fill[red, fill opacity=0.5] (0.218, 0.218) circle(0.15);
					\fill[red, fill opacity=0.5] (0.782, 0.218) circle(0.15);
				}
				$};
			\node(111) at (0,0) {$
				\tikz[baseline={([yshift=-.5ex]current bounding box.center)}, scale=0.7, rotate=90]
				{
					\draw[dotted] (.5,.5) circle(0.707);
					\draw[knot] (0.5, 1.207) to[out=-90, in=45] (0.325, 0.825);
					\draw[knot] (0.175, 0.675) to[out=225, in=135] (0.175, 0.325);
					\draw[knot] (0.325, 0.175) to[out=-45, in=90] (0.5, -0.207);
					\draw[knot] (0,0) to (0.175, 0.175);
					\draw[knot] (0.325, 0.325) to (0.425, 0.425);
					\draw[knot] (0.575, 0.575) to (1,1);
					\draw[knot] (0,1) to (0.175, 0.825);
					\draw[knot] (0.325, 0.675) to (0.425, 0.575);
					\draw[knot] (0.575, 0.425) to (1,0);
					\draw[knot] (0.175, 0.175) to[out=45, in=135] (0.325, 0.175);
					\draw[knot] (0.175, 0.325) to[out=-45, in=-135] (0.325, 0.325);
					\draw[knot] (0.425, 0.425) to[out=45, in=135] (0.575, 0.425);
					\draw[knot] (0.425, 0.575) to[out=-45, in=-135] (0.575, 0.575);
					\draw[knot] (0.175, 0.675) to[out=45, in=-45] (0.175, 0.825);
					\draw[knot] (0.325, 0.675) to[out=135, in=-135] (0.325, 0.825);
				}
				$};
			\node(101) at (1,0) {$
				\tikz[baseline={([yshift=-.5ex]current bounding box.center)}, scale=0.7, rotate=90]
				{
					\draw[dotted] (.5,.5) circle(0.707);
					\draw[knot] (0.5, 1.207) to[out=-90, in=45] (0.325, 0.825);
					\draw[knot] (0.175, 0.675) to[out=225, in=135] (0.175, 0.325);
					\draw[knot] (0.325, 0.175) to[out=-45, in=90] (0.5, -0.207);
					\draw[knot] (0,0) to (0.175, 0.175);
					\draw[knot] (0.325, 0.325) to (0.425, 0.425);
					\draw[knot] (0.575, 0.575) to (1,1);
					\draw[knot] (0,1) to (0.175, 0.825);
					\draw[knot] (0.325, 0.675) to (0.425, 0.575);
					\draw[knot] (0.575, 0.425) to (1,0);
					\draw[knot] (0.175, 0.175) to[out=45, in=-45] (0.175, 0.325);
					\draw[knot] (0.325, 0.175) to[out=135, in=-135] (0.325, 0.325);
					\draw[knot] (0.425, 0.425) to[out=45, in=135] (0.575, 0.425);
					\draw[knot] (0.425, 0.575) to[out=-45, in=-135] (0.575, 0.575);
					\draw[knot] (0.175, 0.675) to[out=45, in=-45] (0.175, 0.825);
					\draw[knot] (0.325, 0.675) to[out=135, in=-135] (0.325, 0.825);
				}
				$};
			\node(110) at (2,-.5) {$
				\tikz[baseline={([yshift=-.5ex]current bounding box.center)}, scale=0.7, rotate=90]
				{
					\draw[dotted] (.5,.5) circle(0.707);
					\draw[knot] (0.5, 1.207) to[out=-90, in=45] (0.325, 0.825);
					\draw[knot] (0.175, 0.675) to[out=225, in=135] (0.175, 0.325);
					\draw[knot] (0.325, 0.175) to[out=-45, in=90] (0.5, -0.207);
					\draw[knot] (0,0) to (0.175, 0.175);
					\draw[knot] (0.325, 0.325) to (0.425, 0.425);
					\draw[knot] (0.575, 0.575) to (1,1);
					\draw[knot] (0,1) to (0.175, 0.825);
					\draw[knot] (0.325, 0.675) to (0.425, 0.575);
					\draw[knot] (0.575, 0.425) to (1,0);
					\draw[knot] (0.175, 0.175) to[out=45, in=135] (0.325, 0.175);
					\draw[knot] (0.175, 0.325) to[out=-45, in=-135] (0.325, 0.325);
					\draw[knot] (0.425, 0.425) to[out=45, in=-45] (0.425, 0.575); 
					\draw[knot] (0.575, 0.425) to[out=135, in=-135] (0.575, 0.575);
					\draw[knot] (0.175, 0.675) to[out=45, in=-45] (0.175, 0.825);
					\draw[knot] (0.325, 0.675) to[out=135, in=-135] (0.325, 0.825);
				}
				$};
			\node(100) at (3,-.5) {$
				\tikz[baseline={([yshift=-.5ex]current bounding box.center)}, scale=0.7, rotate=90]
				{
					\draw[dotted] (.5,.5) circle(0.707);
					\draw[knot] (0.5, 1.207) to[out=-90, in=45] (0.325, 0.825);
					\draw[knot] (0.175, 0.675) to[out=225, in=135] (0.175, 0.325);
					\draw[knot] (0.325, 0.175) to[out=-45, in=90] (0.5, -0.207);
					\draw[knot] (0,0) to (0.175, 0.175);
					\draw[knot] (0.325, 0.325) to (0.425, 0.425);
					\draw[knot] (0.575, 0.575) to (1,1);
					\draw[knot] (0,1) to (0.175, 0.825);
					\draw[knot] (0.325, 0.675) to (0.425, 0.575);
					\draw[knot] (0.575, 0.425) to (1,0);
					\draw[knot] (0.175, 0.175) to[out=45, in=-45] (0.175, 0.325);
					\draw[knot] (0.325, 0.175) to[out=135, in=-135] (0.325, 0.325);
					\draw[knot] (0.425, 0.425) to[out=45, in=-45] (0.425, 0.575); 
					\draw[knot] (0.575, 0.425) to[out=135, in=-135] (0.575, 0.575);
					\draw[knot] (0.175, 0.675) to[out=45, in=-45] (0.175, 0.825);
					\draw[knot] (0.325, 0.675) to[out=135, in=-135] (0.325, 0.825);
				}
				$};
			\node(011) at (0,-1) {$
				\tikz[baseline={([yshift=-.5ex]current bounding box.center)}, scale=0.7, rotate=90]
				{
					\draw[dotted] (.5,.5) circle(0.707);
					\draw[knot] (0.5, 1.207) to[out=-90, in=45] (0.325, 0.825);
					\draw[knot] (0.175, 0.675) to[out=225, in=135] (0.175, 0.325);
					\draw[knot] (0.325, 0.175) to[out=-45, in=90] (0.5, -0.207);
					\draw[knot] (0,0) to (0.175, 0.175);
					\draw[knot] (0.325, 0.325) to (0.425, 0.425);
					\draw[knot] (0.575, 0.575) to (1,1);
					\draw[knot] (0,1) to (0.175, 0.825);
					\draw[knot] (0.325, 0.675) to (0.425, 0.575);
					\draw[knot] (0.575, 0.425) to (1,0);
					\draw[knot] (0.175, 0.175) to[out=45, in=135] (0.325, 0.175);
					\draw[knot] (0.175, 0.325) to[out=-45, in=-135] (0.325, 0.325);
					\draw[knot] (0.425, 0.425) to[out=45, in=135] (0.575, 0.425);
					\draw[knot] (0.425, 0.575) to[out=-45, in=-135] (0.575, 0.575);
					\draw[knot] (0.175, 0.825) to[out=-45, in=-135] (0.325, 0.825);
					\draw[knot] (0.175, 0.675) to[out=45, in=135] (0.325, 0.675);
				}
				$};
			\node(001) at (1,-1) {$
				\tikz[baseline={([yshift=-.5ex]current bounding box.center)}, scale=0.7, rotate=90]
				{
					\draw[dotted] (.5,.5) circle(0.707);
					\draw[knot] (0.5, 1.207) to[out=-90, in=45] (0.325, 0.825);
					\draw[knot] (0.175, 0.675) to[out=225, in=135] (0.175, 0.325);
					\draw[knot] (0.325, 0.175) to[out=-45, in=90] (0.5, -0.207);
					\draw[knot] (0,0) to (0.175, 0.175);
					\draw[knot] (0.325, 0.325) to (0.425, 0.425);
					\draw[knot] (0.575, 0.575) to (1,1);
					\draw[knot] (0,1) to (0.175, 0.825);
					\draw[knot] (0.325, 0.675) to (0.425, 0.575);
					\draw[knot] (0.575, 0.425) to (1,0);
					\draw[knot] (0.175, 0.175) to[out=45, in=-45] (0.175, 0.325);
					\draw[knot] (0.325, 0.175) to[out=135, in=-135] (0.325, 0.325);
					\draw[knot] (0.425, 0.425) to[out=45, in=135] (0.575, 0.425);
					\draw[knot] (0.425, 0.575) to[out=-45, in=-135] (0.575, 0.575);
					\draw[knot] (0.175, 0.825) to[out=-45, in=-135] (0.325, 0.825);
					\draw[knot] (0.175, 0.675) to[out=45, in=135] (0.325, 0.675);
				}
				$};
			\node(010) at (2,-1.5) {$
				\tikz[baseline={([yshift=-.5ex]current bounding box.center)}, scale=0.7, rotate=90]
				{
					\draw[dotted] (.5,.5) circle(0.707);
					\draw[knot] (0.5, 1.207) to[out=-90, in=45] (0.325, 0.825);
					\draw[knot] (0.175, 0.675) to[out=225, in=135] (0.175, 0.325);
					\draw[knot] (0.325, 0.175) to[out=-45, in=90] (0.5, -0.207);
					\draw[knot] (0,0) to (0.175, 0.175);
					\draw[knot] (0.325, 0.325) to (0.425, 0.425);
					\draw[knot] (0.575, 0.575) to (1,1);
					\draw[knot] (0,1) to (0.175, 0.825);
					\draw[knot] (0.325, 0.675) to (0.425, 0.575);
					\draw[knot] (0.575, 0.425) to (1,0);
					\draw[knot] (0.175, 0.175) to[out=45, in=135] (0.325, 0.175);
					\draw[knot] (0.175, 0.325) to[out=-45, in=-135] (0.325, 0.325);
					\draw[knot] (0.425, 0.425) to[out=45, in=-45] (0.425, 0.575); 
					\draw[knot] (0.575, 0.425) to[out=135, in=-135] (0.575, 0.575);
					\draw[knot] (0.175, 0.825) to[out=-45, in=-135] (0.325, 0.825);
					\draw[knot] (0.175, 0.675) to[out=45, in=135] (0.325, 0.675);
				}
				$};
			\node(000) at (3,-1.5) {$
				\tikz[baseline={([yshift=-.5ex]current bounding box.center)}, scale=0.7, rotate=90]
				{
					\draw[dotted] (.5,.5) circle(0.707);
					\draw[knot] (0.5, 1.207) to[out=-90, in=45] (0.325, 0.825);
					\draw[knot] (0.175, 0.675) to[out=225, in=135] (0.175, 0.325);
					\draw[knot] (0.325, 0.175) to[out=-45, in=90] (0.5, -0.207);
					\draw[knot] (0,0) to (0.175, 0.175);
					\draw[knot] (0.325, 0.325) to (0.425, 0.425);
					\draw[knot] (0.575, 0.575) to (1,1);
					\draw[knot] (0,1) to (0.175, 0.825);
					\draw[knot] (0.325, 0.675) to (0.425, 0.575);
					\draw[knot] (0.575, 0.425) to (1,0);
					\draw[knot] (0.175, 0.175) to[out=45, in=-45] (0.175, 0.325);
					\draw[knot] (0.325, 0.175) to[out=135, in=-135] (0.325, 0.325);
					\draw[knot] (0.425, 0.425) to[out=45, in=-45] (0.425, 0.575); 
					\draw[knot] (0.575, 0.425) to[out=135, in=-135] (0.575, 0.575);
					\draw[knot] (0.175, 0.825) to[out=-45, in=-135] (0.325, 0.825);
					\draw[knot] (0.175, 0.675) to[out=45, in=135] (0.325, 0.675);
				}
				$};
			\node(211) at (0,-3) {$\tikz[baseline={([yshift=-.5ex]current bounding box.center)}, scale=0.7, rotate=180]
				{
					\draw[dotted] (.5,.5) circle(0.707);
					\draw[knot] (0,0) to[out=45, in=-90] (0.375,0.5) to[out=90, in=-90] (0.3, 0.782) to[out=90, in=180] (0.5, 0.915) to[out=0, in=180] (1.207, 0.5);
					\draw[knot] (0,1) to[out=-45,in=90] (0.2, 0.782) to[out=-90,in=0] (-0.207,0.5);
					\draw[knot, overcross] (1,0) to[out=135, in=-90] (0.625,0.5) to[out=90, in=-135] (1,1);
				}$};
			\node(201) at (1,-3) {$\tikz[baseline={([yshift=-.5ex]current bounding box.center)}, scale=0.7, rotate=180]
				{
					\draw[dotted] (.5,.5) circle(0.707);
					\draw[knot] (1.207, 0.5) to[out=180, in=0] (0.5, 0.915) to[out=180, in=0] (0.218, 0.8) to[out=180, in=-45] (0,1);
					\draw[knot] (0,0) to[out=45, in=-90] (0.375,0.5) to[out=90, in=0] (0.218, 0.7) to[out=180, in=0] (-0.207, 0.5);
					\draw[knot, overcross] (1,0) to[out=135, in=-90] (0.625,0.5) to[out=90, in=-135] (1,1);
				}$};
			\node(210) at (2,-4) {$\tikz[baseline={([yshift=-.5ex]current bounding box.center)}, scale=0.7, rotate=180]
				{
					\draw[dotted] (.5,.5) circle(0.707);
					\draw[knot] (1.207, 0.5) to[out=180, in=0] (0.5, 0.915) to[out=180, in=90] (0.3,0.782) to[out=-90, in=180] (0.5,0.625);
					\draw[knot] (0,1) to[out=-45, in=90] (0.2, 0.782) to[out=-90, in=0] (-0.207, 0.5);
					\draw[knot, overcross] (0,0) to[out=45, in=180] (0.5,0.375) to[out=0, in=135] (1,0);
					\draw[knot, overcross] (0.5,0.625) to[out=0, in=-135] (1,1);
				}$};
			\node(200) at (3,-4) {$\tikz[baseline={([yshift=-.5ex]current bounding box.center)}, scale=0.7, rotate=180]
				{
					\draw[dotted] (.5,.5) circle(0.707);
					\draw[knot] (1.207, 0.5) to[out=180, in=0] (0.5, 0.915) to[out=180, in=0] (0.218, 0.8) to[out=180, in=-45] (0,1);
					\draw[knot, overcross] (0,0) to[out=45, in=180] (0.5,0.375) to[out=0, in=135] (1,0);
					\draw[knot] (-0.207,0.5) to[out=0, in=180] (0.218, 0.7) to[out=0, in=180] (0.5,0.625);
					\draw[knot, overcross] (0.5,0.625) to[out=0, in=-135] (1,1);
				}$};
			\node(222) at (0.5,-6) {$
				\tikz[baseline={([yshift=-.5ex]current bounding box.center)}, scale=0.7]
				{
					\draw[dotted] (.5,.5) circle(0.707);
					\draw[knot] (-0.207, 0.5) to[out=0, in=180] (0.5, 0.085);
					\draw[knot] (1.207, 0.5) to[out=180, in=0] (0.5, 0.085);
					\draw[knot, overcross] (0,0) to[out=45, in=-90] (0.375,0.5) to[out=90, in=-45] (0,1);
					\draw[knot, overcross] (1,0) to[out=135, in=-90] (0.625,0.5) to[out=90, in=-135] (1,1);
				}
				$};
			\node(220) at (2.5,-6) {$
				\tikz[baseline={([yshift=-.5ex]current bounding box.center)}, scale=0.7]
				{
					\draw[dotted] (.5,.5) circle(0.707);
					\draw[knot] (-0.207, 0.5) to[out=0, in=180] (0.5, 0.085);
					\draw[knot] (1.207, 0.5) to[out=180, in=0] (0.5, 0.085);
					\draw[knot, overcross] (0,0) to[out=45, in=180] (0.5,0.375) to[out=0, in=135] (1,0);
					\draw[knot, overcross] (0,1) to[out=-45, in=180] (0.5,0.625) to[out=0, in=-135] (1,1);
				}
				$};
			\draw[->] (201) -- (200);
			\draw[->, thick, blue] (001) -- node[left, pos=0.4]{$\circled{3b}$}(201);
			\draw[white, ultra thick] (001) -- (200);
			\draw[->, thick, blue] (001) --node[left, pos=0.4]{$\circled{1b}$} (200);
			\draw[->, thick, blue] (201) -- node[left, pos=0.5]{$\circled{4b}$} (220);
			\draw[->, thick, blue] (200) --node[left, pos=0.25]{$\circled{2b}$} (220);
			\draw[->] (111) -- (110);
			\draw[->] (101) -- (100);
			\draw[white, ultra thick] (011) -- (010);
			\draw[->] (011) -- (010);
			\draw[->] (001) -- node[above, pos=0.4]{$\circled{5b}$}(000);
			\draw[->] (111) -- (101);
			\draw[->] (111) -- (011);
			\draw[->] (011) -- (001);
			\draw[->] (101) -- (001);
			\draw[->] (110) -- (100);
			\draw[->] (110) -- (010);
			\draw[->] (010) -- (000);
			\draw[->] (100) -- (000);
			\draw[->] (211) -- (201);
			\draw[white, ultra thick] (211) -- (210);
			\draw[white, ultra thick] (210) -- (200);
			\draw[->] (211) -- (210);
			\draw[->] (210) -- (200);
			\draw[->] (222) -- (220);
			\draw[->, thick, red] (100) to[out=-60, in=60] (200);
			\draw[white, ultra thick] (010) -- (210);
			\draw[->, thick, red] (010) -- (210);
			\draw[->, thick, red] (010) to[out=-45, in=90] (200);
			\draw[->, thick, red] (210) -- (220);
			\draw[->, thick, red] (200) to[out=-90,in=60] (220);
			\draw[->, ultra thick, white] (000) to[out=-60, in=45] (220);
			\draw[->, thick, PineGreen] (000) to[out=-60, in=45]node[right]{$E_{-1}'$} (220);
		}
	}
	\caption{Local picture without closures}
	\label{fig:R3local}
\end{figure}
\vfill
\clearpage

\begin{table}[h]
\tikz[]{
    \draw[ultra thick] (0,0) rectangle (12,-12);
    \draw[ultra thick] (0,0) -- node[below, sloped]{closure} node[above,sloped]{map} (2,-2);
    \draw[ultra thick] (0,-2) -- (12,-2); 
    \draw[ thick] (0,-4) -- (12,-4); 
    \draw[ thick] (0,-6) -- (12,-6); 
    \draw[ thick] (0,-8) -- (12,-8);
    \draw[ thick] (0,-10) -- (12,-10); 
    \draw[ultra thick] (2,0) -- (2,-12);
    \draw[ thick] (4,0) -- (4,-12);
    \draw[ thick] (6,0) -- (6,-12);
    \draw[ thick] (8,0) -- (8,-12);
    \draw[ thick] (10,0) -- (10,-12);
    \node at (3,-1) {$\circled{2a} \circ \circled{1a}$};
    \node at (5,-1) {$E_{-1}$};
    \node at (7,-1) {$\circled{2b} \circ \circled{1b}$};
    \node at (9,-1) {$E_{-1}'$};
    \node at (11,-1) {$\circled{5b}$};
    \node at (1, -3) {$
    \tikz[baseline={([yshift=-.5ex]current bounding box.center)}, scale=0.7]
	{
    \draw[dotted] (.5,.5) circle(0.707);
    \draw[knot, Black, thick] (0,1) to[out=135, in=180, looseness = 2] (-0.207, 0.5);
    \draw[knot, Black, thick] (1,1) to[out=45, in=0, looseness = 2] (1.207, 0.5);
    \draw[knot, Black, thick] (0,0) to[out=-135, in=-45, looseness=2.5] (1,0);
    \node[scale=0.7] at (0,0) {$\bullet$};
    \node[scale=0.7] at (-0.207,0.5) {$\bullet$};
    \node[scale=0.7] at (0,1) {$\bullet$};
    \node[scale=0.7] at (1,1) {$\bullet$};
    \node[scale=0.7] at (1.207,0.5) {$\bullet$};
    \node[scale=0.7] at (1,0) {$\bullet$};
    }
    $};
    \node at (1,-5) {$
    \tikz[baseline={([yshift=-.5ex]current bounding box.center)}, scale=0.7]
	{
    \draw[dotted] (.5,.5) circle(0.707);
    \draw[knot, Black, thick] (0,1) to[out=135, in=180, looseness = 2] (-0.207, 0.5);
    \draw[knot, Black, thick] (0,0) to[out=-135, in=-135, looseness=2] (1.3,-0.3) to[out=45, in=45, looseness=2] (1,1);
    \draw[knot, Black, thick] (1,0) to[out=-45, in=0, looseness=2] (1.207, 0.5);
    \node[scale=0.7] at (0,0) {$\bullet$};
    \node[scale=0.7] at (-0.207,0.5) {$\bullet$};
    \node[scale=0.7] at (0,1) {$\bullet$};
    \node[scale=0.7] at (1,1) {$\bullet$};
    \node[scale=0.7] at (1.207,0.5) {$\bullet$};
    \node[scale=0.7] at (1,0) {$\bullet$};
    }
    $};
    \node at (1,-7) {$
    \tikz[baseline={([yshift=-.5ex]current bounding box.center)}, scale=0.7]
	{
    \draw[dotted] (.5,.5) circle(0.707);
    \draw[knot, Black, thick] (0,1) to[out=135, in=45, looseness = 2.5] (1,1);
    \draw[knot, Black, thick] (0,0) to[out=-135, in=-45, looseness=2.5] (1,0);
    \draw[knot, Black, thick] (-0.207,0.5) to[out=180, in=180, looseness=2] (0.5, 1.75) to[out=0, in=0, looseness=2] (1.207,0.5);
    \node[scale=0.7] at (0,0) {$\bullet$};
    \node[scale=0.7] at (-0.207,0.5) {$\bullet$};
    \node[scale=0.7] at (0,1) {$\bullet$};
    \node[scale=0.7] at (1,1) {$\bullet$};
    \node[scale=0.7] at (1.207,0.5) {$\bullet$};
    \node[scale=0.7] at (1,0) {$\bullet$};
    }
    $};
    \node at (1,-9)  {$
    \tikz[baseline={([yshift=-.5ex]current bounding box.center)}, scale=0.7]
	{
    \draw[dotted] (.5,.5) circle(0.707);
    \draw[knot, Black, thick] (1,1) to[out=45, in=0, looseness = 2] (1.207, 0.5);
    \draw[knot, Black, thick] (0,0) to[out=-135, in=180, looseness=2] (-0.207,0.5);
    \draw[knot, Black, thick] (1,0) to[out=-45, in=-45, looseness=2] (-0.3,-0.3) to[out=135, in=135, looseness=2] (0,1);
    \node[scale=0.7] at (0,0) {$\bullet$};
    \node[scale=0.7] at (-0.207,0.5) {$\bullet$};
    \node[scale=0.7] at (0,1) {$\bullet$};
    \node[scale=0.7] at (1,1) {$\bullet$};
    \node[scale=0.7] at (1.207,0.5) {$\bullet$};
    \node[scale=0.7] at (1,0) {$\bullet$};
    }
    $};
    \node at (1,-11)  {$
    \tikz[baseline={([yshift=-.5ex]current bounding box.center)}, scale=0.7]
	{
    \draw[dotted] (.5,.5) circle(0.707);
    \draw[knot, Black, thick] (0,0) to[out=-135, in=180, looseness=2] (-0.207,0.5);
    \draw[knot, Black, thick] (1,0) to[out=-45, in=0, looseness=2] (1.207, 0.5);
    \draw[knot, Black, thick] (0,1) to[out=135, in=45, looseness = 2.5] (1,1);
    \node[scale=0.7] at (0,0) {$\bullet$};
    \node[scale=0.7] at (-0.207,0.5) {$\bullet$};
    \node[scale=0.7] at (0,1) {$\bullet$};
    \node[scale=0.7] at (1,1) {$\bullet$};
    \node[scale=0.7] at (1.207,0.5) {$\bullet$};
    \node[scale=0.7] at (1,0) {$\bullet$};
    }
    $};
    \node at (3,-3) {$
    \tikz[baseline={([yshift=-.5ex]current bounding box.center)}, scale=.35]{
    \draw[knot, -] (0,2) .. controls (0,1.75) and (1,1.75) .. (1,2);
    \draw[knot, -] (0,2) .. controls (0,2.25) and (1,2.25) .. (1,2);
    \draw[knot, -] (0,-1) .. controls (0,-1.25) and (1,-1.25) .. (1,-1);
    \draw[dashed, knot, -] (0,-1) .. controls (0,-0.75) and (1,-0.75) .. (1,-1);
    \draw[knot, -] (2,0) .. controls (2,-.25) and (3,-.25) .. (3,0);
    \draw[dashed, knot, -] (2,0) .. controls (2,.25) and (3,.25) .. (3,0);
    \draw[knot, -] (2,0) .. controls (2,-0.75) and (3,-0.75) .. (3,0);
    \draw[knot, -] (3,0) .. controls (3,1.35) and (1,1.35) .. (1,2);
    \draw[knot, -] (0,0) -- (0,2);
    \draw[knot, -] (1,0) .. controls (1,0.75) and (2,0.75) .. (2,0);
    \draw[knot, -] (0,0) -- (0,-1);
    \draw[knot, -] (1,0) -- (1,-1);
    \draw[knot, -] (-2,2) .. controls (-2,1.75) and (-1,1.75) .. (-1,2);
    \draw[knot, -] (-2,2) .. controls (-2,2.25) and (-1,2.25) .. (-1,2);
    \draw[knot, -] (-2,-1) .. controls (-2,-1.25) and (-1,-1.25) .. (-1,-1);
    \draw[dashed, knot, -] (-2,-1) .. controls (-2,-0.75) and (-1,-0.75) .. (-1,-1);
    \draw[knot, -] (-2,2) -- (-2,-1);
    \draw[knot, -] (-1,2) -- (-1,-1);
    }
    $};
    \node at (3,-5) {$
    \tikz[baseline={([yshift=-.5ex]current bounding box.center)}, scale=.35]{
    \draw[knot, -] (0,2) .. controls (0,1.75) and (1,1.75) .. (1,2);
    \draw[knot, -] (0,2) .. controls (0,2.25) and (1,2.25) .. (1,2);
    \draw[knot, -] (0,-1) .. controls (0,-1.25) and (1,-1.25) .. (1,-1);
    \draw[dashed, knot, -] (0,-1) .. controls (0,-0.75) and (1,-0.75) .. (1,-1);
    \draw[knot, -] (2,0) .. controls (2,-.25) and (3,-.25) .. (3,0);
    \draw[dashed, knot, -] (2,0) .. controls (2,.25) and (3,.25) .. (3,0);
    \draw[knot, -] (2,0) .. controls (2,-0.75) and (3,-0.75) .. (3,0);
    \draw[knot, -] (3,0) .. controls (3,1.35) and (1,1.35) .. (1,2);
    \draw[knot, -] (0,0) -- (0,2);
    \draw[knot, -] (1,0) .. controls (1,0.75) and (2,0.75) .. (2,0);
    \draw[knot, -] (0,0) -- (0,-1);
    \draw[knot, -] (1,0) -- (1,-1);
    }
    $};
    \node at (3,-7) {$
    \tikz[baseline={([yshift=-.5ex]current bounding box.center)}, scale=.35]{
	\draw[knot, -]  (1,2) .. controls (1,3) and (0,3) .. (0,4);
	\draw[knot, -]  (2,2) .. controls (2,3) and (3,3) .. (3,4);
	\draw[knot, -] (1,4) .. controls (1,3) and (2,3) .. (2,4);
	\draw[knot, -] (0,4) .. controls (0,3.75) and (1,3.75) .. (1,4);
	\draw[knot, -] (0,4) .. controls (0,4.25) and (1,4.25) .. (1,4);
	\draw[knot, -] (2,4) .. controls (2,3.75) and (3,3.75) .. (3,4);
	\draw[knot, -] (2,4) .. controls (2,4.25) and (3,4.25) .. (3,4);
	\draw[knot, -] (1,1) .. controls (1,0.75) and (2,0.75) .. (2,1);
	\draw[knot, dashed, -] (1,1) .. controls (1,1.25) and (2,1.25) .. (2,1);
    \draw[knot, -] (1,1) -- (1,2);
    \draw[knot, -] (2,1) -- (2,2);
    \draw[knot, -] (4,2) .. controls (4,1.75) and (5,1.75) .. (5,2);
    \draw[knot, -, dashed] (4,2) .. controls (4,2.25) and (5,2.25) .. (5,2);
    \draw[knot, -] (4,4) .. controls (4,3.75) and (5,3.75) .. (5,4);
    \draw[knot, -] (4,4) .. controls (4,4.25) and (5,4.25) .. (5,4);
    \draw[knot, -] (4,4) -- (4,2);
    \draw[knot, -] (5,4) -- (5,2);
    \draw[knot, -] (4,2) .. controls (4,1.25) and (5,1.25) .. (5,2);
    }
    $};
    \node at (3,-9) {$
    \tikz[baseline={([yshift=-.5ex]current bounding box.center)}, scale=.35, rotate=180]{
	\draw[knot, -]  (1,2) .. controls (1,3) and (0,3) .. (0,4);
	\draw[knot, -]  (2,2) .. controls (2,3) and (3,3) .. (3,4);
	\draw[knot, -] (1,4) .. controls (1,3) and (2,3) .. (2,4);
	\draw[knot, -, dashed] (0,4) .. controls (0,3.75) and (1,3.75) .. (1,4);
	\draw[knot, -] (0,4) .. controls (0,4.25) and (1,4.25) .. (1,4);
	\draw[knot, -, dashed] (2,4) .. controls (2,3.75) and (3,3.75) .. (3,4);
	\draw[knot, -] (2,4) .. controls (2,4.25) and (3,4.25) .. (3,4);
	\draw[knot, -] (1,1) .. controls (1,0.75) and (2,0.75) .. (2,1);
	\draw[knot, -] (1,1) .. controls (1,1.25) and (2,1.25) .. (2,1);
    \draw[knot, -] (1,1) -- (1,2);
    \draw[knot, -] (2,1) -- (2,2);
    \draw[knot, -, dashed] (4,2) .. controls (4,1.75) and (5,1.75) .. (5,2);
    \draw[knot, -] (4,2) .. controls (4,2.25) and (5,2.25) .. (5,2);
    \draw[knot, -, dashed] (4,4) .. controls (4,3.75) and (5,3.75) .. (5,4);
    \draw[knot, -] (4,4) .. controls (4,4.25) and (5,4.25) .. (5,4);
    \draw[knot, -] (4,4) -- (4,2);
    \draw[knot, -] (5,4) -- (5,2);
    \draw[knot, -] (4,2) .. controls (4,1.25) and (5,1.25) .. (5,2);
    }
    $};
    \node at (3,-11) {$
    \tikz[baseline={([yshift=-.5ex]current bounding box.center)}, scale=.35]{
    \draw[knot, -] (0,2) .. controls (0,1.75) and (1,1.75) .. (1,2);
    \draw[knot, -] (0,2) .. controls (0,2.25) and (1,2.25) .. (1,2);
    \draw[knot, -] (2,2) .. controls (2,1.75) and (3,1.75) .. (3,2);
    \draw[knot, -] (2,2) .. controls (2,2.25) and (3,2.25) .. (3,2);
    \draw[knot, -] (0,-1) .. controls (0,-1.25) and (1,-1.25) .. (1,-1);
    \draw[dashed, knot, -] (0,-1) .. controls (0,-0.75) and (1,-0.75) .. (1,-1);
    \draw[knot, -] (2,-1) .. controls (2,-1.25) and (3,-1.25) .. (3,-1);
    \draw[dashed, knot, -] (2,-1) .. controls (2,-0.75) and (3,-0.75) .. (3,-1);
    \draw[knot, -] (2,0) .. controls (2,0.75) and (3,0.75) .. (3,0);
    \draw[knot, -] (1,-1) .. controls (1,0.65) and (3,0.65) .. (3,2);
    \draw[knot, -] (0,0) -- (0,2);
    \draw[knot, -] (1,2) .. controls (1,1.25) and (2,1.25) .. (2,2);
    \draw[knot, -] (0,-1) -- (0,0);
    \draw[knot, -] (2,-1) -- (2,0);
    \draw[knot, -] (3,-1) -- (3,0);
    }
    $};
    \node[scale=1.25] at (5,-3) {$0$};
    \node[scale=1.25] at (5,-5) {$0$};
    \node[scale=1.25] at (5,-7) {$0$};
    \node[scale=1.25] at (5,-9) {$0$};
    \node[scale=1.25] at (5,-11) {$0$};
    \node at (7,-3) {$
    \tikz[baseline={([yshift=-.5ex]current bounding box.center)}, scale=.35]{
	\draw[knot, -]  (1,2) .. controls (1,3) and (0,3) .. (0,4);
	\draw[knot, -]  (2,2) .. controls (2,3) and (3,3) .. (3,4);
	\draw[knot, -] (1,4) .. controls (1,3) and (2,3) .. (2,4);
	\draw[knot, -] (0,4) .. controls (0,3.75) and (1,3.75) .. (1,4);
	\draw[knot, -] (0,4) .. controls (0,4.25) and (1,4.25) .. (1,4);
	\draw[knot, -] (2,4) .. controls (2,3.75) and (3,3.75) .. (3,4);
	\draw[knot, -] (2,4) .. controls (2,4.25) and (3,4.25) .. (3,4);
    \draw[knot, -] (1,1) -- (1,2);
    \draw[knot, -] (2,1) -- (2,2);
	\draw[knot, -] (1,1) .. controls (1,0.75) and (2,0.75) .. (2,1);
	\draw[knot, dashed, -] (1,1) .. controls (1,1.25) and (2,1.25) .. (2,1);
    \draw[knot, -] (4,1) .. controls (4,0.75) and (5,0.75) .. (5,1);
    \draw[knot, -, dashed] (4,1) .. controls (4,1.25) and (5,1.25) .. (5,1);
    \draw[knot, -] (4,2) -- (4,1);
    \draw[knot, -] (5,2) -- (5,1);
    \draw[knot, -] (4,2) .. controls (4,2.75) and (5,2.75) .. (5,2);
}
    $};
    \node[scale=1.25] at (7,-5) {$0$};
    \node at (7,-7) {$
    \tikz[baseline={([yshift=-.5ex]current bounding box.center)}, scale=.35]{
	\draw[knot, -]  (1,2) .. controls (1,3) and (0,3) .. (0,4);
	\draw[knot, -]  (2,2) .. controls (2,3) and (3,3) .. (3,4);
	\draw[knot, -] (1,4) .. controls (1,3) and (2,3) .. (2,4);
	\draw[knot, -] (0,4) .. controls (0,3.75) and (1,3.75) .. (1,4);
	\draw[knot, -] (0,4) .. controls (0,4.25) and (1,4.25) .. (1,4);
	\draw[knot, -] (2,4) .. controls (2,3.75) and (3,3.75) .. (3,4);
	\draw[knot, -] (2,4) .. controls (2,4.25) and (3,4.25) .. (3,4);
	\draw[knot, -] (1,2) .. controls (1,1.75) and (2,1.75) .. (2,2);
	\draw[knot, dashed, -] (1,2) .. controls (1,2.25) and (2,2.25) .. (2,2);
    \draw[knot, -] (1,2) .. controls (1,1.25) and (2,1.25) .. (2,2);
    \draw[knot, -] (-2,1) .. controls (-2,0.75) and (-1,0.75) .. (-1,1);
    \draw[knot, -, dashed] (-2,1) .. controls (-2,1.25) and (-1,1.25) .. (-1,1);
    \draw[knot, -] (-2,4) .. controls (-2,3.75) and (-1,3.75) .. (-1,4);
    \draw[knot, -] (-2,4) .. controls (-2,4.25) and (-1,4.25) .. (-1,4);
    \draw[knot, -] (-2,4) -- (-2,1);
    \draw[knot, -] (-1,4) -- (-1,1);
    }
    $};
    \node[scale=1.25] at (7,-9) {$0$};
    \node[scale=1.25] at (7,-11) {$0$};
    \node at (9,-3) {$
    \tikz[baseline={([yshift=-.5ex]current bounding box.center)}, scale=.35]{
    \draw[knot, -] (-2,2) .. controls (-2,1.75) and (-1,1.75) .. (-1,2);
    \draw[knot, -] (-2,2) .. controls (-2,2.25) and (-1,2.25) .. (-1,2);
    \draw[knot, -] (-2,-1) .. controls (-2,-1.25) and (-1,-1.25) .. (-1,-1);
    \draw[dashed, knot, -] (-2,-1) .. controls (-2,-0.75) and (-1,-0.75) .. (-1,-1);
    \draw[knot, -] (-2,2) -- (-2,-1);
    \draw[knot, -] (-1,2) -- (-1,-1);
    \draw[knot, -] (0,2) .. controls (0,1.75) and (1,1.75) .. (1,2);
    \draw[knot, -] (0,2) .. controls (0,2.25) and (1,2.25) .. (1,2);
    \draw[knot, -] (0,2) -- (0,0.5);
    \draw[knot, -] (1,2) -- (1,0.5);
    \draw[knot, -] (0,0.5) .. controls (0, 0.25) and (1, 0.25) .. (1,0.5);
    \draw[knot, -, dashed] (0,0.5) .. controls (0, 0.75) and (1, 0.75) .. (1,0.5);
    \draw[knot, -] (0,0.5) .. controls (0,-0.25) and (1,-0.25) .. (1,0.5);
    }
    $};
    \node at (9,-5) {$
    \tikz[baseline={([yshift=-.5ex]current bounding box.center)}, scale=.35]{
    \begin{scope}[xshift=-4cm]
    \draw[knot, -] (2,2) .. controls (2,1.75) and (3,1.75) .. (3,2);
    \draw[knot, -] (2,2) .. controls (2,2.25) and (3,2.25) .. (3,2);
    \draw[knot, -] (2,-1) .. controls (2,-1.25) and (3,-1.25) .. (3,-1);
    \draw[dashed, knot, -] (2,-1) .. controls (2,-0.75) and (3,-0.75) .. (3,-1);
    \draw[knot, -] (2,2) -- (2,-1);
    \draw[knot, -] (3,2) -- (3,-1);
    \end{scope}
    \draw[knot, -] (0,-1) .. controls (0,-1.25) and (1,-1.25) .. (1,-1);
    \draw[dashed, knot, -] (0,-1) .. controls (0,-0.75) and (1,-0.75) .. (1,-1);
    \draw[knot, -] (0,-1) -- (0,0.5);
    \draw[knot, -] (1,-1) -- (1,0.5);
    \draw[knot, -] (0,0.5) .. controls (0, 0.25) and (1, 0.25) .. (1,0.5);
    \draw[knot, -, dashed] (0,0.5) .. controls (0, 0.75) and (1, 0.75) .. (1,0.5);
    \draw[knot, -] (0,0.5) .. controls (0,1.25) and (1,1.25) .. (1,0.5);
    }
    $};
    \node at (9,-7) {$
    \tikz[baseline={([yshift=-.5ex]current bounding box.center)}, scale=.35]{
    \draw[knot, -] (-2,2) .. controls (-2,1.75) and (-1,1.75) .. (-1,2);
    \draw[knot, -] (-2,2) .. controls (-2,2.25) and (-1,2.25) .. (-1,2);
    \draw[knot, -] (-2,-1) .. controls (-2,-1.25) and (-1,-1.25) .. (-1,-1);
    \draw[dashed, knot, -] (-2,-1) .. controls (-2,-0.75) and (-1,-0.75) .. (-1,-1);
    \draw[knot, -] (-2,2) -- (-2,-1);
    \draw[knot, -] (-1,2) -- (-1,-1);
    \draw[knot, -] (0,2) .. controls (0,1.75) and (1,1.75) .. (1,2);
    \draw[knot, -] (0,2) .. controls (0,2.25) and (1,2.25) .. (1,2);
    \draw[knot, -] (0,2) -- (0,0.5);
    \draw[knot, -] (1,2) -- (1,0.5);
    \draw[knot, -] (2,2) .. controls (2,1.75) and (3,1.75) .. (3,2);
    \draw[knot, -] (2,2) .. controls (2,2.25) and (3,2.25) .. (3,2);
    \draw[knot, -] (2,-1) .. controls (2,-1.25) and (3,-1.25) .. (3,-1);
    \draw[dashed, knot, -] (2,-1) .. controls (2,-0.75) and (3,-0.75) .. (3,-1);
    \draw[knot, -] (2,2) -- (2,-1);
    \draw[knot, -] (3,2) -- (3,-1);
    \draw[knot, -] (0,0.5) .. controls (0, 0.25) and (1, 0.25) .. (1,0.5);
    \draw[knot, -, dashed] (0,0.5) .. controls (0, 0.75) and (1, 0.75) .. (1,0.5);
    \draw[knot, -] (0,0.5) .. controls (0,-0.25) and (1,-0.25) .. (1,0.5);
    }
    $};
    \node at (9,-9) {$
    \tikz[baseline={([yshift=-.5ex]current bounding box.center)}, scale=.35]{
    \draw[knot, -] (2,2) .. controls (2,1.75) and (3,1.75) .. (3,2);
    \draw[knot, -] (2,2) .. controls (2,2.25) and (3,2.25) .. (3,2);
    \draw[knot, -] (2,-1) .. controls (2,-1.25) and (3,-1.25) .. (3,-1);
    \draw[dashed, knot, -] (2,-1) .. controls (2,-0.75) and (3,-0.75) .. (3,-1);
    \draw[knot, -] (2,2) -- (2,-1);
    \draw[knot, -] (3,2) -- (3,-1);
    \draw[knot, -] (0,-1) .. controls (0,-1.25) and (1,-1.25) .. (1,-1);
    \draw[dashed, knot, -] (0,-1) .. controls (0,-0.75) and (1,-0.75) .. (1,-1);
    \draw[knot, -] (0,-1) -- (0,0.5);
    \draw[knot, -] (1,-1) -- (1,0.5);
    \draw[knot, -] (0,0.5) .. controls (0, 0.25) and (1, 0.25) .. (1,0.5);
    \draw[knot, -, dashed] (0,0.5) .. controls (0, 0.75) and (1, 0.75) .. (1,0.5);
    \draw[knot, -] (0,0.5) .. controls (0,1.25) and (1,1.25) .. (1,0.5);
    }
    $};
    \node at (9,-11) {$
    \tikz[baseline={([yshift=-.5ex]current bounding box.center)}, scale=.35]{
    \begin{scope}[xshift=-6cm]
    \draw[knot, -] (2,2) .. controls (2,1.75) and (3,1.75) .. (3,2);
    \draw[knot, -] (2,2) .. controls (2,2.25) and (3,2.25) .. (3,2);
    \draw[knot, -] (2,-1) .. controls (2,-1.25) and (3,-1.25) .. (3,-1);
    \draw[dashed, knot, -] (2,-1) .. controls (2,-0.75) and (3,-0.75) .. (3,-1);
    \draw[knot, -] (2,2) -- (2,-1);
    \draw[knot, -] (3,2) -- (3,-1);
    \end{scope}
    \draw[knot, -] (-2,2) .. controls (-2,1.75) and (-1,1.75) .. (-1,2);
    \draw[knot, -] (-2,2) .. controls (-2,2.25) and (-1,2.25) .. (-1,2);
    \draw[knot, -] (-2,-1) .. controls (-2,-1.25) and (-1,-1.25) .. (-1,-1);
    \draw[dashed, knot, -] (-2,-1) .. controls (-2,-0.75) and (-1,-0.75) .. (-1,-1);
    \draw[knot, -] (-2,2) -- (-2,-1);
    \draw[knot, -] (-1,2) -- (-1,-1);
    \draw[knot, -] (0,-1) .. controls (0,-1.25) and (1,-1.25) .. (1,-1);
    \draw[dashed, knot, -] (0,-1) .. controls (0,-0.75) and (1,-0.75) .. (1,-1);
    \draw[knot, -] (0,-1) -- (0,0.5);
    \draw[knot, -] (1,-1) -- (1,0.5);
    \draw[knot, -] (0,0.5) .. controls (0, 0.25) and (1, 0.25) .. (1,0.5);
    \draw[knot, -, dashed] (0,0.5) .. controls (0, 0.75) and (1, 0.75) .. (1,0.5);
    \draw[knot, -] (0,0.5) .. controls (0,1.25) and (1,1.25) .. (1,0.5);
    }
    $};
    \node at (11, -3) {$
    \tikz[baseline={([yshift=-.5ex]current bounding box.center)}, scale=.35]{
	\draw[knot, -] (0,0) .. controls (0,1.5) and (1,1.5) .. (1,3);
	\draw[knot, -] (1,0) .. controls (1,1.5) and (2,1.5) .. (2,0);
	\draw[knot, -] (3,0) .. controls (3,1.5) and (2,1.5) .. (2,3);
	\draw[knot, -] (0,0) .. controls (0,-.25) and (1,-.25) .. (1,0);
	\draw[dashed, knot, -] (0,0) .. controls (0,.25) and (1,.25) .. (1,0);
	\draw[knot, -] (2,0) .. controls (2,-.25) and (3,-.25) .. (3,0);
	\draw[dashed, knot, -] (2,0) .. controls (2,.25) and (3,.25) .. (3,0);
	\draw[knot, -] (1,3) .. controls (1,2.75) and (2,2.75) .. (2,3);
	\draw[knot, -] (1,3) .. controls (1,3.25) and (2,3.25) .. (2,3);
}
    $};
    \node at (11, -5) {$
    \tikz[baseline={([yshift=-.5ex]current bounding box.center)}, scale=.35]{
    \draw[knot, -] (1,0) .. controls (1,-0.25) and (2,-0.25) .. (2,0);
	\draw[knot, -, dashed] (1,0) .. controls (1,0.25) and (2,0.25) .. (2,0);
    \draw[knot, -] (0,3) .. controls (0,2.75) and (1,2.75) .. (1,3);
	\draw[knot, -] (0,3) .. controls (0,3.25) and (1,3.25) .. (1,3);
    \draw[knot, -] (2,3) .. controls (2,2.75) and (3,2.75) .. (3,3);
	\draw[knot, -] (2,3) .. controls (2,3.25) and (3,3.25) .. (3,3);
    \draw[knot, -] (0,3) .. controls (0,1.5) and (1,1.5) .. (1,0);
    \draw[knot, -] (3,3) .. controls (3,1.5) and (2,1.5) .. (2,0);
    \draw[knot, -] (1,3) .. controls (1,1.5) and (2,1.5) .. (2,3);
}
    $};
    \node at (11, -7) {$
    \tikz[baseline={([yshift=-.5ex]current bounding box.center)}, scale=.35]{
    \draw[knot, -] (1,0) .. controls (1,-0.25) and (2,-0.25) .. (2,0);
	\draw[knot, -, dashed] (1,0) .. controls (1,0.25) and (2,0.25) .. (2,0);
    \draw[knot, -] (0,3) .. controls (0,2.75) and (1,2.75) .. (1,3);
	\draw[knot, -] (0,3) .. controls (0,3.25) and (1,3.25) .. (1,3);
    \draw[knot, -] (2,3) .. controls (2,2.75) and (3,2.75) .. (3,3);
	\draw[knot, -] (2,3) .. controls (2,3.25) and (3,3.25) .. (3,3);
    \draw[knot, -] (0,3) .. controls (0,1.5) and (1,1.5) .. (1,0);
    \draw[knot, -] (3,3) .. controls (3,1.5) and (2,1.5) .. (2,0);
    \draw[knot, -] (1,3) .. controls (1,1.5) and (2,1.5) .. (2,3);
}
    $};
    \node at (11, -9) {$
    \tikz[baseline={([yshift=-.5ex]current bounding box.center)}, scale=.35]{
    \draw[knot, -] (-2,0) .. controls (-2,-0.25) and (-1,-0.25) .. (-1,0);
    \draw[knot, -, dashed] (-2,0) .. controls (-2,0.25) and (-1,0.25) .. (-1,0);
    \draw[knot, -] (-2,3) .. controls (-2,3-0.25) and (-1,3-0.25) .. (-1,3);
    \draw[knot, -] (-2,3) .. controls (-2,3.25) and (-1,3.25) .. (-1,3);
    \draw[knot, -] (-2,0) -- (-2,3);
    \draw[knot, -] (-1,0) -- (-1,3);
	\draw[knot, -] (0,0) .. controls (0,1.5) and (1,1.5) .. (1,3);
	\draw[knot, -] (1,0) .. controls (1,1.5) and (2,1.5) .. (2,0);
	\draw[knot, -] (3,0) .. controls (3,1.5) and (2,1.5) .. (2,3);
	\draw[knot, -] (0,0) .. controls (0,-.25) and (1,-.25) .. (1,0);
	\draw[dashed, knot, -] (0,0) .. controls (0,.25) and (1,.25) .. (1,0);
	\draw[knot, -] (2,0) .. controls (2,-.25) and (3,-.25) .. (3,0);
	\draw[dashed, knot, -] (2,0) .. controls (2,.25) and (3,.25) .. (3,0);
	\draw[knot, -] (1,3) .. controls (1,2.75) and (2,2.75) .. (2,3);
	\draw[knot, -] (1,3) .. controls (1,3.25) and (2,3.25) .. (2,3);
}
    $};
    \node at (11, -11) {$
    \tikz[baseline={([yshift=-.5ex]current bounding box.center)}, scale=.35]{
    \draw[knot, -] (1,0) .. controls (1,-0.25) and (2,-0.25) .. (2,0);
	\draw[knot, -, dashed] (1,0) .. controls (1,0.25) and (2,0.25) .. (2,0);
    \draw[knot, -] (0,3) .. controls (0,2.75) and (1,2.75) .. (1,3);
	\draw[knot, -] (0,3) .. controls (0,3.25) and (1,3.25) .. (1,3);
    \draw[knot, -] (2,3) .. controls (2,2.75) and (3,2.75) .. (3,3);
	\draw[knot, -] (2,3) .. controls (2,3.25) and (3,3.25) .. (3,3);
    \draw[knot, -] (0,3) .. controls (0,1.5) and (1,1.5) .. (1,0);
    \draw[knot, -] (3,3) .. controls (3,1.5) and (2,1.5) .. (2,0);
    \draw[knot, -] (1,3) .. controls (1,1.5) and (2,1.5) .. (2,3);
    \draw[knot, -] (5,0) .. controls (5,-0.25) and (4,-0.25) .. (4,0);
    \draw[knot, -, dashed] (5,0) .. controls (5,0.25) and (4,0.25) .. (4,0);
    \draw[knot, -] (5,3) .. controls (5,3-0.25) and (4,3-0.25) .. (4,3);
    \draw[knot, -] (5,3) .. controls (5,3.25) and (4,3.25) .. (4,3);
    \draw[knot, -] (5,0) -- (5,3);
    \draw[knot, -] (4,0) -- (4,3);
}
    $};
}
\caption{}
\label{tab:Etable}
\end{table}
		
To finish the proof, we need keep track of signs. We will do this sign computation for one closure, and leave the remaining cases for the reader to check. Consider Figure \ref{fig:R3closure}. In this case, note that the only nontrivial maps are $h_{2a}\circ h_{1a}$ and $E_{-1}' \circ d'$; see the fourth row of Table \ref{tab:Etable}. Choose consistent decorations on the two Reidemeister III diagrams and fix the same orientation for the family of metrics so that the sign assignments on the top cubes coincide. For the death maps arising from removing crossings, we fix the identification between the unlinks so that the caps are applied to the components from which the decoration arrow starts. Since the red maps together with the green chain homotopies are equal to the Reidemeister II map on odd Khovanov homology, we have that
\[
\epsilon_3\epsilon_4=-1,\qquad \text{and} \qquad \epsilon_5\epsilon_6=1.
\]
Then, using the fact that removing the first crossing on the left diagram induces a chain map, we have $\epsilon=-\epsilon_1\epsilon_2\epsilon_3$. Also, since the top cube of the right diagram is a chain complex, we have $\epsilon'=-\epsilon_1\epsilon_2\epsilon_5$. The sign on $h_{2a} \circ h_{1a}$ is given by $\epsilon\epsilon_4$ and the sign on $E_{-1}'\circ d'$ is given by $\epsilon' \epsilon_6$. Applying the equations above,
\[
\epsilon\epsilon_4=-\epsilon_1\epsilon_2\epsilon_3\epsilon_4=\epsilon_1\epsilon_2=-\epsilon'\epsilon_5=-\epsilon'\epsilon_6.
\]
Hence, the signs on the two maps are opposite, and cancel with one another.
\end{proof}

By Lemma \ref{lemma:R3computation}, $\widehat{H}=(G-g')\circ d'$. On one hand, $G-g'$ is also a chain homotopy between the identity map and $\bar{e}'\circ (\bar{e}')^{-1}$. On the other hand, given any two chain homotopies $I$ and $J$ between the identity map and $\bar{e}'\circ (\bar{e}')^{-1}$, the diagonal maps $I \circ d'$ and $J \circ d'$ are homotopic. Therefore, $\widehat{H}$ is the desired homotopy equivalence, concluding the proof.

\pagebreak

\begin{figure}[h]
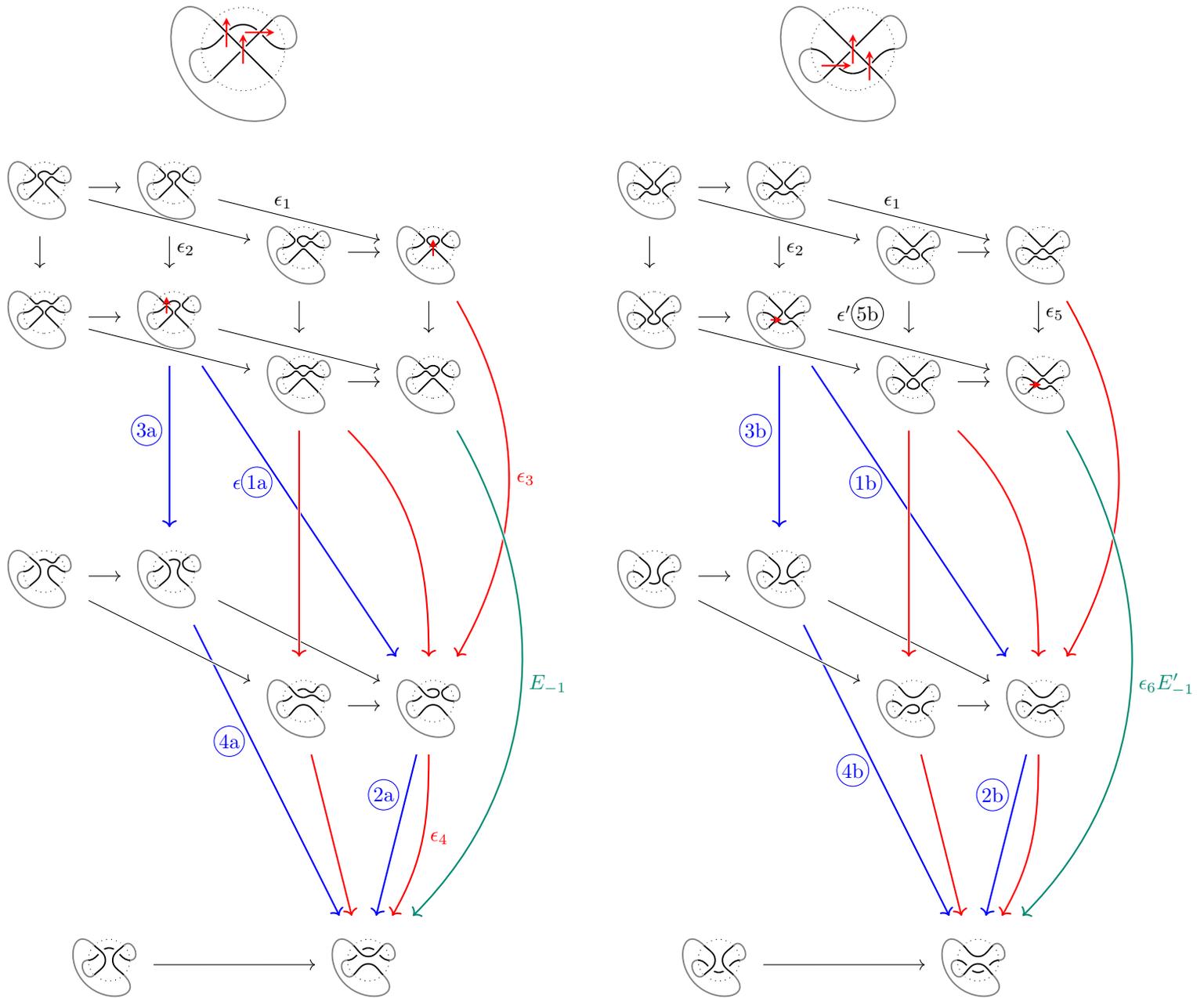

	\makebox[\textwidth][c]{
		\tikz[scale=2.2]{
	    \node at (1.5,1) {$
		\tikz[baseline={([yshift=-.5ex]current bounding box.center)}, scale=1]
		{
			\draw[dotted] (.5,.5) circle(0.707);
			\draw[knot] (-0.207, 0.5) to[out=0, in=180] (0.5, 0.915);
			\draw[knot] (1.207, 0.5) to[out=180, in=0] (0.5, 0.915);
			\draw[knot, overcross] (0,0) -- (0.45, 0.45);
			\draw[knot] (0.45, 0.45) -- (0.55, 0.55);
			\draw[knot, overcross] (0.55, 0.55) -- (1,1);
			\draw[knot, overcross] (0,1) -- (0.45, 0.55);
			\draw[knot] (0,1) -- (0.01, 0.99);
			\draw[knot, overcross] (0.45, 0.55) -- (0.55, 0.45);
			\draw[knot] (0.55, 0.45) -- (1,0);
			\draw[-stealth, red, thick] (0.5, 0.5-0.25) -- (0.5, 0.5+0.25);
			\draw[-stealth, red, thick] (0.782-0.25, 0.782) -- (0.782+0.25, 0.782);
			\draw[-stealth, red, thick] (0.5-0.282, 0.782-0.25) -- (0.5-0.282, 0.782+0.25);
			
			\draw[knot, gray] (1,1) to[out=45, in=0, looseness = 2] (1.207, 0.5);
			\draw[knot, gray] (0,0) to[out=-135, in=180, looseness=2] (-0.207,0.5);
			\draw[knot, gray] (1,0) to[out=-45, in=-45, looseness=2] (-0.3,-0.3) to[out=135, in=135, looseness=2] (0,1);
		        }
		$};
			\node(111) at (0,0) {$
				\tikz[baseline={([yshift=-.5ex]current bounding box.center)}, scale=0.5, rotate=-90]
				{
					\draw[dotted] (.5,.5) circle(0.707);
					\draw[knot] (0.5, 1.207) to[out=-90, in=45] (0.325, 0.825);
					\draw[knot] (0.175, 0.675) to[out=225, in=135] (0.175, 0.325);
					\draw[knot] (0.325, 0.175) to[out=-45, in=90] (0.5, -0.207);
					\draw[knot] (0,0) to (0.175, 0.175);
					\draw[knot] (0.325, 0.325) to (0.425, 0.425);
					\draw[knot] (0.575, 0.575) to (1,1);
					\draw[knot] (0,1) to (0.175, 0.825);
					\draw[knot] (0.325, 0.675) to (0.425, 0.575);
					\draw[knot] (0.575, 0.425) to (1,0);
					\draw[knot] (0.175, 0.175) to[out=45, in=135] (0.325, 0.175);
					\draw[knot] (0.175, 0.325) to[out=-45, in=-135] (0.325, 0.325);
					\draw[knot] (0.425, 0.425) to[out=45, in=135] (0.575, 0.425);
					\draw[knot] (0.425, 0.575) to[out=-45, in=-135] (0.575, 0.575);
					\draw[knot] (0.175, 0.675) to[out=45, in=-45] (0.175, 0.825);
					\draw[knot] (0.325, 0.675) to[out=135, in=-135] (0.325, 0.825);
					\draw[knot, gray] (0.5,-0.207) to[out=-90, in=-45, looseness = 2] (1,0);
					\draw[knot, gray] (0.5,1.207) to[out=90, in=135, looseness=2] (0,1);
					\draw[knot, gray] (0,0) to[out=-135, in=-135, looseness=2] (1.3,-0.3) to[out=45, in=45, looseness=2] (1,1);
				}
				$};
			\node(101) at (1,0) {$
				\tikz[baseline={([yshift=-.5ex]current bounding box.center)}, scale=0.5, rotate=-90]
				{
					\draw[dotted] (.5,.5) circle(0.707);
					\draw[knot] (0.5, 1.207) to[out=-90, in=45] (0.325, 0.825);
					\draw[knot] (0.175, 0.675) to[out=225, in=135] (0.175, 0.325);
					\draw[knot] (0.325, 0.175) to[out=-45, in=90] (0.5, -0.207);
					\draw[knot] (0,0) to (0.175, 0.175);
					\draw[knot] (0.325, 0.325) to (0.425, 0.425);
					\draw[knot] (0.575, 0.575) to (1,1);
					\draw[knot] (0,1) to (0.175, 0.825);
					\draw[knot] (0.325, 0.675) to (0.425, 0.575);
					\draw[knot] (0.575, 0.425) to (1,0);
					\draw[knot] (0.175, 0.175) to[out=45, in=135] (0.325, 0.175);
					\draw[knot] (0.175, 0.325) to[out=-45, in=-135] (0.325, 0.325);
					\draw[knot] (0.425, 0.425) to[out=45, in=135] (0.575, 0.425);
					\draw[knot] (0.425, 0.575) to[out=-45, in=-135] (0.575, 0.575);
					\draw[knot] (0.175, 0.825) to[out=-45, in=-135] (0.325, 0.825);
					\draw[knot] (0.175, 0.675) to[out=45, in=135] (0.325, 0.675);
					\draw[knot, gray] (0.5,-0.207) to[out=-90, in=-45, looseness = 2] (1,0);
					\draw[knot, gray] (0.5,1.207) to[out=90, in=135, looseness=2] (0,1);
					\draw[knot, gray] (0,0) to[out=-135, in=-135, looseness=2] (1.3,-0.3) to[out=45, in=45, looseness=2] (1,1);
				}
				$};
			\node(110) at (2,-.5) {$
				\tikz[baseline={([yshift=-.5ex]current bounding box.center)}, scale=0.5, rotate=-90]
				{
					\draw[dotted] (.5,.5) circle(0.707);
					\draw[knot] (0.5, 1.207) to[out=-90, in=45] (0.325, 0.825);
					\draw[knot] (0.175, 0.675) to[out=225, in=135] (0.175, 0.325);
					\draw[knot] (0.325, 0.175) to[out=-45, in=90] (0.5, -0.207);
					\draw[knot] (0,0) to (0.175, 0.175);
					\draw[knot] (0.325, 0.325) to (0.425, 0.425);
					\draw[knot] (0.575, 0.575) to (1,1);
					\draw[knot] (0,1) to (0.175, 0.825);
					\draw[knot] (0.325, 0.675) to (0.425, 0.575);
					\draw[knot] (0.575, 0.425) to (1,0);
					\draw[knot] (0.175, 0.175) to[out=45, in=135] (0.325, 0.175);
					\draw[knot] (0.175, 0.325) to[out=-45, in=-135] (0.325, 0.325);
					\draw[knot] (0.425, 0.425) to[out=45, in=-45] (0.425, 0.575); 
					\draw[knot] (0.575, 0.425) to[out=135, in=-135] (0.575, 0.575);
					\draw[knot] (0.175, 0.675) to[out=45, in=-45] (0.175, 0.825);
					\draw[knot] (0.325, 0.675) to[out=135, in=-135] (0.325, 0.825);
					\draw[knot, gray] (0.5,-0.207) to[out=-90, in=-45, looseness = 2] (1,0);
					\draw[knot, gray] (0.5,1.207) to[out=90, in=135, looseness=2] (0,1);
					\draw[knot, gray] (0,0) to[out=-135, in=-135, looseness=2] (1.3,-0.3) to[out=45, in=45, looseness=2] (1,1);
				}
				$};
			\node(100) at (3,-.5) {$
				\tikz[baseline={([yshift=-.5ex]current bounding box.center)}, scale=0.5, rotate=-90]
				{
					\draw[dotted] (.5,.5) circle(0.707);
					\draw[knot] (0.5, 1.207) to[out=-90, in=45] (0.325, 0.825);
					\draw[knot] (0.175, 0.675) to[out=225, in=135] (0.175, 0.325);
					\draw[knot] (0.325, 0.175) to[out=-45, in=90] (0.5, -0.207);
					\draw[knot] (0,0) to (0.175, 0.175);
					\draw[knot] (0.325, 0.325) to (0.425, 0.425);
					\draw[knot] (0.575, 0.575) to (1,1);
					\draw[knot] (0,1) to (0.175, 0.825);
					\draw[knot] (0.325, 0.675) to (0.425, 0.575);
					\draw[knot] (0.575, 0.425) to (1,0);
					\draw[knot] (0.175, 0.175) to[out=45, in=135] (0.325, 0.175);
					\draw[knot] (0.175, 0.325) to[out=-45, in=-135] (0.325, 0.325);
					\draw[knot] (0.425, 0.425) to[out=45, in=-45] node(X)[pos=0.5]{} (0.425, 0.575); 
					\draw[knot] (0.575, 0.425) to[out=135, in=-135] node(Y)[pos=0.5]{} (0.575, 0.575);
					\draw[knot] (0.175, 0.825) to[out=-45, in=-135] (0.325, 0.825);
					\draw[knot] (0.175, 0.675) to[out=45, in=135] (0.325, 0.675);
					\draw[knot, gray] (0.5,-0.207) to[out=-90, in=-45, looseness = 2] (1,0);
					\draw[knot, gray] (0.5,1.207) to[out=90, in=135, looseness=2] (0,1);
					\draw[knot, gray] (0,0) to[out=-135, in=-135, looseness=2] (1.3,-0.3) to[out=45, in=45, looseness=2] (1,1);
					\draw[red, stealth-, thick, shorten >=.1mm, shorten <=.1mm] (X.north) -- (Y.south);
				}
				$};
			\node(011) at (0,-1) {$
				\tikz[baseline={([yshift=-.5ex]current bounding box.center)}, scale=0.5, rotate=-90]
				{
					\draw[dotted] (.5,.5) circle(0.707);
					\draw[knot] (0.5, 1.207) to[out=-90, in=45] (0.325, 0.825);
					\draw[knot] (0.175, 0.675) to[out=225, in=135] (0.175, 0.325);
					\draw[knot] (0.325, 0.175) to[out=-45, in=90] (0.5, -0.207);
					\draw[knot] (0,0) to (0.175, 0.175);
					\draw[knot] (0.325, 0.325) to (0.425, 0.425);
					\draw[knot] (0.575, 0.575) to (1,1);
					\draw[knot] (0,1) to (0.175, 0.825);
					\draw[knot] (0.325, 0.675) to (0.425, 0.575);
					\draw[knot] (0.575, 0.425) to (1,0);
					\draw[knot] (0.175, 0.175) to[out=45, in=-45] (0.175, 0.325);
					\draw[knot] (0.325, 0.175) to[out=135, in=-135] (0.325, 0.325);
					\draw[knot] (0.425, 0.425) to[out=45, in=135] (0.575, 0.425);
					\draw[knot] (0.425, 0.575) to[out=-45, in=-135] (0.575, 0.575);
					\draw[knot] (0.175, 0.675) to[out=45, in=-45] (0.175, 0.825);
					\draw[knot] (0.325, 0.675) to[out=135, in=-135] (0.325, 0.825);
					\draw[knot, gray] (0.5,-0.207) to[out=-90, in=-45, looseness = 2] (1,0);
					\draw[knot, gray] (0.5,1.207) to[out=90, in=135, looseness=2] (0,1);
					\draw[knot, gray] (0,0) to[out=-135, in=-135, looseness=2] (1.3,-0.3) to[out=45, in=45, looseness=2] (1,1);
				}
				$};
			\node(001) at (1,-1) {$
				\tikz[baseline={([yshift=-.5ex]current bounding box.center)}, scale=0.5, rotate=-90]
				{
					\draw[dotted] (.5,.5) circle(0.707);
					\draw[knot] (0.5, 1.207) to[out=-90, in=45] (0.325, 0.825);
					\draw[knot] (0.175, 0.675) to[out=225, in=135] (0.175, 0.325);
					\draw[knot] (0.325, 0.175) to[out=-45, in=90] (0.5, -0.207);
					\draw[knot] (0,0) to (0.175, 0.175);
					\draw[knot] (0.325, 0.325) to (0.425, 0.425);
					\draw[knot] (0.575, 0.575) to (1,1);
					\draw[knot] (0,1) to (0.175, 0.825);
					\draw[knot] (0.325, 0.675) to (0.425, 0.575);
					\draw[knot] (0.575, 0.425) to (1,0);
					\draw[knot] (0.175, 0.175) to[out=45, in=-45] node(X)[pos=0.5]{} (0.175, 0.325);
					\draw[knot] (0.325, 0.175) to[out=135, in=-135] node(Y)[pos=0.5]{} (0.325, 0.325);
					\draw[knot] (0.425, 0.425) to[out=45, in=135] (0.575, 0.425);
					\draw[knot] (0.425, 0.575) to[out=-45, in=-135] (0.575, 0.575);
					\draw[knot] (0.175, 0.825) to[out=-45, in=-135] (0.325, 0.825);
					\draw[knot] (0.175, 0.675) to[out=45, in=135] (0.325, 0.675);
					\draw[knot, gray] (0.5,-0.207) to[out=-90, in=-45, looseness = 2] (1,0);
					\draw[knot, gray] (0.5,1.207) to[out=90, in=135, looseness=2] (0,1);
					\draw[knot, gray] (0,0) to[out=-135, in=-135, looseness=2] (1.3,-0.3) to[out=45, in=45, looseness=2] (1,1);
					\draw[red, stealth-, thick, shorten >=.1mm, shorten <=.1mm] (X.north) -- (Y.south);
				}
				$};
			\node(010) at (2,-1.5) {$
				\tikz[baseline={([yshift=-.5ex]current bounding box.center)}, scale=0.5, rotate=-90]
				{
					\draw[dotted] (.5,.5) circle(0.707);
					\draw[knot] (0.5, 1.207) to[out=-90, in=45] (0.325, 0.825);
					\draw[knot] (0.175, 0.675) to[out=225, in=135] (0.175, 0.325);
					\draw[knot] (0.325, 0.175) to[out=-45, in=90] (0.5, -0.207);
					\draw[knot] (0,0) to (0.175, 0.175);
					\draw[knot] (0.325, 0.325) to (0.425, 0.425);
					\draw[knot] (0.575, 0.575) to (1,1);
					\draw[knot] (0,1) to (0.175, 0.825);
					\draw[knot] (0.325, 0.675) to (0.425, 0.575);
					\draw[knot] (0.575, 0.425) to (1,0);
					\draw[knot] (0.175, 0.175) to[out=45, in=-45] (0.175, 0.325);
					\draw[knot] (0.325, 0.175) to[out=135, in=-135] (0.325, 0.325);
					\draw[knot] (0.425, 0.425) to[out=45, in=-45] (0.425, 0.575); 
					\draw[knot] (0.575, 0.425) to[out=135, in=-135] (0.575, 0.575);
					\draw[knot] (0.175, 0.675) to[out=45, in=-45] (0.175, 0.825);
					\draw[knot] (0.325, 0.675) to[out=135, in=-135] (0.325, 0.825);
					\draw[knot, gray] (0.5,-0.207) to[out=-90, in=-45, looseness = 2] (1,0);
					\draw[knot, gray] (0.5,1.207) to[out=90, in=135, looseness=2] (0,1);
					\draw[knot, gray] (0,0) to[out=-135, in=-135, looseness=2] (1.3,-0.3) to[out=45, in=45, looseness=2] (1,1);
				}
				$};
			\node(000) at (3,-1.5) {$
				\tikz[baseline={([yshift=-.5ex]current bounding box.center)}, scale=0.5, rotate=-90]
				{
					\draw[dotted] (.5,.5) circle(0.707);
					\draw[knot] (0.5, 1.207) to[out=-90, in=45] (0.325, 0.825);
					\draw[knot] (0.175, 0.675) to[out=225, in=135] (0.175, 0.325);
					\draw[knot] (0.325, 0.175) to[out=-45, in=90] (0.5, -0.207);
					\draw[knot] (0,0) to (0.175, 0.175);
					\draw[knot] (0.325, 0.325) to (0.425, 0.425);
					\draw[knot] (0.575, 0.575) to (1,1);
					\draw[knot] (0,1) to (0.175, 0.825);
					\draw[knot] (0.325, 0.675) to (0.425, 0.575);
					\draw[knot] (0.575, 0.425) to (1,0);
					\draw[knot] (0.175, 0.175) to[out=45, in=-45] (0.175, 0.325);
					\draw[knot] (0.325, 0.175) to[out=135, in=-135] (0.325, 0.325);
					\draw[knot] (0.425, 0.425) to[out=45, in=-45] (0.425, 0.575); 
					\draw[knot] (0.575, 0.425) to[out=135, in=-135] (0.575, 0.575);
					\draw[knot] (0.175, 0.825) to[out=-45, in=-135] (0.325, 0.825);
					\draw[knot] (0.175, 0.675) to[out=45, in=135] (0.325, 0.675);
					\draw[knot, gray] (0.5,-0.207) to[out=-90, in=-45, looseness = 2] (1,0);
					\draw[knot, gray] (0.5,1.207) to[out=90, in=135, looseness=2] (0,1);
					\draw[knot, gray] (0,0) to[out=-135, in=-135, looseness=2] (1.3,-0.3) to[out=45, in=45, looseness=2] (1,1);
				}
				$};
			\node(211) at (0,-3) {$
				\tikz[baseline={([yshift=-.5ex]current bounding box.center)}, scale=0.5]
				{
					\draw[dotted] (.5,.5) circle(0.707);
					\draw[knot] (-0.207, 0.5) to[out=0, in=180] (0.5, 0.915) to[out=0, in=180] (0.782, 0.8) to[out=0, in=-135] (1,1);
					\draw[knot] (1.207, 0.5) to[out=180, in=0] (0.782, 0.7) to[out=180, in=90] (0.625,0.5) to[out=-90, in=135] (1,0);
					\draw[knot, overcross] (0,0) to[out=45, in=-90] (0.375,0.5) to[out=90, in=-45] (0,1);
					\draw[knot, gray] (1,1) to[out=45, in=0, looseness = 2] (1.207, 0.5);
					\draw[knot, gray] (0,0) to[out=-135, in=180, looseness=2] (-0.207,0.5);
					\draw[knot, gray] (1,0) to[out=-45, in=-45, looseness=2] (-0.3,-0.3) to[out=135, in=135, looseness=2] (0,1);
				}
				$};
			\node(201) at (1,-3) {$
				\tikz[baseline={([yshift=-.5ex]current bounding box.center)}, scale=0.5]
				{
					\draw[dotted] (.5,.5) circle(0.707);
					\draw[knot] (-0.207, 0.5) to[out=0, in=180] (0.5, 0.915) to[out=0, in=90] (0.7, 0.782) to[out=-90, in=90] (0.625,0.5) to[out=-90, in=135] (1,0);
					\draw[knot] (1,1) to[out=-135, in=90] (0.8, 0.782) to[out=-90, in=180] (1.207, 0.5);
					\draw[knot, overcross] (0,0) to[out=45, in=-90] (0.375,0.5) to[out=90, in=-45] (0,1);
					\draw[knot, gray] (1,1) to[out=45, in=0, looseness = 2] (1.207, 0.5);
					\draw[knot, gray] (0,0) to[out=-135, in=180, looseness=2] (-0.207,0.5);
					\draw[knot, gray] (1,0) to[out=-45, in=-45, looseness=2] (-0.3,-0.3) to[out=135, in=135, looseness=2] (0,1);
				}
				$};
			\node(210) at (2,-4) {$
				\tikz[baseline={([yshift=-.5ex]current bounding box.center)}, scale=0.5]
				{
					\draw[dotted] (.5,.5) circle(0.707);
					\draw[knot] (-0.207, 0.5) to[out=0, in=180] (0.5, 0.915);
					\draw[knot, overcross] (0,0) to[out=45, in=180] (0.5,0.375) to[out=0, in=135] (1,0);
					\draw[knot, overcross] (0,1) to[out=-45, in=180] (0.5,0.625);
					\fill[white, fill opacity=1] (0.782, 0.782) circle(0.15);
					\draw[knot] (0.5, 0.915) to[out=0, in=180] (0.782, 0.8) to[out=0, in=-135] (1,1);
					\draw[knot] (0.5,0.625) to [out=0, in=180] (0.782, 0.7) to[out=0, in=180] (1.207, 0.5);
					\draw[knot, gray] (1,1) to[out=45, in=0, looseness = 2] (1.207, 0.5);
					\draw[knot, gray] (0,0) to[out=-135, in=180, looseness=2] (-0.207,0.5);
					\draw[knot, gray] (1,0) to[out=-45, in=-45, looseness=2] (-0.3,-0.3) to[out=135, in=135, looseness=2] (0,1);
				}
				$};
			\node(200) at (3,-4) {$
				\tikz[baseline={([yshift=-.5ex]current bounding box.center)}, scale=0.5]
				{
					\draw[dotted] (.5,.5) circle(0.707);
					\draw[knot] (-0.207, 0.5) to[out=0, in=180] (0.5, 0.915);
					\draw[knot, overcross] (0,0) to[out=45, in=180] (0.5,0.375) to[out=0, in=135] (1,0);
					\draw[knot, overcross] (0,1) to[out=-45, in=180] (0.5,0.625);
					\fill[white, fill opacity=1] (0.782, 0.782) circle(0.15);
					\draw[knot] (0.5, 0.915) to[out=0, in=90] (0.7, 0.782) to[out=-90, in=0] (0.5,0.625);
					\draw[knot] (1,1) to[out=-135, in=90] (0.8, 0.782) to[out=-90, in=180] (1.207, 0.5);
					\draw[knot, gray] (1,1) to[out=45, in=0, looseness = 2] (1.207, 0.5);
					\draw[knot, gray] (0,0) to[out=-135, in=180, looseness=2] (-0.207,0.5);
					\draw[knot, gray] (1,0) to[out=-45, in=-45, looseness=2] (-0.3,-0.3) to[out=135, in=135, looseness=2] (0,1);
				}
				$};
			\node(222) at (0.5,-6) {$
				\tikz[baseline={([yshift=-.5ex]current bounding box.center)}, scale=0.5]
				{
					\draw[dotted] (.5,.5) circle(0.707);
					\draw[knot] (-0.207, 0.5) to[out=0, in=180] (0.5, 0.915);
					\draw[knot] (1.207, 0.5) to[out=180, in=0] (0.5, 0.915);
					\draw[knot, overcross] (0,0) to[out=45, in=-90] (0.375,0.5) to[out=90, in=-45] (0,1);
					\draw[knot, overcross] (1,0) to[out=135, in=-90] (0.625,0.5) to[out=90, in=-135] (1,1);
					\draw[knot, gray] (1,1) to[out=45, in=0, looseness = 2] (1.207, 0.5);
					\draw[knot, gray] (0,0) to[out=-135, in=180, looseness=2] (-0.207,0.5);
					\draw[knot, gray] (1,0) to[out=-45, in=-45, looseness=2] (-0.3,-0.3) to[out=135, in=135, looseness=2] (0,1);
				}
				$};
			\node(220) at (2.5,-6) {$
				\tikz[baseline={([yshift=-.5ex]current bounding box.center)}, scale=0.5]
				{
					\draw[dotted] (.5,.5) circle(0.707);
					\draw[knot] (-0.207, 0.5) to[out=0, in=180] (0.5, 0.915);
					\draw[knot] (1.207, 0.5) to[out=180, in=0] (0.5, 0.915);
					\draw[knot, overcross] (0,0) to[out=45, in=180] (0.5,0.375) to[out=0, in=135] (1,0);
					\draw[knot, overcross] (0,1) to[out=-45, in=180] (0.5,0.625) to[out=0, in=-135] (1,1);
					\draw[knot, gray] (1,1) to[out=45, in=0, looseness = 2] (1.207, 0.5);
					\draw[knot, gray] (0,0) to[out=-135, in=180, looseness=2] (-0.207,0.5);
					\draw[knot, gray] (1,0) to[out=-45, in=-45, looseness=2] (-0.3,-0.3) to[out=135, in=135, looseness=2] (0,1);
				}
				$};
			\draw[->] (201) -- (200);
			\draw[->, thick, blue] (001) -- node[left, pos=0.4]{$\circled{3a}$} (201);
			\draw[white, ultra thick] (001) -- (200);
			\draw[->, thick, blue] (001) -- node[left, pos=0.4]{$\epsilon\circled{1a}$} (200);
			\draw[->, thick, blue] (201) -- node[left, pos=0.4]{$\circled{4a}$} (220);
			\draw[->, thick, blue] (200) -- node[left, pos=0.25]{$\circled{2a}$} (220);
			\draw[->] (111) -- (110);
			\draw[->] (101) -- node[above, pos=0.4]{$\epsilon_1$}(100);
			\draw[white, ultra thick] (011) -- (010);
			\draw[->] (011) -- (010);
			\draw[->] (001) -- (000);
			\draw[->] (111) -- (101);
			\draw[->] (111) -- (011);
			\draw[->] (011) -- (001);
			\draw[->] (101) -- node[right, pos=0.4]{$\epsilon_2$} (001);
			\draw[->] (110) -- (100);
			\draw[->] (110) -- (010);
			\draw[->] (010) -- (000);
			\draw[->] (100) -- (000);
			\draw[->] (211) -- (201);
			\draw[white, ultra thick] (211) -- (210);
			\draw[white, ultra thick] (210) -- (200);
			\draw[->] (211) -- (210);
			\draw[->] (210) -- (200);
			\draw[->] (222) -- (220);
			\draw[->, thick, red] (100) to[out=-60, in=60] node[right]{$\epsilon_3$} (200);
			\draw[white, ultra thick] (010) -- (210);
			\draw[->, thick, red] (010) --node[right, pos=0.8]{} (210);
			\draw[->, thick, red] (010) to[out=-45, in=90] node[right]{} (200);
			\draw[->, thick, red] (210) -- node[right]{} (220);
			\draw[->, thick, red] (200) to[out=-90,in=60]node[right]{$\epsilon_4$} (220);
			\draw[->, ultra thick, white] (000) to[out=-60, in=45] (220);
			\draw[->, thick, PineGreen] (000) to[out=-60, in=45]node[right]{$E_{-1}$} (220);
		}
		\quad
		\tikz[scale=2.2]{
\node at (1.5,1) {$
	\tikz[baseline={([yshift=-.5ex]current bounding box.center)}, scale=1]
	{
		\draw[dotted] (.5,.5) circle(0.707);
		\draw[knot] (-0.207, 0.5) to[out=0, in=180] (0.5, 0.085);
		\draw[knot] (1.207, 0.5) to[out=180, in=0] (0.5, 0.085);
		\draw[knot, overcross] (0,0) -- (0.45, 0.45);
		\draw[knot] (0.45, 0.45) -- (0.55, 0.55);
		\draw[knot, overcross] (0.55, 0.55) -- (1,1);
		\draw[knot, overcross] (0,1) -- (0.45, 0.55);
		\draw[knot] (0,1) -- (0.01, 0.99);
		\draw[knot, overcross] (0.45, 0.55) -- (0.55, 0.45);
		\draw[knot, overcross] (0.55, 0.45) -- (1,0);	
		\draw[-stealth, red, thick] (0.5, 0.5-0.25) -- (0.5, 0.5+0.25);
		\draw[-stealth, red, thick] (0.218-0.25, 0.218) -- (0.218+0.25, 0.218);
		\draw[-stealth, red, thick] (0.782, 0.218-0.25) -- (0.782, 0.218+0.25);
		\draw[knot, gray] (1,1) to[out=45, in=0, looseness = 2] (1.207, 0.5);
		\draw[knot, gray] (0,0) to[out=-135, in=180, looseness=2] (-0.207,0.5);
		\draw[knot, gray] (1,0) to[out=-45, in=-45, looseness=2] (-0.3,-0.3) to[out=135, in=135, looseness=2] (0,1);
	}
	$};
			\node(111) at (0,0) {$
				\tikz[baseline={([yshift=-.5ex]current bounding box.center)}, scale=0.5, rotate=90]
				{
					\draw[dotted] (.5,.5) circle(0.707);
					\draw[knot] (0.5, 1.207) to[out=-90, in=45] (0.325, 0.825);
					\draw[knot] (0.175, 0.675) to[out=225, in=135] (0.175, 0.325);
					\draw[knot] (0.325, 0.175) to[out=-45, in=90] (0.5, -0.207);
					\draw[knot] (0,0) to (0.175, 0.175);
					\draw[knot] (0.325, 0.325) to (0.425, 0.425);
					\draw[knot] (0.575, 0.575) to (1,1);
					\draw[knot] (0,1) to (0.175, 0.825);
					\draw[knot] (0.325, 0.675) to (0.425, 0.575);
					\draw[knot] (0.575, 0.425) to (1,0);
					\draw[knot] (0.175, 0.175) to[out=45, in=135] (0.325, 0.175);
					\draw[knot] (0.175, 0.325) to[out=-45, in=-135] (0.325, 0.325);
					\draw[knot] (0.425, 0.425) to[out=45, in=135] (0.575, 0.425);
					\draw[knot] (0.425, 0.575) to[out=-45, in=-135] (0.575, 0.575);
					\draw[knot] (0.175, 0.675) to[out=45, in=-45] (0.175, 0.825);
					\draw[knot] (0.325, 0.675) to[out=135, in=-135] (0.325, 0.825);
					\draw[knot, gray] (0.5,-0.207) to[out=-90, in=-45, looseness = 2] (1,0);
					\draw[knot, gray] (0.5,1.207) to[out=90, in=135, looseness=2] (0,1);
					\draw[knot, gray] (0,0) to[out=-135, in=-135, looseness=2] (-0.3,1.3) to[out=45, in=45, looseness=2] (1,1);
				}
				$};
			\node(101) at (1,0) {$
				\tikz[baseline={([yshift=-.5ex]current bounding box.center)}, scale=0.5, rotate=90]
				{
					\draw[dotted] (.5,.5) circle(0.707);
					\draw[knot] (0.5, 1.207) to[out=-90, in=45] (0.325, 0.825);
					\draw[knot] (0.175, 0.675) to[out=225, in=135] (0.175, 0.325);
					\draw[knot] (0.325, 0.175) to[out=-45, in=90] (0.5, -0.207);
					\draw[knot] (0,0) to (0.175, 0.175);
					\draw[knot] (0.325, 0.325) to (0.425, 0.425);
					\draw[knot] (0.575, 0.575) to (1,1);
					\draw[knot] (0,1) to (0.175, 0.825);
					\draw[knot] (0.325, 0.675) to (0.425, 0.575);
					\draw[knot] (0.575, 0.425) to (1,0);
					\draw[knot] (0.175, 0.175) to[out=45, in=-45] (0.175, 0.325);
					\draw[knot] (0.325, 0.175) to[out=135, in=-135] (0.325, 0.325);
					\draw[knot] (0.425, 0.425) to[out=45, in=135] (0.575, 0.425);
					\draw[knot] (0.425, 0.575) to[out=-45, in=-135] (0.575, 0.575);
					\draw[knot] (0.175, 0.675) to[out=45, in=-45] (0.175, 0.825);
					\draw[knot] (0.325, 0.675) to[out=135, in=-135] (0.325, 0.825);
					\draw[knot, gray] (0.5,-0.207) to[out=-90, in=-45, looseness = 2] (1,0);
					\draw[knot, gray] (0.5,1.207) to[out=90, in=135, looseness=2] (0,1);
					\draw[knot, gray] (0,0) to[out=-135, in=-135, looseness=2] (-0.3,1.3) to[out=45, in=45, looseness=2] (1,1);
				}
				$};
			\node(110) at (2,-.5) {$
				\tikz[baseline={([yshift=-.5ex]current bounding box.center)}, scale=0.5, rotate=90]
				{
					\draw[dotted] (.5,.5) circle(0.707);
					\draw[knot] (0.5, 1.207) to[out=-90, in=45] (0.325, 0.825);
					\draw[knot] (0.175, 0.675) to[out=225, in=135] (0.175, 0.325);
					\draw[knot] (0.325, 0.175) to[out=-45, in=90] (0.5, -0.207);
					\draw[knot] (0,0) to (0.175, 0.175);
					\draw[knot] (0.325, 0.325) to (0.425, 0.425);
					\draw[knot] (0.575, 0.575) to (1,1);
					\draw[knot] (0,1) to (0.175, 0.825);
					\draw[knot] (0.325, 0.675) to (0.425, 0.575);
					\draw[knot] (0.575, 0.425) to (1,0);
					\draw[knot] (0.175, 0.175) to[out=45, in=135] (0.325, 0.175);
					\draw[knot] (0.175, 0.325) to[out=-45, in=-135] (0.325, 0.325);
					\draw[knot] (0.425, 0.425) to[out=45, in=-45] (0.425, 0.575); 
					\draw[knot] (0.575, 0.425) to[out=135, in=-135] (0.575, 0.575);
					\draw[knot] (0.175, 0.675) to[out=45, in=-45] (0.175, 0.825);
					\draw[knot] (0.325, 0.675) to[out=135, in=-135] (0.325, 0.825);
					\draw[knot, gray] (0.5,-0.207) to[out=-90, in=-45, looseness = 2] (1,0);
					\draw[knot, gray] (0.5,1.207) to[out=90, in=135, looseness=2] (0,1);
					\draw[knot, gray] (0,0) to[out=-135, in=-135, looseness=2] (-0.3,1.3) to[out=45, in=45, looseness=2] (1,1);
				}
				$};
			\node(100) at (3,-.5) {$
				\tikz[baseline={([yshift=-.5ex]current bounding box.center)}, scale=0.5, rotate=90]
				{
					\draw[dotted] (.5,.5) circle(0.707);
					\draw[knot] (0.5, 1.207) to[out=-90, in=45] (0.325, 0.825);
					\draw[knot] (0.175, 0.675) to[out=225, in=135] (0.175, 0.325);
					\draw[knot] (0.325, 0.175) to[out=-45, in=90] (0.5, -0.207);
					\draw[knot] (0,0) to (0.175, 0.175);
					\draw[knot] (0.325, 0.325) to (0.425, 0.425);
					\draw[knot] (0.575, 0.575) to (1,1);
					\draw[knot] (0,1) to (0.175, 0.825);
					\draw[knot] (0.325, 0.675) to (0.425, 0.575);
					\draw[knot] (0.575, 0.425) to (1,0);
					\draw[knot] (0.175, 0.175) to[out=45, in=-45] (0.175, 0.325);
					\draw[knot] (0.325, 0.175) to[out=135, in=-135] (0.325, 0.325);
					\draw[knot] (0.425, 0.425) to[out=45, in=-45] (0.425, 0.575); 
					\draw[knot] (0.575, 0.425) to[out=135, in=-135] (0.575, 0.575);
					\draw[knot] (0.175, 0.675) to[out=45, in=-45] (0.175, 0.825);
					\draw[knot] (0.325, 0.675) to[out=135, in=-135] (0.325, 0.825);
					\draw[knot, gray] (0.5,-0.207) to[out=-90, in=-45, looseness = 2] (1,0);
					\draw[knot, gray] (0.5,1.207) to[out=90, in=135, looseness=2] (0,1);
					\draw[knot, gray] (0,0) to[out=-135, in=-135, looseness=2] (-0.3,1.3) to[out=45, in=45, looseness=2] (1,1);
				}
				$};
			\node(011) at (0,-1) {$
				\tikz[baseline={([yshift=-.5ex]current bounding box.center)}, scale=0.5, rotate=90]
				{
					\draw[dotted] (.5,.5) circle(0.707);
					\draw[knot] (0.5, 1.207) to[out=-90, in=45] (0.325, 0.825);
					\draw[knot] (0.175, 0.675) to[out=225, in=135] (0.175, 0.325);
					\draw[knot] (0.325, 0.175) to[out=-45, in=90] (0.5, -0.207);
					\draw[knot] (0,0) to (0.175, 0.175);
					\draw[knot] (0.325, 0.325) to (0.425, 0.425);
					\draw[knot] (0.575, 0.575) to (1,1);
					\draw[knot] (0,1) to (0.175, 0.825);
					\draw[knot] (0.325, 0.675) to (0.425, 0.575);
					\draw[knot] (0.575, 0.425) to (1,0);
					\draw[knot] (0.175, 0.175) to[out=45, in=135] (0.325, 0.175);
					\draw[knot] (0.175, 0.325) to[out=-45, in=-135] (0.325, 0.325);
					\draw[knot] (0.425, 0.425) to[out=45, in=135] (0.575, 0.425);
					\draw[knot] (0.425, 0.575) to[out=-45, in=-135] (0.575, 0.575);
					\draw[knot] (0.175, 0.825) to[out=-45, in=-135] (0.325, 0.825);
					\draw[knot] (0.175, 0.675) to[out=45, in=135] (0.325, 0.675);
					\draw[knot, gray] (0.5,-0.207) to[out=-90, in=-45, looseness = 2] (1,0);
					\draw[knot, gray] (0.5,1.207) to[out=90, in=135, looseness=2] (0,1);
					\draw[knot, gray] (0,0) to[out=-135, in=-135, looseness=2] (-0.3,1.3) to[out=45, in=45, looseness=2] (1,1);
				}
				$};
		\node(001) at (1,-1) {$
			\tikz[baseline={([yshift=-.5ex]current bounding box.center)}, scale=0.5, rotate=90]
			{
				\draw[dotted] (.5,.5) circle(0.707);
				\draw[knot] (0.5, 1.207) to[out=-90, in=45] (0.325, 0.825);
				\draw[knot] (0.175, 0.675) to[out=225, in=135] (0.175, 0.325);
				\draw[knot] (0.325, 0.175) to[out=-45, in=90] (0.5, -0.207);
				\draw[knot] (0,0) to (0.175, 0.175);
				\draw[knot] (0.325, 0.325) to (0.425, 0.425);
				\draw[knot] (0.575, 0.575) to (1,1);
				\draw[knot] (0,1) to (0.175, 0.825);
				\draw[knot] (0.325, 0.675) to (0.425, 0.575);
				\draw[knot] (0.575, 0.425) to (1,0);
				\draw[knot] (0.175, 0.175) to[out=45, in=-45] (0.175, 0.325);
				\draw[knot] (0.325, 0.175) to[out=135, in=-135] (0.325, 0.325);
				\draw[knot] (0.425, 0.425) to[out=45, in=135] (0.575, 0.425);
				\draw[knot] (0.425, 0.575) to[out=-45, in=-135] (0.575, 0.575);
				\draw[knot] (0.175, 0.825) to[out=-45, in=-135] node(X)[pos=0.5]{} (0.325, 0.825);
				\draw[knot] (0.175, 0.675) to[out=45, in=135] node(Y)[pos=0.5]{} (0.325, 0.675);
				\draw[knot, gray] (0.5,-0.207) to[out=-90, in=-45, looseness = 2] (1,0);
				\draw[knot, gray] (0.5,1.207) to[out=90, in=135, looseness=2] (0,1);
				\draw[knot, gray] (0,0) to[out=-135, in=-135, looseness=2] (-0.3,1.3) to[out=45, in=45, looseness=2] (1,1);
				\draw[red, stealth-, thick, shorten >=.1mm, shorten <=.1mm] (X.east) -- (Y.west);
			}
			$};
			\node(010) at (2,-1.5) {$
				\tikz[baseline={([yshift=-.5ex]current bounding box.center)}, scale=0.5, rotate=90]
				{
					\draw[dotted] (.5,.5) circle(0.707);
					\draw[knot] (0.5, 1.207) to[out=-90, in=45] (0.325, 0.825);
					\draw[knot] (0.175, 0.675) to[out=225, in=135] (0.175, 0.325);
					\draw[knot] (0.325, 0.175) to[out=-45, in=90] (0.5, -0.207);
					\draw[knot] (0,0) to (0.175, 0.175);
					\draw[knot] (0.325, 0.325) to (0.425, 0.425);
					\draw[knot] (0.575, 0.575) to (1,1);
					\draw[knot] (0,1) to (0.175, 0.825);
					\draw[knot] (0.325, 0.675) to (0.425, 0.575);
					\draw[knot] (0.575, 0.425) to (1,0);
					\draw[knot] (0.175, 0.175) to[out=45, in=135] (0.325, 0.175);
					\draw[knot] (0.175, 0.325) to[out=-45, in=-135] (0.325, 0.325);
					\draw[knot] (0.425, 0.425) to[out=45, in=-45] (0.425, 0.575); 
					\draw[knot] (0.575, 0.425) to[out=135, in=-135] (0.575, 0.575);
					\draw[knot] (0.175, 0.825) to[out=-45, in=-135] (0.325, 0.825);
					\draw[knot] (0.175, 0.675) to[out=45, in=135] (0.325, 0.675);
					\draw[knot, gray] (0.5,-0.207) to[out=-90, in=-45, looseness = 2] (1,0);
					\draw[knot, gray] (0.5,1.207) to[out=90, in=135, looseness=2] (0,1);
					\draw[knot, gray] (0,0) to[out=-135, in=-135, looseness=2] (-0.3,1.3) to[out=45, in=45, looseness=2] (1,1);
				}
				$};
	\node(000) at (3,-1.5) {$
		\tikz[baseline={([yshift=-.5ex]current bounding box.center)}, scale=0.5, rotate=90]
		{
			\draw[dotted] (.5,.5) circle(0.707);
			\draw[knot] (0.5, 1.207) to[out=-90, in=45] (0.325, 0.825);
			\draw[knot] (0.175, 0.675) to[out=225, in=135] (0.175, 0.325);
			\draw[knot] (0.325, 0.175) to[out=-45, in=90] (0.5, -0.207);
			\draw[knot] (0,0) to (0.175, 0.175);
			\draw[knot] (0.325, 0.325) to (0.425, 0.425);
			\draw[knot] (0.575, 0.575) to (1,1);
			\draw[knot] (0,1) to (0.175, 0.825);
			\draw[knot] (0.325, 0.675) to (0.425, 0.575);
			\draw[knot] (0.575, 0.425) to (1,0);
			\draw[knot] (0.175, 0.175) to[out=45, in=-45] (0.175, 0.325);
			\draw[knot] (0.325, 0.175) to[out=135, in=-135] (0.325, 0.325);
			\draw[knot] (0.425, 0.425) to[out=45, in=-45] (0.425, 0.575); 
			\draw[knot] (0.575, 0.425) to[out=135, in=-135] (0.575, 0.575);
			\draw[knot] (0.175, 0.825) to[out=-45, in=-135] node(X)[pos=0.5]{} (0.325, 0.825);
			\draw[knot] (0.175, 0.675) to[out=45, in=135] node(Y)[pos=0.5]{} (0.325, 0.675);
			\draw[knot, gray] (0.5,-0.207) to[out=-90, in=-45, looseness = 2] (1,0);
			\draw[knot, gray] (0.5,1.207) to[out=90, in=135, looseness=2] (0,1);
			\draw[knot, gray] (0,0) to[out=-135, in=-135, looseness=2] (-0.3,1.3) to[out=45, in=45, looseness=2] (1,1);
			\draw[red, stealth-, thick, shorten >=.1mm, shorten <=.1mm] (X.east) -- (Y.west);
		}
		$};
			\node(211) at (0,-3) {$\tikz[baseline={([yshift=-.5ex]current bounding box.center)}, scale=0.5, rotate=180]
				{
					\draw[dotted] (.5,.5) circle(0.707);
					\draw[knot] (0,0) to[out=45, in=-90] (0.375,0.5) to[out=90, in=-90] (0.3, 0.782) to[out=90, in=180] (0.5, 0.915) to[out=0, in=180] (1.207, 0.5);
					\draw[knot] (0,1) to[out=-45,in=90] (0.2, 0.782) to[out=-90,in=0] (-0.207,0.5);
					\draw[knot, overcross] (1,0) to[out=135, in=-90] (0.625,0.5) to[out=90, in=-135] (1,1);
					\draw[knot, gray] (1,1) to[out=45, in=0, looseness = 2] (1.207, 0.5);
					\draw[knot, gray] (0,0) to[out=-135, in=180, looseness=2] (-0.207,0.5);
					\draw[knot, gray] (1,0) to[out=-45, in=-45, looseness=2] (1.3,1.3) to[out=135, in=135, looseness=2] (0,1);
				}$};
			\node(201) at (1,-3) {$\tikz[baseline={([yshift=-.5ex]current bounding box.center)}, scale=0.5, rotate=180]
				{
					\draw[dotted] (.5,.5) circle(0.707);
					\draw[knot] (1.207, 0.5) to[out=180, in=0] (0.5, 0.915) to[out=180, in=0] (0.218, 0.8) to[out=180, in=-45] (0,1);
					\draw[knot] (0,0) to[out=45, in=-90] (0.375,0.5) to[out=90, in=0] (0.218, 0.7) to[out=180, in=0] (-0.207, 0.5);
					\draw[knot, overcross] (1,0) to[out=135, in=-90] (0.625,0.5) to[out=90, in=-135] (1,1);
					\draw[knot, gray] (1,1) to[out=45, in=0, looseness = 2] (1.207, 0.5);
					\draw[knot, gray] (0,0) to[out=-135, in=180, looseness=2] (-0.207,0.5);
					\draw[knot, gray] (1,0) to[out=-45, in=-45, looseness=2] (1.3,1.3) to[out=135, in=135, looseness=2] (0,1);
				}$};
			\node(210) at (2,-4) {$\tikz[baseline={([yshift=-.5ex]current bounding box.center)}, scale=0.5, rotate=180]
				{
					\draw[dotted] (.5,.5) circle(0.707);
					\draw[knot] (1.207, 0.5) to[out=180, in=0] (0.5, 0.915) to[out=180, in=90] (0.3,0.782) to[out=-90, in=180] (0.5,0.625);
					\draw[knot] (0,1) to[out=-45, in=90] (0.2, 0.782) to[out=-90, in=0] (-0.207, 0.5);
					\draw[knot, overcross] (0,0) to[out=45, in=180] (0.5,0.375) to[out=0, in=135] (1,0);
					\draw[knot, overcross] (0.5,0.625) to[out=0, in=-135] (1,1);
					\draw[knot, gray] (1,1) to[out=45, in=0, looseness = 2] (1.207, 0.5);
					\draw[knot, gray] (0,0) to[out=-135, in=180, looseness=2] (-0.207,0.5);
					\draw[knot, gray] (1,0) to[out=-45, in=-45, looseness=2] (1.3,1.3) to[out=135, in=135, looseness=2] (0,1);
				}$};
			\node(200) at (3,-4) {$\tikz[baseline={([yshift=-.5ex]current bounding box.center)}, scale=0.5, rotate=180]
				{
					\draw[dotted] (.5,.5) circle(0.707);
					\draw[knot] (1.207, 0.5) to[out=180, in=0] (0.5, 0.915) to[out=180, in=0] (0.218, 0.8) to[out=180, in=-45] (0,1);
					\draw[knot, overcross] (0,0) to[out=45, in=180] (0.5,0.375) to[out=0, in=135] (1,0);
					\draw[knot] (-0.207,0.5) to[out=0, in=180] (0.218, 0.7) to[out=0, in=180] (0.5,0.625);
					\draw[knot, overcross] (0.5,0.625) to[out=0, in=-135] (1,1);
					\draw[knot, gray] (1,1) to[out=45, in=0, looseness = 2] (1.207, 0.5);
					\draw[knot, gray] (0,0) to[out=-135, in=180, looseness=2] (-0.207,0.5);
					\draw[knot, gray] (1,0) to[out=-45, in=-45, looseness=2] (1.3,1.3) to[out=135, in=135, looseness=2] (0,1);
				}$};
			\node(222) at (0.5,-6) {$
				\tikz[baseline={([yshift=-.5ex]current bounding box.center)}, scale=0.5]
				{
					\draw[dotted] (.5,.5) circle(0.707);
					\draw[knot] (-0.207, 0.5) to[out=0, in=180] (0.5, 0.085);
					\draw[knot] (1.207, 0.5) to[out=180, in=0] (0.5, 0.085);
					\draw[knot, overcross] (0,0) to[out=45, in=-90] (0.375,0.5) to[out=90, in=-45] (0,1);
					\draw[knot, overcross] (1,0) to[out=135, in=-90] (0.625,0.5) to[out=90, in=-135] (1,1);
					\draw[knot, gray] (1,1) to[out=45, in=0, looseness = 2] (1.207, 0.5);
					\draw[knot, gray] (0,0) to[out=-135, in=180, looseness=2] (-0.207,0.5);
					\draw[knot, gray] (1,0) to[out=-45, in=-45, looseness=2] (-0.3,-0.3) to[out=135, in=135, looseness=2] (0,1);
				}
				$};
			\node(220) at (2.5,-6) {$
				\tikz[baseline={([yshift=-.5ex]current bounding box.center)}, scale=0.5]
				{
					\draw[dotted] (.5,.5) circle(0.707);
					\draw[knot] (-0.207, 0.5) to[out=0, in=180] (0.5, 0.085);
					\draw[knot] (1.207, 0.5) to[out=180, in=0] (0.5, 0.085);
					\draw[knot, overcross] (0,0) to[out=45, in=180] (0.5,0.375) to[out=0, in=135] (1,0);
					\draw[knot, overcross] (0,1) to[out=-45, in=180] (0.5,0.625) to[out=0, in=-135] (1,1);
					\draw[knot, gray] (1,1) to[out=45, in=0, looseness = 2] (1.207, 0.5);
					\draw[knot, gray] (0,0) to[out=-135, in=180, looseness=2] (-0.207,0.5);
					\draw[knot, gray] (1,0) to[out=-45, in=-45, looseness=2] (-0.3,-0.3) to[out=135, in=135, looseness=2] (0,1);
				}
				$};
			\draw[->] (201) -- (200);
			\draw[->, thick, blue] (001) -- node[left, pos=0.4]{$\circled{3b}$}(201);
			\draw[white, ultra thick] (001) -- (200);
			\draw[->, thick, blue] (001) --node[left, pos=0.4]{$\circled{1b}$} (200);
			\draw[->, thick, blue] (201) -- node[left, pos=0.5]{$\circled{4b}$}(220);
			\draw[->, thick, blue] (200) --node[left, pos=0.25]{$\circled{2b}$} (220);
			\draw[->] (111) --  (110);
			\draw[->] (101) -- node[above, pos=0.4]{$\epsilon_1$}(100);
			\draw[white, ultra thick] (011) -- (010);
			\draw[->] (011) -- (010);
			\draw[->] (001) -- node[above, pos=0.2]{$\epsilon'\circled{5b}$}(000);
			\draw[->] (111) -- (101);
			\draw[->] (111) --(011);
			\draw[->] (011) -- (001);
			\draw[->] (101) --node[right, pos=0.4]{$\epsilon_2$} (001);
			\draw[->] (110) -- (100);
			\draw[->] (110) -- (010);
			\draw[->] (010) -- (000);
			\draw[->] (100) -- node[right, pos=0.4]{$\epsilon_5$}(000);
			\draw[->] (211) -- (201);
			\draw[white, ultra thick] (211) -- (210);
			\draw[white, ultra thick] (210) -- (200);
			\draw[->] (211) -- (210);
			\draw[->] (210) -- (200);
			\draw[->] (222) -- (220);
			\draw[->, thick, red] (100) to[out=-60, in=60] (200);
			\draw[white, ultra thick] (010) -- (210);
			\draw[->, thick, red] (010) -- (210);
			\draw[->, thick, red] (010) to[out=-45, in=90] (200);
			\draw[->, thick, red] (210) -- (220);
			\draw[->, thick, red] (200) to[out=-90,in=60] (220);
			\draw[->, ultra thick, white] (000) to[out=-60, in=45] (220);
			\draw[->, thick, PineGreen] (000) to[out=-60, in=45]node[right]{$\epsilon_6E_{-1}'$} (220);
		}
	}
	\caption{Computation for one closure}
	\label{fig:R3closure}
\end{figure}

\begin{remark}
We conclude by describing an alternative approach to Reidemeister III moves. We thank Matt Stoffregen again for sharing this idea.

First, decompose the Reidemeister III maps of Figure \ref{fig:R3CH1} further to obtain the diagram in Figure \ref{fig:R3further}. Then Lemma \ref{lemma:R3computation} implies that the composition of the three maps in the red rectangle of Figure \ref{fig:R3further} is chain homotopic to the identity map. Alternatively, one can prove this using an algebraic argument similar to the one outlined in the remark of Subsection \ref{ss:r2}.
	\begin{figure}[h]
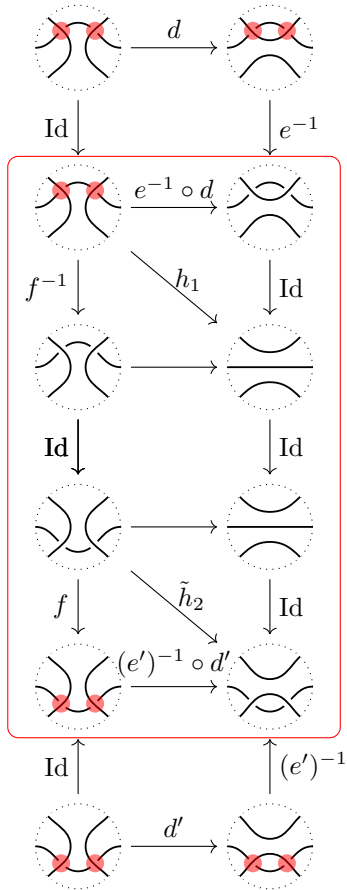

	\centering
	\tikz[baseline={([yshift=-.5ex]current bounding box.center)}, scale=0.85]
	{
		\node(A) at (3,0) 
		{$
			\tikz[baseline={([yshift=-.5ex]current bounding box.center)}, scale=0.8]
			{
				\draw[dotted] (.5,.5) circle(0.707);
				\draw[knot] (-0.207, 0.5) to[out=0, in=180] (0.5, 0.085);
				\draw[knot] (1.207, 0.5) to[out=180, in=0] (0.5, 0.085);
				\draw[knot, overcross] (0,0) to[out=45, in=-90] (0.375,0.5) to[out=90, in=-45] (0,1);
				\draw[knot, overcross] (1,0) to[out=135, in=-90] (0.625,0.5) to[out=90, in=-135] (1,1);
				\fill[red, fill opacity=0.5] (0.218, 0.218) circle(0.15);
				\fill[red, fill opacity=0.5] (0.782, 0.218) circle(0.15);
			}
			$};
		
		\node(B) at (6,0) 	{$
			\tikz[baseline={([yshift=-.5ex]current bounding box.center)}, scale=0.8]
			{
				\draw[dotted] (.5,.5) circle(0.707);
				\draw[knot] (-0.207, 0.5) to[out=0, in=180] (0.5, 0.085);
				\draw[knot] (1.207, 0.5) to[out=180, in=0] (0.5, 0.085);
				\draw[knot, overcross] (0,0) to[out=45, in=180] (0.5,0.375) to[out=0, in=135] (1,0);
				\draw[knot, overcross] (0,1) to[out=-45, in=180] (0.5,0.625) to[out=0, in=-135] (1,1);
				\fill[red, fill opacity=0.5] (0.218, 0.218) circle(0.15);
				\fill[red, fill opacity=0.5] (0.782, 0.218) circle(0.15);
			}
			$};
		
		\node(A1) at (3,2.5) 	{$
			\tikz[baseline={([yshift=-.5ex]current bounding box.center)}, scale=0.8]
			{
				\draw[dotted] (.5,.5) circle(0.707);
				\draw[knot] (-0.207, 0.5) to[out=0, in=180] (0.5, 0.085);
				\draw[knot] (1.207, 0.5) to[out=180, in=0] (0.5, 0.085);
				\draw[knot, overcross] (0,0) to[out=45, in=-90] (0.375,0.5) to[out=90, in=-45] (0,1);
				\draw[knot, overcross] (1,0) to[out=135, in=-90] (0.625,0.5) to[out=90, in=-135] (1,1);
				\fill[red, fill opacity=0.5] (0.218, 0.218) circle(0.15);
				\fill[red, fill opacity=0.5] (0.782, 0.218) circle(0.15);
			}
			$};
		\node(B1) at (6,2.5) {$
			\tikz[baseline={([yshift=-.5ex]current bounding box.center)}, scale=0.8]
			{
				\draw[dotted] (.5,.5) circle(0.707);
				\draw[knot] (-0.207, 0.5) to[out=0, in=180] (0.5, 0.085);
				\draw[knot] (1.207, 0.5) to[out=180, in=0] (0.5, 0.085);
				\draw[knot, overcross] (0,0) to[out=45, in=180] (0.5,0.375) to[out=0, in=135] (1,0);
				\draw[knot, overcross] (0,1) to[out=-45, in=180] (0.5,0.625) to[out=0, in=-135] (1,1);
			}
			$};
		\node(C) at (3,5) {$
			\tikz[baseline={([yshift=-.5ex]current bounding box.center)}, scale=0.8]
			{
				\draw[dotted] (.5,.5) circle(0.707);
				\draw[knot] (-0.207, 0.5) to[out=0, in=180] (0.5, 0.085);
				\draw[knot] (1.207, 0.5) to[out=180, in=0] (0.5, 0.085);
				\draw[knot, overcross] (0,0) to[out=45, in=-90] (0.375,0.5) to[out=90, in=-45] (0,1);
				\draw[knot, overcross] (1,0) to[out=135, in=-90] (0.625,0.5) to[out=90, in=-135] (1,1);
			}
			$};
		\node(D) at (6,5) {$
			\tikz[baseline={([yshift=-.5ex]current bounding box.center)}, scale=0.8]
			{
				\draw[dotted] (.5,.5) circle(0.707);
				\draw[knot] (-0.207, 0.5) -- (1.207, 0.5);
				\draw[knot, overcross] (0,0) to[out=45, in=180] (0.5,0.25) to[out=0, in=135] (1,0);
				\draw[knot, overcross] (0,1) to[out=-45, in=180] (0.5,0.75) to[out=0, in=-135] (1,1);
			}
			$};
		\node(E) at (3,7.5) {$
			\tikz[baseline={([yshift=-.5ex]current bounding box.center)}, scale=0.8]
			{
				\draw[dotted] (.5,.5) circle(0.707);
				\draw[knot] (-0.207, 0.5) to[out=0, in=180] (0.5, 0.915);
				\draw[knot] (1.207, 0.5) to[out=180, in=0] (0.5, 0.915);
				\draw[knot, overcross] (0,0) to[out=45, in=-90] (0.375,0.5) to[out=90, in=-45] (0,1);
				\draw[knot, overcross] (1,0) to[out=135, in=-90] (0.625,0.5) to[out=90, in=-135] (1,1);
			}
			$};
		\node(F) at (6,7.5) {$
			\tikz[baseline={([yshift=-.5ex]current bounding box.center)}, scale=0.8]
			{
				\draw[dotted] (.5,.5) circle(0.707);
				\draw[knot] (-0.207, 0.5) -- (1.207, 0.5);
				\draw[knot, overcross] (0,0) to[out=45, in=180] (0.5,0.25) to[out=0, in=135] (1,0);
				\draw[knot, overcross] (0,1) to[out=-45, in=180] (0.5,0.75) to[out=0, in=-135] (1,1);
			}
			$};
		\node(E1) at (3,10) {$
			\tikz[baseline={([yshift=-.5ex]current bounding box.center)}, scale=0.8]
			{
				\draw[dotted] (.5,.5) circle(0.707);
				\draw[knot] (-0.207, 0.5) to[out=0, in=180] (0.5, 0.915);
				\draw[knot] (1.207, 0.5) to[out=180, in=0] (0.5, 0.915);
				\draw[knot, overcross] (0,0) to[out=45, in=-90] (0.375,0.5) to[out=90, in=-45] (0,1);
				\draw[knot, overcross] (1,0) to[out=135, in=-90] (0.625,0.5) to[out=90, in=-135] (1,1);
				\fill[red, fill opacity=0.5] (0.782, 0.782) circle(0.15);
				\fill[red, fill opacity=0.5] (0.5-0.282, 0.782) circle(0.15);
			}
			$};
		\node(F1) at (6,10) {$
			\tikz[baseline={([yshift=-.5ex]current bounding box.center)}, scale=0.8]
			{
				\draw[dotted] (.5,.5) circle(0.707);
				\draw[knot] (-0.207, 0.5) to[out=0, in=180] (0.5, 0.915);
				\draw[knot] (1.207, 0.5) to[out=180, in=0] (0.5, 0.915);
				\draw[knot, overcross] (0,0) to[out=45, in=180] (0.5,0.375) to[out=0, in=135] (1,0);
				\draw[knot, overcross] (0,1) to[out=-45, in=180] (0.5,0.625) to[out=0, in=-135] (1,1);
			}
			$};
		\node(G) at (3,12.5) {$
			\tikz[baseline={([yshift=-.5ex]current bounding box.center)}, scale=0.8]
			{
				\draw[dotted] (.5,.5) circle(0.707);
				\draw[knot] (-0.207, 0.5) to[out=0, in=180] (0.5, 0.915);
				\draw[knot] (1.207, 0.5) to[out=180, in=0] (0.5, 0.915);
				\draw[knot, overcross] (0,0) to[out=45, in=-90] (0.375,0.5) to[out=90, in=-45] (0,1);
				\draw[knot, overcross] (1,0) to[out=135, in=-90] (0.625,0.5) to[out=90, in=-135] (1,1);
				\fill[red, fill opacity=0.5] (0.782, 0.782) circle(0.15);
				\fill[red, fill opacity=0.5] (0.5-0.282, 0.782) circle(0.15);
			}
			$};
		\node(H) at (6,12.5) {$
			\tikz[baseline={([yshift=-.5ex]current bounding box.center)}, scale=0.8]
			{
				\draw[dotted] (.5,.5) circle(0.707);
				\draw[knot] (-0.207, 0.5) to[out=0, in=180] (0.5, 0.915);
				\draw[knot] (1.207, 0.5) to[out=180, in=0] (0.5, 0.915);
				\draw[knot, overcross] (0,0) to[out=45, in=180] (0.5,0.375) to[out=0, in=135] (1,0);
				\draw[knot, overcross] (0,1) to[out=-45, in=180] (0.5,0.625) to[out=0, in=-135] (1,1);
				\fill[red, fill opacity=0.5] (0.782, 0.782) circle(0.15);
				\fill[red, fill opacity=0.5] (0.5-0.282, 0.782) circle(0.15);
			}
			$};

		\draw[->] (A) to node[pos=0.5, above]{$d'$} (B);
		\draw[->] (A1) to node[pos=0.5, above]{$(e')^{-1}\circ d'$} (B1);
		\draw[->] (C) to node[pos=0.5, above]{} (D);
		\draw[->] (E) to node[pos=0.5, above]{} (F);
		\draw[->] (E1) to node[pos=0.5, above]{$e^{-1}\circ d$} (F1);
		\draw[->] (G) to node[pos=0.5, above]{$d$} (H);
		\draw[->] (A) to node[pos=0.5, left]{$\Id$} (A1);
		\draw[->] (B) to node[pos=0.5, right]{$(e')^{-1}$} (B1);
		\draw[->] (C) to node[pos=0.5, left]{$f$} (A1);
		\draw[->] (D) to node[pos=0.5, right]{$\Id$} (B1);
		\draw[->] (E) to node[pos=0.5, left]{$\Id$} (C);
		\draw[->] (E) to node[pos=0.5, left]{$\Id$} (C);
		\draw[->] (E) to node[pos=0.5, left]{$\Id$} (C);
		\draw[->] (F) to node[pos=0.5, right]{$\Id$} (D);
		\draw[->] (G) to node[pos=0.5, left]{$\Id$} (E1);
		\draw[->] (H) to node[pos=0.5, right]{$e^{-1}$} (F1);
		\draw[->] (E1) to node[pos=0.5, left]{$f^{-1}$} (E);
		\draw[->] (F1) to node[pos=0.5, right]{$\Id$} (F);		
		\draw[->] (C) to node[pos=0.35, right=1mm]{$\tilde{h}_2$} (B1);
		\draw[->] (E1) to node[pos=0.4, right]{$h_1$} (F);
				\begin{scope}[on background layer]
			\draw[red,rounded corners]
			($(A1)+(-1.1,-0.8)$) rectangle ($(F1)+(1.1,0.8)$);
		\end{scope}
		
	}
	\caption{Further decomposition of Figure \ref{fig:R3CH1}}
	\label{fig:R3further}
\end{figure}
\end{remark}

\bibliographystyle{alpha}
\bibliography{OKH_and_2knots}

\end{document}